\documentclass[11pt,a4paper]{amsart}

\usepackage[english]{babel}
\usepackage[T1]{fontenc}
\usepackage[utf8]{inputenc}

\usepackage{geometry}
\usepackage[nodisplayskipstretch]{setspace}

\usepackage[alphabetic]{amsrefs} 
\DefineSimpleKey{bib}{primaryclass}{}
\DefineSimpleKey{bib}{archiveprefix}{}
\newcommand*{\bracketize}[1]{[#1]}
\BibSpec{arxiv}{%
  +{}{\PrintAuthors}{author}
  +{,}{ \textit}{title}
  +{}{ \parenthesize}{date}
  +{,}{ arXiv:}{eprint}
  +{}{ \bracketize}{primaryclass}
}
\usepackage[dvipsnames]{xcolor}
\usepackage{graphicx}
\usepackage{tikz}
\usetikzlibrary{positioning,arrows.meta,cd,babel}
\usepackage[percent]{overpic}
\usepackage[shortlabels]{enumitem}
\setlist[itemize]{topsep=3pt,itemsep=3pt}

\usepackage{amssymb}      
\usepackage{mathtools}    
\usepackage{mathrsfs}     
\usepackage{dsfont}       
\usepackage{xfrac}        

\definecolor{myRED}{HTML}{E6332A}
\definecolor{myGRAY}{HTML}{8E8E8E}
\definecolor{myMIDDGRAY}{HTML}{E5E5E5}
\definecolor{myMIDGRAY}{HTML}{E2E2E2}
\definecolor{myDARKGRAY}{HTML}{404647}
\definecolor{myLIGHTGRAY}{HTML}{F9F9F9}
\definecolor{myBLUE}{HTML}{03468F}
\definecolor{myPURPLE}{HTML}{8C54D0}
\definecolor{myGREEN}{HTML}{007355}
\definecolor{myYELLOW}{HTML}{FFD300}

\usepackage{hyperref}
\hypersetup{
  colorlinks=true,
  linkcolor=blue,
  urlcolor=blue,
  citecolor=black
}

\numberwithin{equation}{section}

\newtheorem{theorem}{Theorem}[section]
\newtheorem{lemma}[theorem]{Lemma}
\newtheorem{proposition}[theorem]{Proposition}
\newtheorem{corollary}[theorem]{Corollary}
\newtheorem{claim}[theorem]{Claim}

\newtheorem*{theorem*}{Theorem}
\newtheorem*{pushinglemma*}{Pushing Lemma}
\newtheorem*{lemmasep*}{Proposition 7.6}
\newtheorem*{definition*}{Definition}

\usepackage{alphalph}
\makeatletter
\newtheorem{thmx}{Theorem}

\makeatother

\theoremstyle{definition}
\newtheorem{definition}[theorem]{Definition}

\newtheorem{remark}[theorem]{Remark}

\newcommand{\mycomment}[1]{}
\newcommand\sbullet[1][.75]{\mathbin{\vcenter{\hbox{\scalebox{#1}{$\bullet$}}}}}

\newcommand{\sspc}{\hspace*{0.04cm}}
\newcommand{\spc}{\hspace*{0.08cm}}
\newcommand{\pp}{{\sspc\prime}}

\renewcommand{\O}{O}

\newcommand{\OO}{\mathcal O}
\newcommand{\OOstart}{\mathcal O_{\mathrm{in}}}
\newcommand{\OOend}{\mathcal O_{\mathrm{out}}}

\newcommand{\R}{\mathbb R}
\newcommand{\N}{\mathbb N}
\newcommand{\Z}{\mathbb Z}

\newcommand{\D}{\mathcal D}
\newcommand{\F}{\mathcal F}

\newcommand{\W}{\mathcal W}

\newcommand{\thinsubsetneq}{\mathrel{\scalebox{1}[.85]{$\subsetneq$}}}
\newcommand{\thinsupsetneq}{\mathrel{\scalebox{1}[.85]{$\supsetneq$}}}

\newcommand{\orb}{\mathrm{Orb}(f)}
\newcommand{\orbphi}{\mathrm{Orb}(\phi)}

\newcommand{\fb}{\partial_L C_\O}
\newcommand{\Ltop}{\partial_L^{\textup{top}\,}}
\newcommand{\Lbot}{\partial_L^{\textup{bot}\,}}
\newcommand{\Rtop}{\partial_R^{\textup{top}\,}}
\newcommand{\Rbot}{\partial_R^{\textup{bot}\,}}
\newcommand{\bb}{\partial_R C_\O}

\newcommand{\fasym}{\stackrel{+}{\sim}}
\newcommand{\Dasym}{\stackrel{D}{\sim}}
\newcommand{\notfasym}{\not\stackrel{+}{\sim}}
\newcommand{\basym}{\stackrel{-}{\sim}}
\newcommand{\notbasym}{\not\stackrel{-}{\sim}}

\title[Refined methods in Foliated Brouwer Theory]{Refined methods in Foliated Brouwer Theory}

\author[N. Schuback]{Nelson Schuback}
\address{Sorbonne Université and Université Paris Cité, CNRS, IMJ-PRG, F-75005 Paris, France}
\email{\href{mailto:nelson.schuback@imj-prg.fr}{nelson.schuback@imj-prg.fr}}
\thanks{This project has received funding from the European Union’s Horizon 2020 research and innovation programme under the Marie Skłodowska-Curie grant agreement No 945332. \includegraphics*[scale = 0.03]{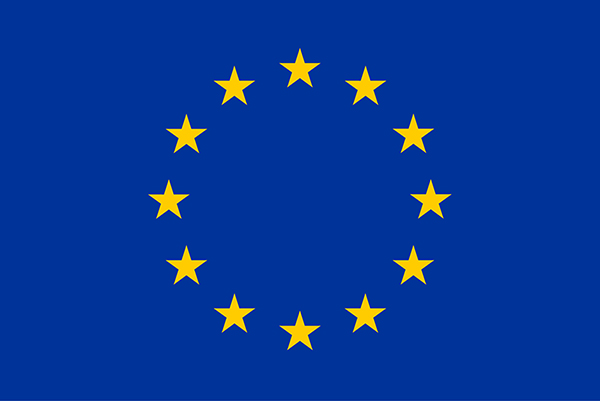}}

\begin{document}

\begin{abstract}
    A Brouwer homeomorphism is a fixed-point free, orientation-preserving homeomorphism of the plane. A foundational result of Le Calvez \cite{lec1} establishes that every such homeomorphism $f$ admits an oriented planar foliation $\F$ such that every point $x \in \R^2$ can be connected to its image $f(x)$ by a path positively transverse to $\F$. This provides a powerful framework for analyzing the dynamics of $f$ by studying how its orbits cross the leaves of $\mathcal{F}$.\break
    In this article, we refine this framework by identifying additional qualitative dynamical information about $f$ that is encoded in $\F$, which can be systematically recovered through the concept of proper transverse trajectories. Later, we investigate the possible combinatorial configurations of these proper trajectories for finite collections of orbits and characterize their simplest forms. As a key application, this refined framework is used in a forthcoming work to offer a new perspective on Homotopy Brouwer Theory, originally introduced by Handel \cite{handel99}.
\end{abstract}

\maketitle


\vspace*{0.42cm}

\tableofcontents

\setlength{\baselineskip}{1.25\baselineskip}
\newpage
\section{Introduction}

Let \(f:\mathbb{R}^2 \longrightarrow \mathbb{R}^2\) be an orientation-preserving and fixed-point free homeomorphism, often referred to as a \textit{Brouwer homeomorphism}. Such maps are named after L. E. J. Brouwer, which first studied them in the early twentieth century. In his seminal work \cite{Brouwer}, Brouwer proved the Brouwer Translation Theorem (BTT), which asserts that through every point of the plane passes a \textit{Brouwer line} $\lambda$ of $f$, that is, a proper embedding $\lambda: \R \longrightarrow \R^2$ such that
$$f(L(\lambda)) \subset L(\lambda)\quad \text{ and } \quad f^{-1}(R(\lambda)) \subset R(\lambda).$$
Recall that, according to the Jordan curve theorem, $\lambda$ admits a left side and a right side, denoted by \(L(\lambda)\) and \(R(\lambda)\), given by the connected components of \(\R^2\setminus \lambda\) lying to the left and right of \(\lambda\).

In particular, the BTT implies that the plane can be covered by unbounded simply connected open $f$-invariant subsets, on which \(f\) is conjugate to a translation. This means that every point is wandering under $f$ and, thus, the limit set of every orbit of $f$ is empty. In other words, every orbit of $f$ is the image of a proper embedding of $\Z$ into $\R^2$. 
This fundamental description of the dynamics of Brouwer homeomorphisms have served as basis for the development of Brouwer Theory in the twentieth century (see \cite{GUILLOU1994331} for a historic account). 

In the early 2000s, Le Calvez introduced a foliated version of the BTT \cite{lec1} (and its equivariant form in \cite{lec2}), thereby initiating the so-called Foliated Brouwer Theory, a powerful framework for studying Brouwer homeomorphisms.

\begin{theorem*}[Le Calvez \cite{lec1}]
There exists an oriented topological foliation $\F$ of the plane such that every leaf $\phi \in \F$ is a Brouwer line of $f$.
\end{theorem*}

Any foliation $\F$ of the plane by Brouwer lines of $f$ is called a \textit{transverse foliation} of $f$.
Although such a transverse foliation is not unique, its existence already yields a robust structure for analyzing the dynamics of $f$.
Indeed, one may connect any point \(x \in \mathbb{R}^2\) to its image $f(x)$ by a path \(\gamma_x : [0,1] \longrightarrow \mathbb{R}^2\) that is positively transverse to \(\F\), in the sense that \(\gamma_x\) intersects any leaf $\phi\in\F$ at most once, always passing from $R(\phi)$ to $L(\phi)$. This allows one to represent the orbit $\O = \O(f,x)$ of $x$ under $f$ by an embedded line $\Gamma_\O:\R^2\longrightarrow \R^2$ given by the concatenation
\[
\Gamma_\O := \prod_{n \in \mathbb{Z}} \gamma_{f^n(x)}.
\]
We refer to such $\Gamma_\O$ as a \textit{transverse trajectory} of the orbit $\O$, with respect to the foliation $\F$.
Since the choice of paths $\gamma_{f^n(x)}$ is not unique, the orbit $\O$ admits uncountably many transverse trajectories $\Gamma_\O$ with respect to the same foliation $\F$, many of which not being properly embedded.\break
Nevertheless, all such transverse trajectories $\Gamma_\O$ intersect exactly the same set of leaves of $\F$.
Although this approach loses part of the qualitative dynamical information of the orbit, it has proved remarkably effective (particularly in its equivariant form) for reproducing classical results in surface dynamics and for establishing new powerful criteria for detecting chaotic behavior in surface homeomorphisms isotopic to the identity, through the celebrated Forcing Theory.

There are some problems, however, that cannot be addressed using the current methods of Foliated Brouwer Theory. For instance, when tackling the study of Brouwer mapping classes in Homotopy Brouwer Theory (see \cite{handel99}, \cite{LEROUX_2017}, \cite{BAVARD_2017}), it is important to know whether the transverse trajectories $\{\Gamma_\O\}_{\O \in \OO}$ of a finite collection $\OO$ of orbits of $f$ can be chosen to be properly-embedded and with minimal pairwise intersections.

In this work, we refine the methods of Foliated Brouwer Theory to address such problems, centered around future applications to the study of Homotopy Brouwer Theory. 

\pagebreak




\newpage
 
\subsection*{Description of the results}
Let us consider $f:\R^2\longrightarrow\R^2$ to be a Brouwer homeomorphism, and \(\F\) to be a transverse foliation of \(f\). The set of all orbits of \(f\) is denoted by $\orb$.

To each orbit $\O \in \orb$, we associate the set of leaves $C_\O \subset \F$ crossed by $\O$, defined as
$$C_\O = \{\phi \in \F \mid \phi \cap \Gamma_\O \neq \varnothing\},$$ where $\Gamma_\O$ represent any transverse trajetory of $\O$. This set is independent of the choice of $\Gamma_\O$. 

By abuse of notation, one may treat $C_\O$ as a subset of $\R^2$ by identifying $C_\O = \bigcup_{\phi\in C_\O}\phi$.  In Section~\ref{sec:prelim_foliations}, we show that $C_\O\subset\R^2$ is homeomorphic to a plane and trivially foliated by $\F$. Moreover, the boundary $\partial C_\O$ is formed by a locally finite collection of leaves of $\F$. At last,  we define two distinguished subsets $\partial_L C_\O$ and $\partial_R C_\O$ of the boundary $\partial C_\O$, given by 
$$ \partial_L C_\O := \partial C_\O \cap \Bigl(\ \bigcap_{\sspc\phi \in C_\O\sspc} L(\phi)\Bigr) \quad \text{ and } \quad \partial_R C_\O := \partial C_\O \cap \biggl(\ \bigcap_{\sspc\phi \in C_\O\sspc} R(\phi)\biggr)$$
and a total order $\prec$ among the leaves in $\partial_L C_\O$, and similarly for $\partial_R C_\O$ (see Section~\ref{sec:order_limit_leaves}).

\begin{figure}[h!]
    \center
    \vspace*{-0.15cm}\begin{overpic}[width=8cm, height=3.7cm, tics=10]{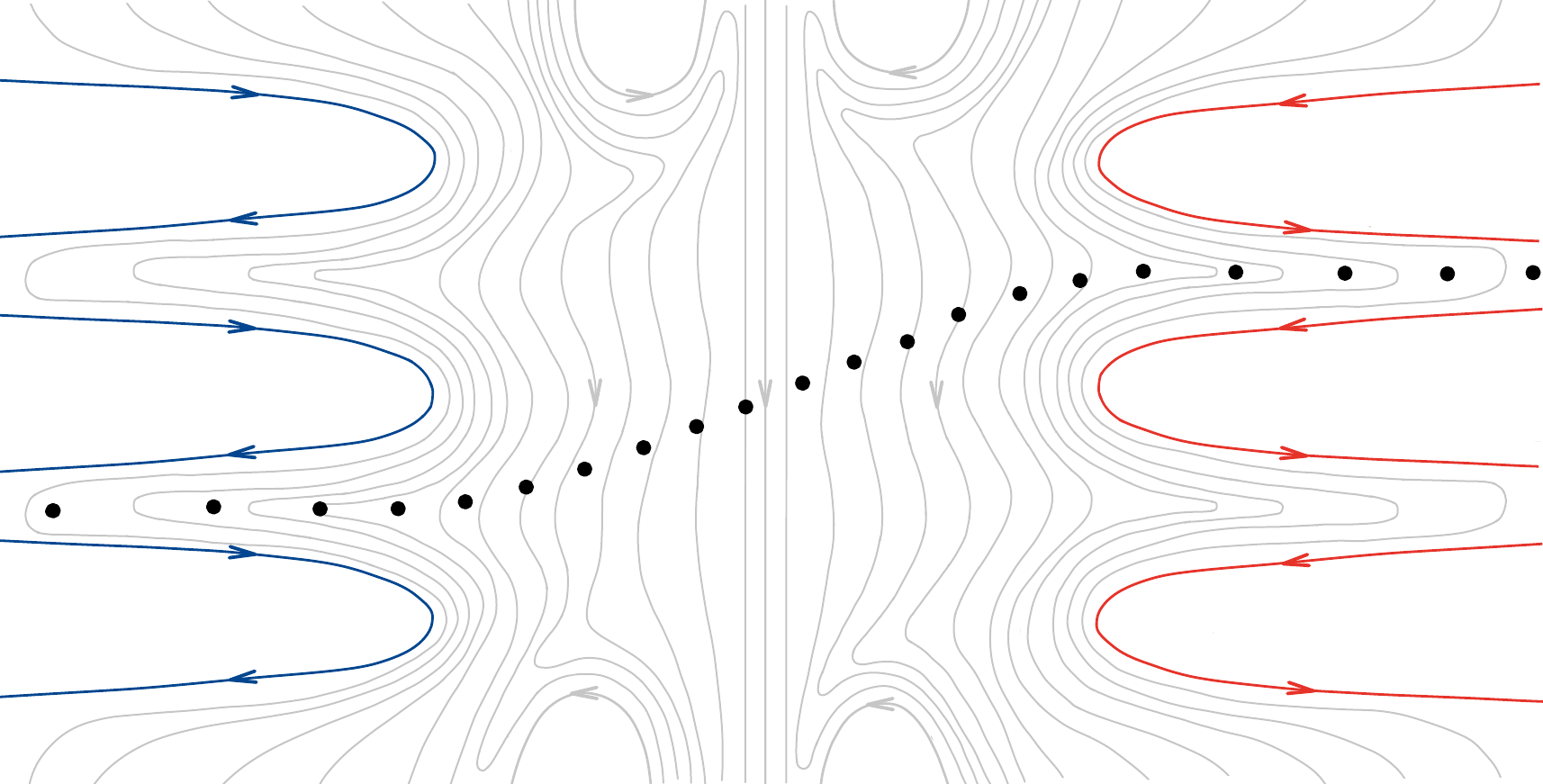}
        \put (54.4,15) {\colorbox{white}{$\rule{0cm}{0.3cm}\quad $}}
        \put (55,15.3) {{\color{myGRAY}\large$\displaystyle C_\O $}}
        \put (41.5,24) {\colorbox{white}{\color{myDARKGRAY}\large$\rule{0cm}{0.26cm}\  $}}
         \put (41.2,24) {{\color{black}\large$\displaystyle \O$}}
        \put (3,22) {\color{myBLUE}\large$\displaystyle \bb$}
        \put (85,22) {\color{myRED}\large$\displaystyle \partial_L C_\O$}
\end{overpic}
\end{figure}

 \vspace*{-0.2cm}
Let $\text{Cut}_\prec(\partial_L C_\O)$ denote the set of all cuts of the totally ordered set $(\partial_L C_\O,\prec)$, that is, the set of all pairs $(A,B)$ such that $A\sspc,\sspc B \subset \partial_L C_\O$ are disjoint subsets of $\partial_L C_\O$ satisfying
\begin{itemize}[leftmargin = 1.5cm]
    \item $A \cup B = \partial_L C_\O$,
    \item For any leaf $\phi \in A$ and any leaf $\phi' \in B$, it holds that $\phi \prec \phi'$.
\end{itemize}   
Similarly, let $\text{Cut}_\prec(\partial_R C_\O)$ denote the set of all cuts of the totally ordered set $(\partial_R C_\O,\prec)$.

\begin{definition}
    A proper transverse trajectory $\Gamma_\O$ of an orbit $\O \in \orb$ is a properly embedded path $\Gamma_\O:\R\longrightarrow \R^2$ positively transverse to $\F$ passing through every point of $\O$.
\end{definition}

The following result shows that every orbit admits a proper transverse trajectory and that any such trajectories carry more qualitative information about the orbit than non-proper trajectories.

\begin{thmx}\label{theorem:1A-intro}
    Every orbit $\O \in \orb$ is canonically associated to two cuts
    \begin{align*}
    (\sspc\partial_L^\text{\sspc top\sspc } \O\sspc,\spc\partial_L^\text{\sspc bot\sspc } \O\sspc) \in \textup{Cut}_\prec(\partial_L C_\O) \quad \ \text{ and } \ \quad 
    (\sspc \partial_R^\text{\sspc top\sspc } \O\sspc,\spc\partial_R^\text{\sspc bot\sspc } \O\sspc) \in \textup{Cut}_\prec(\partial_R C_\O).
\end{align*}
    In addition, the orbit $\O$ admits a proper transverse trajectory $\Gamma_\O$, and it must satisfy\begin{align*}
    &\Ltop \O = \partial_L C_\O \cap L(\Gamma_\O)\spc, \spc \quad \Rtop \O = \partial_L C_\O \cap L(\Gamma_\O), \\[0.1ex]  &\Lbot \O = \partial_R C_\O \cap R(\Gamma_\O)\spc, \quad \Rbot \O = \partial_R C_\O \cap R(\Gamma_\O).
\end{align*}

\vspace*{-0.25cm}
\begin{figure}[h!]
    \center
	\begin{overpic}[width=8cm, height=3.7cm, tics=10]{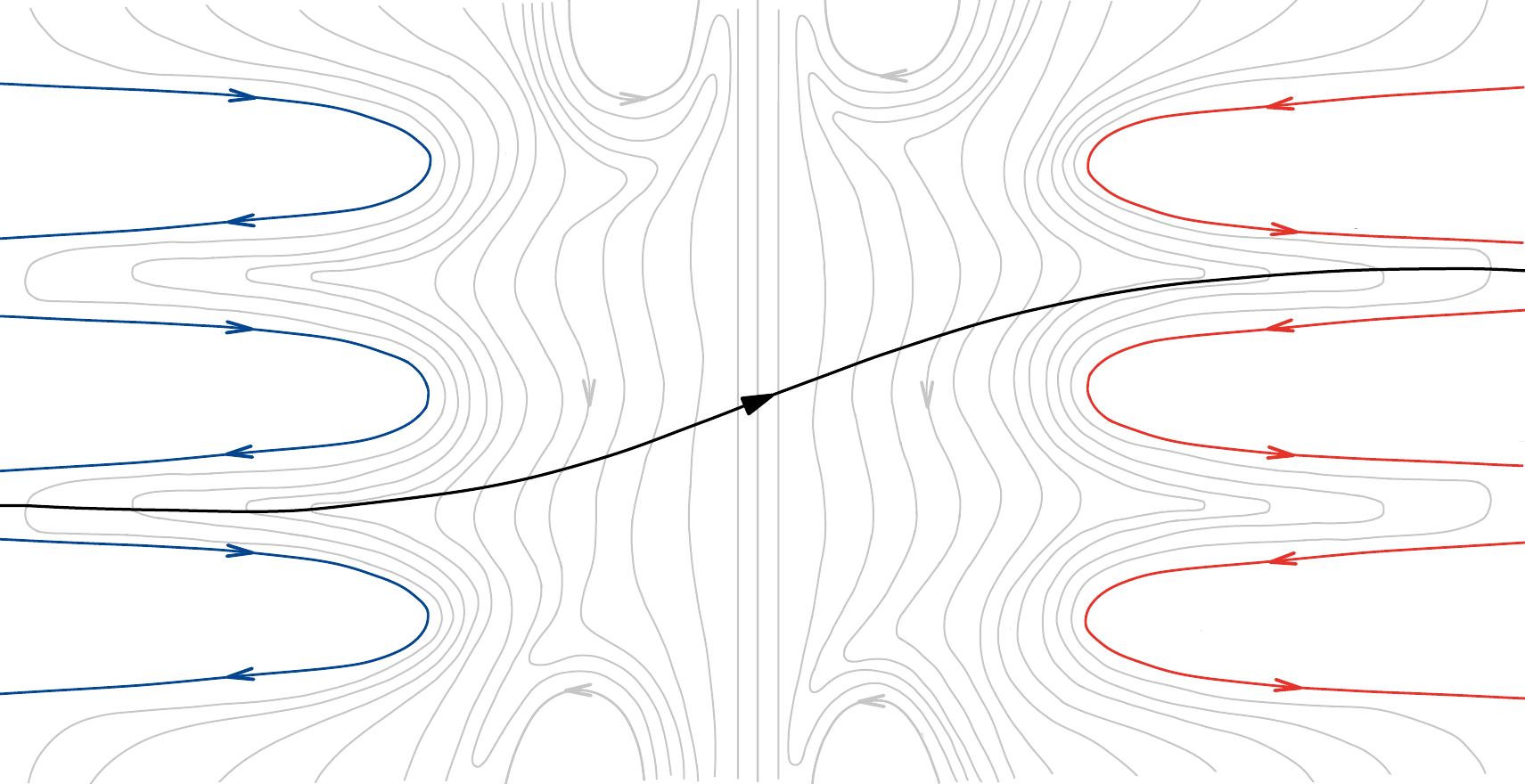}
        \put (54.4,15) {\colorbox{white}{$\rule{0cm}{0.3cm}\ \ \  $}}
        \put (55.5,15.3) {{\color{myGRAY}\large$\displaystyle C_\O $}}
        \put (41.5,25) {\colorbox{white}{\color{myDARKGRAY}\large$\rule{0cm}{0.3cm}\ \spc $}}
        \put (41.8,25) {{\color{black}\large$\displaystyle \Gamma_\O$}}
         \put (85,21.8) {\color{myRED}\large$\displaystyle \Lbot \O $}
        \put (85,35) {\color{myRED}\large$\displaystyle \Ltop \O $}
        \put (85,8) {\color{myRED}\large$\displaystyle \Lbot \O $}
        \put (3,21.8) {\color{myBLUE}\large$\displaystyle \Rtop \O $}
        \put (3,35.5) {\color{myBLUE}\large$\displaystyle \Rtop \O $}
        \put (3,8.2) {\color{myBLUE}\large$\displaystyle \Rbot \O $}
\end{overpic}
\end{figure}
\end{thmx}

\vspace*{-0.9cm}
\newpage

As a consequence of Theorem~\ref{theorem:1A-intro}, we obtain a finer description of the behavior of orbits in terms of the structure of the transverse foliation \(\F\), given by the set $C_\O$ and the associated cuts $(\sspc\partial_L^\text{\sspc top\sspc } \O\sspc,\spc\partial_L^\text{\sspc bot\sspc } \O\sspc)$ and $(\sspc \partial_R^\text{\sspc top\sspc } \O\sspc,\spc\partial_R^\text{\sspc bot\sspc } \O\sspc)$, or dually, by its proper transverse trajectories. This qualitative data is referred throughout the text as the \emph{$\F$-asymptotic behavior} of $\O$. 

Being able to describe the \(\F\)-asymptotic behavior refines the framework, for instance by allowing us to distinguish orbits of $f$ even when they cross the same leaves of \(\F\), something not achievable before. This is illustrated in the following example.

\begin{figure}[h!]
    \center
    \hspace*{-5cm}\begin{overpic}[width=7.3cm, height=3.6cm, tics=10]{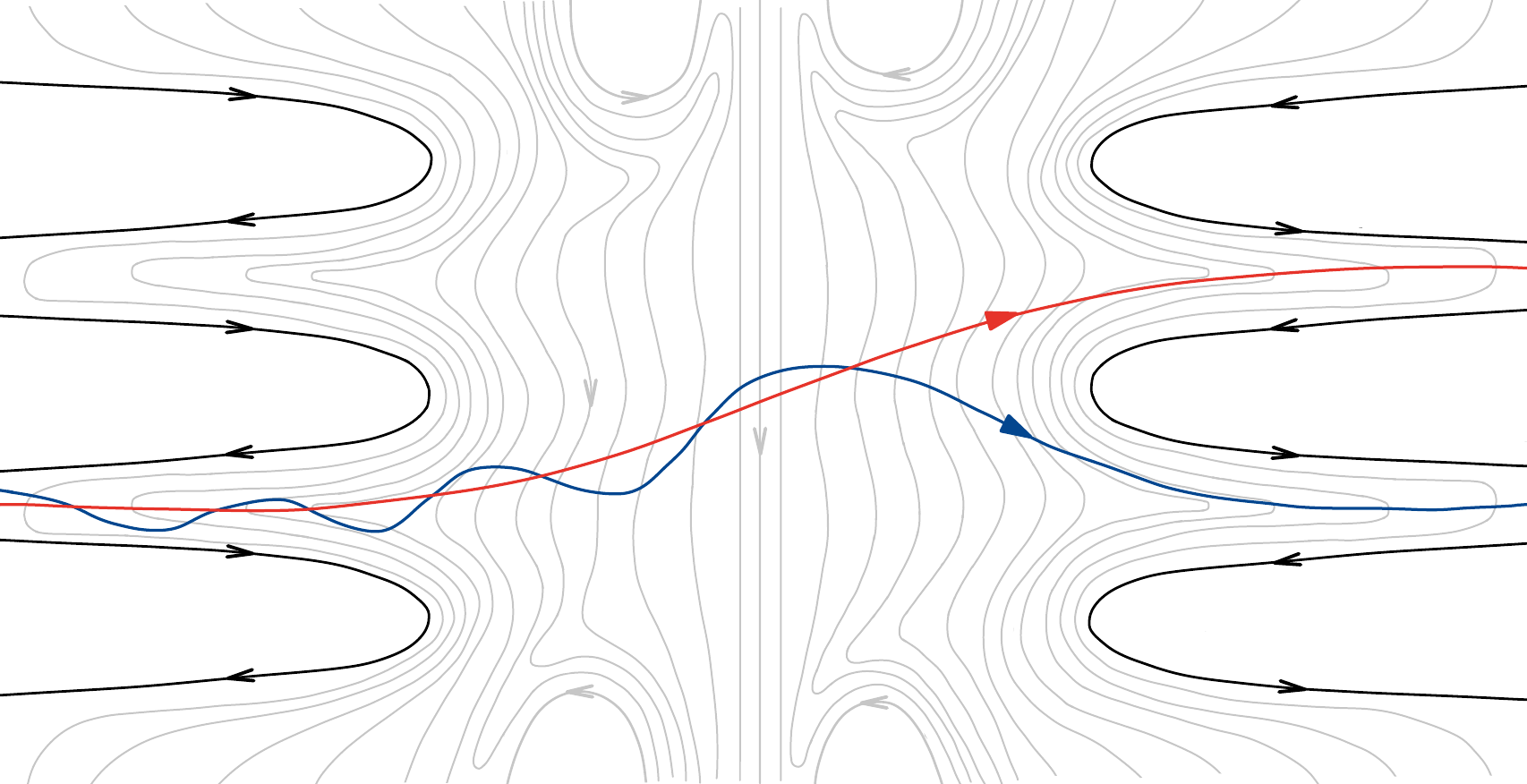}
        \put (58,15) {\colorbox{white}{$\rule{0cm}{0.3cm}\ \  $}}
        \put (58.3,15.3) {{\color{myBLUE}\large$\displaystyle \Gamma_{\O^\pp} $}}
        \put (57.6,32) {\colorbox{white}{\color{myDARKGRAY}\large$\rule{0cm}{0.3cm}\   \  $}}
        \put (58.3,31.6) {{\color{myRED}\large$\displaystyle \Gamma_\O$}}
        \put (129.6,42.5) {{\color{black}\normalsize$\displaystyle C_\O$}{\color{black}\normalsize$\displaystyle \ =\ $}{\color{black}\normalsize$\displaystyle C_{\O^\pp}$}}
        \put (123,33) {{\color{black}\normalsize$\displaystyle \Rtop \O$}{\color{black}\normalsize$\displaystyle \ =\ $}{\color{black}\normalsize$\displaystyle \Rtop {\O^\pp}$}}
        \put (123,23) {{\color{black}\normalsize$\displaystyle \Rbot \O$}{\color{black}\normalsize$\displaystyle \ =\ $}{\color{black}\normalsize$\displaystyle \Rbot {\O^\pp}$}}
        \put (123,13) {{\color{black}\normalsize$\displaystyle \Ltop \O$}{\color{black}\normalsize$\displaystyle \ \thinsubsetneq\ $}{\color{black}\normalsize$\displaystyle \Ltop {\O^\pp}$}}
        \put (123,3) {{\color{black}\normalsize$\displaystyle \Lbot \O$}{\color{black}\normalsize$\displaystyle \ \thinsupsetneq \ $}{\color{black}\normalsize$\displaystyle \Lbot {\O^\pp}$}}

\end{overpic}
\end{figure}

\vspace*{-0.3cm}
\begin{definition} 
    Two orbits $\O, \O^{\sspc \prime} \in \orb$ are called \emph{forward $\F$-asymptotic}, denoted $\O \fasym \O^\pp$, if there exists a leaf $\phi \in C_\O \cap C_{\O^\pp}$ such that $L(\phi) \cap C_\O = L(\phi) \cap C_{\O^\pp}$,
and, in addition, 
$$ \Ltop \O = \Ltop \O^\pp \quad \text{ and } \quad \Lbot \O = \Lbot \O^\pp.$$
    Similarly, we define the \emph{backward $\F$-asymptotic} relation, denoted $\O \basym \O^\pp$, by replacing $L$ for $R$.
\end{definition}

 
By comparing the different $\F$-asymptotic behaviors allowed by the underlying foliation \(\F\),  we obtain two relations $\lesssim_L$ and $\lesssim_R$ on the set $\orb$, which describes of how the orbits are arranged with respect to each other. For a definition of these relations, we refer to Section~\ref{sec:preorders}. Most importantly, for each $\phi \in \F$, these relations restrict to total preorders on the subset 
$$\orbphi := \{\O \in \orb \mid \phi \in C_\O\}.$$
Moreover, any two orbits $\O,\O^\pp \in \orb$ satisfy the following equivalences:
\begin{align*}
         (\O \lesssim_L \O^{\sspc\prime} \ \text{ and }\  \O^{\sspc\prime} \lesssim_L \O) \iff \O\fasym \O^{\sspc\prime} , \\
         (\O \lesssim_R \O^{\sspc\prime} \ \text{ and }\  \O^{\sspc\prime} \lesssim_R \O) \iff \O\basym \O^{\sspc\prime}.
\end{align*}

\begin{thmx}\label{theorem:1B-intro}
    Let $\phi \in \F$, and let $\OO\subset \orbphi$ be a finite collection of orbits crossing $\phi$.  Assume that \(\F\) is generic for \(\OO\), i.e., no leaf in $\F$ intersects more than one orbit in \(\OO\). Then, there exists a family $\bigl\{\Gamma_{\O}\bigr\}_{\O \in \OO}$ of proper transverse trajectories of the orbits in $\OO$ that satisfy
    \mycomment{-0.12cm}
    $$\Gamma_{\O} \cap \Gamma_{\O^\pp} \cap \overline{\rule{0pt}{3.6mm}L(\phi)} = \varnothing, \quad \forall \sspc \O,\O^\pp \in \OO,\  \O \neq \O^\pp.$$

    \mycomment{-0.22cm}
    \noindent Moreover, the order at which the intersection points $p_\O \in \Gamma_{\O} \cap \phi$ appear along the leaf $\phi$ is compatible with the relation $\lesssim_L$, in the sense that $p_\O < p_{\O^\pp}$ along the leaf $\phi$ only if $\O \lesssim_L \O^\pp$. 
\end{thmx}


\vspace*{-0.4cm}

\begin{figure}[h!]
    \center\begin{overpic}[width=7.2cm, height=3.6cm, tics=10]{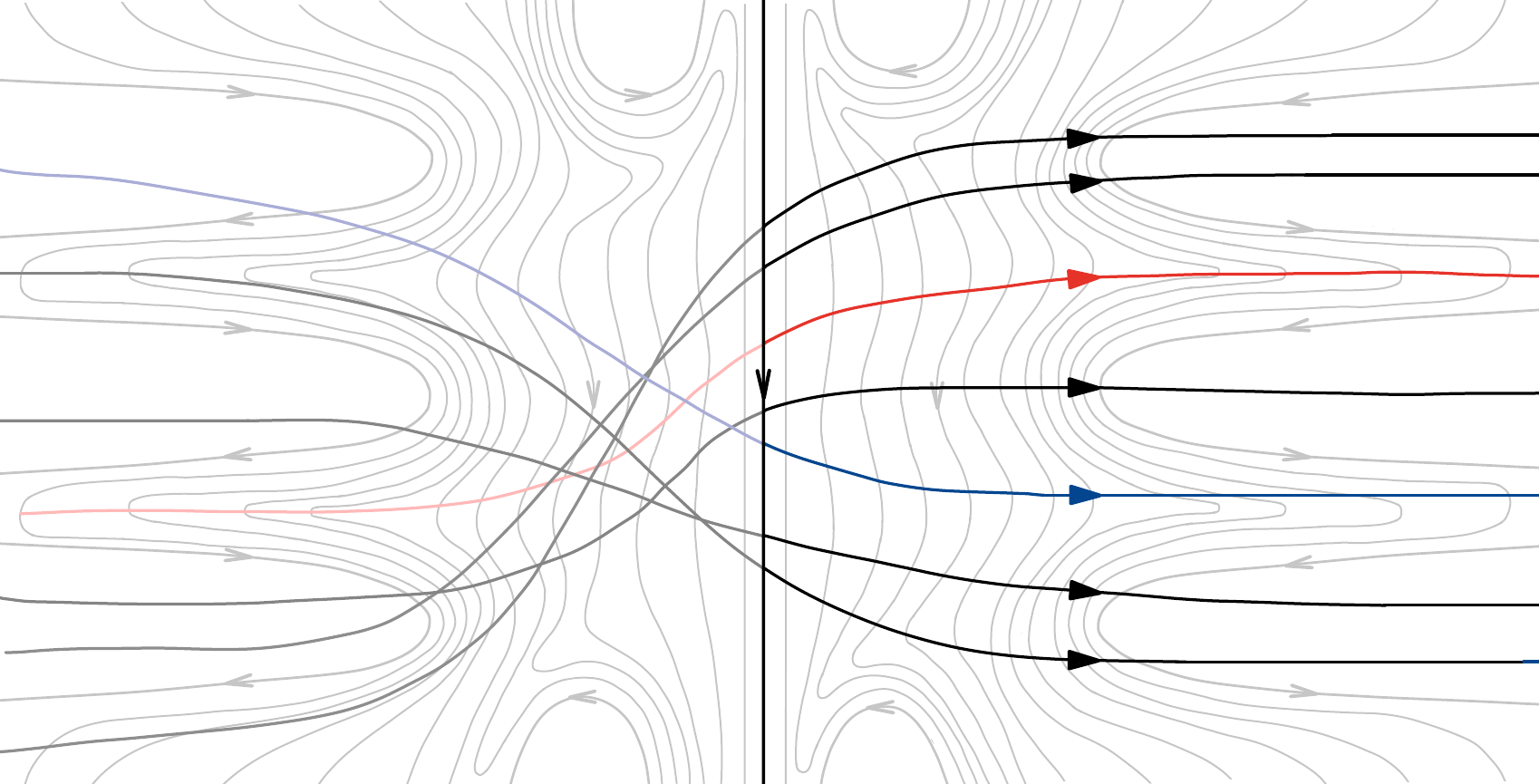}
        \put (44.5,-6) {\colorbox{white}{\color{black}\large$\displaystyle\  \phi \spc$}}
         \put (103.5,15.5) {{\color{myBLUE}\large$\displaystyle \Gamma_{\O^\pp} $}}
         \put (103.5,30) {{\color{myRED}\large$\displaystyle \Gamma_\O$}}
         \put (-25,20) {{\color{myRED}\large$\displaystyle \O \ $}{\color{black}\large$\displaystyle \lesssim_L$}{\color{myBLUE}\large $\spc\displaystyle \O^\pp$}}
\end{overpic}
\end{figure}

\vspace*{0.1cm}
A similar result holds if we replace $\overline{\rule{0pt}{3.6mm}L(\phi)}$ and $\lesssim_L$ for their right analogues $\overline{\rule{0pt}{3.6mm}R(\phi)}$ and $\lesssim_R$. 

\newpage

Finally, we introduce the notion of \emph{weak $\F$-transverse intersection} between two orbits, which generalizes the notion of $\F$-transverse intersection defined by Le Calvez and Tal in \cite{LCTal2018}. Roughly speaking, two orbits $\O, \O^\pp \in \orb$ have a weak $\F$-transverse intersection if they are ordered by $\lesssim_L$ and $\lesssim_R$ exclusively in opposite ways.

\begin{definition}\label{def:weak_transverse_intersection}
    Two orbits $\O, \O^{\sspc \prime} \in \orb$ are said to have a \textit{weak $\F$-transverse intersection} if they satisfy $\O\not\fasym \O^\pp$ and $\O\not\basym \O^\pp$, and, in addition, they are oppositely ordered by $\lesssim_L$ and $\lesssim_R\sspc$, meaning that either $\O \lesssim_L \O^{\sspc \prime}$ and $\O^{\sspc \prime} \lesssim_R \O$, or $\O^{\sspc \prime} \lesssim_L \O$ and $\O \lesssim_R \O^{\sspc \prime}$.
\end{definition}

In the figure below, we compare the classic and weak notions of $\F$-transverse intersection.

\vspace*{0.1cm}
\begin{figure}[h!]
    \center
    \vspace*{-0.1cm}
    \begin{overpic}[width=6.6cm, height=3.7cm, tics=10]{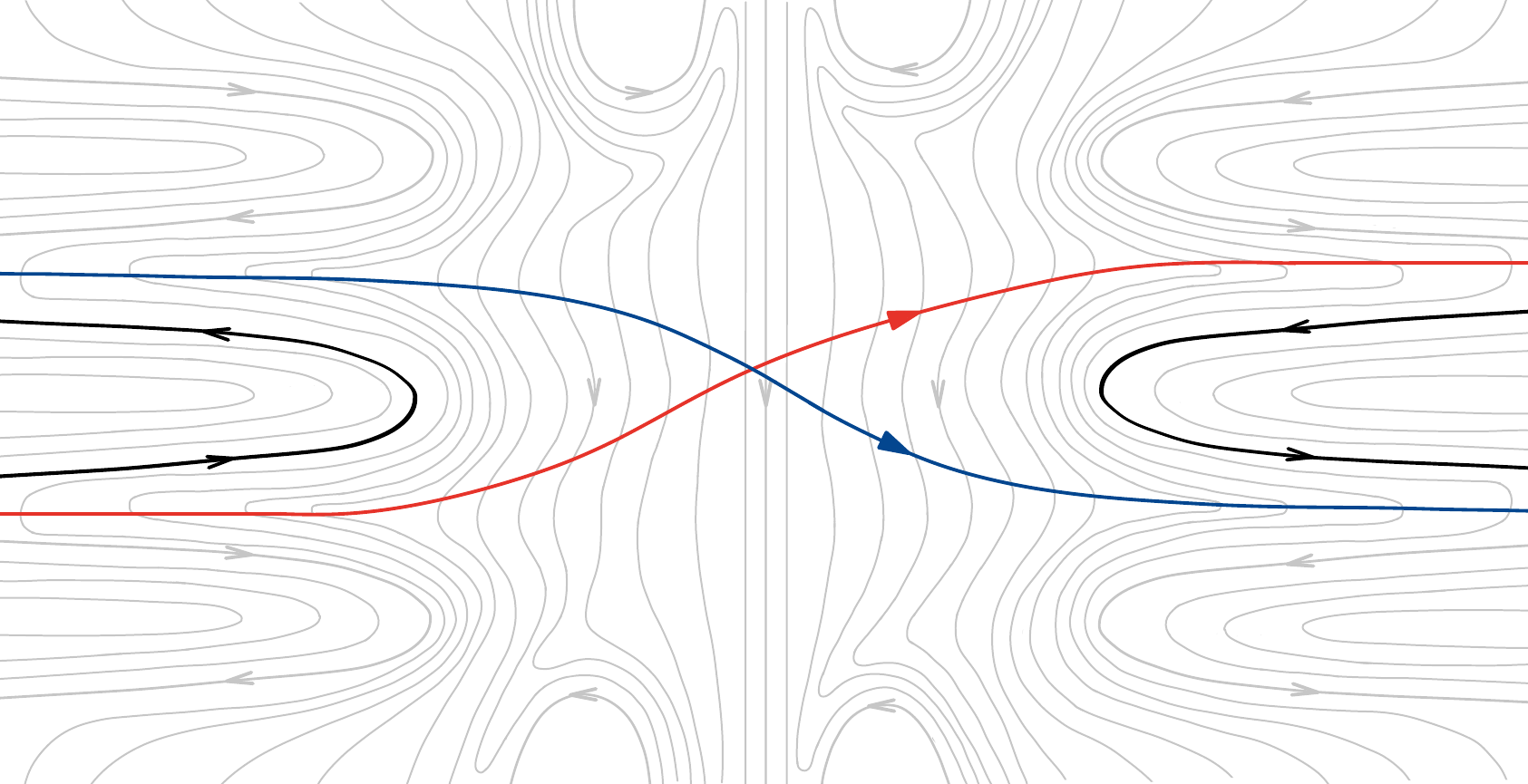}
        \put (7.5,-10) {\large weak $\F$-transverse intersection}
\end{overpic}
\hspace{1cm}
\begin{overpic}[width=6.6cm, height=3.7cm, tics=10]{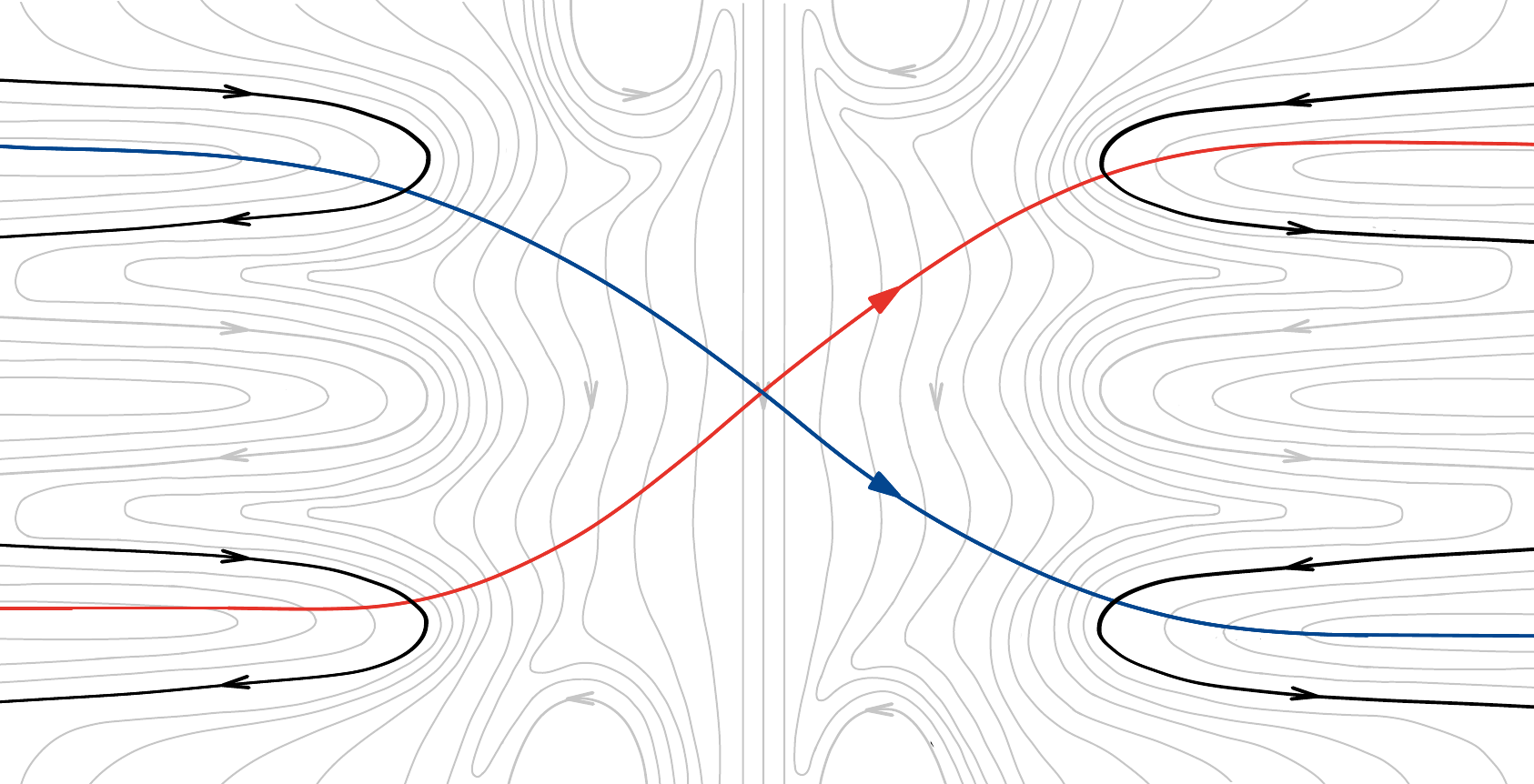}
        \put (15,-10) {\large $\F$-transverse intersection}
\end{overpic}
\end{figure}
\vspace*{0.65cm}
Finally, we conclude by presenting a structural result, which is a key step towards future applications of Foliated Brouwer Theory to the study of Brouwer mapping classes. 

Before stating Theorem \ref{thm:1C-alternative}, we need to introduce some therminology.
Denote the open unit disk by $\mathbb D^2 = \{(x,y) \in \R^2 \mid x^2 + y^2 < 1\}$. 
A foliation of $\mathbb D^2$ is said to be \emph{nicely extended to the boundary} if each leaf of the foliation can be extended to a simple arc in the closed disk $\overline{\mathbb D^2}$\break with distinct endpoints on the boundary $\partial \mathbb D^2$. For any foliation $\F$ of the plane, and any homeomorphism $\varphi:\R^2 \longrightarrow \mathbb D^2$, we denote by $\varphi(\F)$ the foliation of $\mathbb D^2$ formed by the images of the leaves of $\F$ under $\varphi$. 
Finally, a \emph{diagram} is a finite collection of line segments in $\mathbb D^2$ such that with pairwise distinct endpoints on the boundary $\partial \mathbb D^2$. 


\begin{thmx}\label{thm:1C-alternative}
    Let $\OO\subset \orb$ be a finite collection of orbits, and assume \(\F\) is generic for $\OO$. Then,
    there exists a family $\{\Gamma_\O\}_{\O \in \OO}$ of proper transverse trajectories for $\OO$, and a homeomorphism $\varphi:\R^2 \longrightarrow \mathbb D^2$ that satisfy the following properties:
    \begin{itemize}[leftmargin=1.3cm]
        \item[\textup{\textbf{(i)}\ \sspc }]  The foliation $\varphi(\F)$ is nicely extended to the boundary.
        \item[\textup{\textbf{(ii)}}\spc ]  The family $\{\varphi(\Gamma_\O)\}_{\O \in \OO}$ is a diagram in $\mathbb D^2$.
        \item[\textup{\textbf{(iii)}}]  Two distinct segments $\varphi(\Gamma_\O)$ and $\varphi(\Gamma_{\O^\pp})$ intersect if, and only if, the orbits $\O$ and $\O^\pp$ have a weak $\F$-transverse intersection.
    \end{itemize}
\end{thmx}

\vspace*{0cm}
\begin{figure}[h!]
    \center\begin{overpic}[width=4.8cm, height=4.8cm, tics=10]{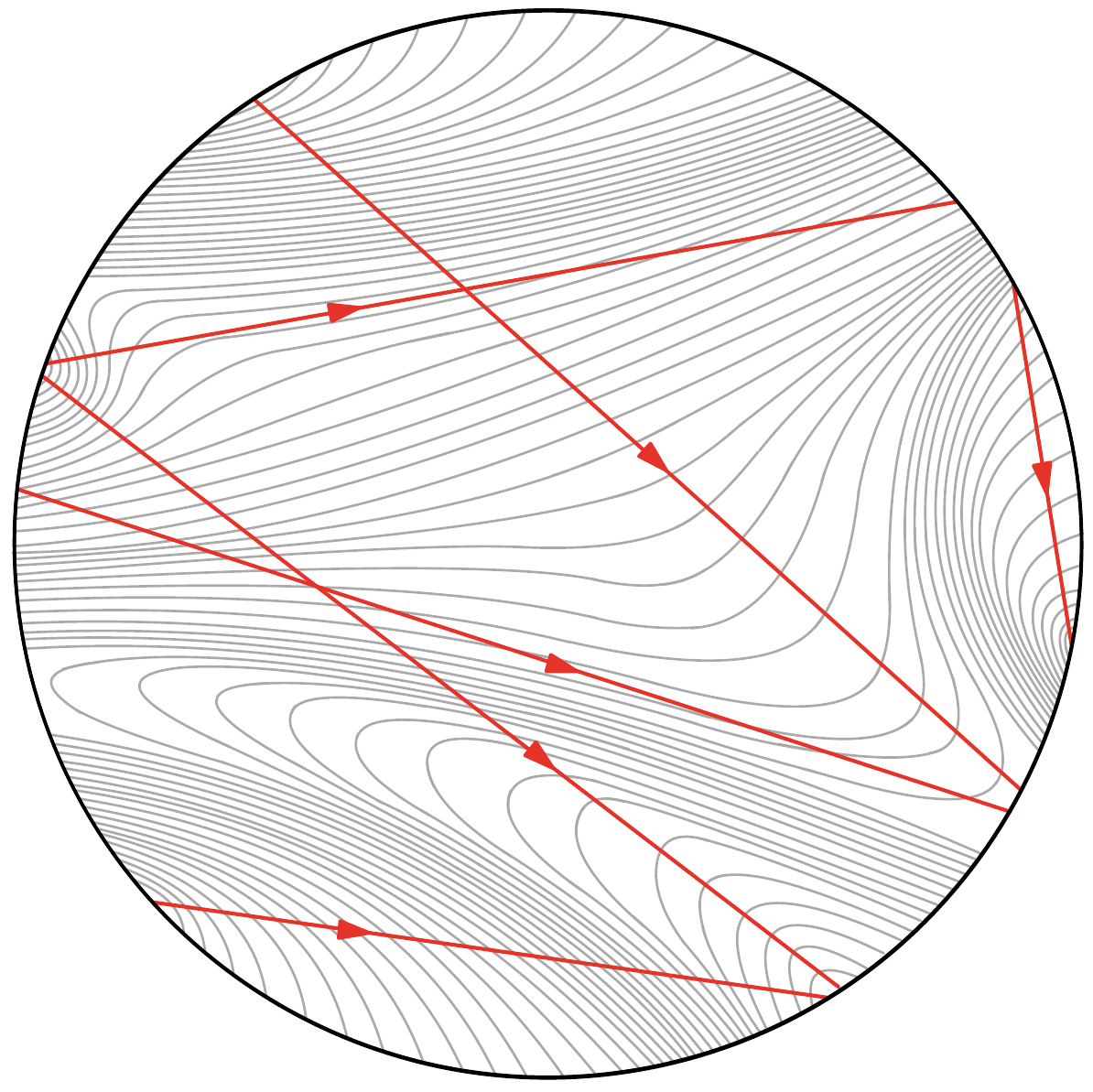}
         \put (103.5,30) {{\color{myRED}\large$\displaystyle \{\varphi(\Gamma_\O)\}_{\O \in \OO}$}}
\end{overpic}
\end{figure}


\newpage
\section{Preliminaries}

\subsection{Planar topology} 
Let $\R^2$ be the Euclidean plane, equipped with the topology induced by the Euclidean metric $d$. An open ball of radius $r>0$ centered at $x\in \R^2$ is denoted by $B(x,r)$. 

A \textit{path} $\gamma: I\longrightarrow \R^2$ is a continuous map defined on an interval $I\subset \R$. 
A path $\gamma$ is said to be a reparametrization of another path $\gamma':I'\longrightarrow \R^2$ if there exists a strictly increasing map $\varphi:\R\longrightarrow\R$ that satisfies $\varphi(I) = I'$ and $\gamma= \gamma' \circ \varphi\vert_I$. A \textit{loop} $\gamma:\mathbb {S}^1\longrightarrow \R^2$ is a continuous map defined on the unit circle $\mathbb{S}^1$. A path or a loop is said to be \textit{simple} if it is injective. 
By abuse of notation, if there is no ambiguity, we may consider paths and loops as subset of the plane by identifying them with their images.

A \textit{topological line} is a path $\ell: \R \longrightarrow \R^2$ which is injective and proper. Observe that any topological line admits a natural orientation induced by the positive orientation on $\R$. Similarly, we define a \textit{half-line} as a path $\ell^+: [\spc0,\infty) \longrightarrow \R^2$ that is injective and proper. Often enough, it is convenient to consider half-lines defined on the interval $(-\infty,0\spc]$ as well.

In the sphere compactification of the plane $\mathbb{S}^2 = \R^2 \sqcup \{\infty\}$, every topological line $\ell$ can be extended to a simple loop on $\mathbb S^2$ by adding the point at infinity. According to the Jordan curve theorem \cite{jordan1887cours}, the complement of any simple loop on the sphere has exactly two connected components. By the Schoenflies theorem (see \cite{schoenflies} for an elementary proof), each of these connected components is homeomorphic to $\R^2$. In our context, this means that every topological line $\ell$ separates the plane into two sets, each of which is homeomorphic to the plane. These sets are called the left and right sides of $\ell$, denoted \( L(\ell) \) and \( R(\ell) \), according to the orientation of $\ell$. 


Now, we present a fundamental generalization of the Schoenflies theorem, which provides a criterion for extending homeomorphisms defined on subsets of the sphere to the entire sphere.  This theorem is particularly useful in the context of low-dimensional topology because it allows one to handle ``sufficiently nice'' subsets of the sphere with more flexibility.

Here, by \textit{arc} we mean a simple path $\gamma: [\sspc0,1\sspc] \longrightarrow \mathbb S^2$. A \textit{Y-set} is a subset of $\mathbb S^2$ which is the union of three arcs intersecting at a single point that is an endpoint of each arc.  Note that the orientation of $\mathbb S^2$ induces a cyclic order on the arcs constituting an Y-set in a neighborhood of the intersection point. 

\begin{theorem*}[Homma-Schoenflies, \cite{homma1953extension}]\label{thm:homma_schoenflies}
    Let $X\subset \mathbb S^2$ be a compact, connected and locally-connected set. Then, an injective and continuous map $h:X\longrightarrow \mathbb S^2$ can be extended to an orientation-preserving homeomorphism of $\mathbb S^2$ if, and only if, the map $h$ preserves the cyclic order of any Y-set contained in $X$.
\end{theorem*}

Applying the Homma-Schoenflies theorem is not necessarily straightforward, as it requires checking the preservation of the cyclic order of Y-sets. However, there are many classic examples of applications of this theorem (see \cite[Appendix]{leroux2006pointsfixes}). Along this preliminary section, we will guide the reader through the process of applying the Homma-Schoenflies theorem as we introduce the new concepts and tools.

A family $\{A_i\}_{i \in I}$ of subsets of $\R^2$ is called \textit{locally-finite} if, for every compact set  $K\subset\R^2$, we have that $\{i \in I \mid A_i \cap K \neq \varnothing\}$ is a finite set.

\begin{lemma}\label{lemma:patt}
    Let $\{X_i\}_{i \in I}$ be a family of locally-finite and pairwise disjoint closed subsets of $\R^2$. Then, there exists a locally-finite family $\{U_i\}_{i\in I}$ of pairwise disjoint open sets satisfying $X_i \subset U_i\spc$.
\end{lemma}

\begin{proof}
    For each $i \in I$, consider the set $Y_i:=\bigcup_{j \in I} X_j \setminus X_i$. First, we show that $Y_i$ is closed. For a given $i \in I$, let $(x_n)_{n\in \N}$ be a sequence of points in $Y_i$ that converges to a point $x\in \R^2$. Observe that, since the family $\{X_j\}_{j \in I}$ is locally-finite, the set $\{j \in I \mid d(x, X_j) < r\}$ is finite for any $r>0$. Consequently, there exists an index $j_0 \in I\setminus \{i\}$ and a subsequence $(x_{n_k})_{k\in \N}$ that satisfies $x_{n_k} \in X_{j_0}$ for every $k\in \N$. Since $X_{j_0}$ is a closed set, we have that $x \in X_{j_0}$. Since $X_{j_0}$ is disjoint from $X_i$, 
    we have that $x \not\in X_i$. This allows us to conclude that $x \in Y_i$, which proves that the set $Y_i$ is closed.
    
    We can now define, for each $i \in I$, the set 
$$ U_i := \left\{z \in \R^2 \mid d(z, X_i)<\min\bigl(\sspc d(z, Y_i)\sspc, \sspc 1\sspc\bigr)\right\}.$$
Since $X_i$ and $Y_i$ are disjoint sets, we observe that
$$ d(z, X_i) = 0 < \min\bigl(\sspc d(z, Y_i)\sspc, \sspc 1\sspc\bigr), \quad \forall z \in X_i\sspc.$$
This proves that $X_i\sspc \subset U_i$ for every $i \in I$. Next, we remark that
$$ d(z, X_i) <\min\bigl(\sspc d(z, Y_i)\sspc, \sspc 1\sspc\bigr) \leq d(z, X_j), \quad \forall j\neq i,\ \forall  z \in U_i.$$
This implies that $U_i$ is disjoint from $U_j$ for every $j \in I\setminus \{i\}$. Finally, note that the function $$g_i:z \longmapsto d(z, X_i) - \min\bigl(\sspc d(z, Y_i)\sspc, \sspc 1\sspc\bigr)$$ is continuous. Thus, by writing $U_i = g_i^{-1}((-\infty,0))$, we conclude that $U_i$ is an open set.

To conclude the proof, we need to show that the family $\{U_i\}_{i \in I}$ is locally-finite. For that, we consider for each $i \in I$ the closed ball $$B(X_i\sspc,1):=\{z \in \R^2 \mid d(z, X_i) \leq 1\}.$$ Since the family $\{X_i\}_{i \in I}$ is locally-finite, the family of closed balls $\{B(X_i\sspc,1)\}_{i \in I}$ is also locally-finite. From the definition of the set $U_i$, we can observe that 
$$ U_i\subset B(X_i\sspc,1), \quad \forall i \in I.$$
Therefore, we conclude that the family of open sets $\{U_i\}_{i \in I}$ also is locally-finite.
\end{proof}

We provide a first example of application of the Homma-Schoenflies theorem. 

\vspace{0.3cm}

\begin{itemize}[leftmargin=0.6cm]
    \item Let $\{\ell_{n}\}_{n\in \mathbb Z}$ be a locally-finite family of pairwise disjoint topological lines on $\R^2$ such that 
    $$ L(\ell_{n_1})\subset L(\ell_{n_2}) \subset L(\ell_{n_3}),\quad \forall n_1<n_2<n_3\in \Z.$$
     In this context, the union of the lines in the family $\{\ell_n\}_{n \in \Z}$ and the point at infinity forms a compact, connected, and locally-connected subset of $\mathbb S^2$. It is worth noting that this would not be the case if the family $\{\ell_n\}_{n \in \Z}$ 
was not locally-finite.
By applying the Homma-Schoenflies theorem, we can obtain an orientation-preserving homeomorphism on the plane that maps each line $\ell_n$ to the vertical line $\{n\}\times \R$ oriented downwards. 
\end{itemize}

\mycomment{0.4cm}
\noindent For more details on the proof of this application, we refer the reader to \cite[Appendix B]{LEROUX_2017}.

\mycomment{0.2cm}

\subsection{Topological planar foliations}\label{sec:prelim_foliations}

{\ A} \textit{topological planar foliation} \(\mathcal{F}\) is a partition of \(\mathbb{R}^2\) into disjoint topological lines,  called \textit{leaves}, that admits a local parallel structure, i.e., for any \(x \in \mathbb{R}^2\), there exists a closed neighborhood \(U\subset \R^2 \) of \(x\) and a homeomorphism \(\varphi: U \to [0,1] \times [0,1]\) that satisfies
\begin{itemize}
    \item[$\sbullet$] For every leaf \(\phi \in \F\), each connected component of the set $\phi \sspc\cap\sspc U$ is of the form \[\varphi^{-1}(\{t\} \times [0,1]), \quad \text{ for some }t \in [0,1]. \]
\end{itemize}
\noindent In the particular case of foliations of $\R^2$, the condition above can be exchanged for the condition that the entire set $\phi \sspc\cap\sspc U$ is of the form \(\varphi^{-1}(\{t\} \times [0,1])\) for some \(t \in [0,1] \).

A \textit{cross-section} of \(\F\) over a leaf $\phi \in \F$ is a compact arc \(\sigma \subset \mathbb{R}^2\) that intersects $\phi$ and is transverse to \(\F\), meaning that the arc $\sigma$ intersects each leaf of \(\F\) at most once. 

Let $\phi \in \F$ be a leaf, and let $\{\sigma_k\}_{k \in \mathbb \Z^*}$ be a locally-finite collection of pairwise disjoint compact cross-sections of $\F$ over the leaf $\phi$. Assume that, for every $k >0$ it holds that:
\begin{itemize}
    \item[$\sbullet$] The cross-sections $\sigma_{-k}$ and $\sigma_{k}$ intersect the same leaves in $\F$;
    \item[$\sbullet$] The sub-arc $u_k\subset\phi$ joining $\sigma_{-k}$ and $\sigma_{k}$ satisfies $u_k \subset u_{k+1}$;
    \item[$\sbullet$] The only leaf of $\F$ intersecting all cross-sections in $\{\sigma_k\}_{k \in \mathbb \N^*}$ is the leaf $\phi$.
    \item[$\sbullet$] Every leaf of $\F$ that intersects $\sigma_{k+1}$ also intersects $\sigma_k$. 
\end{itemize}
\mycomment{0.3cm}
The closed set $V\subset \R^2$, obtained as the union of all compact sub-arcs of leaves in $\F$ joining the cross-sections $\sigma_{-k}$ and $\sigma_{k}$, for some $k>0$, is called a \textit{trivial-neighborhood} of the leaf $\phi$. Observe that, for any open neighborhood $U\subset \R^2$ of $\phi$, we can consider sufficiently small cross-sections $\{\sigma_k\}_{k \in \mathbb \Z^*}$ so that the neighborhood $V$ is contained in $U$.

It is worth noting that, trivial-neighborhoods of leaves can be taken to be arbitrarily small.  More precisely, for any  leaf $\phi \in \F$, and any open neighborhood $U\subset \R^2$ of $\phi$, there exists  a trivial-neighborhood $V$ of $\phi$ such that $V\subset U$. Such a trivial-neighborhood $V$ can be obtained by taking the cross-sections in the collection $\{\sigma_k\}_{k \in \mathbb \Z^*}$ to be sufficiently small.

\vspace{0.2cm}

By the Homma-Schoenflies theorem, for any trivial-neighborhood $V$ of a leaf $\phi \in \F$,  we can find a homeomorphism $\varphi: \R^2\longrightarrow \R^2$ that satisfies the following properties:
\begin{itemize}
    \item[$\sbullet$] The leaf $\phi$ is mapped by $\varphi$ into the vertical line $\{0\}\times \R$;
    \item[$\sbullet$] The set $V$ is mapped by $\varphi$ into the set $\bigcup_{n \in \N^*} [\sspc-\frac{1}{n}, \frac{1}{n}\sspc]\times [\sspc -n,n\sspc]$.
    \item[$\sbullet$] For each leaf $\phi' \in \F$, the set $\varphi(\phi'\cap V)$ is a vertical line $\{t\}\times \R$, for some $t\in \R$.
\end{itemize}

As a direct application of Lemma \ref{lemma:patt}, we have the following lemma.

\begin{lemma}\label{lemma:trivial_neighborhoods}
    {\ Let} $\{\phi_i\}_{i \in I}$ be a locally-finite family of pairwise disjoint leaves of \(\F\).  Then, for any closed set $X \subset \R^2$ that is disjoint from all leaves in $\{\phi_i\}_{i \in I}$, there exists a locally-finite family of pairwise disjoint closed sets $\{V_i\}_{i \in I}$ such that, for each $i \in I$,  the set $\sspc V_i\sspc $ is a trivial-neighborhood of the leaf $\sspc \phi_i\sspc $ that is disjoint from $X$.
\end{lemma}

\begin{proof}
    Observe that $\{\phi_i\}_{i \in I} \cup \{X\}$ forms a locally-finite family of pairwise disjoint closed subsets of $\R^2$. By applying Lemma \ref{lemma:patt}, we can find a locally-finite family of pairwise disjoint open sets $\{U_i\}_{i \in I}$ such that, for each $i \in I$, the set $U_i$ is a neighborhood of the leaf $\phi_i$ and is disjoint from $X$. Since trivial-neighborhoods of leaves can be taken to be arbitrarily small, for each $i \in I$, we can consider a trivial-neighborhood $V_i$ of the leaf $\phi_i$ such that $V_i \subset U_i$. Since $\{U_i\}_{i \in I}$ is locally-finite, so is the family $\{V_i\}_{i \in I}$. This concludes the proof.
\end{proof}



We can endow any foliation \(\F\) with the \textit{quotient topology}, which corresponds to the finest topology on $\F$ for which the quotient map $\pi_{_\F}:\R^2\longrightarrow \F$ that sends each point to the leaf passing through it is continuous. In this topology, the open sets of the foliation \(\F\) are those whose preimages under $\pi_{_\F}$ are open in \(\R^2\). Throughout this text, we will always consider the quotient topology on foliations, unless otherwise stated.

A subset $A\subset \R^2$ is called \textit{saturated} (with respect to $\F$) if it contains all leaves of \(\F\) that intersect it. The restriction of a foliation \(\F\) to a saturated subset $A\subset \R^2$ is denoted by $\F\vert_A$.  Note that, in the case of saturated sets we have $\F\vert_A = \pi_{_\F}(A)$. 
By a slight abuse of notation, we often treat saturated sets $A\subset \R^2$ as subsets of the foliation $\F$ by identifying $A = \F\vert_A$. 
This is the abuse that allows us to denote saturated sets as subsets $A\subset \F$, or to refer to leaves as elements $\phi \in A$.
This convention introduces no topological ambiguity when referring to an open saturated set or to its boundary. Indeed, from the definition of the quotient topology,  a saturated set $A\subset \R^2$ is open if, and only if, the restriction $\F\vert_A$ is an open subset of \(\F\). Additionally, the boundary $\partial A$ of a saturated set $A$ may be understood equivalently as the boundary of $A$ inside $\mathbb{R}^2$, or as the boundary of \(\F|_A\) inside \(\F\), with no risk of confusion.

This next lemma states that the connectedness is also preserved through this notation.

\begin{lemma}\label{lemma:consistent_abuse_notation_piF}
    A saturated set \(A\subset \R^2\) is connected if, and only if$\sspc$, $\sspc \F\vert_A$ is connected.
\end{lemma}

    \begin{proof} {\ We} first prove ($\implies$). Let us assume that \(A\subset \R^2\) is a connected saturated set.  Since $\pi_{_\F}$ is continuous, the image $\pi_{_\F}(A)$ is connected. 

    {\ ($\impliedby$)} Let us assume that \(A\subset \R^2\) is a saturated set such that $\F\vert_A$ is connected.  Suppose by contradiction that \(A\) is not connected. Therefore, there exists two non-empty open subsets $U,V \subset \R^2$ such that $A\subset U \cup V$, and such that $U \cap A$ is disjoint from $V\cap A$.
    Since each leaf $\phi \in A$ is connected, then either $\phi \cap U$ or $\phi \cap V$ is empty. Since $A$ is saturated and covered by the union $U\cup V$, we have that each leaf $\phi \in A$ is contained in either $U$ or $V$. 
    In other words, $U \cap A$ and $V \cap A$ are disjoint saturated sets that cover $A$. Since $\pi_{_\F}$ is a quotient map, $\F\vert_{U\cup A}=\pi_{_\F}(U \cap A)$ and $\F\vert_{V\cup A}=\pi_{_\F}(V \cap A)$ are open subsets of $\F\vert_A$ that are disjoint and cover $\F\vert_A$. This contradicts the assumption that $\F\vert_A$ is connected.
    \end{proof}

\mycomment{-0.21cm}
Equipped with the quotient topology, the foliation $\F$ becomes a simply-connected and non-Hausdorff 1-dimensional manifold, often called the \textit{leaf space} of \(\F\). As a matter of fact, the classical result due to Haefliger and Reeb \cite{haefligerreeb1957} proves that every simply-connected, non-Hausdorff 1-manifold can be realized as the leaf space of a planar foliation. 
Note that, any simply-connected (non-Hausdorff) manifold is orientable. In the 1-dimensional context, this corresponds to the existence of an atlas of charts mapping the local orientation at each point of the manifold to the positive orientation of \(\R\). By choosing an orientation, we fix a local notion of left and right sides around each point. Since by removing a point we separate any simply-connected non-Hausdorff 1-manifold into two connected components, this local notion of left and right sides is extended globally to each of these connected components.

In the context of the leaf space, this means that for each leaf $\phi \in \F$, the set $\F\setminus\{\phi\}$ has exactly two connected components, one labeled as the left side $L(\phi)$ and the other as the right side \(R(\phi)\). Each leaf of $\mathcal{F}$ thereby inherits an orientation, and this assignment is locally coherent, since it is determined by the local orientation at each point of the leaf space.  In terms of local trivializing maps of the foliation $\F$, this means that one can choose each trivializing map $\varphi: U \to [0,1] \times [0,1]$ so that the orientation of every leaf $\phi \in \F$ is mapped to the positive direction of the second coordinate on the set $\varphi(U \cap \phi) = \{t\} \times [0,1]$. This is, in fact, the classical definition of an oriented foliation. This is, in fact, the classical definition of an oriented foliation. We choose this approach so that the reader to see why the every planar foliation admits an orientation.

 Once the foliation $\F$ is equipped with an orientation, it inherits a naturally defined partial order, where $\phi\leq \phi'$ if and only if $L(\phi') \subset L(\phi)$. We remark that, pairs of non-separable leaves $\phi,\phi'\in \F$ must satisfy either $R(\phi) \cap R(\phi') = \varnothing$ or $L(\phi) \cap L(\phi') = \varnothing$ and, consequently, they cannot be comparable under this partial order relation. Recall that, an interval of the foliation $\F$ with respect to this partial order consists in a subset of $\F$ that contains all leaves between any two of its leaves. Observe that leaf intervals are not necessarily homeomorphic to an interval of the real line $\R$, as the leaf space of $\F$ is not necessarily Hausdorff.

A saturated set $A\subset \R^2$ is said to be \textit{non-separating} if for every leaf $\phi \in A$, one of the connected components of $\R^2\setminus\phi$ contains all leaves in $A\setminus \{\phi\}$. 
 
 \begin{lemma}\label{lemma:connected}
    Let \(A\subset \R^2\) be an open, connected saturated set. Then, the following holds:
    \begin{itemize}[leftmargin=1.3cm]
        \item[\textup{\textbf{(i)}}] The restriction $\F\vert_A$ is an interval of the foliation $\F$.
        \item[\textup{\textbf{(ii)}}] The boundary $\partial A$ is a locally-finite and non-separating collection of leaves in \(\F\).
        \item[\textup{\textbf{(iii)}}] The set $A\subset\R^2$ is homeomorphic to a plane.
    \end{itemize} 
\end{lemma}

\begin{proof}
    \textit{Proof of $(i)$:} Suppose by contradiction that there exist $\phi,\phi' \in \F\vert_A$ such that $\phi< \phi'$, and $\phi'' \in \F\setminus \F\vert_A$ such that $\phi < \phi'' < \phi'$. In this case, the sets $L(\phi'')$ and $R(\phi'')$ are disjoint open sets whose union covers $A$ and such that $L(\phi'') \cap A$ and $R(\phi'') \cap A$ are both non-empty. This contradicts the assumption that $A$ is connected.

    \textit{Proof of $(ii)$:} Since $A\subset \R^2$ is open and connected, for any leaf $\phi \in \partial A$, the set $A$ must be contained in one of the two connected components of $\R^2\setminus \phi$. Therefore, each of the remaining leaves in $\partial A\setminus\{\phi\}$ is contained in the same connected component of $\R^2\setminus \phi$ that contains $A$. This implies that the set $\partial A$ is non-separating.

    Next, observe that any open set $U\subset \R^2$ that locally trivializes the foliation $\F$ can at most intersect two leaves in $\partial A$. This is because, for any three leaves of $\F$ intersecting $U$, we get that one of them must separate the other two. Let $K\subset \R^2$ be a compact set, and let $\{U_i\}_{0<i<k}$ be a finite cover of $K$ by open sets that locally trivialize $\F$. Since each set in the cover $\{U_i\}_{0<i<k}$ intersects at most two leaves in $\partial A$, the set $K$ intersects at most finitely many leaves in $\partial A$. This implies that the collection of leaves in $\partial A$ is locally-finite.

    \textit{Proof of $(iii)$:} Knowing that the boundary $\partial A$ is non-separating, we can write $A$ as the complement of the set $\bigcup_{\phi \in \partial A} K(\phi)$, where $K(\phi)$ is defined as the closure of the connected component of $\R^2\setminus \phi$ that does not contain $A$.  Observe that, by adding the point at infinity, the set $\bigcup_{\phi \in \partial A} K(\phi)\cup \{\infty\}$ becomes a compact and connected subset of $\mathbb S^2$. 
    Since $A$ is the complement of a compact and connected subset of $\mathbb S^2$, we conclude that it is simply-connected.
    According to the Riemann mapping theorem, being a simply-connected open subset of $\R^2$, the set $A$ must be homeomorphic to the plane. This completes the proof of the lemma.
\end{proof}
 
A saturated set $A\subset \R^2$ is said to be \textit{trivially foliated} if the restriction $\F\vert_A$ is totally ordered by the natural order on the foliation $\F$.

\begin{definition}\label{def:leaf_domain}
    A \textit{leaf domain} is a saturated set that is open, connected  and trivially foliated. 
\end{definition}

Following Lemma \ref{lemma:consistent_abuse_notation_piF} and Lemma \ref{lemma:connected}, we observe that non-empty leaf domains of $\F$ correspond exactly to the subsets of $\F$ that are homeomorphic to $\R$. Moreover, the boundary of a leaf domain is formed by a locally-finite and non-separating collection of leaves in \(\F\).

The following lemma is elementary, but we choose to include it for transparency.

\begin{lemma}\label{lemma:intersection_of_domains}
    A finite intersection of leaf domains, also is a leaf domain.
\end{lemma}

\begin{proof}
    This follows from the fact that, in a simply-connected non-Hausdorff 1-manifold, the intersection of two subsets homeomorphic to $\R$ is either empty or homeomorphic to $\R$.
\end{proof}

The boundary of a leaf domain $A\subset \F$ admits two important subsets, defined as
$$ \partial_R A := \partial A \cap \biggl(\ \bigcap_{\spc\phi \in A\spc} R(\phi)\biggr) \quad \ \ \text{ and }\ \  \quad \partial_L A := \partial A \cap \biggl(\ \bigcap_{\spc\phi \in A\spc} L(\phi)\biggr).$$ 

\noindent These sets are disjoint by definition and their union may not cover the entire boundary $\partial A$, as illustrated in the figure below. Moreover, the set $A$ lies on the left side of each leaf in $\partial_R A$ and on the right side of every leaf in $\partial_L A$. 

\begin{figure}[h!]
    \center
    \mycomment{0.2cm}\begin{overpic}[width=8.3cm, height=3.8cm, tics=10]{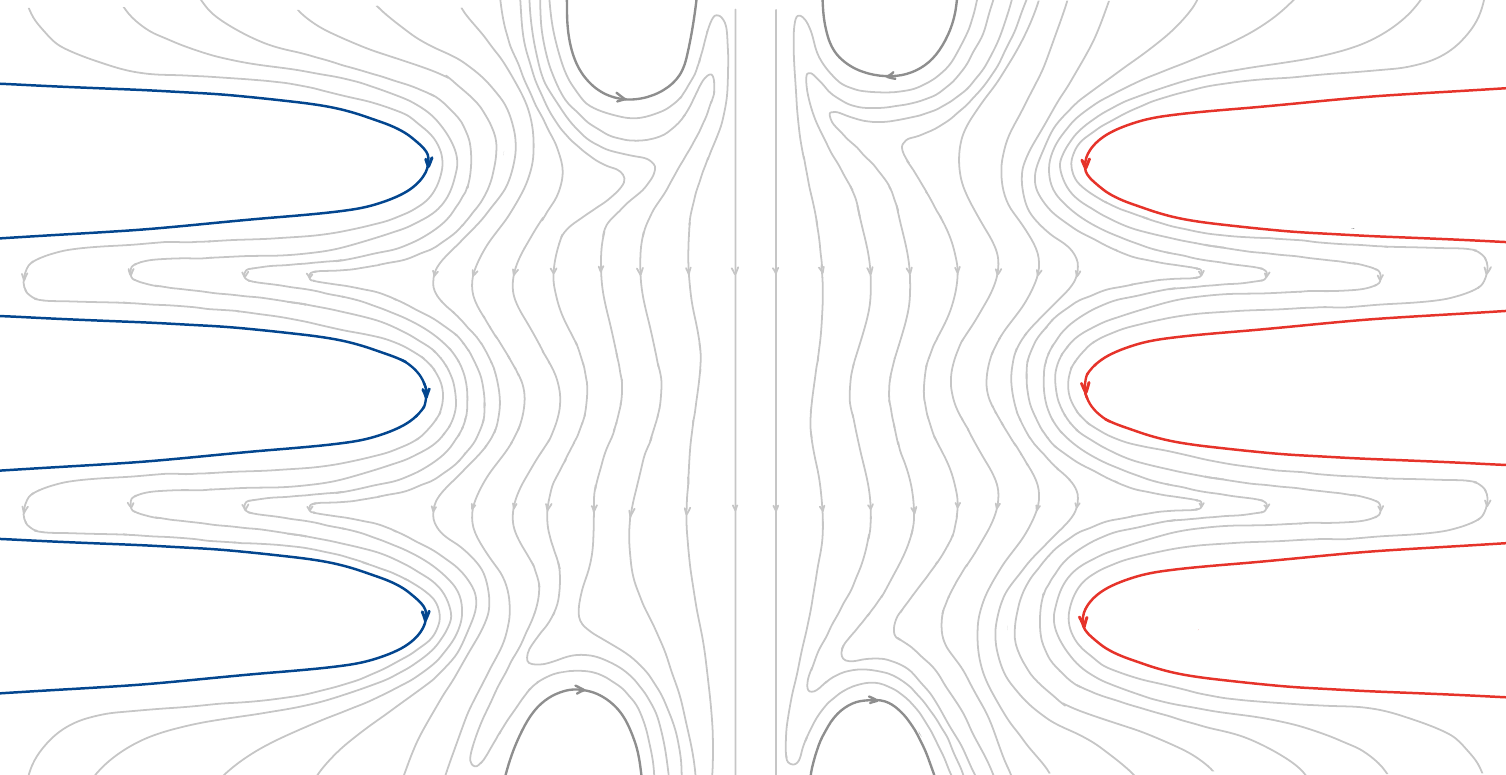}
        \put (47,21) {\colorbox{white}{\color{myDARKGRAY}\large$\displaystyle A$}}
        \put (6,21) {\color{myBLUE}\large$\displaystyle \partial_R A$}
        \put (84,21) {\color{myRED}\large$\displaystyle \partial_L A$}
\end{overpic}
\end{figure}

The sets $\partial_R A$ and $\partial_L A$ are each totally ordered by a relation $\prec\sspc$, defined as follows:
\begin{itemize}[leftmargin=1.3cm]
    \item First, fix a leaf of reference $\phi_* \in A$. For any two leaves $\phi, \phi' \in \partial_L A$ (resp. $\partial_R A$),  we define $\phi\prec \phi'$ if and only if there exist two disjoint paths $\gamma,\gamma':[0,1]\longrightarrow \R^2$ with endpoints $\gamma(1) \in \phi\spc$, $\gamma'(1) \in \phi'\sspc$ and  $\gamma(0), \gamma'(0) \in  \phi_*\spc$, ordered $\gamma(0)< \gamma'(0) $ according to the orientation of $\phi_*$. 
\end{itemize}

\mycomment{0.2cm}
\noindent We remark that $\prec\sspc$ is well-defined and independent of the choice of reference leaf $\phi_*\in A$. 
Later, in Section \ref{sec:order_limit_leaves}, we provide a more formal alternative definition of this order relation.

A path $\gamma:I\longrightarrow \R^2$ defined on an interval $I \subset \R$ is said to be \emph{positively transverse} to an oriented foliation $\F$ if the map $\pi_{_\F}\circ\gamma:I\longrightarrow \F$ is strictly increasing with respect to the natural order on $\F$. In other words, a positively transverse path intersects each leaf of $\F$ at most once and it does so crossing each leaf from right to left.

Each positively transverse path $\gamma$ defines a saturated set $C_\gamma\subset \R^2$ which is the union of all leaves in $\F$ that are intersected by $\gamma$. From the definition of positively transverse paths, it follows that the set $C_\gamma$ (seen as a subset of the foliation $\F$) is homeomorphic to an interval of the real line $\R$. In particular, this proves the following elementary lemma.

\begin{lemma}
\label{lemma:positively_transverse_path}
    Let $I\subset \R$ be an open interval, and let $\gamma:I\longrightarrow \R^2$ be path positively transverse to the foliation $\F$. Then, the saturated set $C_\gamma$ is a leaf domain of $\F$.
\end{lemma}

Conversely, for any leaf domain $A\subset \F$, and any pair of points $x,y \in A$ that do not lie on the same leaf, there exists a positively transverse path $\gamma:I\longrightarrow \R^2$ that satisfies either $\gamma(0)=x$ and $\gamma(1)=y$, or $\gamma(0)=y$ and $\gamma(1)=x$. This can be seen by considering a homeomorphism mapping $A$ into $\R^2$ foliated by parallel lines oriented in the same direction.

We conclude this section with two examples of applications of the Homma–Schoenflies theorem that illustrate how it can be used to describe the foliation near transverse paths. 

\mycomment{0.15cm}
\noindent \textbf{(1) } For any path $\gamma:[0,1]\longrightarrow \R^2$ positively transverse to a foliation $\F$, there exists an orientation-preserving homeomorphism $\varphi:\R^2\longrightarrow \R^2$ that satisfies the following properties:
\begin{itemize}
    \item[$\sbullet$] The path $\gamma$ is mapped by $\varphi$ into the horizontal arc $[0,1]\times\{0\}$.
    \item[$\sbullet$] For $i\in \{0,1\}$, the leaf $\phi_{\gamma(i)}\in \F$ passing through $\gamma(i)$ is mapped by $\varphi$ into the vertical line $\{i\}\times \R$ oriented downwards;
    \item[$\sbullet$] The exists a closed set $V\subset \R^2$ containing $\gamma$ that is mapped by $\varphi$ into the  square $[0,1]\times [\sspc -1,1\sspc]$ such that, for any leaf $\phi\in \F$ intersecting $V$, the set $\phi\cap V$ is mapped by $\varphi$ into a vertical arc of the form $\{t\}\times [\sspc -1,1\sspc]$, for some $t\in [0,1]$.
\end{itemize}

\mycomment{0.35cm}
\noindent \textbf{(2) }  For any topological line $\Gamma:\R\longrightarrow \R^2$ positively transverse to $\F$, there exists an orientation-preserving homeomorphism $\varphi:\R^2\longrightarrow \R^2$ that satisfies the following properties:
\begin{itemize}
    \item[$\sbullet$] The topological line $\Gamma$ is mapped by $\varphi$ into the horizontal line $\R\times\{0\}$;
    \item[$\sbullet$] The exists a closed set $V\subset \R^2$ containing $\Gamma$ that is mapped by $\varphi$ into the horizontal strip $\R\times [\sspc -1,1\sspc]$ and such that, for any leaf $\phi\in \F$ intersecting $V$, the set $\phi\cap V$ is mapped by $\varphi$ into a vertical arc of the form $\{t\}\times [\sspc -1,1\sspc]$ oriented downwards.
\end{itemize}

\subsection{Foliated Brouwer Theory}\label{sec:foliated_brouwer}
Let \( f:\mathbb{R}^2 \longrightarrow \mathbb{R}^2 \) be a \textit{Brouwer homeomorphism}, that is, an orientation-preserving and fixed-point free homeomorphism of the plane.

The orbit of a point \( x \in \mathbb{R}^2 \) under \( f \) is defined as the set
$$ \mathcal{O}(f, x) = \{f^n(x) \mid n \in \mathbb{Z}\} \subset \mathbb{R}^2.$$
Sometimes, it may be useful to denote the set of orbits of \( f \) by $\orb$. As proved in \cite{Brouwer}, every orbit of $f$ is wandering and, thus, it has an empty limit set. This means that every orbit of $f$ is the image of a proper embedding of $\mathbb{Z}$ into the plane.

A \textit{Brouwer line} for \( f \) is an oriented topological line \( \lambda\) on \( \mathbb{R}^2 \) that satisfies
$$ f(L(\lambda)) \subset L(\lambda) \quad \text{ and } \quad f^{-1}(R(\lambda)) \subset R(\lambda),$$
where \( L(\lambda) \) and \( R(\lambda) \) denote the left and right sides of \( \lambda \). The classical \textit{Brouwer translation theorem} \cite{Brouwer} asserts that through every point of the plane passes a Brouwer line for \( f \).

\begin{theorem*}[Le Calvez, \cite{lec1}] There exists an oriented topological foliation \(\F\) of the plane such that every leaf $\phi \in \F$ is a Brouwer line of $f$.
\end{theorem*}

Any such foliation of the plane by Brouwer lines of \( f \) is called a \textit{transverse foliation} of \( f \). It is worth noting that a Brouwer homeomorphism admits several transverse foliations, which  may exhibit significantly different structures and non-homeomorphic leaf spaces, for instance.

Let us fix a transverse foliation $\F$ of a Brouwer homeomorphism \( f \). In this contex, one can show that every point \( x \in \mathbb{R}^2 \) may be connected to its image \( f(x) \) by a path $\gamma_x:[0,1]\longrightarrow \R^2$ that is positively transverse to the foliation $\F$, in the sense that $\gamma_x$ intersects each leaf of $\phi \in \F$ at most once, always from right $R(\phi)$ to left $L(\phi)$. By concatenating the paths $\gamma_{f^n(x)}$ for all integers $n\in \Z$, we obtain a positively transverse path 
$$ \Gamma := \prod_{n\in \Z} \gamma_{f^n(x)},$$
that connects the entire orbit $\mathcal{O}(f,x)$, called a \textit{transverse trajectory} of $\O(f,x)$ (with respect to the foliation $\F$). Since the transverse foliation $\F$ is usually fixed, we often omit its mention when referring to transverse trajectories.
It is worth noting that an orbit $\O(f,x)$ admits uncountably many transverse trajectories, which depend on the choice of paths $\{\gamma_{f^n(x)}\}_{n \in \Z}$.

Indenpendently of the choice of transverse paths $\{\gamma_{f^n(x)}\}_{n \in \Z}$, the leaves that are intersected by the transverse trajectory $\Gamma_\O$ of $\O=\O(f,x)$ are determined by the orbit itself.  In fact, the set $C_\O \subset \F$ of leaves \textit{crossed} by the orbit $\O$ can be equivalently defined as
$$ C_\O = \{\phi \in \F \mid \phi \cap \Gamma_\O \neq \varnothing\} = \{\phi \in \F \mid \O \text{ intersects } R(\phi) \text{ and } L(\phi)\}.$$

We remark that, if the choice of paths $\{\gamma_{f^n(x)}\}_{n \in \Z}$ is locally-finite on the plane, then the transverse trajectory $\Gamma_\O$ is a topological line, and it is called a \textit{proper transverse trajectory} of the orbit $\O$. The converse is also true.

\textbf{ Generic transverse foliations:}\
Let $\OO=\{\O_1,...,\O_r\}$ be a finite collection of orbits of $f$. Throughout this text, we may treat $\OO$ as a subset of $\R^2$ by identifying it with $\OO = \bigcup_{i=1}^{\sspc r} \O_i$. A transverse foliation $\F$ of \( f \) is said to be \textit{generic} for the collection $\OO$ if each leaf $\phi \in \F$ intersects at most one orbit in $\OO$. This condition is indeed generic, as we explain below.

Assume that a leaf $\phi \in \F$ intersects two orbits in $\OO$, at points $x,x^\pp\in \phi$. Let $B$ be a closed ball centered at $x$ that is sufficiently small so that $B \cap \OO = \{x\}$ and $f(B) \cap B = \varnothing.$ Observe that, we can perturbe $\F$ inside the ball $B$ in such a way that $x$ no longer lies on the leaf passing through $x^\pp$. Since $B$ satisfies $f(B) \cap B = \varnothing$, the leaves of $\F$ continue to be Brouwer lines of \( f \) after the perturbation. Now, since $\OO$ is a finite collection of orbits of $f$, and each orbit of $f$ is a closed discrete subset of $\R^2$, we have that for any compact set $K\subset \R^2$, the set $K\cap \OO$ is finite. Therefore, by repeating the perturbation process described above on a locally-finite family of disjoint balls covering $\OO$, we can obtain a transverse foliation $\F'$ of \( f \) that is generic for $\OO$.

\addtocontents{toc}{\protect\setcounter{tocdepth}{2}}

\section{Asymptotic behavior of orbits}\label{chap:asymptotic behavior}

Since $f$ is a Brouwer homeomorphism, every orbit of $f$ goes to infinity, in the sense that, for any $x \in \R^2$ the set $\O=\bigcup_{n \in \Z} f^n(x)$ is the image of a proper embedding of $\Z$ into $\R^2$.  Without an additional structure on the system, we cannot precisely describe the way that the orbits of $f$ escape to infinity. Recently, the most efficient way to describe the asymptotic behavior of orbits of $f$ has been through transverse foliations and transverse trajectories.

As explained in Section \ref{sec:foliated_brouwer}, once we consider a planar foliation $\F$ transverse to $f$, we can describe each orbit $\O$ of $f$ in terms of a transverse trajectory $\Gamma_\O$. The information carried by an arbitrary transverse trajectory $\Gamma_\O$ is equivalent to the information carried by the set
$$C_\O := \left\{\spc\phi \in \F \mid \spc \phi\sspc \cap \sspc\Gamma_\O \neq \varnothing\spc \right\}.$$
The set $C_\O$ is independent of the choice of transverse trajectory $\Gamma_\O$, and it corresponds to the set of leaves $\phi \in \F$ such that both sides $R(\phi)$ and $L(\phi)$ are intersected by the orbit $\O$. Until now, this has been the information used to describe the asymptotic behavior of orbits of Brouwer homeomorphisms in the framework of Foliated Brouwer Theory. 


\vspace*{0.22cm}
\textbf{Limitation of the framework:}\ 

Let $x \in \R^2$ be a point in the plane, and let $\O$ be the orbit of $x$ under of $f$. For each $n\in \Z$, let $\phi_n \in \F$ be the leaf of the foliation $\F$ containing the point $f^n(x)$. Note that $(\phi_n)_{n\in\Z}$ is a bi-infinite sequence of leaves in $C_\O$ that is increasing with respect to the natural order of the foliation $\F$, meaning, $L(\phi_{n+1})\subset L(\phi_n)$ for all $n\in \Z$. 
Although $\O$ must escape to infinity, the sequence  $(\phi_n)_{n\in\Z}$ may accumulate at certain leaves of $\F$. Since $(\phi_n)_{n\in\Z}$ is increasing, there are can two types of leaves accumulated by the sequence $(\phi_n)_{n\in\Z}$. Either a leaf $\phi_L\in \F$, which is a limit leaf of the forward sequence $(\phi_n)_{n\geq0}$, or a leaf $\phi_R \in\F$, which is a limit leaf of the backward sequence $(\phi_{n})_{n\leq0}$. It is worth noting, since $\F$ is not necessarily Hausdorff, each of these sequences may admit multiple limit leaves, as illustrated in the figure below.

\begin{figure}[h!]
    \center
    \vspace*{-0.1cm}\begin{overpic}[width=8.5cm, height=3.5cm, tics=10]{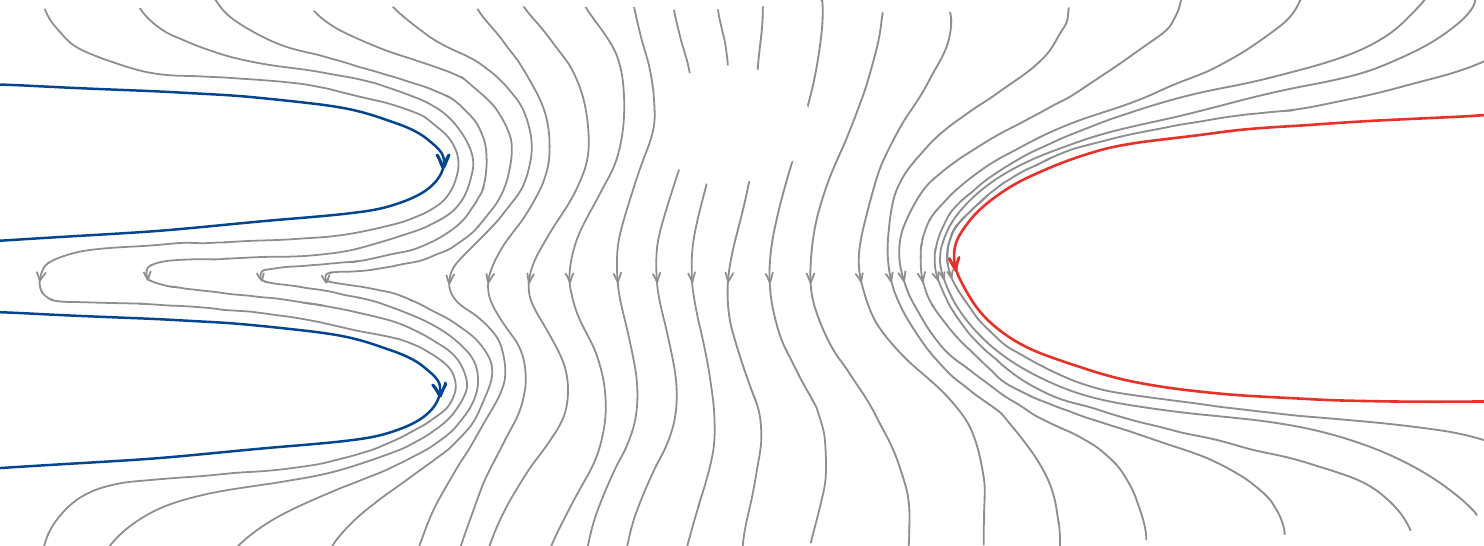}
        \put (68,20) {\color{myRED}\large$\displaystyle \phi_L $}
        \put (19.5,27.5) {\color{myBLUE}\large$\displaystyle \phi_R $}
        \put (19.5,10.4) {\color{myBLUE}\large$\displaystyle \phi_R $}
        
        \put (40,30) {\colorbox{white}{\color{myDARKGRAY}$\rule{0cm}{0.33cm} \quad \quad \quad \ $}}
        \put (41,30.8) {\large\color{myDARKGRAY}$\displaystyle (\phi_n)_{n \in \Z} $}
\end{overpic}
\end{figure}

\newpage

\vspace*{-0.1cm}
A priori, each point \(f^n(x) \in \O\) can be located anywhere along its underlying leaf $\phi_n$. Thus, if the sequence $(\phi_n)_{n\geq0}$ converges to a leaf $\phi_L$, the orbit $\O$ could possibly display a weird asymptotic behavior with respect to $\phi_L$, such as the one illustrated in the figure below.  This weird behavior can be characterized by the following property:
\begin{equation*}
    \text{``Every transverse trajectory } \Gamma_\O \text{ associated to } \O \text{ accumulates at the leaf } \phi_L\sspc\text{''}.
\end{equation*}
Such weird behavior impedes the existence of proper transverse trajectories for the orbit $\O$.

\vspace*{0.1cm}
\begin{figure}[h!]
    \center\begin{overpic}[width=8.5cm,height=3.5cm, tics=10]{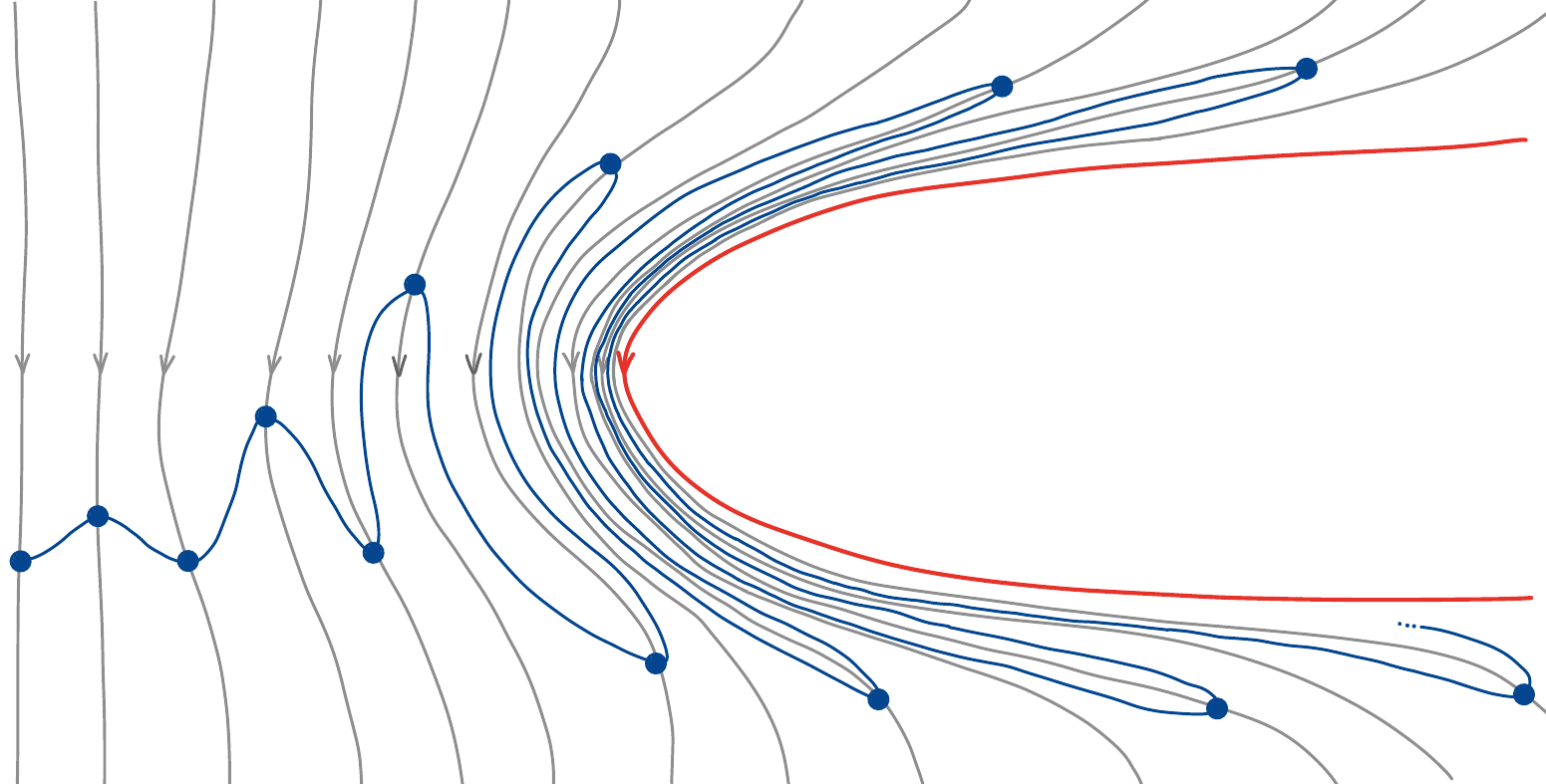}
        \put (47,20) {\color{myRED}\large$\displaystyle \phi_L $}
        \put (7,6.5) {\color{myBLUE}\large$\displaystyle \O $}
        \put (36,37) {\colorbox{white}{\color{myBLUE}\large$\displaystyle \Gamma_\O$}}
\end{overpic}
\end{figure}

\vspace*{-0.1cm}
In this section, we overcome this limitation by proving that every orbit of $f$ indeed admits proper transverse trajectories (see Theorem \ref{thmx:proper_trajectories_restate}). Moreover, we show that a proper transverse trajectory $\Gamma_\O$ associated to an orbit $\O$ encodes additional information about the asymptotic behavior of $\O$ that is not captured by $C_\O$ alone, but rather by $C_\O$ together with four distinguished subsets of its boundary $\partial C_\O$, namely $\Ltop \O\sspc$, $\Lbot \O\sspc$, $\Rtop \O\sspc$ and $\Rbot \O\sspc$.

\begin{figure}[h!]
    \center
    \vspace*{0.2cm}\begin{overpic}[width=8.5cm, height=3.5cm,tics=10]{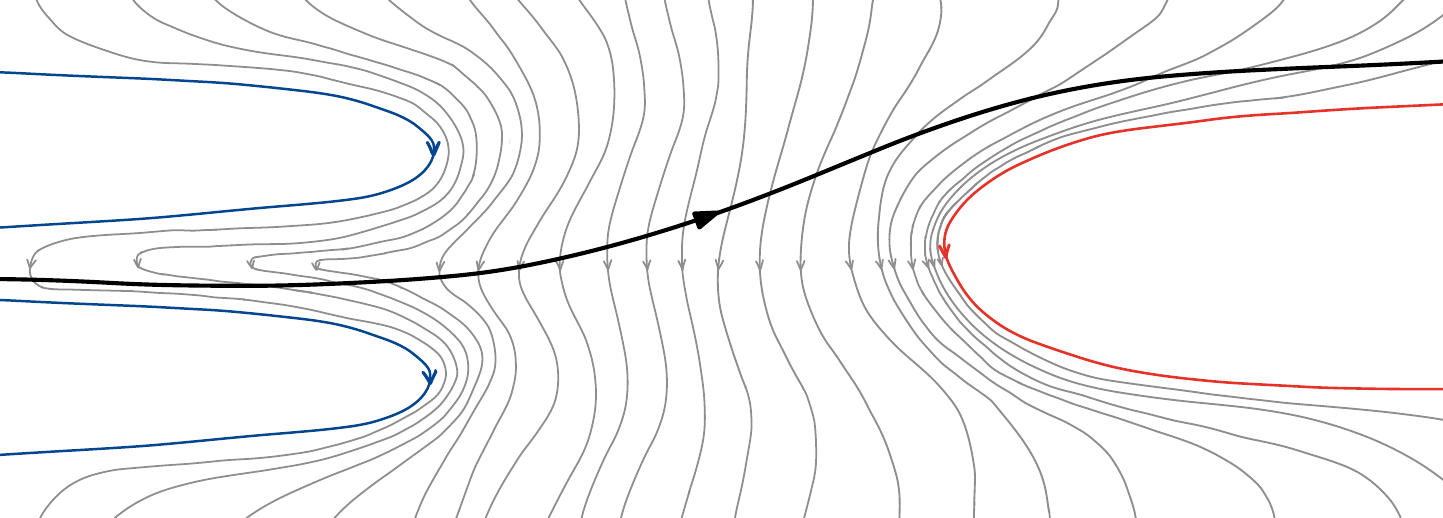}
        \put (72,20) {\color{myRED}\large$\displaystyle \Lbot \O $}
        \put (13,27) {\color{myBLUE}\large$\displaystyle \Rtop \O $}
        \put (13,9.4) {\color{myBLUE}\large$\displaystyle \Rbot \O $}
        \put (45.6,27.8) {\colorbox{white}{\color{myDARKGRAY}\large$\displaystyle \rule{0cm}{0.3cm} \   $}}
        \put (45,28) {{\color{myDARKGRAY}\large$\displaystyle \Gamma_\O $}}
\end{overpic}
\end{figure}

\subsection{Crossed leaves and limit leaves}\label{sec:crossed_leaves_and_limit_leaves}

Let $\orb$ be the set of orbits of $f$. For each orbit $\O \in \orb$, we define the set
$$ C_\O:= C(\F, \O\sspc) = \left\{\spc\phi \in \F \mid \spc \phi\sspc \cap \sspc\Gamma_\O \neq \varnothing\spc \right\},$$
where $\Gamma_\O$ represents any$\sspc$\footnote{All transverse trajectories $\Gamma_\O$ associated with an orbit $\O\in \orb$ intersect the same leaves in $\F$.} transverse trajectory of the orbit $\O$. 
Following the abuse of notation introduced in Section \ref{sec:prelim_foliations}, we can see the set $C_\O$ as a saturated subset of $\R^2$, formed by the union of all leaves contained in $C_\O$. 
According to Lemma \ref{lemma:positively_transverse_path} and Lemma \ref{lemma:connected}, the saturated set $C_\O \subset \R^2$ is a leaf domain of $\F$ and, consequently, it is homeomorphic to a trivially foliated plane, and its boundary $\partial C_\O$ is a locally-finite and non-separating collection of leaves in $\F$.
The set $C_\O$ is often referred to as the \textit{leaf domain} of the orbit $\O$, and by abuse of notation, we can treat $C_\O$ as a subset of the foliation $\F$ formed by the leaves contained in $C_\O$.
Throughout this work, a leaf of $\F$ is said to be \textit{crossed} by the orbit $\O$ if it belongs to the set $C_\O$. 

Since the set $C_\O$ is a leaf domain, its boundary $\partial C_\O$ admits two distinguished subsets, namely the \textit{left-limit} and \textit{right-limit} leaves of $C_\O$, which are defined as follows:
$$ \partial_L C_\O := \partial C_\O \cap \biggl(\ \bigcap_{\sspc\phi \in C_\O\sspc} L(\phi)\biggr) \quad \text{ and } \quad \partial_R C_\O := \partial C_\O \cap \biggl(\ \bigcap_{\sspc\phi \in C_\O\sspc} R(\phi)\biggr).$$

\vspace*{0.1cm}
\noindent Note that, by fixing a basepoint $x \in O$, and considering the bi-infinite sequence $(\phi_n)_{n\in\Z}$ formed by the leaves $\phi_n \in \F$ passing through the points $f^n(x)$, we observe that the forward sequence $(\phi_n)_{n\geq0}$ converges simultaneously to all leaves in $\partial_L C_\O$ (with respect to the quotient topology on $\F$), while the backward sequence $(\phi_n)_{n\leq0}$ converges to all leaves in $\partial_R C_\O$. This can be deduced from the fact that, for any leaf $\phi \in C_\O$, there exists $n \geq 0$ such that $$L(\phi_{-n}) \subset L(\phi) \subset L(\phi_n).$$
Observe that the sets $\partial_L C_\O$ and $\partial_R C_\O$ are disjoint by definition, and their union may not cover the entire boundary set $\partial C_\O$, as illustrated in the figure below.

\begin{figure}[h!]
    \center
    \vspace*{0cm}\begin{overpic}[width=8cm, height=3.8
        cm, tics=10]{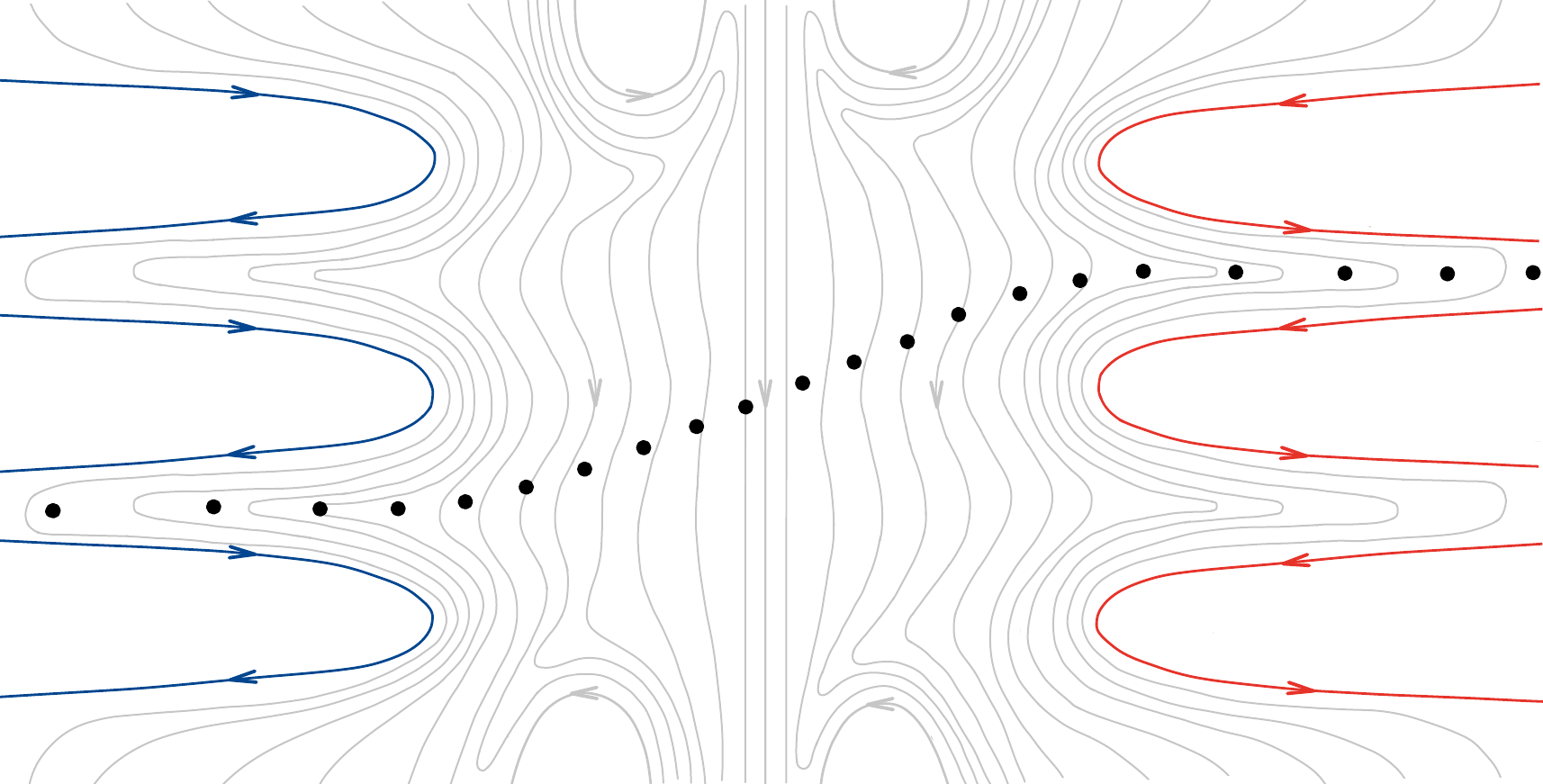}
        \put (54.4,15) {\colorbox{white}{$\rule{0cm}{0.3cm}\quad $}}
        \put (55,15.3) {{\color{myGRAY}\large$\displaystyle C_\O $}}
        \put (41.5,24) {\colorbox{white}{\color{myDARKGRAY}\large$\rule{0cm}{0.26cm}\  $}}
         \put (41.2,24) {{\color{black}\large$\displaystyle \O$}}
        \put (3,22) {\color{myBLUE}\large$\displaystyle \bb$}
        \put (85,22) {\color{myRED}\large$\displaystyle \partial_L C_\O$}
\end{overpic}
\end{figure}

\vspace*{-0.4cm}

\begin{remark} For any pair of orbits $\O,\O^\pp\in \orb$, the following properties hold:
    \begin{itemize}[leftmargin =1.2cm]
        \item[\textbf{(i)}] \textbf{Common leaf $\ $vs. Separated:} The orbits $\O$ and $\O^\pp$ respect a natural dichotomy: Either their leaf domains $C_\O$ and $C_{\O^\pp}$ intersect, in which case every leaf $\phi \in C_\O \cap C_{\O^\pp}$ is called a \textit{common leaf} of $\O$ and $\O'$. Or the sets $C_\O$ and $C_{\O^\pp}$ are disjoint, in which case we say that the orbits $\O$ and $\O^\pp$ are \textit{separated}, because there exists a leaf $\phi \in \F$  such that $\O$ and $\O^\pp$ are contained in different components of $\R^2\setminus \phi$.
        For example, the leaf $\phi \in \partial C_\O$ that bounds the connected component of $\R^2\setminus C_\O$ which contains $C_{\O^\pp}$.

        \vspace*{0.25cm}
        \begin{figure}[h!]
    \center 
    \hspace*{0.3cm}\begin{overpic}[width=3.6cm, height=2cm, tics=10]{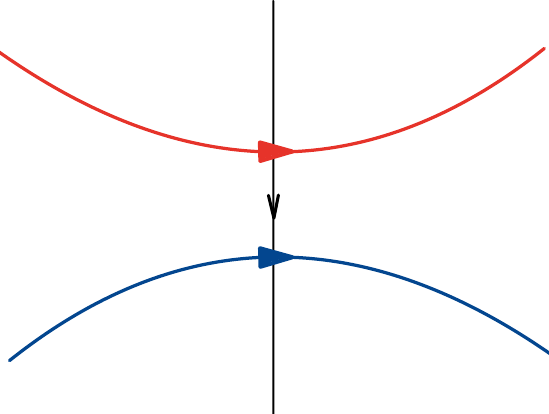}
        \put (37,55) {{\color{black}\large$\displaystyle \phi$}}
        \put (105,52) {{\color{myRED}\large$\displaystyle \Gamma_\O$}}
        \put (105,3) {{\color{myBLUE}\large$\displaystyle \Gamma_{\O^\pp}$}}
        \put (26,-14) {{\color{black}\large Common leaf}}
\end{overpic}
\hspace*{3.7cm}\begin{overpic}[width=3.6cm, height=2cm, tics=10]{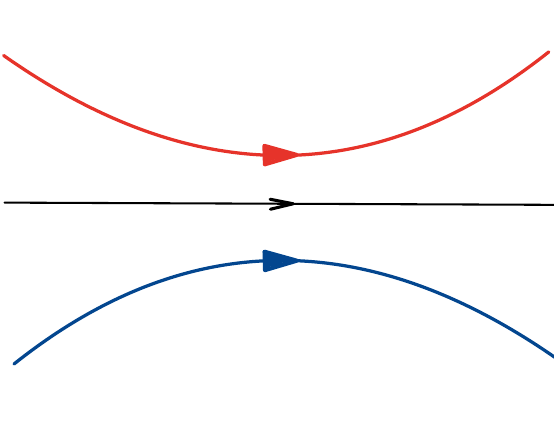}
    \put (-12.5,29) {{\color{black}\large$\displaystyle \phi$}}
    \put (105,52) {{\color{myRED}\large$\displaystyle \Gamma_\O$}}
        \put (105,3) {{\color{myBLUE}\large$\displaystyle \Gamma_{\O^\pp}$}}
    \put (33,-14) {{\color{black}\large Separated}}
\end{overpic}
\end{figure}

        \vspace*{0.4cm}
    
     \item[\textbf{(ii)}] If the orbits $\O$ and $\O^{\sspc\prime}$ satisfy $\partial_L C_\O \cap \partial_R C_{\O^{\sspc\prime}} \neq \varnothing$, then they must also satisfy
    $$ C_\O \cap C_{\O^{\sspc\prime}} = \varnothing \quad \text{ and } \quad \partial_L C_\O \cap \partial_R C_{\O^{\sspc\prime}} = \{\phi\}.$$

    \item[\textbf{(iii)}] If two orbits $\O$ and $\O^{\sspc\prime}$ satisfy $\partial_L C_\O \cap \partial_L C_{\O^{\sspc\prime}} \neq \varnothing$, then they also satisfy
    $$ C_\O \cap C_{\O^{\sspc\prime}} \neq \varnothing \quad \text{ and } \quad \partial_L C_\O = \partial_L C_{\O^{\sspc\prime}}.$$

    \begin{figure}[h!]
    \center
    \vspace*{-0.1cm}
    \begin{overpic}[width=6.6cm, height=3.9cm, tics=10]{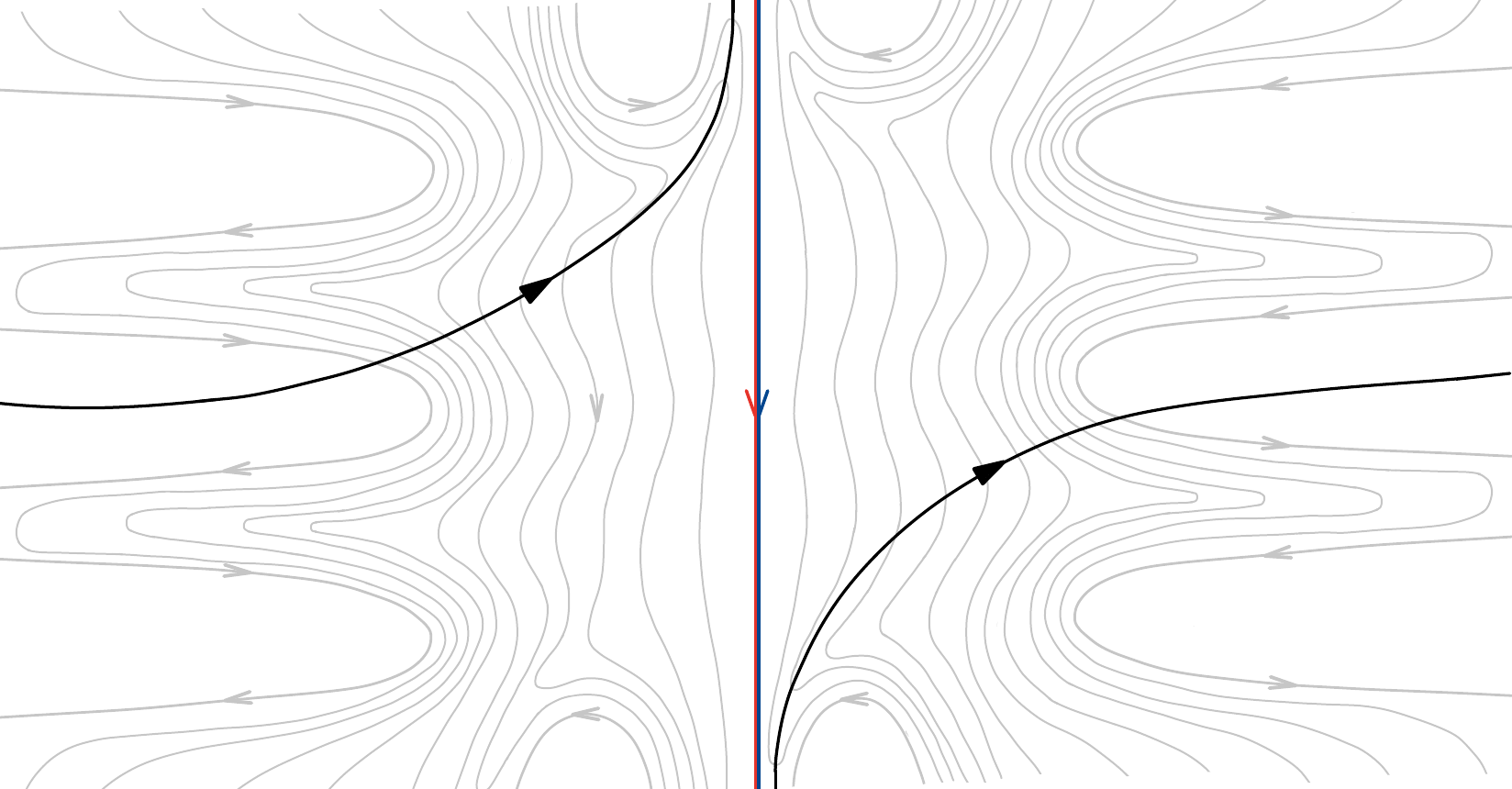}
        \put (52,27) {\colorbox{white}{$\rule{0cm}{0.3cm}\quad\ \ \ \spc $}}
        \put (33,27) {\colorbox{white}{$\rule{0cm}{0.3cm}\quad\ \ \  $}}
         \put (33,27.5) {{\color{myRED}$\displaystyle \fb$}}
         \put (52,27.2) {{\color{myBLUE}$\displaystyle \partial_R C_{\O'}$}}
        \put (-9,27) {\color{black}\large$\displaystyle \Gamma_\O$}
        \put (103,28.5) {\color{black}\large$\displaystyle \Gamma_{\O^\pp}$}
\end{overpic}
\hspace{1cm}
\begin{overpic}[width=6.6cm, height=3.9cm, tics=10]{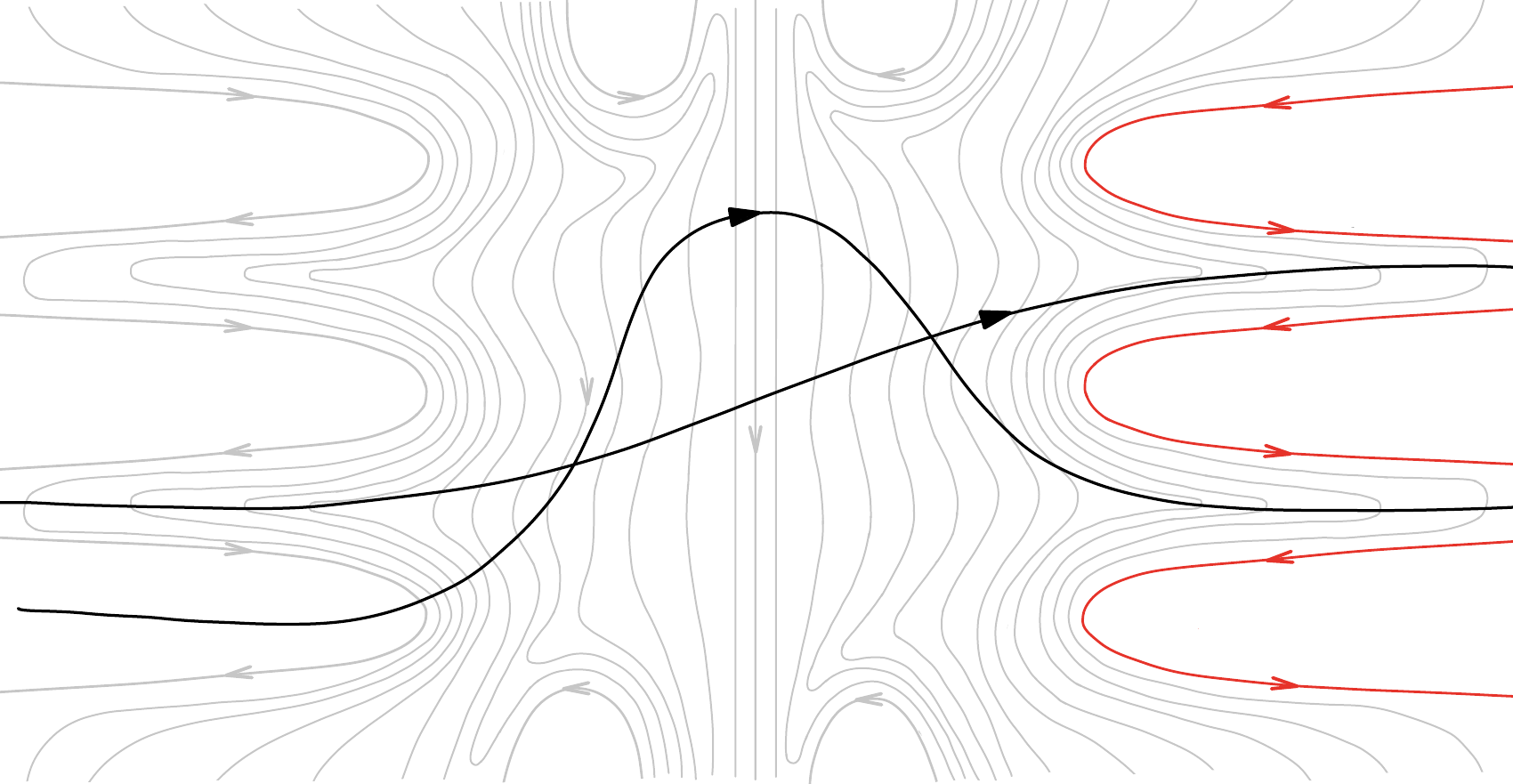}
        \put (52.9,23) {\colorbox{white}{$\rule{0cm}{0.3cm}\ \ \spc $}}
        \put (52.4,23) {{\color{black}\large$\displaystyle \Gamma_{\O^\pp} $}}
        \put (35,41.2) {\colorbox{white}{$\rule{0cm}{0.26cm}\ \spc \spc$}}
        \put (35,41) {{\color{black}\large$\displaystyle \Gamma_\O$}}
        
        \put (85,27.5) {\color{myRED}$\displaystyle \partial_L C_\O = \partial_L C_{\O^\pp}$}
\end{overpic}
\end{figure}

    \end{itemize}

\end{remark}

\subsection*{Definition of order $\prec$ on the sets $\fb$ and $\bb$.}\label{sec:order_limit_leaves}

Fix $\O \in \orb$ and a basepoint $x \in \O$. Since $\partial C_\O$ is a locally-finite collection of leaves disjoint from the closed set $\sspc\phi_0 \sspc\cup\sspc \O\sspc$, we can apply Lemma \ref{lemma:trivial_neighborhoods} to affirm that there exists a locally-finite family $\{V_\phi\}_{\phi \in \partial C_\O}$ of pairwise disjoint closed sets such that, for each $\phi \in \partial C_\O$, the set $V_\phi$ is a trivial-neighborhood of $\phi$ disjoint from $\sspc\phi_0 \cup \O\spc$.

For each $n\in\Z$, let $\phi_n \in \F$ be the leaf passing through the point $f^n(x)$, and define 
\mycomment{-0.12cm}
$$ J_n := \left\{\spc\phi \in \partial_L C_\O \cup \partial_R C_\O \mid V_\phi \cap \phi_n \neq \varnothing\spc\right\}.$$

\mycomment{-0.23cm}
\noindent
Observe that, by construction, $J_n \subset \partial_L C_\O$ when $n>0$, and $J_n \subset \partial_R C_\O$ when $n<0$.

\begin{figure}[h!]
    \center
    \vspace*{0.33cm}\begin{overpic}[width=6cm, height=5cm, tics=10]{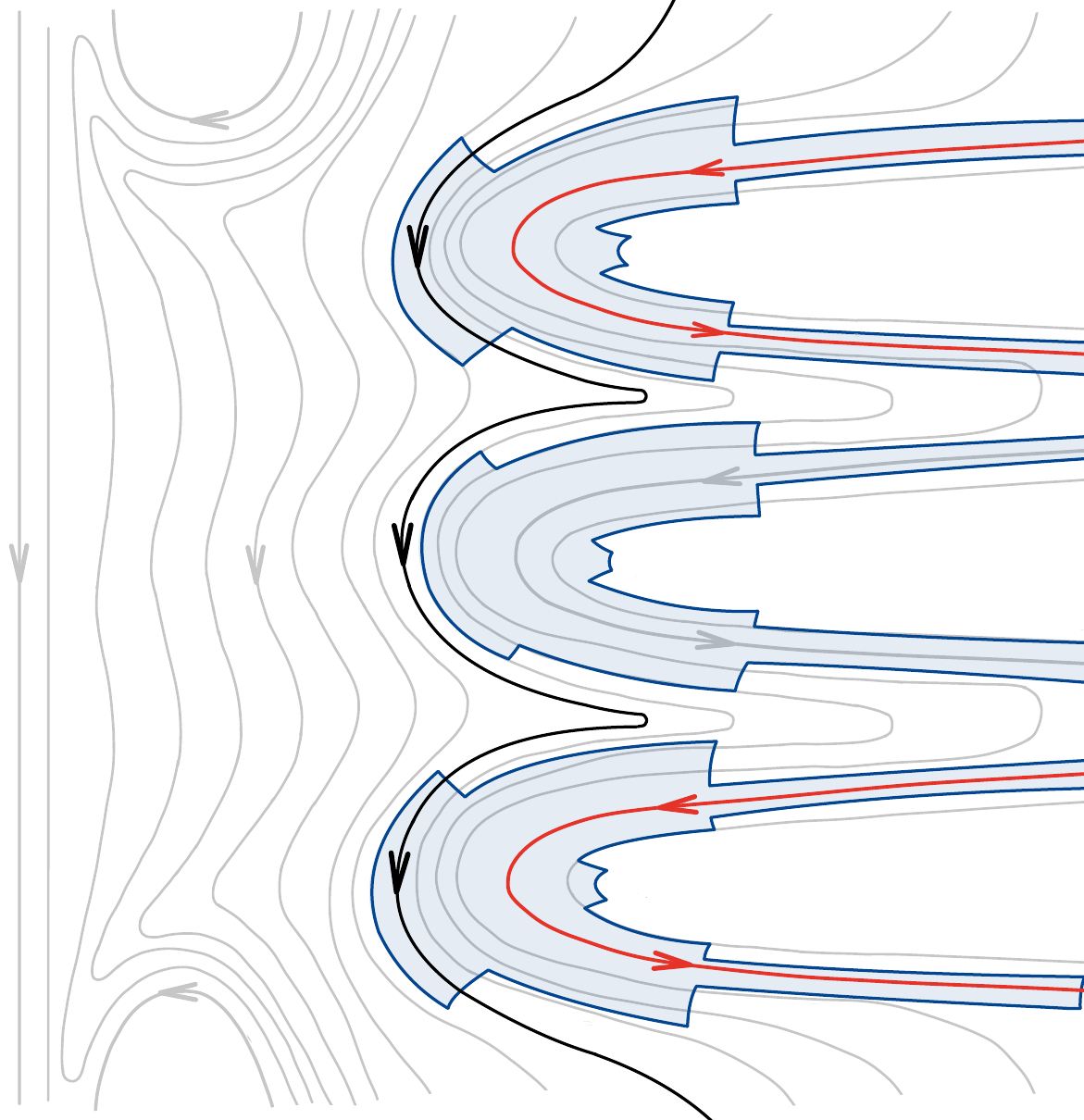}
        \put (59,87) {{\color{black}\large$\displaystyle \phi_n$}}
         \put (75,63) {{\color{myRED}\large$\displaystyle \phi \in J_n$}}
        \put (26,64.2) {\colorbox{white}{\color{myBLUE}\large$\displaystyle \rule{0cm}{0.3cm}\ \  $}}
        \put (27,64.9) {{\color{myBLUE}\large$\displaystyle V_\phi$}}
         \put (75,39.5) {{\color{myGRAY}\large$\displaystyle \phi'' \not\in J_n$}}
         \put (26,41.5) {\colorbox{white}{\color{myBLUE}\large$\displaystyle \rule{0cm}{0.3cm}\ \ \spc $}}
        \put (26,42) {{\color{myBLUE}\large$\displaystyle V_{\phi''}$}}
         \put (75,15.5) {{\color{myRED}\large$\displaystyle \phi' \in J_n$}}
         \put (24.5,15.2) {\colorbox{white}{\color{myBLUE}\large$\displaystyle \rule{0cm}{0.3cm}\ \  $}}
        \put (24.5,16) {{\color{myBLUE}\large$\displaystyle V_{\phi'}$}}
\end{overpic}
\end{figure}

\vspace*{-0.5cm}
\begin{lemma}
\label{lemma:Jn}
    The sequence of sets $(J_n)_{n\geq0}$ forms a filtration of $\partial_L C_\O$, in the sense that
    \begin{itemize}[leftmargin = 1.2cm]
        \item[\textup{\textbf{(i)}}] $J_0 = \varnothing$,
        \item[\textup{\textbf{(ii)}}]  $J_n \subset J_{n+1}\sspc$, for any $n\geq0$,
        \item[\textup{\textbf{(iii)}}] $\partial_L C_\O = \bigcup_{n\geq0} J^{n}$.
    \end{itemize}
    \mycomment{0.2cm}
    Similarly, the sequence of sets $(J_n)_{n\leq0}$ forms a filtration of the set $\partial_R C_\O$.
\end{lemma}

\mycomment{-0.3cm}
\begin{proof}
    We prove the lemma for the sequence $(J_n)_{n\geq0}$, the proof for $(J_n)_{n\leq0}$ is analogous.

Since the leaf $\phi_0$ is disjoint from each set in the family $\{V_\phi\}_{\phi \in \partial_L C_\O}$, it holds that $J_0 = \varnothing$. 
Next, as the sequence of leaves $(\phi_n)_{n\geq0}$ is increasing in $\F$, if some leaf $\phi_n$ intersects $V_\phi$ for a given integer $n \geq 0$, then the leaf $\phi_{n+1}$ also intersects $V_\phi$. In other words, it holds that
\begin{equation*}\label{eq:inc}
    J_n \subset J_{n+1}, \quad \forall n\geq0.
\end{equation*}
Finally, note that for each leaf $\phi \in C_\O$, there exists some $n>0$ such that $L(\phi_n) \subset L(\phi)$. Therefore, for each leaf $\phi \in \partial_L C_\O$, there exists the smallest positive integer $N_\phi>0$ satisfying 
$$\phi_{N_\phi} \cap V_\phi \neq \varnothing.$$
This implies that $\partial_L C_\O = \bigcup_{n\geq0} J^{n}$, and thus we conclude the proof of the lemma.
\end{proof}

\mycomment{-0.9cm}
First, note that, for each $n \in \Z$ and $\phi \in J_n$, the set $\alpha_n(\phi):=\phi_n \cap V_\phi$ is a non-empty compact sub-arc of the leaf $\phi_n$. This gives us, for each $n \in \Z$, a family $A_n:=\{\alpha_n(\phi)\}_{\phi \in J_n}$ of pairwise disjoint sub-arcs of the leaf $\phi_n$. Next, note that the orientation of $\phi_n$ induces a natural total order on the set $A_n$, which is defined as follows:
For any two arcs $\alpha,\alpha'\in A_n$, we write \(\alpha \prec_n \alpha'\) if, along the oriented leaf \(\phi_n\), the arc \(\alpha'\) is positioned after \(\alpha\). Through the bijection $\phi\longmapsto\alpha_n(\phi)\spc$ we import this order to  $J_n\sspc$. Given two leaves $\phi,\phi'\in J_n$, we say that
\begin{equation*}
    \text{$\phi \prec_n \phi'\sspc $  \quad if \quad  $\alpha_n(\phi) \prec_n \alpha_n(\phi')$.}
\end{equation*}

Observe that $\prec_n$ and $\prec_{n+1}$ are compatible with the inclusion $J_n \subset J_{n+1}$ for any $n\geq 0$. This means that, if two leaves $\phi,\phi'\in J_n$ satisfy $\phi \prec_n \phi'$, then they also satisfy $\phi \prec_{n+1} \phi'$. 
This ensures that the set
$ \partial_L C_\O = \bigcup_{n>0} J_{n} $
inherits a total order $\prec$ that extends the sets in the filtration $(J_n)_{n\geq0}$, in the sense that,  for any $\phi,\phi'\in \partial_L C_\O$, we define the relation
\begin{equation*}
    \phi \prec \phi'\quad \text{ if there exists } n>0 \text{ such that } \phi \prec_n \phi'.  
\end{equation*}

\begin{figure}[h!]
    \center
    \vspace*{0.12cm}\begin{overpic}[width=6cm, height=5cm, tics=10]{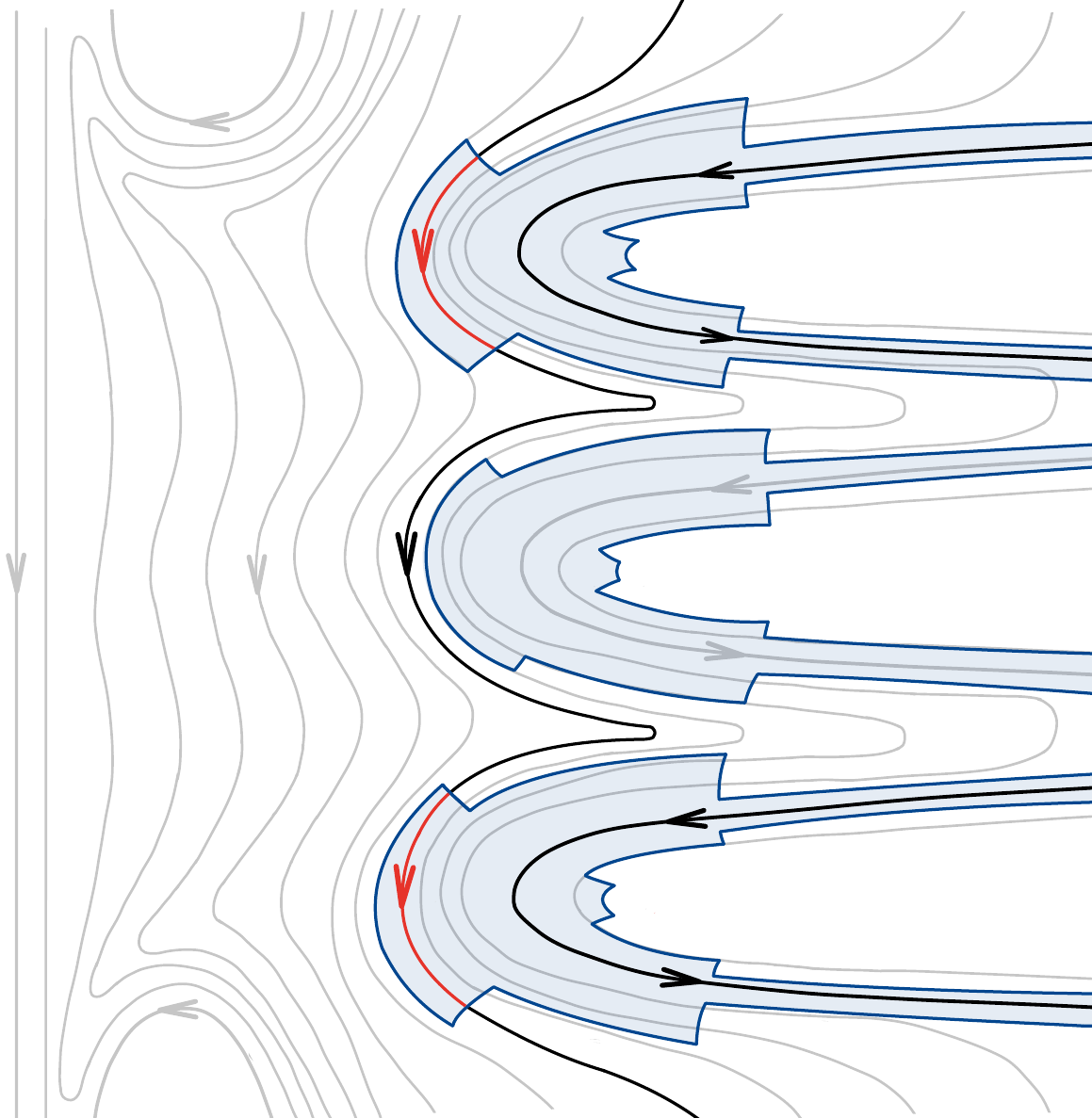}
        \put (57,86) {\colorbox{white}{\color{myBLUE}\large$\displaystyle \rule{0cm}{0.34cm}\ \ \ $}}
        \put (59,86) {{\color{black}\large$\displaystyle \phi_n$}}
         \put (75,62.5) {{\color{black}\large$\displaystyle \phi \in J_n$}}
        \put (18,63) {\colorbox{white}{\color{myBLUE}\large$\displaystyle \rule{0cm}{0.33cm}\quad \ \ \ $}}
        \put (18,63.5) {{\color{myRED}\large$\displaystyle \alpha_n(\phi)$}}
         \put (-56,39.8) {{\color{black}\large$\displaystyle \alpha_n(\phi) \prec_n \alpha_n(\phi')$}}
         \put (13,15.2) {\colorbox{white}{\color{myBLUE}\large$\displaystyle \rule{0cm}{0.34cm}\quad \ \ \ \ \spc $}}
        \put (14,15.5) {{\color{myRED}\large$\displaystyle \alpha_n(\phi')$}}
         \put (75,14) {{\color{black}\large$\displaystyle \phi' \in J_n$}}
\end{overpic}
\end{figure}

We remark that the set $\partial_R C_\O$ inherits an analogous total order $\prec$ from the filtration $(J_n)_{n\leq0}$.

\mycomment{-0.1cm}
\begin{remark}
    There are two important properties of the order $\prec$ that we want to highlight:
    \begin{itemize}[leftmargin = 1cm]
    \item[\textbf{(i)}] The definition of $\prec$ does not change if the trivial-neighborhoods in the family $\{V_\phi\}_{\phi \in \partial C_\O}$ are replaced by smaller ones. This ensures that our construction is independent of the specific choice of trivial-neighborhoods, as long as they are taken small enough.

    \item[\textbf{(ii)}] By taking sufficiently small cross-sections $\sigma_\phi\subset V_\phi$ over each leaf $\phi \in \partial_L C_\O$, and then extending them until they reach the leaf $\phi_0$, we can obtain a family $\{\gamma_\phi\}_{\phi \in \partial_L C_\O}$ of pairwise disjoint paths $\gamma_\phi:[0,1]\longrightarrow\R^2$ positively transverse to the foliation $\F$ that joins the points $\gamma_\phi(0) \in \phi_0$ and $\gamma_\phi(1) \in \phi\sspc$. This allows us to express the order $\prec$ in terms of the order of the endpoints $\{\gamma_\phi(0)\}_{\phi \in \partial_L C_\O}$ along the leaf $\phi_0$, meaning that
    $$ \phi \prec \phi' \iff \gamma_\phi(0) < \gamma_{\phi'}(0) \text{ according to the orientation of } \phi_0.$$
    \end{itemize} 
    \end{remark}

\subsection{Orbit induced cuts}\label{sec:convergence_cuts}

Let $\text{Cut}_\prec(\partial_L C_\O)$ be the set of cuts of $(\partial_L C_\O,\prec)$, that is, the set of all pairs $(A,B)$ such that $A\sspc,\sspc B \subset \partial_L C_\O$ are disjoint subsets of $\partial_L C_\O$ satisfying
\begin{itemize}[leftmargin = 1.5cm]
    \item $A \cup B = \partial_L C_\O$,
    \item For any leaf $\phi \in A$ and any leaf $\phi' \in B$, it holds that $\phi \prec \phi'$.
\end{itemize}   
\mycomment{0.2cm}
Similarly, let $\text{Cut}_\prec(\partial_R C_\O)$ denote the set of all cuts of the totally ordered set $(\partial_R C_\O,\prec)$.


\begin{proposition}\label{prop:orbit_cuts}
    Every orbit $\O \in \orb$ admits two canonically defined cuts
    \begin{align*}
    (\sspc\partial_L^\text{\sspc top\sspc } \O\sspc,\spc\partial_L^\text{\sspc bot\sspc } \O\sspc) \in \textup{Cut}_\prec(\partial_L C_\O) \quad \ \text{ and } \ \quad 
    (\sspc \partial_R^\text{\sspc top\sspc } \O\sspc,\spc\partial_R^\text{\sspc bot\sspc } \O\sspc) \in \textup{Cut}_\prec(\partial_R C_\O).
\end{align*}
\begin{figure}[h!]
    \center
    \vspace*{-0.4cm}\begin{overpic}[width=9cm, height=4.2cm, tics=10]{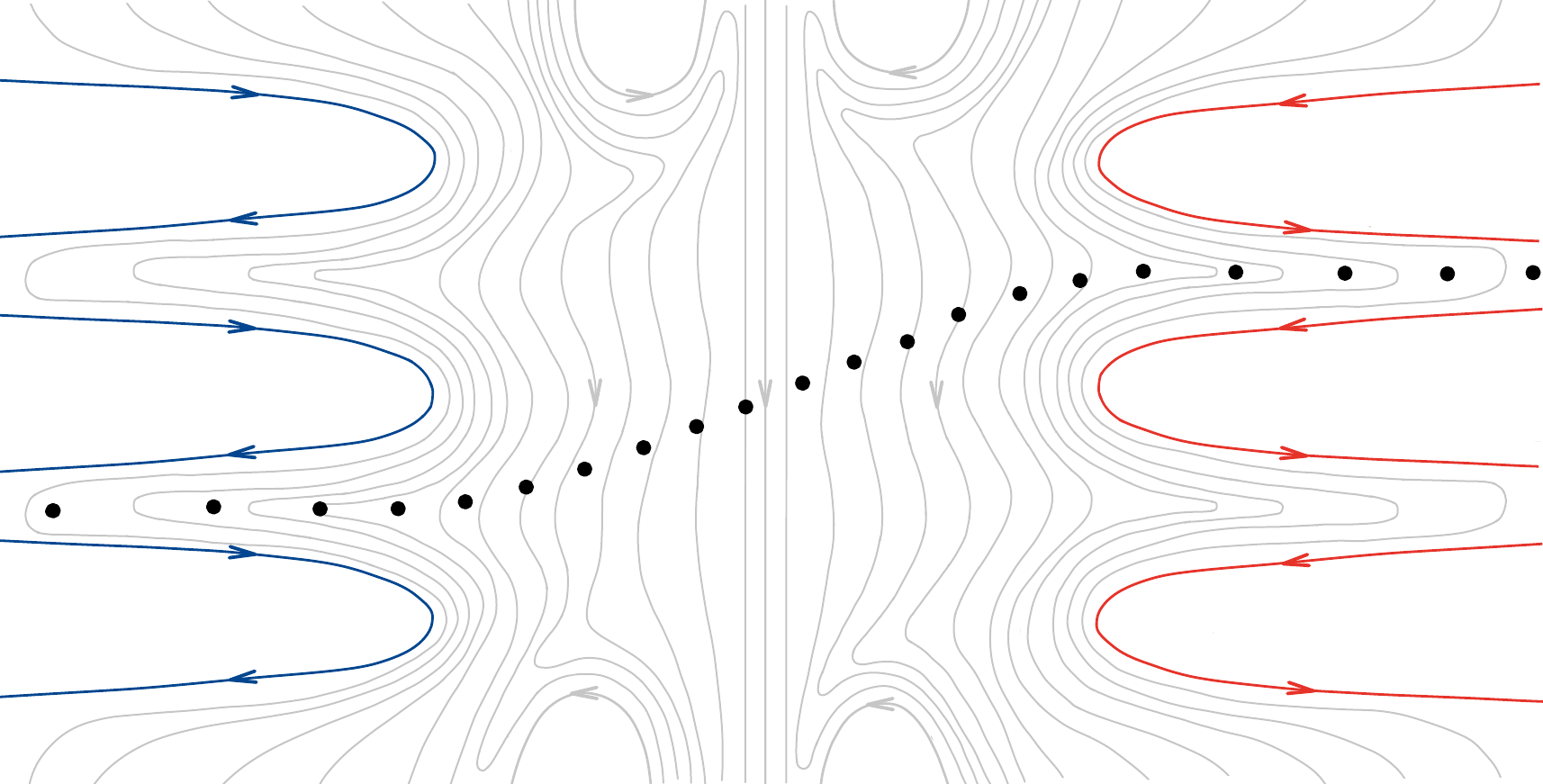}
        \put (54.4,15) {\colorbox{white}{$\rule{0cm}{0.3cm}\quad $}}
        \put (55,15.3) {{\color{myGRAY}\large$\displaystyle C_\O $}}
        \put (43.5,25) {\colorbox{white}{\color{myDARKGRAY}\large$\rule{0cm}{0.26cm}\  $}}
         \put (43.5,25) {{\color{black}\large$\displaystyle \O$}}
        \put (84.5,21.7) {\color{myRED}\large$\displaystyle \Lbot \O $}
        \put (84.5,35.4) {\color{myRED}\large$\displaystyle \Ltop \O $}
        \put (84.5,8) {\color{myRED}\large$\displaystyle \Lbot \O $}
        \put (4,21.5) {\color{myBLUE}\large$\displaystyle \Rtop \O $}
        \put (4,35.6) {\color{myBLUE}\large$\displaystyle \Rtop \O $}
        \put (4,8) {\color{myBLUE}\large$\displaystyle \Rbot \O $}
\end{overpic}
\end{figure}
\end{proposition}

\vspace*{-0.9cm}
\newpage

\mycomment{-0.5cm}
It is worth remarking the sets $\partial_L^\text{\sspc top\sspc } \O$, $\partial_L^\text{\sspc bot\sspc } \O$, $\partial_R^\text{\sspc top\sspc } \O$ and $\partial_R^\text{\sspc bot\sspc } \O$ depend not only on the orbit $\O$, but also on the foliation $\F$. However, we often omit the dependence on $\F$ in the notation, as it is usually clear from the context.

The leaf domain $C_\O$ and the cuts $(\sspc\partial_L^\text{\sspc top\sspc } \O\sspc,\sspc\partial_L^\text{\sspc bot\sspc } \O\sspc)$ and $(\sspc\partial_R^\text{\sspc top\sspc } \O\sspc,\sspc\partial_R^\text{\sspc bot\sspc } \O\sspc)$ describe the asymptotic behavior of the orbit $\O$ with respect to the foliation $\F$. The qualitative data associated with these objects is strictly greater than the one given by the set $C_\O$ alone. Throughout this work, we refer to this qualitative data as the $\F$\textit{-asymptotic behavior} of $\O$.

Later, in Section \ref{sec:proper_transverse_trajectories}, we show that the cuts $(\sspc\partial_L^\text{\sspc top\sspc } \O\sspc,\sspc\partial_L^\text{\sspc bot\sspc } \O\sspc)$ and $(\sspc\partial_R^\text{\sspc top\sspc } \O\sspc,\sspc\partial_R^\text{\sspc bot\sspc } \O\spc)$ can be interpreted by proper transverse trajectories associated with the orbit $\O$, which are transverse trajectories that are topological lines on the plane. These proper transverse trajectories completely describe the $\F$-asymptotic behavior of the orbit $\O$, differently from an arbitrary (non-proper) transverse trajectory, which only encodes the information carried by $C_\O$.

\begin{proof}
    We begin by considering an orbit $\O \in \orb$, a basepoint $x \in \O$, a family  $\{V_\phi\}_{\phi \in \partial C_\O}$ of trivial-neighborhoods, the leaves $(\phi_n)_{n\in\Z}$, and the sets $(J_n)_{n\in\Z}$ as defined in Section \ref{sec:order_limit_leaves}.

Observe that, by cutting each leaf $\phi_n$ at the point $f^n(x)$, we obtain two halflines
\mycomment{-0.12cm}
\begin{align*}
    \phi^+_n:\sspc[\spc\sspc0,\infty\sspc)\spc\longrightarrow\R^2\sspc \quad \text{ and } \quad
    \phi^-_n:(-\infty,0\sspc]\longrightarrow\R^2,
\end{align*}

\mycomment{-0.23cm}
\noindent
with endpoints $\phi^+_n(0) = \phi^-_n(0) = f^n(x)\spc$ and satisfying $\spc\phi^+_n\sspc,\spc \phi^-_n \subset \phi_n$.
Recall that each set in the family $\{V_\phi\}_{\phi \in \partial C_\O}$ is disjoint from  the orbit $\O$. Consequently, for each  \(n \in \Z\) and \(\phi \in J^n\), the arc \(\alpha_n(\phi) := \phi_n \cap V_\phi\) must lie entirely within one of the two half-lines, either \(\phi_n^+\) or \(\phi_n^-\). This allows us to define, for any $n \in \Z$, a cut $(J_n^-,J_n^+) \in \text{Cut}_\prec(J_n)$ given by the sets
\mycomment{-0.12cm}
\begin{align*}
    J_n^-:=\{\phi \in J_n \mid \alpha^\phi_n \subset \phi_n^-\}\quad \text{ and } \quad
    J_n^+:=\{\phi \in J_n \mid \alpha^\phi_n \subset \phi_n^+\}.
\end{align*}

\begin{figure}[h!]
    \center
    \mycomment{0cm}\begin{overpic}[width=6cm, height=5cm, tics=10]{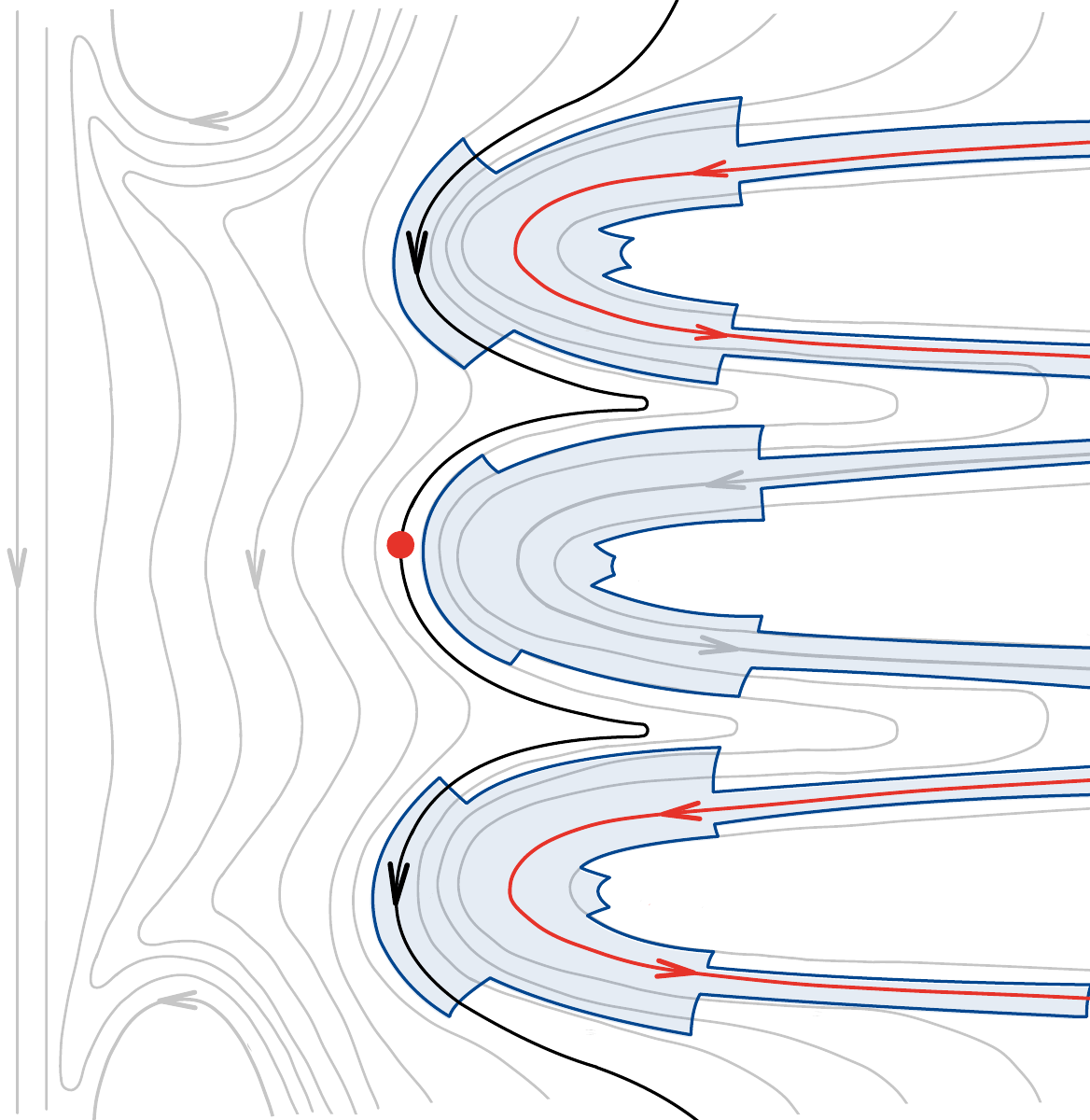}
        \put (58,86.5) {{\color{black}\large$\displaystyle \phi_n$}}
         \put (75,63.3) {{\color{myRED}\large$\displaystyle \phi \in J_n^-$}}
        \put (25,63.5) {\colorbox{white}{\color{myBLUE}\large$\displaystyle \rule{0cm}{0.33cm}\quad $}}
        \put (26,64) {{\color{myBLUE}\large$\displaystyle V_\phi$}}
        \put (15,39) {\colorbox{white}{\color{myBLUE}\large$\displaystyle \rule{0cm}{0.4cm}\quad \quad \sspc $}}
         \put (16,40) {{\color{myRED}\large$\displaystyle f^n(x)$}}
         \put (21.5,15.2) {\colorbox{white}{\color{myBLUE}\large$\displaystyle \rule{0cm}{0.36cm}\quad $}}
        \put (22.5,16) {{\color{myBLUE}\large$\displaystyle V_{\phi'}$}}
         \put (75,15) {{\color{myRED}\large$\displaystyle \phi' \in J_n^+$}}
\end{overpic}
\end{figure}


\mycomment{-0.2cm}
\begin{claim}\label{claim:convergence}
    For each leaf $\phi \in \fb$, there exists $M_\phi>0$ which is the smallest positive integer such that, for any $n \geq M_\phi$, the following two equivalences hold
    \mycomment{-0.12cm}
    \begin{align*}
        \phi \in J_n^+ \iff \phi \in J_{n+1}^+, \\
        \phi \in J_n^- \iff \phi \in J_{n+1}^-.
    \end{align*}

    \mycomment{-0.23cm}
    \noindent Similarly, for each leaf $\phi \in \bb$, there exists $M_\phi<0$ satisfying an analogous property.
\end{claim}

\mycomment{-0.2cm}
\begin{proof}[Proof of Claim \ref{claim:convergence}]
    Fix a leaf $\phi \in \partial_L C_\O$ and denote by $\phi_R \in C_\O$ the right-most leaf that intersects the trivial-neighborhood $V_\phi$. The open set $L(\phi_R) \cap R(\phi)\setminus V_\phi$ has exactly two connected components, which we name $U^+$ and $U^-$ according to the following equivalences:
    \mycomment{-0.12cm}
    \begin{align*}
        \phi\in J_n^+ \iff f^n(x) \in U^+,\\
        \phi\in J_n^- \iff f^n(x) \in U^-.
    \end{align*}

    \mycomment{-0.2cm}
    \indent To simplify the argument, we can profit from Homma-Schoenflies theorem to suppose,  up to conjugacy, that the topological lines $\phi_R$, $\phi$ and $f(\phi)$ are verticals on the plane oriented downwards, the set $U^-$ lies on the upper half-plane, the set $U^+$ lies on the lower half-plane,  and that $f$ maps the points of $\phi$ into $f(\phi)$ as a horizontal translation. 
    This setting is illustrated in the figure below. We represent $V_\phi$ disjoint from $f(\phi)$ because $V_\phi$ is considered small enough, however our proof does not depend on this property, and it is possible that $V_\phi$ intersects $f(\phi)$.

    \begin{figure}[h!]
        \center\begin{overpic}[height=4.8cm, width=4.5cm,tics=10]{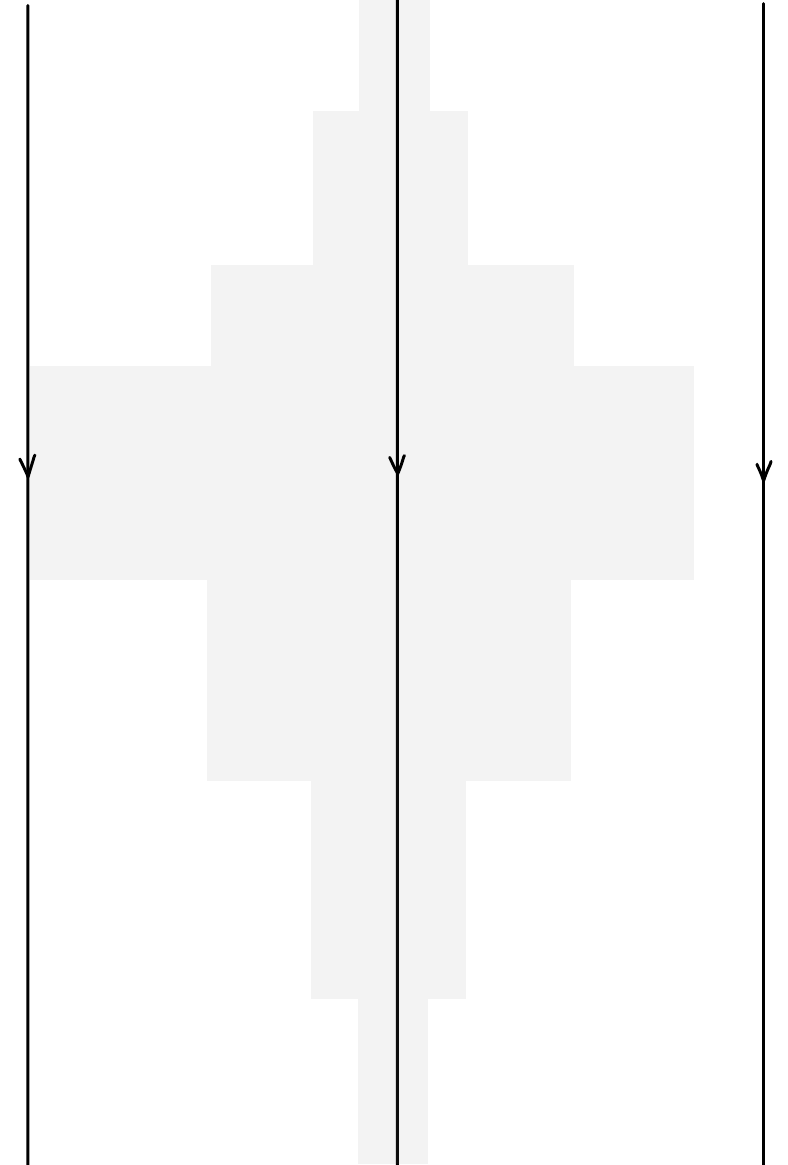}
            \put (54,58) {{\color{black}\large$\displaystyle \phi$}}
            \put (25,58) {{\color{myDARKGRAY}\large$\displaystyle V_\phi$}}
            \put (-14,58) {{\color{black}\large$\displaystyle \phi_R$}}
            \put (99,58) {{\color{black}\large$\displaystyle f(\phi)$}}
            \put (6,88) {{\color{black}\;\large$\displaystyle U^- $}}
            \put (6,10) {{\color{black}\;\large$\displaystyle U^+$}}
    \end{overpic}
    \end{figure}

    \mycomment{-0.4cm}
 Now, we suppose by contradiction that claim of the lemma does not hold. More precisely, we suppose that there exists an increasing sequence of integers $(n_k)_{k \in \N}$ satisfying
    \begin{equation*}
    \phi \in J_{n_k}^- \bigcap J_{n_k + 1}^+.
    \end{equation*}
    According to the equivalence that defines $U^-$ and $U^+$, for every $k \in \N$, it holds that
    $$ z_k:=f^{n_k}(x) \in U^+ \quad \text{ and } \quad f(z_k) \in U^-.$$

    \mycomment{-0.1cm}
    \noindent Since $f$ is a Brouwer homeomorphism, $(z_k)_{k\in\N}$ has no accumulation points on the plane.  As the set $U^+$ is bounded in the horizontal coordinate and contained in the lower half-plane,  the sequence $(z_k)_{k\in\N}$ is unbounded from below on the vertical coordinate. Up to reconsidering a subsequence, we may suppose that $(z_k)_{k\in\N}$ is strictly decreasing on the vertical coordinate. Consequently, the family $\{\alpha_k\}_{k \in \N}$ of horizontal arcs $\alpha_k$ that join $z_k$ to a point $y_k \in \phi$ is,  by construction, locally-finite. This setting is illustrated in the figure below.

    \mycomment{0.3cm}
    \begin{figure}[h!]
        \center\begin{overpic}[height=4.8cm,width=4.5cm, tics=10]{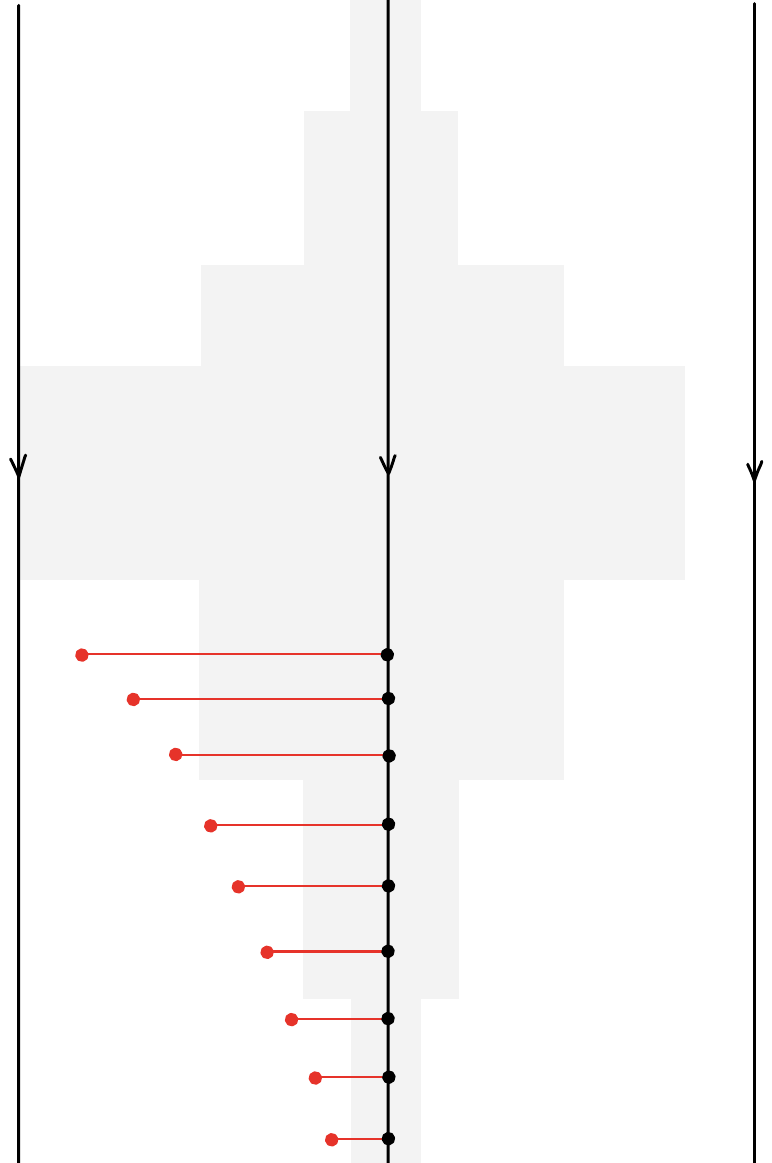}
            \put (53,58) {{\color{black}\large$\displaystyle \phi$}}
            \put (24,48) {{\color{myRED}\large$\displaystyle (\alpha_k)_{k}$}}
            \put (-14,58) {{\color{black}\large$\displaystyle \phi_R$}}
            \put (97,58) {{\color{black}\large$\displaystyle f(\phi)$}}
            \put (6,21) {{\color{myRED}\;\large$\displaystyle {(z_k)}_{k}$}}
            \put (52,4) {{\color{black}\;\large$\displaystyle {(y_k)}_{k}$}}
    \end{overpic}
    \end{figure}
    
    \mycomment{-0.1cm}
    Let $K\subset \R^2$ be the compact arc of zero vertical coordinate connecting $\phi_R$ and $f(\phi)$.  From the hypothesis given by the Homma-Schoenflies theorem, for any $k \in \N$, we have that the point $y_k \in \phi $ and its image $f(y_k) \in f(\phi)$ have exactly the same horizontal coordinate. Consequently, the sequence $(f(y_k))_{k \in \N}$ lies on the lower half-plane, while $(f(z_k))_{k \in \N}$ lies on the upper half-plane. Since the arcs in the family $\{\alpha_k\}_{k\in \N}$ are contained in the set $\overline{\rule{0cm}{0.34cm}L(\phi_R) \cap R(\phi)}$, their images $\{f(\alpha_k)\}_{k \in \N}$ are contained in $\overline{L(\phi_R) \cap R(f(\phi))}$. Note that, since $f(\alpha_k)$ joins the point $f(y_k)$ to $f(z_k)$, which are respectively in the lower and upper half-planes, each arc in the family $\{f(\alpha_k)\}_{k \in \N}$ must intersect the compact set $K$. In other words, $\{f(\alpha_k)\}_{k \in \N}$ is not a locally-finite family of arcs. This setting is illustrated in the figure below.

    \begin{figure}[h!]
        \center
        \mycomment{0.4cm}\begin{overpic}[height=4.8cm,width=4.5cm, tics=10]{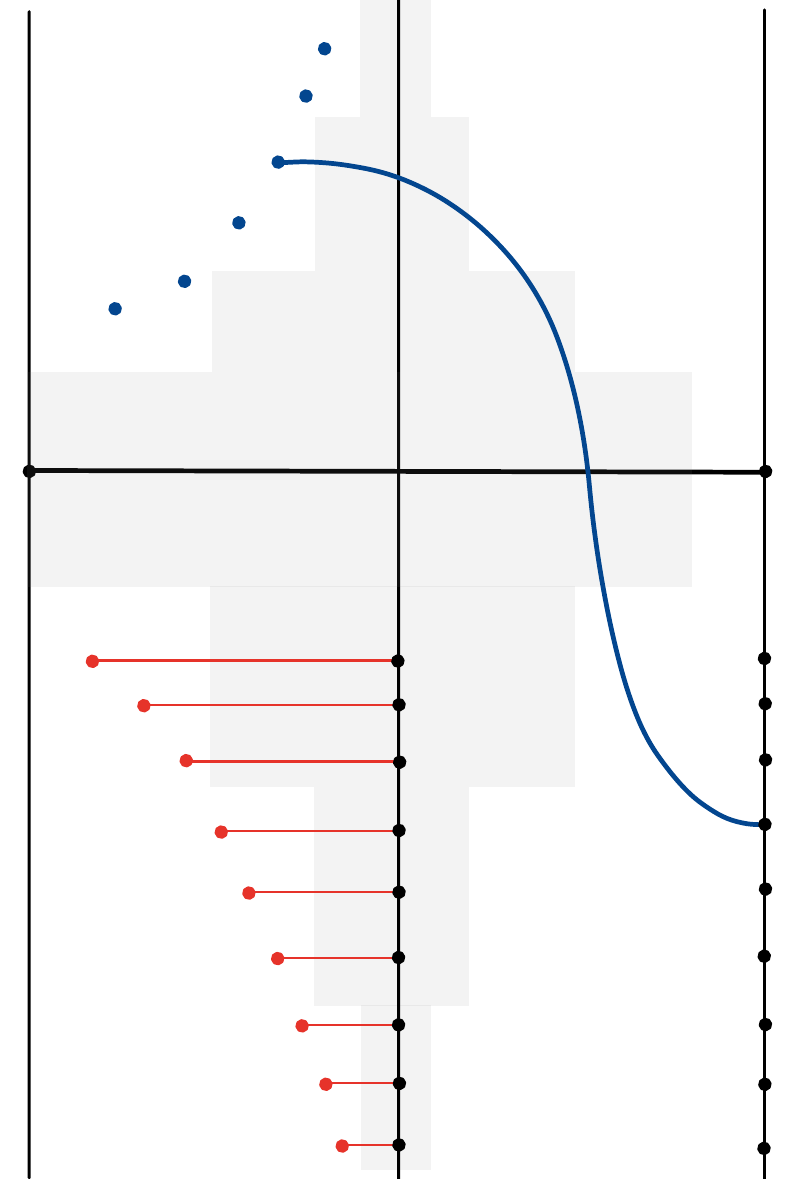}
            \put (63.5,85) {{\color{myBLUE}\large$\displaystyle f(\alpha_k)$}}
            \put (30,63) {{\color{black}\large$\displaystyle K$}}
            \put (-14,58) {{\color{black}\large$\displaystyle \phi_R$}}
            \put (100,58) {{\color{black}\large$\displaystyle f(\phi)$}}
            \put (5,100) {{\color{myBLUE}\;\large$\displaystyle {(f(z_k))}_{k}$}}
            \put (97,4) {{\color{black}\;\large$\displaystyle {(f(y_k))}_{k}$}}
    \end{overpic}
    \end{figure}
    \mycomment{-0.05cm}
    However, since $f$ is a homeomorphism and $\{\alpha_k\}_{k \in \N}$ is locally-finite, $\{f(\alpha_k)\}_{k \in \N}$ should be, in fact, locally-finite. 
    This gives us a contradiction and, thus, concludes the proof.
    \end{proof}

    \mycomment{-0.1cm}
    Due to Claim \ref{claim:convergence}, we can now define the following sets
    \begin{align*}
    \partial_L^\text{\sspc top\sspc } \O  := \{\phi \in \fb \mid \exists M>0,\spc \forall n>M,\spc \phi \in J_n^+ \},\\
    \partial_L^\text{\sspc bot\sspc } \O:=\{\phi \in \fb \mid \exists M>0,\spc \forall n>M,\spc \phi \in J_n^- \}.
\end{align*}
By construction, we have that $(\sspc\partial_L^\text{\sspc top\sspc } \O\sspc,\spc\partial_L^\text{\sspc bot\sspc } \O\sspc) \in \textup{Cut}_\prec(\partial_L C_\O)$. We can define the other two sets $\partial_R^\text{\sspc top\sspc } \O$ and $\partial_R^\text{\sspc bot\sspc } \O$ in a similar way, thus concluding the proof of the proposition.
\end{proof}

\subsection{Proper transverse trajectories}\label{sec:proper_transverse_trajectories}

A proper transverse trajectory \(\Gamma_\O\) of \(\O \in \orb\) is an oriented topological line, positively transverse to the foliation \(\F\), whose image contains  \(\O\).


\addtocounter{thmx}{-3}
\begin{thmx}\label{thmx:proper_trajectories_restate}
Every orbit $\O \in \orb$ admits a proper transverse trajectory $\Gamma_\O$. Moreover, any proper transverse trajectory $\Gamma_\O$ associated to the orbit $\O$ satisfies
\begin{align*}
    \Ltop \O = \partial_L C_\O \cap L(\Gamma_\O)\spc, \quad \Rtop \O = \partial_L C_\O \cap L(\Gamma_\O), \\[0.1ex]  \Lbot \O = \partial_R C_\O \cap R(\Gamma_\O)\spc, \quad \Rbot \O = \partial_R C_\O \cap R(\Gamma_\O).
\end{align*}
\begin{figure}[h!]
    \center
    \vspace*{-0.5cm}\begin{overpic}[width=8.4cm, height=4.1cm, tics=10]{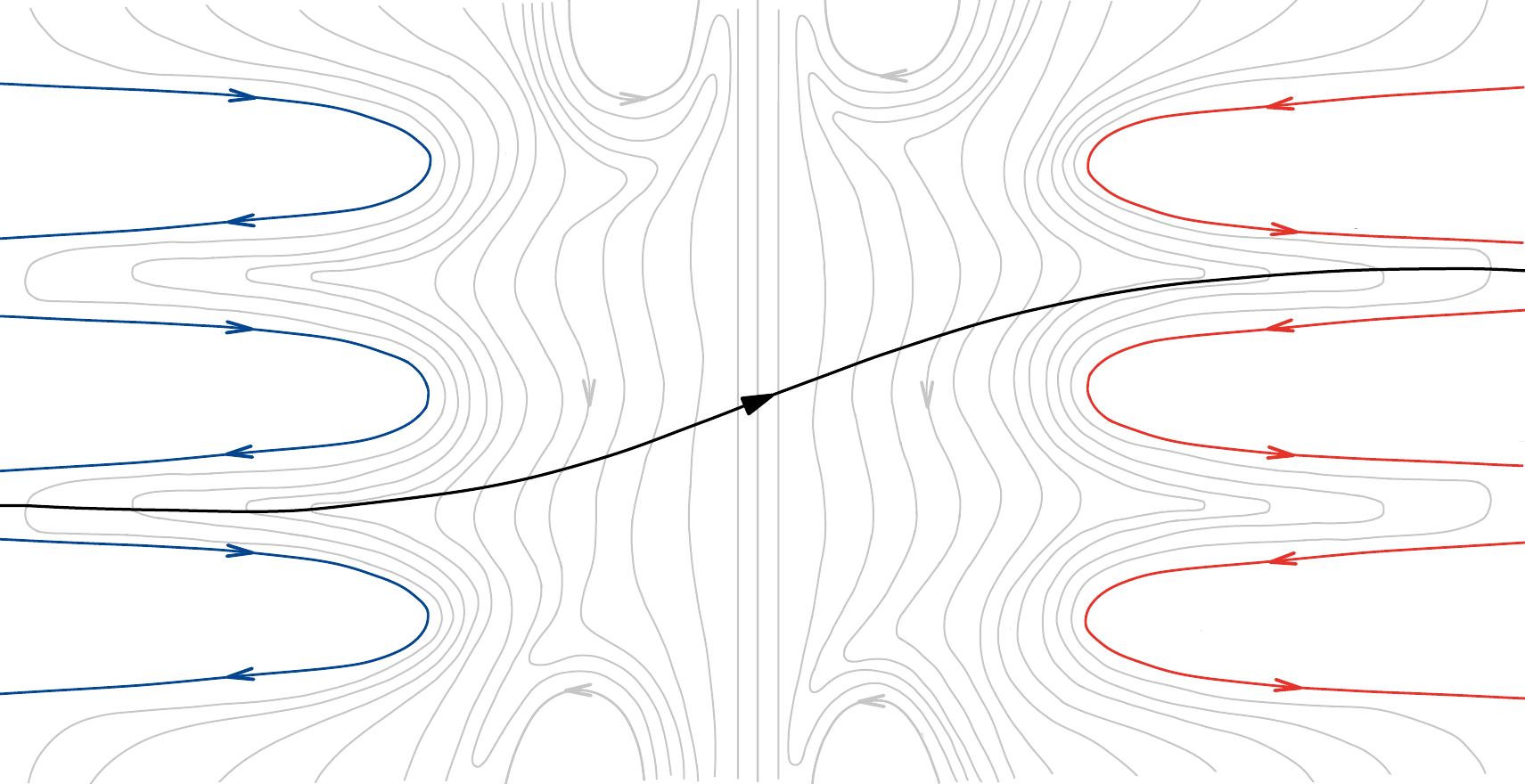}
        \put (54.4,15) {\colorbox{white}{$\rule{0cm}{0.3cm}\ \ \  $}}
        \put (55.5,15.3) {{\color{myGRAY}\large$\displaystyle C_\O $}}
        \put (41.5,26) {\colorbox{white}{\color{myDARKGRAY}\large$\rule{0cm}{0.3cm}\ \spc $}}
        \put (42,26) {{\color{black}\large$\displaystyle \Gamma_\O$}}
         \put (85,22.5) {\color{myRED}\large$\displaystyle \Lbot \O $}
        \put (85,37) {\color{myRED}\large$\displaystyle \Ltop \O $}
        \put (85,8) {\color{myRED}\large$\displaystyle \Lbot \O $}
        \put (3,22.5) {\color{myBLUE}\large$\displaystyle \Rtop \O $}
        \put (3,37) {\color{myBLUE}\large$\displaystyle \Rtop \O $}
        \put (3,8.2) {\color{myBLUE}\large$\displaystyle \Rbot \O $}
\end{overpic}
\end{figure}
   
\end{thmx}

\vspace*{-0.6cm}
\begin{proof} We begin by considering an orbit $\O \in \orb$, a basepoint $x \in \O$, a family  $\{V_\phi\}_{\phi \in \partial C_\O}$ of trivial-neighborhoods, the leaves $(\phi_n)_{n\in\Z}$, and the sets $(J_n)_{n\in\Z}$ as defined in Section \ref{sec:order_limit_leaves}.
Recall that in the proof of Lemma \ref{lemma:Jn} and Proposition \ref{prop:orbit_cuts} we defined, for each leaf $\phi \in \fb$, two integers $N_\phi>0$ and $M_\phi>0$, which are the smallest positive integers satisfying
\begin{align*}
    \phi_{n} \cap V_\phi \neq \varnothing, \quad \forall n\geq N_\phi, \quad \quad \  \\
    \phi \in J_n^+ \iff \phi \in J_{n+1}^+, \quad \forall n\geq M_\phi.
\end{align*}
Observe that, for $n \geq M_\phi$, the following two scenarios \textbf{cannot occur}:
$$ \text{ (a) $\alpha^\phi_n \subset \phi_n^+$ and $\alpha^\phi_{n+1} \subset \phi_{n+1}^-$ } ;\quad \quad   \text{ (b) $\alpha^\phi_n \subset \phi_n^-$ and $\alpha^\phi_{n+1} \subset \phi_{n+1}^+$. } $$
\noindent This means that, for each $n\geq M_\phi$, there exists a path $\gamma_{n,\phi}:[0,1]\longrightarrow\R^2$ positively transverse to the foliation $\F$ that joins $f^n(x)$ to $f^{n+1}(x)$ and is disjoint from the neighborhood $V_\phi$. 

\vspace*{0.5cm}
\begin{figure}[h!]
    \center
    \hspace*{-0.3cm}\begin{overpic}[width=5cm, height=4.4cm, tics=10]{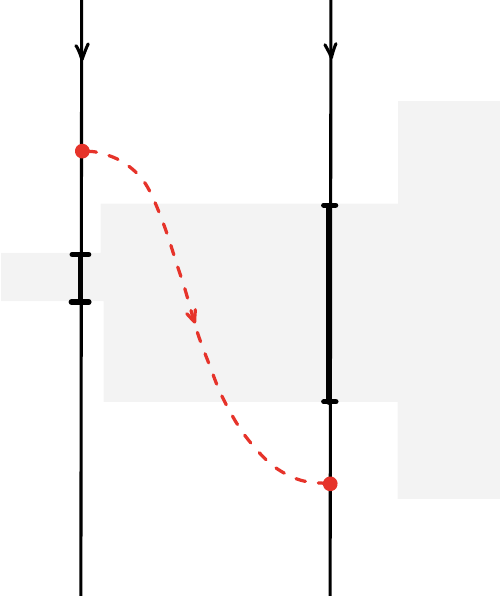}
        \put (25,42) {\color{myRED}\large$\displaystyle \gamma_n $}
        \put (67.5,15) {\colorbox{white}{\color{myRED}\large$\displaystyle f^{n+1}(x) $}}
        \put (67.5,68) {\colorbox{white}{\color{myRED}\large$\displaystyle\rule{0cm}{0.7cm}\quad \quad \quad \quad  $}}
        \put (-12,69) {\color{myRED}\large$\displaystyle f^n(x) $}
        \put (16,-14) {\color{black}\large$\displaystyle \phi \in J_n^- \cap J_{n+1}^+$}
        \put (-22,100) {\color{black}\large\quad \ \  (example of what cannot occur)}
        \put (47,47) {\color{myDARKGRAY}\large$\displaystyle V_\phi $}
        \put (2,87) {\color{myDARKGRAY}$\displaystyle \phi_n $}
        \put (70,87) {\color{myDARKGRAY}$\displaystyle \phi_{n+1} $}
\end{overpic}
\hspace*{2.5cm}\begin{overpic}[width=5cm, height=4.4cm, tics=10]{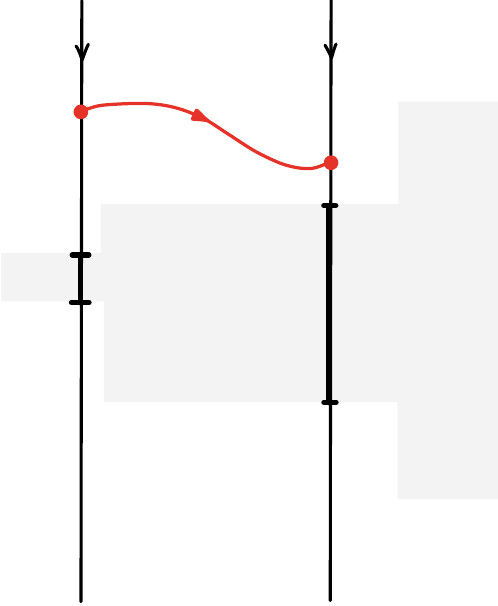}
    \put (38,79) {\color{myRED}\large$\displaystyle \gamma_n $}
        \put (67.5,69) {\colorbox{white}{\color{myRED}\large$\displaystyle \rule{0cm}{0.7cm}f^{n+1}(x) $}}
        \put (-11,70) {\color{myRED}\large$\displaystyle f^n(x) $}
        \put (16,-14) {\color{black}\large$\displaystyle \phi \in J_n^- \cap J_{n+1}^-$}
        \put (67.5,15) {\colorbox{white}{\color{myRED}\large$\displaystyle\rule{0cm}{0.5cm}\quad \quad \quad \quad $}}
        \put (-13,100) {\color{black}\large \quad \ \ (example of what occurs)}
        \put (43,43) {\color{myDARKGRAY}\large$\displaystyle V_\phi $}
        \put (2,87) {\color{myDARKGRAY}$\displaystyle \phi_n $}
        \put (70,87) {\color{myDARKGRAY}$\displaystyle \phi_{n+1} $}
\end{overpic}
\end{figure}

\vspace*{0.6cm}
The first step of the proof is to show that, by modifying the trivial-neighborhoods in the family $\{V_\phi\}_{\phi \in \partial C_\O}$, we can ensure that $N_\phi = M_\phi$ for all $\phi \in \partial C_\O$. 
Indeed, to show this, one should simply observe that each leaf $\phi \in \partial_L C_\O$ also admits the trivial-neighborhood 
\begin{equation}\label{eq:rename}
    V'_\phi:= V_\phi \cap \overline{\rule{0cm}{0.38cm}L(\phi^{M_\phi})}.
\end{equation}
Thus, by redefining the original $V_\phi$ to be instead the newly defined trivial-neighborhood $V'_\phi$, we automatically get the property that $N_\phi = M_\phi$. Therefore, we conclude that after this redefinition, we know that for each $n\geq 0$ and $\phi\in \partial_L C_\O$, the points $f^n(x)$ and $f^{n+1}(x)$ can be connected by a path $\gamma_{n,\phi}$ positively transverse to $\F$ that is disjoint from $V_\phi$.

Now, the second step of the proof is show that, for each $n\geq 0$, we can always construct the paths $\gamma_{n,\phi}$ so that they are the same among all leaves $\phi \in \fb$.

For every $n \geq0$, we define a partition of the set $A^n=\{\alpha_n^\phi\}_{\phi \in J_n}$ given by
\begin{align*}
    A_n^+ : = \{\alpha^\phi_n \in A^n \mid \phi \in J_n^+\} \quad \text{ and } \quad
    A_n^- : = \{\alpha^\phi_n \in A^n \mid \phi \in J_n^-\}.
\end{align*}
 The set $A^n$ is ordered according to the orientation of the leaf $\phi_n$. Note that any element in $A_n^+$ is greater than any element in $A_n^-$. Now, since the family $\{V_\phi\}_{\phi \in \partial_L C_\O}$ is locally finite, any compact sub-arc of $\phi_n$ intersects finitely many arcs of $A^n$. Consequently, if $A_n^+\neq \varnothing$, there exists a minimal element $\alpha_n^\text{min}$ in $A_n^+$. 
Similarly, if $ A_n^-\neq \varnothing$, then there exists a maximal element $\alpha_n^\text{max}$ in $A_n^-$. In the case where $A_n^+ = \varnothing$ or $A_n^- = \varnothing$, such arcs are considered empty. 
Note that, the connected component of $\phi_n\setminus(\alpha_n^\text{max} \cup \alpha_n^\text{min})$ that contains $f^n(x)$ intersects no trivial-neighborhood in $\{V_\phi\}_{\phi \in J^n}$. For each $n \geq 0$, let $V_n^\text{max}$ and $V_n^\text{min}$  be the trivial-neighborhoods in the family $\{V_\phi\}_{\phi \in J^n}$ that contains the arc $\alpha_n^\text{max}$ and $\alpha_n^\text{min}$, respectively. 

For each $n\geq 0$, we define the set
$ D_n := \overline{\rule{0cm}{0.4cm}L(\phi_n) \cap R(\phi_{n+1})}.$
After the redefinition of the trivial-neighborhoods as $V'_\phi:= V_\phi \cap \overline{\rule{0cm}{0.38cm}L(\phi^{M_\phi})}$ in (\ref{eq:rename}), only two scenarios are possible:
\begin{itemize}
    \item[] (1)\  Either $V_n^\text{max} = V_{n+1}^\text{max}$;
    \item[] (2)\  Or the intersection $V_{n+1}^\text{max} \cap D_n$ is contained in $\phi_{n+1}$.
\end{itemize}
\mycomment{0.3cm}
These scenarios are respectively illustrated in the figures below. Note that, if the trivial-neighborhood $V_{n+1}^\text{max}$ is distinct from $V_n^\text{max}$, then it must stop at the leaf $\phi_{n+1}$. This is also a consequence of the redefinition in (\ref{eq:rename}), as it implies that $V_\phi \subset \overline{\rule{0cm}{0.38cm}L(\phi^{M_\phi})}$.

\newpage
\begin{figure}[h!]
    \center 
    \mycomment{0.2cm}\begin{overpic}[width=4.3cm, height=4.1cm, tics=10]{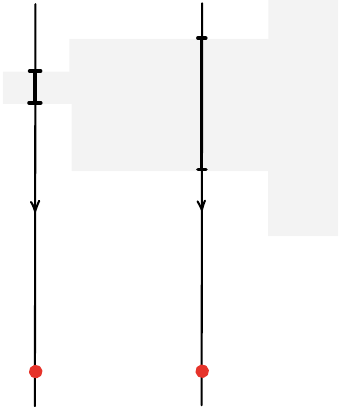}
        \put (26,6) {\color{myRED}$\displaystyle f^{n+1}(x) $}
        \put (-16,6) {\color{myRED}$\displaystyle f^{n}(x) $}
        \put (-22,75) {\color{black}$\displaystyle (1)$:}
        \put (28,69) {\color{myDARKGRAY}\large$\displaystyle V_n^\text{max} $}
        \put (67,69) {\color{myDARKGRAY}\large$\displaystyle V_{n+1}^\text{max} $}
        \put (-4,42) {\color{myDARKGRAY}$\displaystyle \phi_n $}
        \put (38,42) {\color{myDARKGRAY}$\displaystyle \phi_{n+1} $}
\end{overpic}
\hspace*{2.4cm}\begin{overpic}[width=4.3cm,height=4.1cm, tics=10]{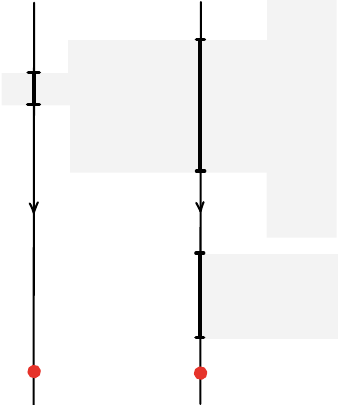}
    \put (26,6) {\color{myRED}$\displaystyle f^{n+1}(x) $}
        \put (-16,6) {\color{myRED}$\displaystyle f^{n}(x) $}
        \put (-22,75) {\color{black}$\displaystyle (2)$:}
        \put (28,69) {\color{myDARKGRAY}\large$\displaystyle V_n^\text{max} $}
        \put (69,21.5) {\color{myDARKGRAY}\large$\displaystyle V_{n+1}^\text{max} $}
        \put (-4,42) {\color{myDARKGRAY}$\displaystyle \phi_n $}
        \put (38,42) {\color{myDARKGRAY}$\displaystyle \phi_{n+1} $}
\end{overpic}
\end{figure}

\noindent We have completely analogous alternatives for the sets $V_n^\text{min}$, for all $n \geq 0$.

Either way, for each $n\geq 0$, we have that the points $f^n(x)$ and $f^{n+1}(x)$ are contained in the same connected component, namely $D^*_n$, of the set $$D_n\setminus(V_n^\text{max} \cup V_n^\text{min} ).$$ This further implies that, for each $n \geq 0$, the points $f^n(x)$ and $f^{n+1}(x)$ can be connected via a path $\gamma_n:[0,1]\longrightarrow\R^2$ that is positively transverse to $\F$ and contained in $D^*_n$. Therefore, since $\gamma_n \subset D^*_n$, it follows 
that $\gamma_n$ is disjoint from every set in the family $\{V_\phi\}_{\phi \in \fb}$.

\mycomment{-0.83cm}
\begin{figure}[h!]
    \center 
    \vspace*{0.2cm}\begin{overpic}[width=4.4cm, height=5.2cm, tics=10]{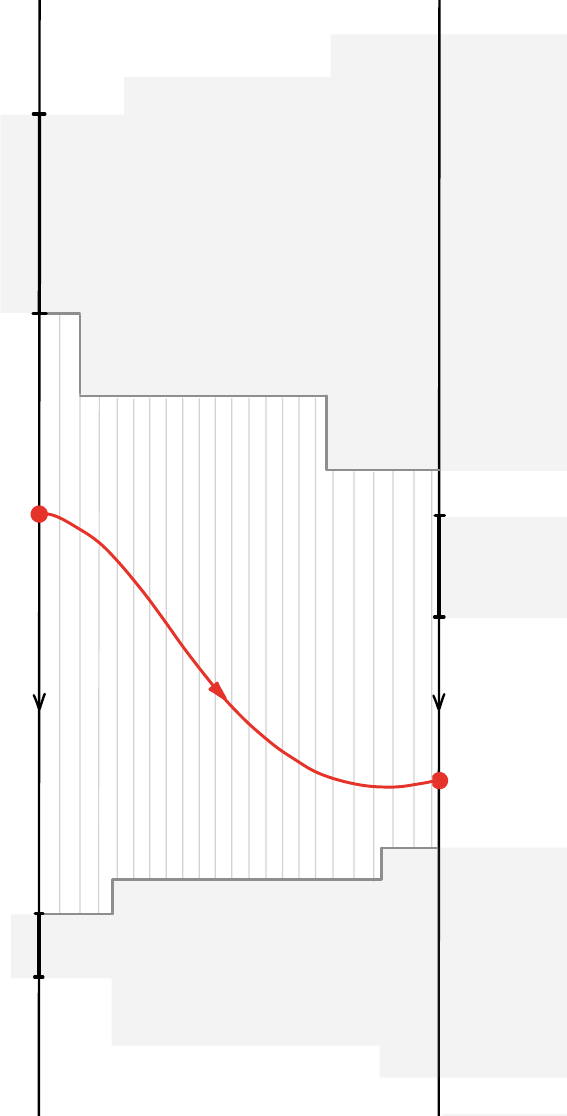}
        \put (68,28.3) {\color{myRED}$\displaystyle f^{n+1}(x) $}
        \put (-14,53) {\color{myRED}$\displaystyle f^{n}(x) $}
        \put (28,77) {\color{myDARKGRAY}\large$\displaystyle V_n^\text{max} $}
        \put (37.5,45) {\color{black}\large$\displaystyle D_n^* $}
        \put (28,12.2) {\color{myDARKGRAY}\large$\displaystyle V_n^\text{min} $}
        \put (67.5,47.1) {\color{myDARKGRAY}$\displaystyle V_{n+1}^\text{max} $}
        \put (24,33) {\color{myRED}\large$\displaystyle \gamma_n $}
        \put (0,105) {\color{myDARKGRAY}$\displaystyle \phi_n $}
        \put (56,105) {\color{myDARKGRAY}$\displaystyle \phi_{n+1} $}
\end{overpic}
\end{figure}

\mycomment{-0.1cm}
By repeating the same construction with respect to the set $\partial_R C_\O$, we end up with a family of positively transverse paths
$\{\gamma_n\}_{n\in \Z}\sspc$, where each $\gamma_n$ connects $ f^n(x)$ to its image $ f^{n+1}(x)$, and is disjoint from $V_{\phi}$, for any leaf $\phi \in \fb \sspc\cup\sspc \bb$.

Finally, the last step is to prove that the family $\{\gamma_n\}_{n\in \Z}\sspc$ is locally finite. For that,
define $$W:=\bigcup_{n \geq 0} \overline{\rule{0cm}{0.4cm}L(\phi_{-n}) \cap R(\phi_{n})}, $$ and observe that every point $z \in W$ has an open neighborhood $U_z\subset \R^2$ that intersects finitely many sets in the family $\{D_n\}_{n \in \Z}$. In particular, each neighborhood $U_z$ intersects finitely many paths in the family $\{\gamma_n\}_{n\in \Z}$.
Let $K \subset \R^2$ be a compact set, and denote $$K^*:=K \setminus \bigcup_{\phi \in  \fb \sspc\cup\sspc \bb} \text{int}{(V_\phi)}.$$
Note that $K^*$ is a compact set contained in $W$ satisfying $\gamma_n \cap K = \gamma_n \cap K^*$, for all $n \in \Z$. Consider a cover of $K^*$ of the form $\{U_z\}_{z \in K^*}$. By taking a finite sub-cover, we conclude that $K^*$ intersects at most finitely many transverse paths in $\{\gamma_n\}_{n\in \Z}$. Hence, $K$ intersects finitely many paths in $\{\gamma_n\}_{n\in \Z}$. This shows that the family $\{\gamma_n\}_{n\in \Z}$ is locally-finite and, hence, $$\Gamma_\O := \prod_{n \in \Z} \gamma_n$$ consists in a proper transverse trajectory associated to the orbit $\O$.

\newpage
The fact that $\Gamma_\O$ satisfies $\Ltop \O = \partial_L C_\O \cap L(\Gamma_\O)$ and $\Lbot \O = \partial_L C_\O \cap R(\Gamma_\O)$ follows directly from the construction. The analogous properties for $\Rtop \O$ and $\Rbot \O$, holds as well. To see that these equalities hold for any proper transverse trajectory $\Gamma_\O$,  it suffices to note that they hold automatically if the trivial-neighborhoods in the family $\{V_\phi\}_{\phi \in \partial C_\O}$ are  disjoint from $\Gamma_\O$. Since in the proper case $\Gamma_\O\subset \R^2$ is a closed set disjoint from $\partial C_\O$, we can apply Lemma \ref{lemma:trivial_neighborhoods}, which ensures that such a family $\{V_\phi\}_{\phi \in \partial C_\O}$ exists. This concludes the proof of Theorem \ref{thmx:proper_trajectories_restate}.
\end{proof}

\subsection{Equivalence classes of $\F$-asymptotic orbits}\label{sec:asymptotic_orbits}


We can relate orbits exhibiting similar $\F$-asymptotic behavior through two equivalence relations on the set $\orb$, denoted $\fasym$ and $\basym$.

\begin{definition} 
    Two orbits $\O, \O^{\sspc \prime} \in \orb$ are called \emph{forward $\F$-asymptotic}, denoted $\O \fasym \O^\pp$, if there exists a leaf $\phi \in C_\O \cap C_{\O^\pp}$ such that $L(\phi) \cap C_\O = L(\phi) \cap C_{\O^\pp}$,
and, in addition, 
$$ \Ltop \O = \Ltop \O^\pp \quad \text{ and } \quad \Lbot \O = \Lbot \O^\pp.$$
    Similarly, we define the \emph{backward $\F$-asymptotic} relation, denoted $\O \basym \O^\pp$, by replacing $L$ for $R$.
\end{definition}





Evidently, $\fasym$ and $\basym$ are equivalence relations on $\orb$. 


\mycomment{0.3cm}
\begin{lemma}\label{lemma:equivalence}
    For any $ \O,\O^{\sspc \prime} \in \orb$, the following statements are equivalent:
    \begin{itemize}[leftmargin=1.5cm]
        \item[\textup{\textbf{(i)}}] The orbits $\O$ and $\O^{\sspc\prime}$  are forward $\F$-asymptotic, i.e. $\O \fasym \O^{\sspc \prime}$.
        \item[\textup{\textbf{(ii)}}] There exist proper transverse trajectories $\spc\Gamma_\O$ and $\spc\Gamma_{\O^{\sspc \prime}}$ associated to $\O$ and $\O^{\sspc \prime}$, respectively, that satisfy
    \(
    \spc \Gamma_\O\left(\spc[\spc t, \infty)\spc\right) \cap \Gamma_{\O^{\sspc \prime}}\left(\spc[\spc t, \infty)\spc\right) \neq \varnothing\sspc, \sspc
    \)
    for every $t \in \R.$
    \end{itemize}
    \mycomment{0.3cm}
    We can obtain a similar result for the relation $\sspc\basym\sspc$ by replacing $\sspc[\spc t, \infty)$ with $(\spc-\infty, t\spc ]$.
\end{lemma}


\begin{figure}[h!]
    \center
    \vspace*{-0.1cm}\begin{overpic}[width=0.55\textwidth,height=3.5cm, tics=10]{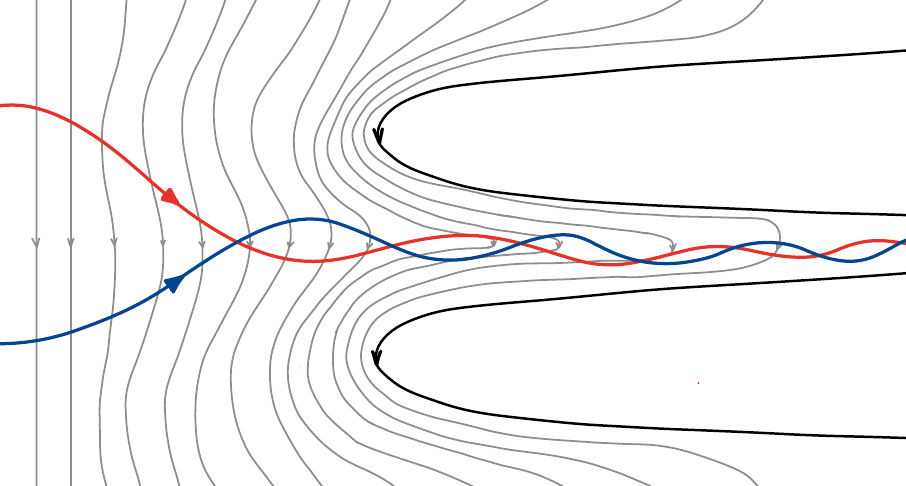}
        \put (-8.5,31) {{\color{myRED}\large$\displaystyle \Gamma_{\O^{\sspc\prime}} $}}
        \put (-7.5,10) {{\color{myBLUE}\large$\displaystyle \Gamma_\O $}}
        \put (57,27.3) {{\color{black}$\displaystyle \Ltop C_\O = \Ltop C_{\O^{\sspc\prime}} $}}
        \put (56,9) {{\color{black}$\displaystyle \Lbot C_\O = \Lbot C_{\O^{\sspc\prime}} $}}
\end{overpic}
\end{figure}

Lemma \ref{lemma:equivalence} admits a dual version which is slightly stronger, presented in Lemma \ref{lemma:dual_equivalence}.

\begin{lemma}\label{lemma:dual_equivalence}
    For any pair of orbits $ \O,\O^{\sspc \prime} \in \orb$ satisfying $C_\O \sspc\cap\sspc C_{\O^\pp}\neq \varnothing$, we have that the following statements are equivalent:
    \begin{itemize}[leftmargin=1.5cm]
        \item[\textup{\textbf{(i)}}] The orbits $\O$ and $\O^{\sspc\prime}$ are not forward $\F$-asymptotic, i.e. $\O \not\fasym \O^{\sspc \prime}$.
        \item[\textup{\textbf{(ii)}}] For any pair of proper transverse trajectories $\spc\Gamma_\O$ and $\spc\Gamma_{\O^{\sspc \prime}}$ associated to $\O$ and $\O^{\sspc \prime}$, respectively, there exists a leaf $\phi_L \in C_\O \cap C_{\O^\pp}$ such that $\spc\Gamma_\O \cap \Gamma_{\O^{\sspc \prime}}  \cap \overline{\rule{0cm}{0.34cm}L(\phi_L)} \spc= \varnothing$.
    \end{itemize}
    \mycomment{0.2cm}
    We obtain a similar result for non-backward $\F$-asymptotic orbits by replacing each $L$ with $R$.
\end{lemma}

\mycomment{0.2cm}

\begin{figure}[h!]
    \center 
    \hspace*{-0.3cm}\begin{overpic}[width=7cm, height=3.7cm, tics=10]{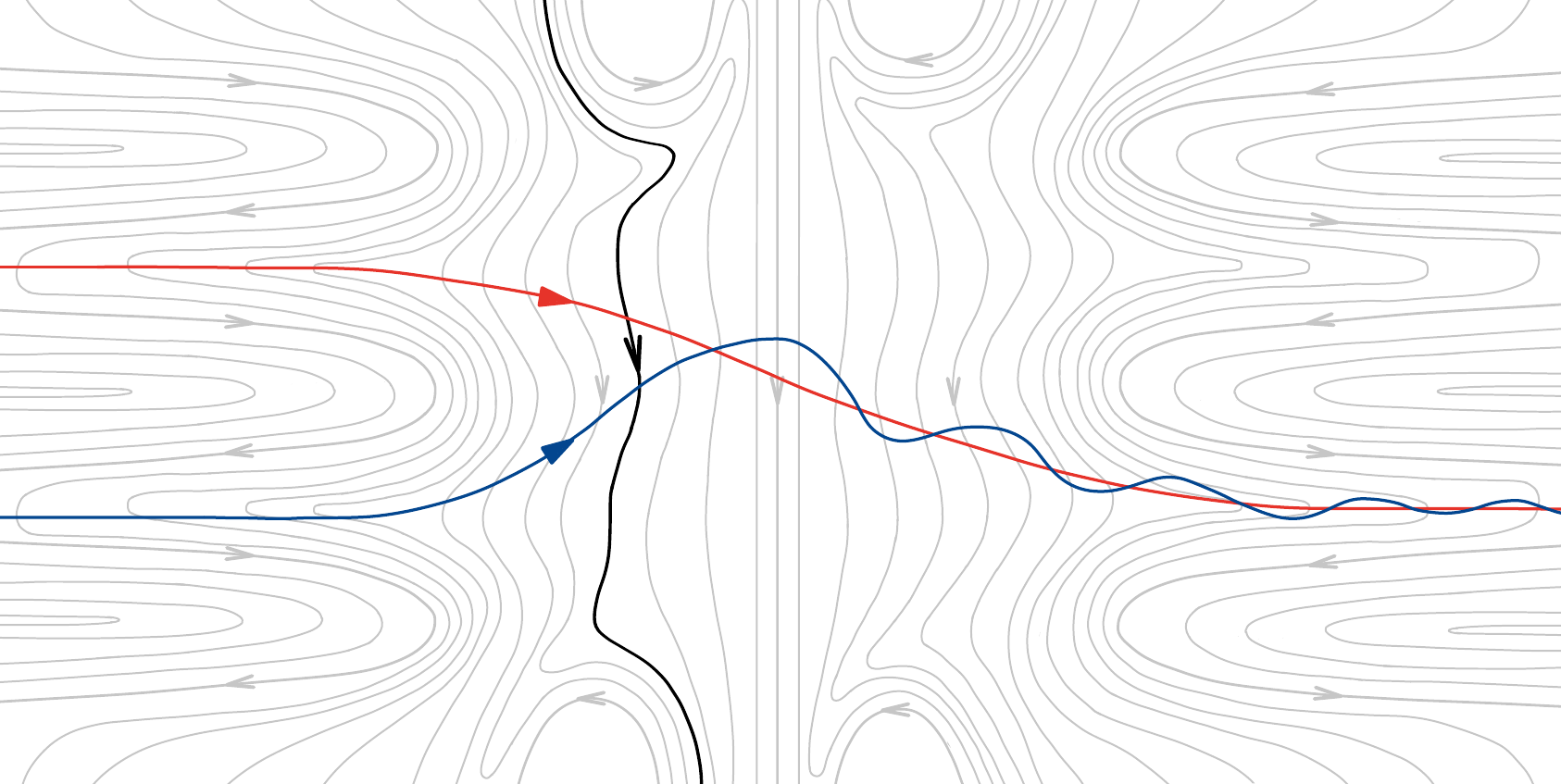}
        \put (32,56) {{\color{black}\large$\displaystyle \phi_R$}}
        \put (39,-10) {{\color{black}\large$\displaystyle \O \notbasym \O'$}}
        \put (-10,35) {{\color{myRED}\large$\displaystyle \Gamma_\O$}}
        \put (-10,17) {{\color{myBLUE}\large$\displaystyle \Gamma_{\O^\pp}$}}
\end{overpic}
\hspace*{1cm}\begin{overpic}[width=7cm, height=3.7cm, tics=10]{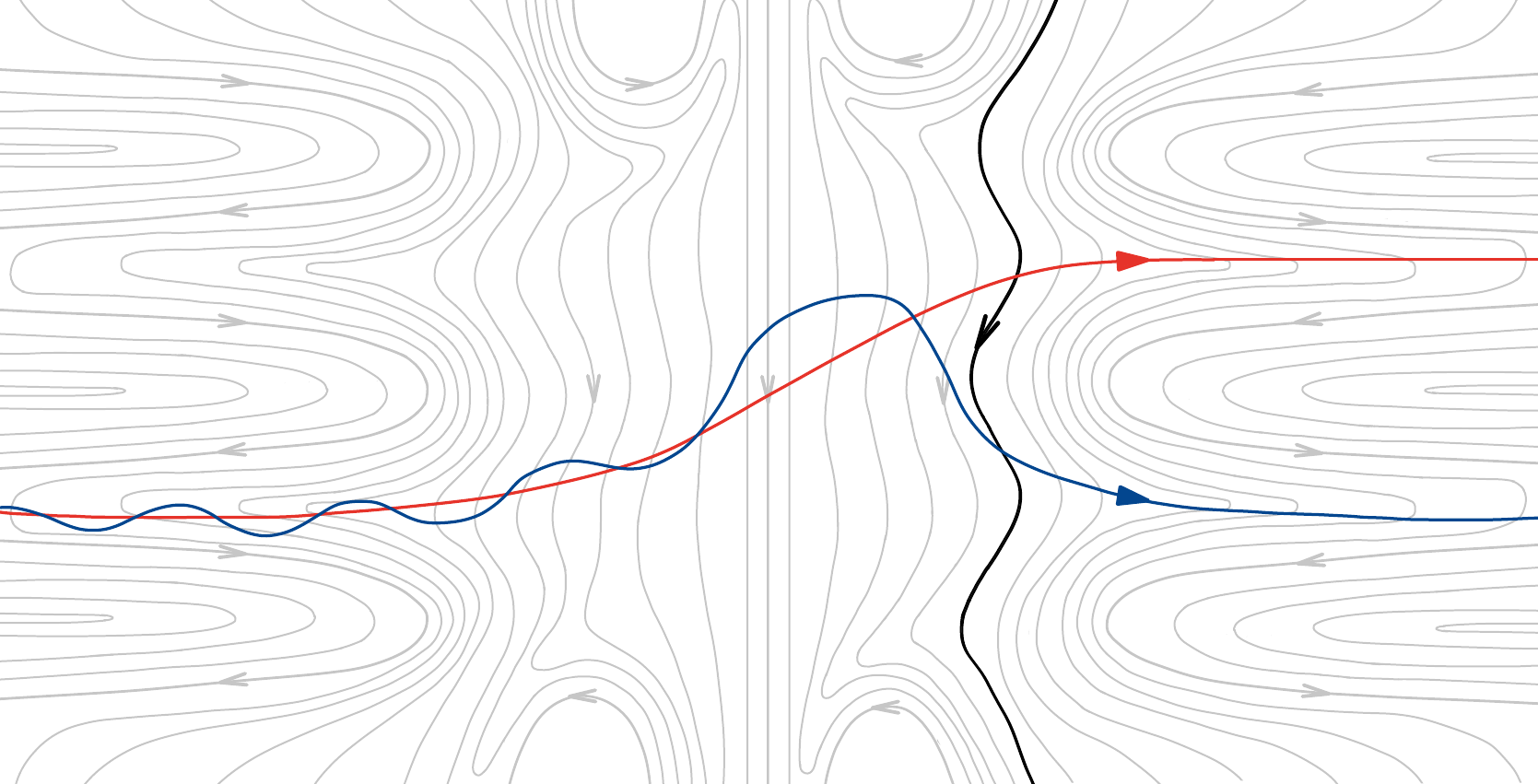}
    \put (65,56) {{\color{black}\large$\displaystyle \phi_L$}}
    \put (39,-10) {{\color{black}\large$\displaystyle \O \notfasym \O'$}}
    \put (104,34) {{\color{myRED}\large$\displaystyle \Gamma_\O$}}
        \put (104,16) {{\color{myBLUE}\large$\displaystyle \Gamma_{\O^\pp}$}}
\end{overpic}
\end{figure}
\newpage

\mycomment{0.2cm}
\begin{remark}
   It is important to remark that the result of Lemma \ref{lemma:equivalence} and Lemma \ref{lemma:dual_equivalence} are no longer true if we consider non-proper transverse trajectories. Consider the following example:

    \mycomment{0.2cm}
    \indent We suppose that $\O,\O^{\sspc \prime} \in \orb$ are two orbits admitting a leaf $\phi \in \Ltop \O \cap \Lbot {\O^{\sspc \prime}}$. These orbits are clearly not forward $\F$-asymptotic. However,  
    as illustrated below, there still exist non-proper transverse trajectories $\Gamma_\O$ and $\Gamma_{\O^{\sspc \prime}}$ associated to $\O$ and $\O^{\sspc\prime}$, respectively,  that satisfy the property \(\spc
    \Gamma_\O\left(\spc[\spc t, \infty)\spc\right) \cap \Gamma_{\O^{\sspc \prime}}\left(\spc[\spc t, \infty)\spc\right) \neq \varnothing\sspc, \sspc
    \) for any $t \in \R$.

    \begin{figure}[h!]
        \center
        \mycomment{0.4cm}\begin{overpic}[width=0.55\textwidth,height=3.7cm, tics=10]{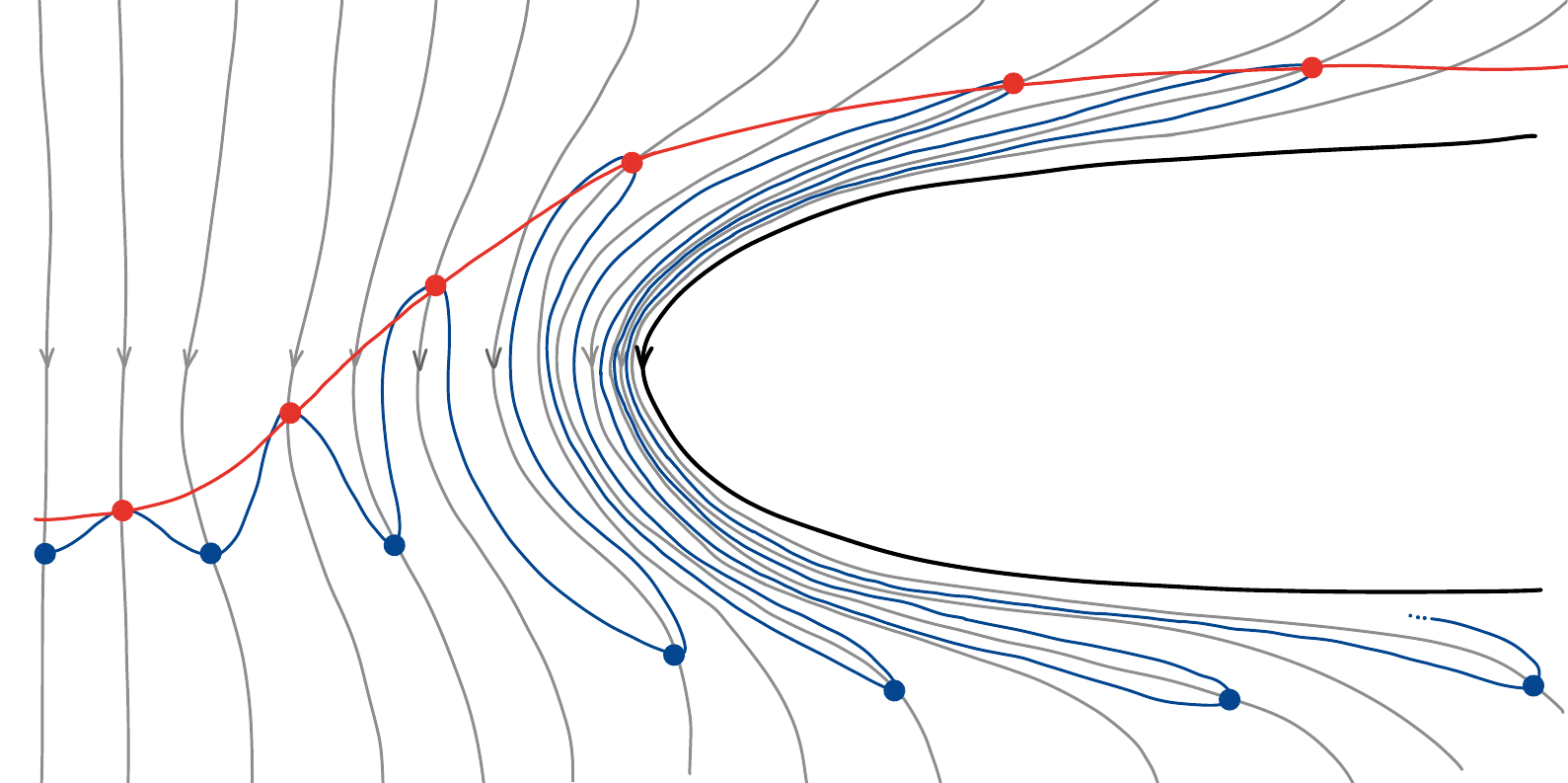}
            \put (52,20) {\color{black}\large$\displaystyle  \Ltop \O \cap \Lbot {\O^{\sspc \prime}}$}
            \put (35,38) {\colorbox{white}{\color{myRED}\large$\displaystyle \Gamma_{\O^{\sspc\prime}}$}}
            \put (33,2) {\colorbox{white}{\color{myBLUE}\large$\displaystyle \Gamma_\O$}}
    \end{overpic}
    \end{figure}

\end{remark}

\mycomment{-0.7cm}
\begin{proof}[Proof of Lemma \ref{lemma:equivalence}]
    $(i) \implies (ii)$: From $(i)$, we have $C_\O \cap C_{\O^{\sspc \prime}}\neq\varnothing$ and $\partial_L C_\O = \partial_L C_{\O^{\sspc \prime}}$.
    Let $\phi_* \in C_\O \cap C_{\O^{\sspc \prime}}$ and let $(\phi_n)_{n\geq0}$ be the increasing sequence of leaves in $\F$ formed by the leaves intersecting the set $(\sspc\O\cup \O^{\sspc \prime}\sspc)\sspc\cap\sspc \overline{\rule{0cm}{0.34cm}L(\phi_*)}$. Note that some leaves of $(\phi_n)_{n\geq0}$ may intersect simultaneously both orbits $\O$ and $\O^{\sspc \prime}$. Next, according to Lemma \ref{lemma:trivial_neighborhoods}, there exists a family of sufficiently small trivial-neighborhoods $\{V_\phi\}_{\phi \in \partial_L C_\O}$ of the leaves in $\partial_L C_\O$ that is pairwise disjoint and disjoint from the closed set $\phi_0 \cup \O \cup \O^{\sspc \prime}$.
    
    Similarly to Section \ref{sec:order_limit_leaves}, we have a filtration $(J_n)_{n\geq 0}$ of the set $\partial_L C_\O$ by sets of the form $J_n=\{\phi \in \partial_L C_\O \mid V_\phi \cap \phi_n\}$.
    Moreover, for each $n\geq 0$ such
    that $\O \cap \phi_n \neq \varnothing$, we have a cut
    $$(J_n^-(\O), J_n^+(\O)) \in \text{Cut}_\prec(J_n).$$ 
    On the other hand, for the values of $n\geq 0$ such that $\O \cap \phi_n = \varnothing$, we define $m(n)>n$ as the smallest integer such that $\O \cap \phi_{m(n)} \neq \varnothing$, and we set 
    \begin{equation*}
        \textup{$J_n^-(\O) := J_n \cap J_{m(n)}^-(\O) \quad$ and $\quad
   J_n^+(\O):= J_n \cap J_{m(n)}^+(\O).$}
    \end{equation*}
    Performing the same adaptation for the cuts induced by the orbit $\O^{\sspc \prime}$, we end up with two cuts $(J_n^-(\O), J_n^+(\O))$ and $(J_n^-(\O^{\sspc \prime}), J_n^+(\O^{\sspc \prime}))$ of the ordered set $(J_n,\prec)$, for every $n\geq0$.

   Since $\O\fasym\O^{\sspc \prime}$, for each leaf $\phi \in \partial_L C_\O$, there exists $M_\phi>0$ such that
    \begin{align*}
        \phi \in \partial_L^\text{\sspc top\sspc } \O \implies \phi\in J_n^-(\O) \cap J_n^-(\O^{\sspc \prime}), \quad \forall n\geq M_\phi,\\
        \phi \in \partial_L^\text{\sspc bot\sspc } \O \implies \phi\in J_n^+(\O) \cap J_n^+(\O^{\sspc \prime}), \quad \forall n\geq M_\phi. 
    \end{align*}
    The same redefinition of the family $\{V_\phi\}_{\phi \in \partial_L C_\O}$ presented in (\ref{eq:rename}), we can ensure that each leaf $\phi \in \fb$ satisfies $\phi \not\in J_n$ for all $n<M_\phi$ and, consequently, for any $n\geq 0$ we have
    \begin{align*}
        \partial_L^\text{\sspc top\sspc } \O  = J_n^- (\O) = J_n^- (\O^{\sspc \prime}) \quad \text{ and } \quad
        \partial_L^\text{\sspc bot\sspc } \O  = J_n^+ (\O) = J_n^+ (\O^{\sspc \prime}).
    \end{align*}
    \indent For each $n\geq 0$, let $p_n$ and $p_n'$ be the points where the leaf $\phi_n$ intersects $\O$ and $\O^{\sspc \prime}$, respectively. In the case where $\phi_n$ intersects only one of these orbits, we consider $p_n = p_n'$. 
    Following arguments similar those in the proof of
    Theorem \ref{thmx:proper_trajectories_restate}, we obtain for each $n\geq0$ two positively transverse paths $\gamma_n, \gamma_n':[0,1]\longrightarrow \R^2$ that are disjoint from all sets in the family $\{V_\phi\}_{\phi \in \fb}$, and such that $\gamma^n_\O(0) = p_n$, $\gamma^n_{\O^{\sspc \prime}}(0) = p_n'$, $\gamma^n_\O(1) = p_{n+1}$ and $\gamma^n_{\O^{\sspc \prime}}(1) = p_{n+1}'$. 
    Observe that $p_n$, $p_n'$, $p_{n+1}$ and $p_{n+1}'$ belong to the same connected component of the set 
    $$ \overline{L(\phi_n) \cap R(\phi_{n+1})} \ \big\backslash \  \bigcup_{\phi \in \partial_L C_\O} \textup{int}(V_\phi).$$
    Therefore, for any pair of open sets $U_n, U_{n+1}\subset \R^2$ satisfying $\phi_n \subset U_n$ and $\phi_{n+1} \subset U_{n+1}$, the same arguments of Section \ref{sec:proper_transverse_trajectories} allows us to construct $\gamma_n$ and $\gamma_n'$ in such a way that they coincide outside the neighborhoods $U_n$ and $U_{n+1}$. Moreover, in the case where $p_n = p_n'$ and $p_{n+1} = p_{n+1}'$, we can assume that $\gamma_n = \gamma_n'$.  In all cases, we have  $\gamma_n \cap \gamma_n' \neq \varnothing$ for all $n\geq0$.
    The argument of Section \ref{sec:proper_transverse_trajectories} proves that the families $\{\gamma_n\}_{n\geq0}$ and $\{\gamma_n'\}_{n\geq0}$ are both locally-finite. Thus, they can be concatenated, and then extended, into proper transverse trajectories $\Gamma$ and $\Gamma'$ associated respectively to the orbits $\O$ and $\O^{\sspc \prime}$. These proper transverse trajectories satisfy item (ii).

\begin{remark}\label{rmk:adapt}
    The argument above can be adapted to any finite subset $\OO\subset \orb$ satisfying
    $$ \O \fasym \O^{\sspc \prime}\sspc, \quad \forall \O,\O^{\sspc\prime}\in \OO.$$
    If each leaf of $\F$ intersects at most one point orbit in $\OO$, this adaptation proves Lemma \ref{lemma:common_half_traj}.
\end{remark}

    $(ii) \implies (i):$ Since $\Gamma$ and $\Gamma'$ intersect each other, there exists a leaf $\phi_* \in C_\O \cap C_{\O^{\sspc \prime}}$.  
    Let $(\phi_n)_{n\geq0}$ be the increasing sequence of leaves in $\F$ formed by the leaves meeting the set $(\O\cup \O^{\sspc \prime})\cap \overline{L(\phi_*)}$.
    Since $\Gamma$ and $\Gamma'$ are proper lines, the set  $\Gamma\cup \Gamma' \subset \R^2$ is closed.  Thus, by applying Lemma \ref{lemma:trivial_neighborhoods}, we know there exists a family $\{V_\phi\}_{\phi \in \partial_L C_\O}$ of pairwise disjoint trivial-neighborhoods of $\partial_L C_\O$ that satisfies $V_\phi \cap (\phi_0 \cup \Gamma \cup \Gamma') = \varnothing$, for any $\phi \in \partial_L C_\O$.

    
    For each $n\geq 0$, we consider $J_n=\{\phi \in \partial_L C_\O \mid V_\phi \cap \phi_n\}$. As in $(i) \implies (ii)$, consider the sets $J_n^- (\O)$, $J_n^+ (\O)$, $J_n^-(\O^{\sspc \prime})$ and $J_n^+(\O^{\sspc \prime})$. Note that, as each $V_\phi$ is disjoint from $\Gamma \sspc \cup \sspc \Gamma'$, hypothesis $(ii)$ implies that, for any $\forall n\geq0$, we have $J_n^+ (\O) = J_n^+ (\O^{\sspc \prime}) $ and $J_n^- (\O) = J_n^- (\O^{\sspc \prime})$.
    This shows that the orbits $\O$ and $\O^{\sspc \prime}$ induce the same cut $(\sspc\partial_L^\text{\sspc top\sspc } \O,\spc \partial_L^\text{\sspc bot\sspc }\O\sspc)$ on $\fb$ and, consequently, we conclude that $\O$ and $\O^{\sspc \prime}$ are forward $\F$-asymptotic.
\end{proof}

\begin{proof}[Proof of Lemma \ref{lemma:dual_equivalence}] 

    \textit{$(i) \implies (ii):$} Suppose that $\O \not\fasym \O^{\sspc \prime}$. 
    According to Lemma \ref{lemma:equivalence}, for any pair of proper transverse trajectories $\Gamma_\O$ and $\Gamma_{\O^{\sspc \prime}}$ of $\O$ and $\O^{\sspc \prime}$, there exists $M>0$ such that
    $$ \Gamma_\O\left(\spc[\spc M, \infty)\spc\right) \cap \Gamma_{\O^{\sspc \prime}}\left(\spc[\spc M, \infty)\spc\right) = \varnothing.$$
    Since $\Gamma_\O$ and $\Gamma_{\O^{\sspc \prime}}$ are topological lines, the intersection $\Gamma_\O \cap \Gamma_{\O^{\sspc \prime}}$ is a closed subset of $\R^2$. Therefore, the set 
    $ {\Gamma_\O}^{-1}(\spc\Gamma_\O \cap \Gamma_{\O^{\sspc \prime}}\sspc)$
    is a closed subset of $\R$ bounded from above by $M$ and, consequently, it admits a maximal element $t_\text{max}$.

    Note that the leaf $\phi_\text{max}\in \F$ passing through the point $\Gamma_\O(t_\text{max})$ belongs to $C_\O \cap C_{\O^{\sspc \prime}}$. Since $C_\O \cap C_{\O^{\sspc \prime}}$ is an open set, there exists a leaf $\phi_L\in \F$ satisfying $L(\phi_L) \subset L(\phi_\text{max})$ and sufficiently close to $\phi_\text{max}$ so that it also belongs to $C_\O \cap C_{\O^{\sspc \prime}}$. This leaf proves $(i) \implies (ii)$.

    \mycomment{0.3cm}

    \textit{$(ii) \implies (i):$} This implication follows trivially by Lemma \ref{lemma:equivalence}.
\end{proof}

\section{Sided preorders of orbits}\label{chap:weak_transverse_intersections}

In Section \ref{chap:asymptotic behavior}, we showed that the notion of proper transverse trajectories provides an effective framework for describing the asymptotic behavior of the orbits in \(\orb\) relative to the transverse foliation \(\F\). In this section, we develop the tools necessary to compare orbits crossing a common leaf while exhibiting different $\F$-asymptotic behaviors.

\mycomment{0.1cm}
We begin this section by exploring the following illustrative setting: 

\mycomment{0.2cm}
Let \(\O, \O^{\sspc\prime} \in \orb\) be two orbits crossing a common leaf, meaning, \(C_\O \cap C_{\O^{\sspc\prime}} \neq \varnothing\).  
Suppose that \(\O\) and \(\O^{\sspc\prime}\) are neither forward nor backward \(\F\)-asymptotic.
Thus, according to Lemma \ref{lemma:dual_equivalence}, for any pair of proper transverse trajectories $\Gamma_\O$ and $\Gamma_{\O^{\sspc\prime}}$ associated with the orbits \(\O\) and \(\O^{\sspc\prime}\), there exists a pair of leaves $\phi_R, \phi_R \in C_\O \cap C_{\O^{\sspc\prime}}$ satisfying $\phi_R < \phi_L$ such that
$$\spc\Gamma_\O \cap \Gamma_{\O^{\sspc \prime}}  \cap \overline{\rule{0cm}{0.34cm}L(\phi_L)} \spc= \varnothing \quad \text{ and } \quad\spc\Gamma_\O \cap \Gamma_{\O^{\sspc \prime}}  \cap \overline{\rule{0cm}{0.34cm}R(\phi_R)} \spc= \varnothing.$$

\newpage
\mycomment{0.2cm}
In this setting, two possible configurations arise:
\begin{itemize}
    \item[\text{(a)}] The points $p_L(\O) \in \Gamma_\O \cap \phi_L$ and $p_L(\O^{\sspc\prime}) \in \Gamma_{\O^{\sspc\prime}} \cap \phi_L$ are ordered along the leaf $\phi_L$ in the same order as the points $p_R(\O) \in \Gamma_\O \cap \phi_R$ and $p_R(\O^{\sspc\prime}) \in \Gamma_{\O^{\sspc\prime}} \cap \phi_R$ along the leaf $\phi_R$.
    
    \item[\text{(b)}] The points $p_L(\O) \in \Gamma_\O \cap \phi_L$ and $p_L(\O^{\sspc\prime}) \in \Gamma_{\O^{\sspc\prime}} \cap \phi_L$ are ordered along the leaf $\phi_L$ in the opposite order as the points $p_R(\O) \in \Gamma_\O \cap \phi_R$ and $p_R(\O^{\sspc\prime}) \in \Gamma_{\O^{\sspc\prime}} \cap \phi_R$ along the leaf $\phi_R$.
\end{itemize}

\mycomment{0.45cm}
\begin{figure}[h!]
    \center 
    \hspace*{0.6cm}\begin{overpic}[width=6cm, height=3cm, tics=10]{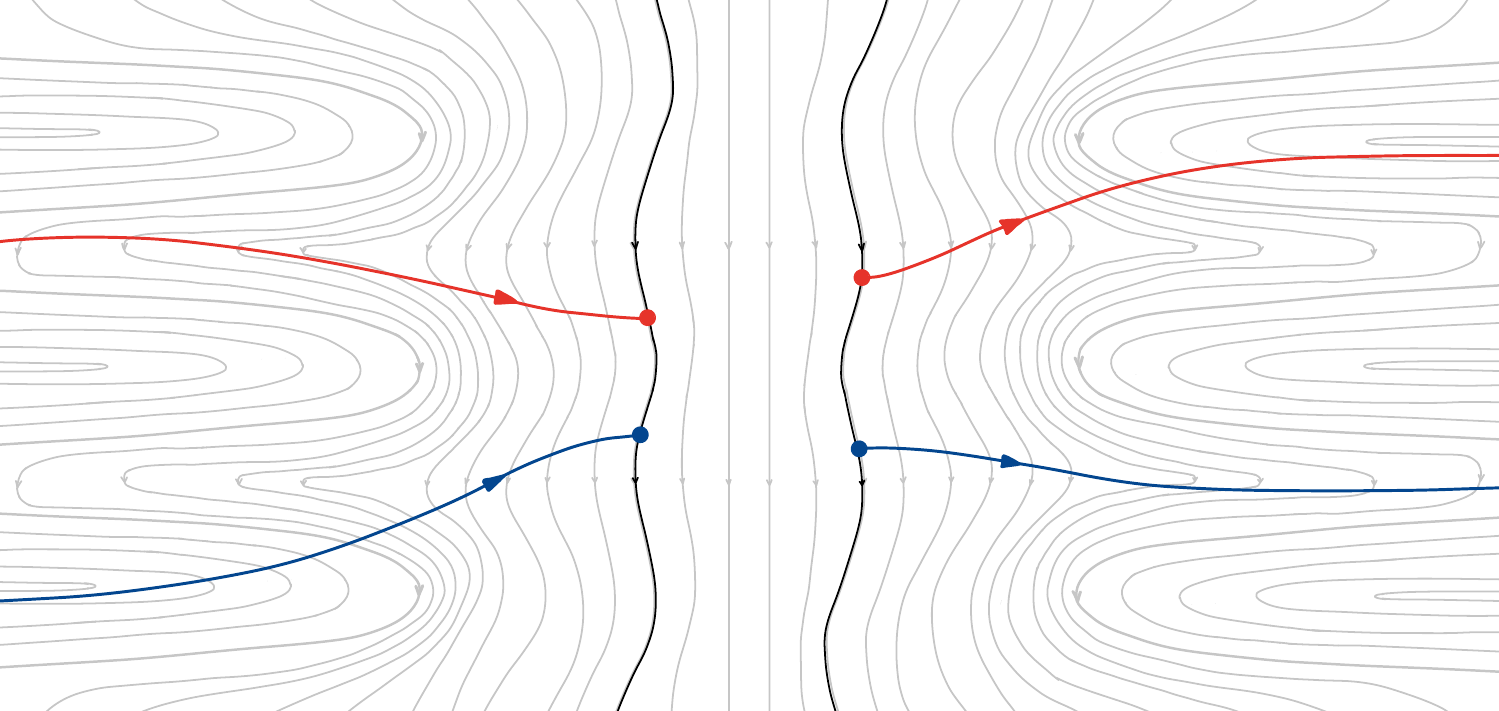}
        \put (-14,33) {{\large\color{myRED}$\displaystyle \Gamma_{\O^{\sspc\prime}} $}}
        \put (-13,6) {{\large\color{myBLUE}$\displaystyle \Gamma_{\O} $}}
        \put (102.5,37) {{\large\color{myRED}$\displaystyle \Gamma_{\O^{\sspc\prime}} $}}
        \put (102.5,13.5) {{\large\color{myBLUE}$\displaystyle \Gamma_{\O} $}}
        \put (-9,49) {{\color{black}$\displaystyle \text{(a)}\ $}}
\end{overpic}
\hspace*{2.3cm}\begin{overpic}[width=6cm,height=3cm, tics=10]{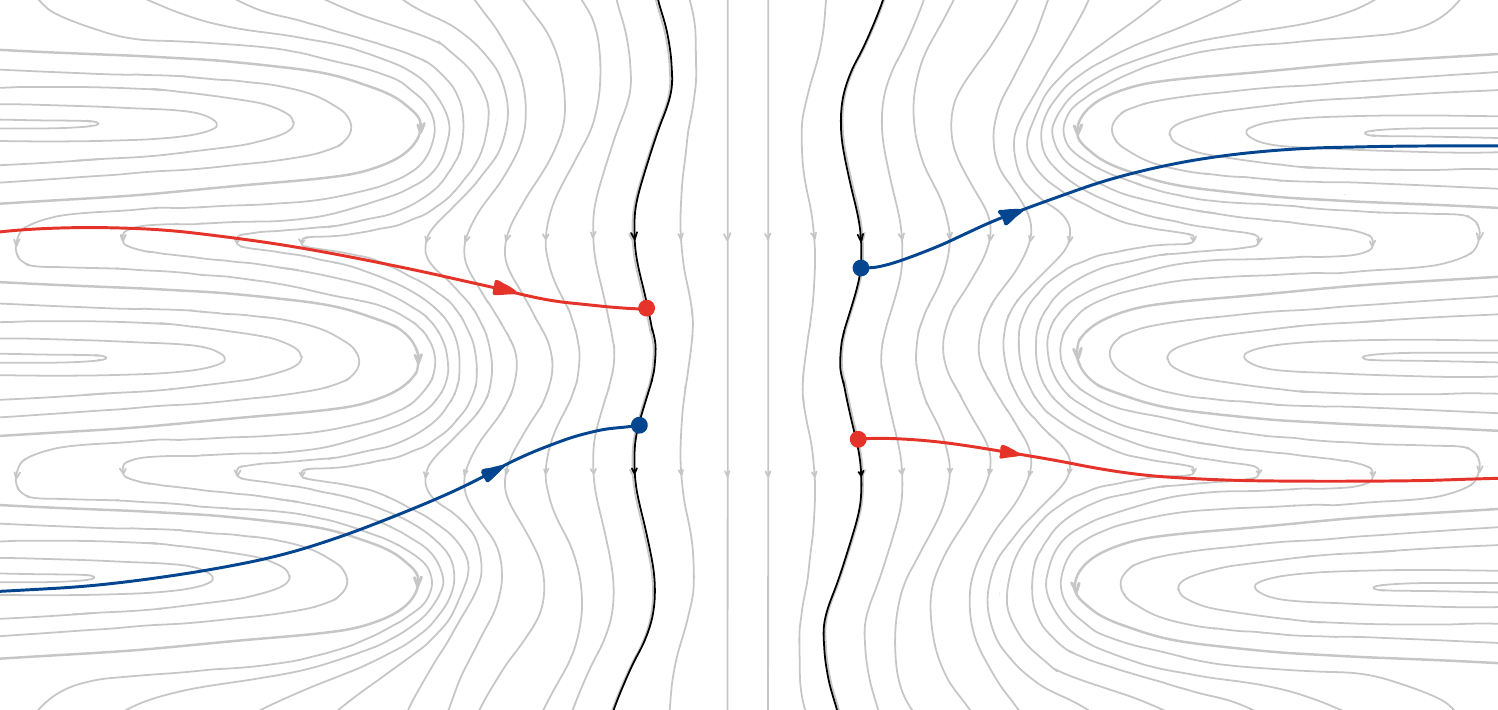}
    \put (-14,33) {{\large\color{myRED}$\displaystyle \Gamma_{\O^{\sspc\prime}} $}}
        \put (-13,6) {{\large\color{myBLUE}$\displaystyle \Gamma_{\O} $}}
        \put (102.5,37) {{\large\color{myBLUE}$\displaystyle \Gamma_{\O} $}}
        \put (102.5,13.5) {{\large\color{myRED}$\displaystyle \Gamma_{\O^{\sspc\prime}} $}}
    \put (-9,49) {{\color{black}$\displaystyle \text{(b)}\ $}}
\end{overpic}
\end{figure}

In Section \ref{sec:preorders}, we show that the configuration in which these orbits \(\O\) and \(\O^{\sspc\prime}\) fall into, depends solely on their asymptotic behavior relative to the foliation \(\F\). 
More precisely,  the ordering of the points \(p_L(\O)\) and \(p_L({\O^{\sspc\prime}})\) along the left leaf $\phi_L$, as well as the ordering of the points \(p_R(\O)\) and \(p_R(\O^\pp)\) along the right leaf $\phi_R$, do not depend on the particular choices of proper transverse trajectories \(\Gamma_\O\) and \(\Gamma_{\O^{\sspc\prime}}\).
This is not the case if $\O$ and $\O^{\sspc\prime}$ are forward or backward $\F$-asymptotic. 

Observe that, in Configuration (b), the trajectories $\Gamma$ and $\Gamma'$ are forced to intersect each other in between the leaves $\phi_R$ and $\phi_L$. Meanwhile, in Configuration (a), the proper transverse trajectories $\Gamma$ and $\Gamma'$ could be disjoint from each other, depending on the arrangement of the points of $\O$ and $\O^{\sspc\prime}$ lying in between the leaves $\phi_R$ and $\phi_L$. 



\subsection{Defining the orbit relations $\lesssim_L$ and $\lesssim_R$}\label{sec:preorders}

For each leaf $\phi \in \F$, we define the set
$$ \orbphi :=\{\O \in \orb \mid \phi \in C_\O\},$$
which corresponds to the set of orbits of $f$ crossing the leaf $\phi$. 

In this section, we define two relations $\lesssim_L$ and $\lesssim_R$ on the set $\orb$ that satisfy
\begin{itemize}[leftmargin=1.7cm]
    \item For each $\phi\in \F$, the relations $\lesssim_L$ and $\lesssim_R$ are total preorder on the set $\orbphi$.
    \item For any pair of orbits $\O, \O^{\sspc\prime} \in \orb$, the following equivalences hold:
    \begin{align*}
        \O\fasym \O^{\sspc\prime} \iff  \O \lesssim_L \O^{\sspc\prime} \ \text{ and }\  \O^{\sspc\prime} \lesssim_L \O, \\
        \O\basym \O^{\sspc\prime} \iff  \O \lesssim_R \O^{\sspc\prime} \ \text{ and }\  \O^{\sspc\prime} \lesssim_R \O.
    \end{align*}
\end{itemize}

\mycomment{0.1cm}
\noindent Before we define $\lesssim_L$ and $\lesssim_R$, observe that the intersection of leaf domains \(C_\O \cap C_{\O^{\sspc\prime}}\) of two orbits \(\O\) and \(\O^{\sspc\prime}\) is also a leaf domain (possibly empty) of the foliation \(\F\) (see Lemma \ref{lemma:intersection_of_domains}). Moreover, the sets of limit leaves \(\partial_L(C_\O \cap C_{\O^{\sspc\prime}})\) and \(\partial_R(C_\O \cap C_{\O^{\sspc\prime}})\) can be totally ordered by a relation \(\prec\), similar to the one constructed in Section \ref{sec:order_limit_leaves} for $\fb$ and $\bb$.

\noindent This sets the machinery used to compare the $\F$-asymptotic behaviors of the orbits \(\O\) and \(\O^{\sspc\prime}\).

\begin{figure}[h!]
    \center
    \vspace*{-0cm}\begin{overpic}[width=9cm, height=4.2cm, tics=10]{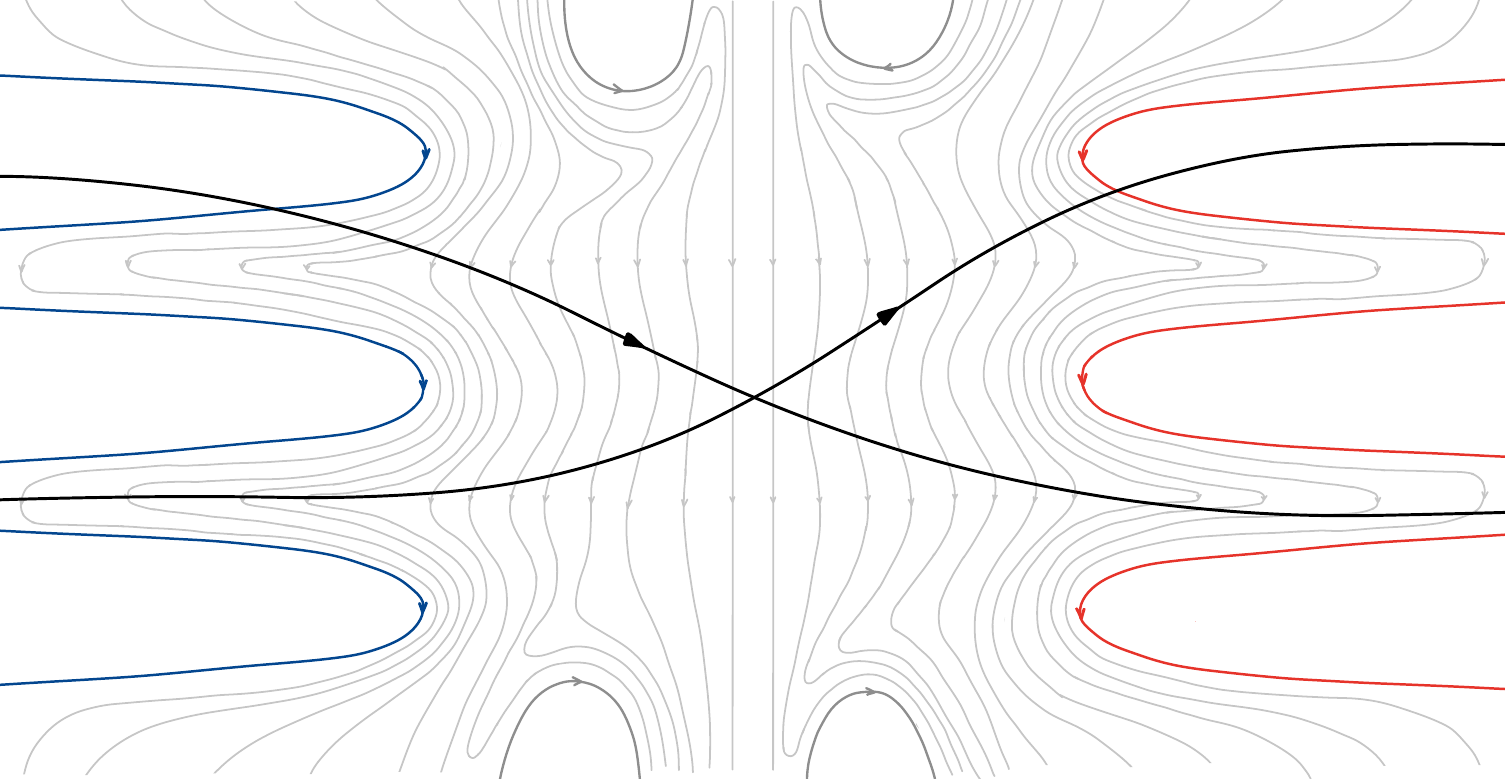}
        \put (39.5,28.7) {\colorbox{white}{\color{black}\normalsize$\rule{0cm}{0.27cm} \ \  $}}
        \put (40,29) {{\color{black}\normalsize$\displaystyle \Gamma_{\O}$}}
        \put (53,30.2) {\colorbox{white}{\color{black}\normalsize$\rule{0cm}{0.27cm} \ \ \ \spc \sspc $}}
        \put (53.5,30.5) {{\color{black}\normalsize$\displaystyle \Gamma_{\O^{\sspc\prime}}$}}
        \put (41.5,12.5) {\colorbox{white}{\color{myGRAY}\large$\rule{0cm}{0.27cm}\ \quad \quad \ \ \ \  $}}
        \put (41.5,12.8) {{\color{myGRAY}\large$\displaystyle C_\O \cap C_{\O^{\sspc\prime}} $}}
        \put (-9,22.5) {\color{myBLUE}$\displaystyle \partial_R(C_\O \cap C_{\O^{\sspc\prime}})$}
        \put (85,22.5) {\color{myRED}$\displaystyle \partial_R(C_\O \cap C_{\O^{\sspc\prime}})$}
\end{overpic}
\end{figure}

\newpage




\pagebreak

\begin{definition}
\label{sec:left-right_relation}
     For any pair of orbits \(\O, \O^{\sspc\prime} \in \orb\), we write 
\end{definition}

\begin{minipage}{\linewidth}
\hspace*{-1.25cm}
\begin{minipage}{0.5\textwidth}
   $$  \O \lesssim_{R} \O^{\sspc \prime}$$  
   
   \vspace*{0.2cm}\noindent if the orbits $\O$ and $\O^\pp$ satisfy $C_\O \cap C_{\O^{\sspc\prime}} \neq \varnothing$ and one of the following conditions:

   \vspace*{0.2cm}
\begin{itemize}[leftmargin=1.3cm]
    \item[\textbf{(R1)}\ ] There exist \( \phi, \phi' \in \partial_R (C_\O \cap C_{\O^{\sspc \prime}}) \), with $\phi \prec \phi'$, such that
    \(
    \phi \in C_\O\spc, \spc\sspc \phi' \in C_{\O^{\sspc \prime}}.
    \)
\end{itemize}

    \center
    \vspace*{0.4cm}\begin{overpic}[width=6cm, height=3cm, tics=10]{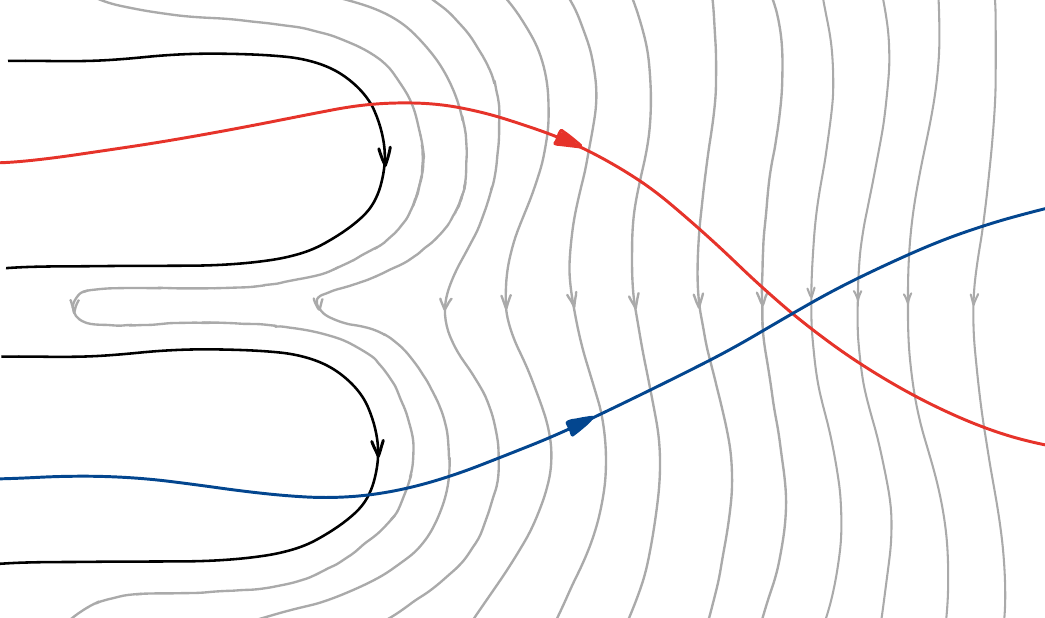}
        \put (60,40) {\colorbox{white}{\color{myRED}\large$\displaystyle \Gamma_{\O} $}}
        \put (60,9) {\colorbox{white}{\color{myBLUE}\large$\displaystyle \Gamma_{\O^\pp} $}}
\end{overpic}

\vspace*{0.1cm}
\begin{itemize}[leftmargin=1.3cm]
    \item[\textbf{(R2)}\ ] There exists a leaf $\phi \in \partial_R (C_\O \cap C_{\O^{\sspc \prime}})$ such that $\phi \in C_{\O}  \cap\sspc \Rtop {\O^{\sspc \prime}}\sspc$.
\end{itemize}

    \center
    \vspace*{0.4cm}\begin{overpic}[width=6cm, height=3cm, tics=10]{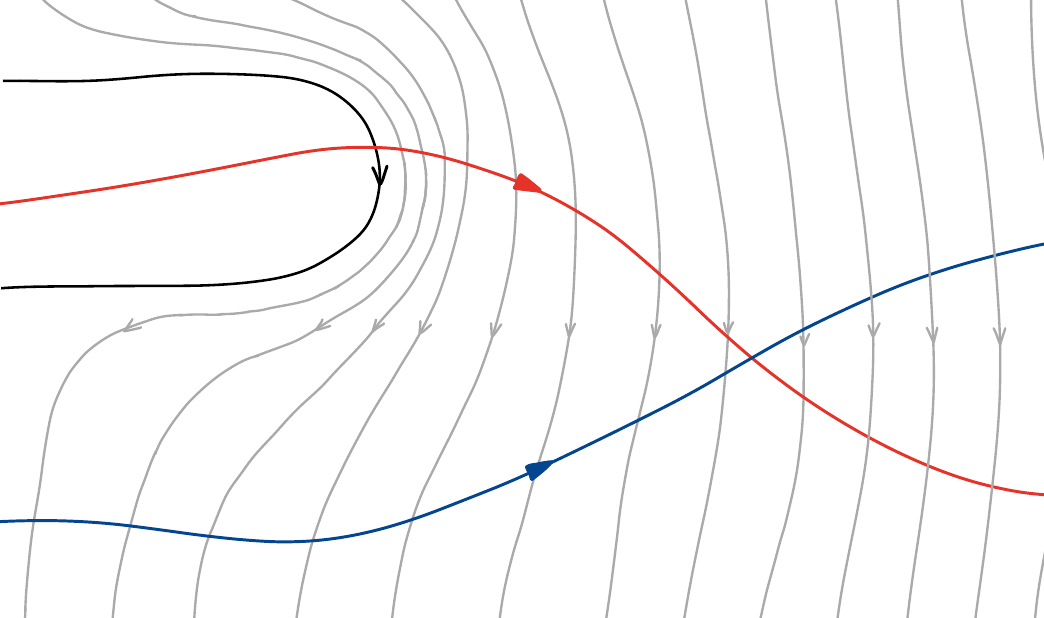}
        \put (57,36) {\colorbox{white}{\color{myRED}\large$\displaystyle \Gamma_{\O} $}}
        \put (57,6) {\colorbox{white}{\color{myBLUE}\large$\displaystyle \Gamma_{\O^\pp} $}}
\end{overpic}

\vspace*{0.1cm}
\begin{itemize}[leftmargin=1.3cm]
    \item[\textbf{(R3)}\ ] There exists a leaf $\phi' \in \partial_R (C_\O \cap C_{\O^{\sspc \prime}})$ such that $\phi' \in C_{\O^{\sspc \prime}} \cap\sspc \Rbot {\O}\sspc$.
\end{itemize}

    \center
    \vspace*{0.4cm}\begin{overpic}[width=6cm, height=3cm, tics=10]{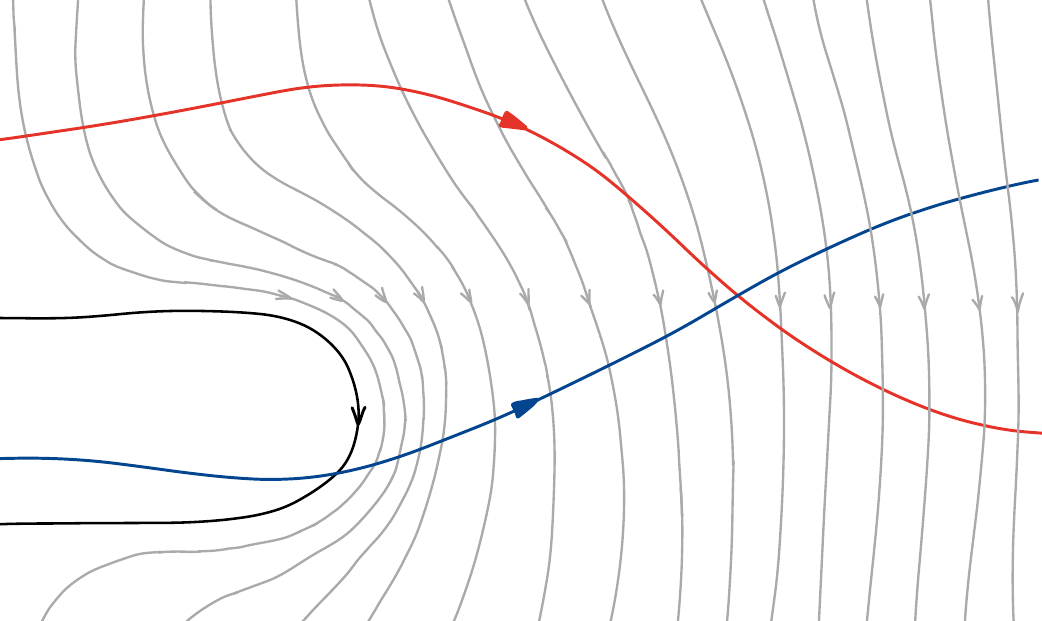}
        \put (56,40) {\colorbox{white}{\color{myRED}\large$\displaystyle \Gamma_{\O} $}}
        \put (56,11) {\colorbox{white}{\color{myBLUE}\large$\displaystyle \Gamma_{\O^\pp} $}}
\end{overpic}

\vspace*{0.1cm}
\begin{itemize}[leftmargin=1.3cm]
    \item[\textbf{(R4)}\ ] $\partial_R (C_\O \cap C_{\O^{\sspc \prime}}) =\partial_R C_\O = \partial_R C_{\O^{\sspc \prime}}$ and, moreover, $\sspc\Rtop {\O} \sspc\subseteq \sspc\Rtop {\O^{\sspc \prime}}\sspc$.
\end{itemize}

    \center
    \vspace*{0.4cm}\begin{overpic}[width=6cm, height=3cm, tics=10]{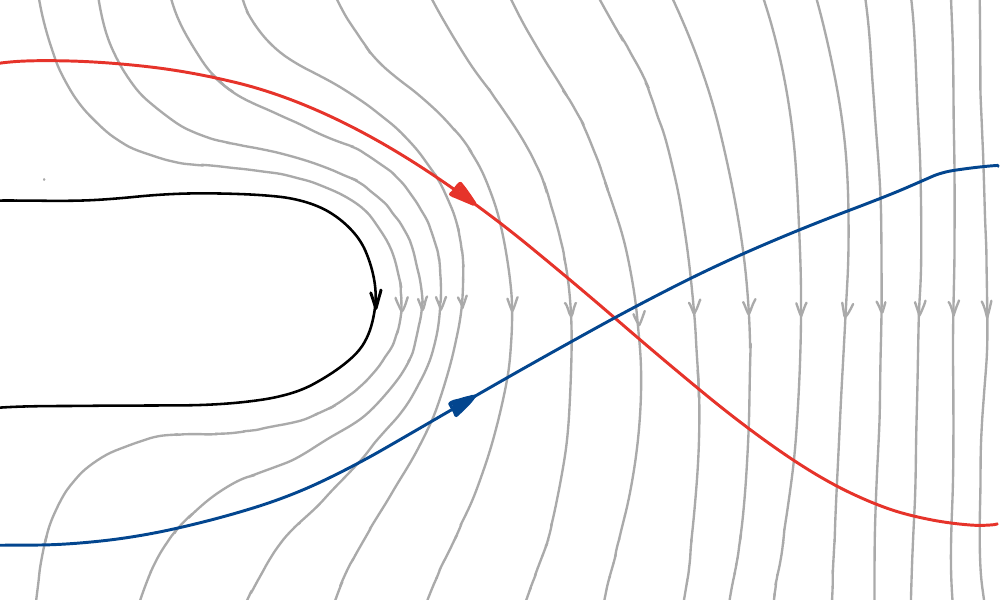}
        \put (48,38) {\colorbox{white}{\color{myRED}\large$\displaystyle \Gamma_{\O} $}}
        \put (48,9) {\colorbox{white}{\color{myBLUE}\large$\displaystyle \Gamma_{\O^\pp} $}}
\end{overpic}
\end{minipage}
\hspace{0.07\linewidth}
\begin{minipage}{0.5\textwidth}
$$  \O \lesssim_{L} \O^{\sspc \prime}$$  

\vspace*{0.2cm}\noindent if the orbits $\O$ and $\O^\pp$ satisfy $C_\O \cap C_{\O^{\sspc\prime}} \neq \varnothing$ and one of the following conditions:

\vspace*{0.2cm}
\begin{itemize}[leftmargin=1.3cm]
    \item[\textbf{(L1)}\ ] There exist \( \phi, \phi' \in \partial_L (C_\O \cap C_{\O^{\sspc \prime}}) \), with $\phi \prec \phi'$, such that
    \(
    \phi \in C_\O\spc, \spc\sspc \phi' \in C_{\O^{\sspc \prime}}.
    \)
\end{itemize}

    \center
    \vspace*{0.4cm}\begin{overpic}[width=6cm, height=3cm, tics=10]{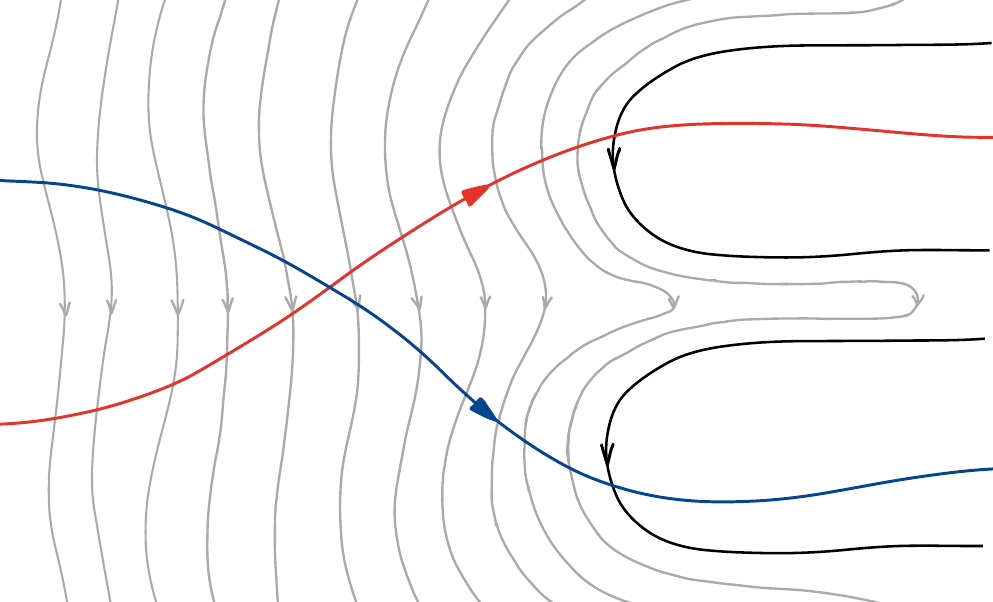}
        \put (35,40) {\colorbox{white}{\color{myRED}\large$\displaystyle \Gamma_{\O}$}}
        \put (34,9) {\colorbox{white}{\color{myBLUE}\large$\displaystyle \Gamma_{\O^\pp} $}}
        
\end{overpic}

\vspace*{0.1cm}
\begin{itemize}[leftmargin=1.3cm]
    \item[\textbf{(L2)}\ ] There exists a leaf $\phi \in \partial_L (C_\O \cap C_{\O^{\sspc \prime}})$ such that $\phi \in C_{\O}  \cap\sspc \Ltop {\O^{\sspc \prime}}\sspc$.
\end{itemize}

    \center
    \vspace*{0.4cm}\begin{overpic}[width=6cm, height=3cm, tics=10]{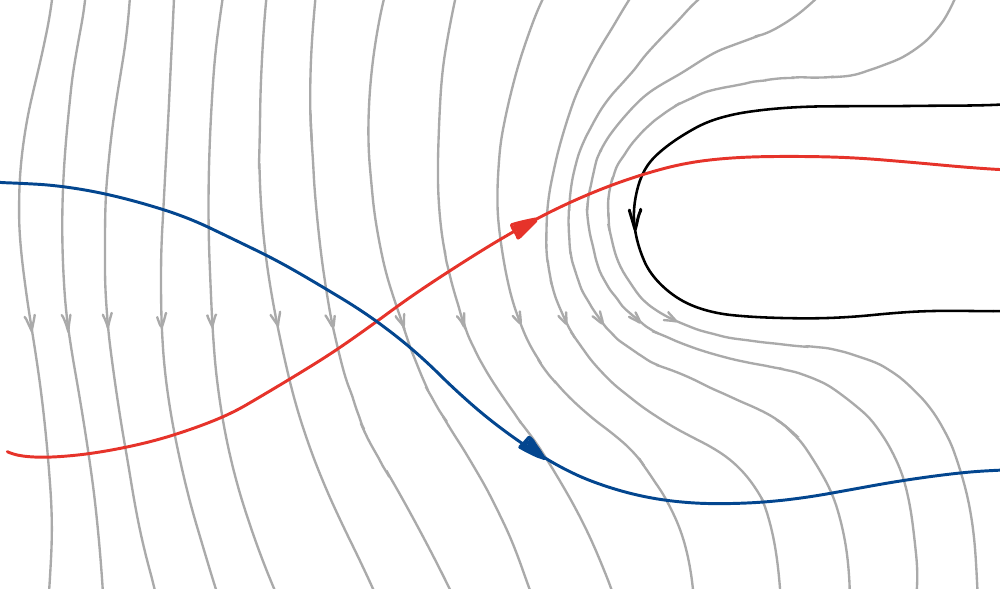}
        \put (39,32.5) {\colorbox{white}{\color{myBLUE}\large$\rule{0cm}{0.29cm} \ \ \  $}}
        \put (39,32.9) {{\color{myRED}\large$\displaystyle \Gamma_{\O} $}}
        \put (37,9) {\colorbox{white}{\color{myBLUE}\large$\rule{0cm}{0.25cm} \ \ \ \sspc $}}
        \put (38,9) {{\color{myBLUE}\large$\displaystyle \Gamma_{\O^\pp} $}}
\end{overpic}

\vspace*{0.1cm}
\begin{itemize}[leftmargin=1.3cm]
    \item[\textbf{(L3)}\ ] There exists a leaf $\phi' \in \partial_L (C_\O \cap C_{\O^{\sspc \prime}})$ such that $\phi' \in C_{\O^{\sspc \prime}} \cap\sspc \Lbot {\O}\sspc$.
\end{itemize}

    \center
    \vspace*{0.4cm}\begin{overpic}[width=6cm, height=3cm, tics=10]{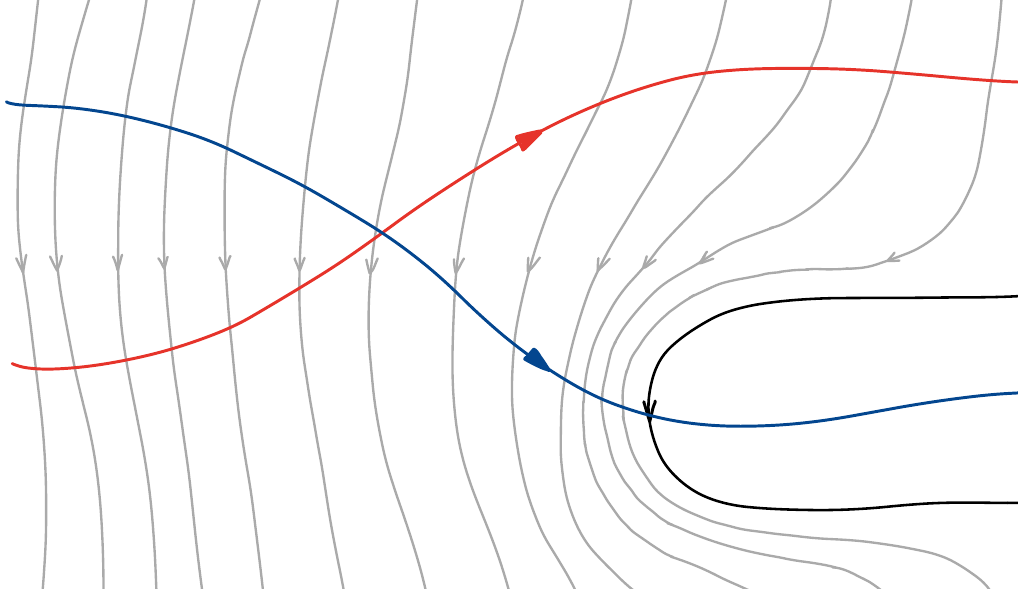}
        \put (38.7,40.5) {\colorbox{white}{\color{myRED}\large$\displaystyle \Gamma_{\O}$}}
        \put (40,11) {\colorbox{white}{\color{myBLUE}\large$\displaystyle \Gamma_{\O^\pp} $}}
\end{overpic}

\vspace*{0.1cm}
\begin{itemize}[leftmargin=1.3cm]
    \item[\textbf{(L4)}\ ] $\partial_L (C_\O \cap C_{\O^{\sspc \prime}}) =\partial_L C_\O = \partial_L C_{\O^{\sspc \prime}}$ and, moreover, $\sspc\Ltop {\O} \sspc\subseteq \sspc\Ltop {\O^{\sspc \prime}}\sspc$.
\end{itemize}

    \center
    \vspace*{0.4cm}\begin{overpic}[width=6cm, height=3cm, tics=10]{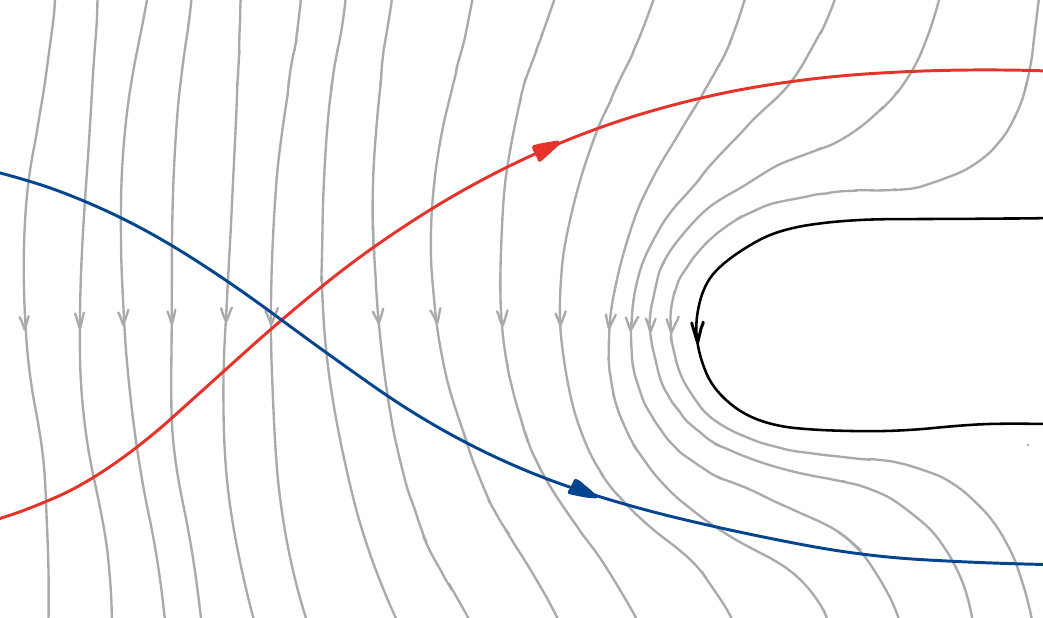}
        \put (37.5,41) {\colorbox{white}{\color{myRED}\large$\displaystyle \Gamma_{\O}$}}
        \put (39,4) {\colorbox{white}{\color{myBLUE}\large$\rule{0cm}{0.25cm} \ \ \ \sspc $}}
        \put (39,4) {{\color{myBLUE}\large$\displaystyle \Gamma_{\O^\pp} $}}
\end{overpic}

\end{minipage}
\end{minipage}
 

\vspace*{0.4cm}
\begin{proposition}\label{prop:preorders}
    The relations $\lesssim_L$ and $\lesssim_R$ satisfy the following properties:
    \begin{itemize}[leftmargin=1.5cm] 

        \item[\textbf{(i)}] Both relations $\lesssim_L$ and $\lesssim_R$ are total preorders on the set $\orbphi$, for every $\phi \in \F$.

        \item[\textbf{(ii)}] For any pair of orbits $\O, \O^{\sspc\prime} \in \orb$, the following equivalences hold
    \begin{align*}
        \O\fasym \O^{\sspc\prime} \iff  \O \lesssim_L \O^{\sspc\prime} \ \text{ and }\  \O^{\sspc\prime} \lesssim_L \O, \\
        \O\basym \O^{\sspc\prime} \iff  \O \lesssim_R \O^{\sspc\prime} \ \text{ and }\  \O^{\sspc\prime} \lesssim_R \O.
    \end{align*}
    \end{itemize}
\end{proposition}

\newpage

\begin{remark}
    The result of Proposition \ref{prop:preorders} is equivalent to say that $\lesssim_L$ and $\lesssim_R$ induce total orders on the quotient sets $\orbphi/{\fasym}$ and $\orbphi/{\basym}$, respectively. To be more precise, we have the following naturally induced relations on the quotient sets:
    \begin{itemize}
        \item The classes $[\O\sspc]^+ ,\sspc [\O^\pp\sspc]^+ \in \orbphi/{\fasym}$ are said to satisfy $[\O\sspc]^+ \lesssim_L [\O^\pp\sspc]^+$ if $\O \lesssim_L \O^\pp$.
        \item The classes $[\O\sspc]^- ,\sspc [\O^\pp\sspc]^-\in \orbphi/{\basym}$ are said to satisfy $[\O\sspc]^- \lesssim_R [\O^\pp\sspc]^-$ if $\O \lesssim_R \O^\pp$.
    \end{itemize}
\end{remark}

\begin{proof}[Proof of Proposition \ref{prop:preorders}]
    Before we prove item \textbf{(i)}, observe that any two orbits $\O$ and $\O^\pp$ that are comparable under $\lesssim_L$ (or $\lesssim_R$) must satisfy $ C_\O \cap C_{\O^\pp} \neq \varnothing$. The converse statement proves that $\lesssim_L$ and $\lesssim_R$ are total relations on the subsets $\orbphi \subset \orb$, for every $\phi \in \F$.

    \textit{\ \ Proof of (i):} 
    Let $\O, \O^{\sspc\prime} \in \orb$ be two orbits satisfying the condition $C_\O \cap C_{\O^{\sspc \prime}} \neq \varnothing$.
    Note that each of the orbits $\O$ and $\O^{\sspc \prime}$ is allowed to cross at most two leaves in $\partial (C_\O \cap C_{\O^{\sspc \prime}})$, one contained in $\partial_R (C_\O \cap C_{\O^{\sspc \prime}})$, and one in  $\partial_L (C_\O \cap C_{\O^{\sspc \prime}})$.
    Moreover, no leaf in $\partial (C_\O \cap C_{\O^{\sspc \prime}})$ can be crossed simultaneously by $\O$ and $\O^{\sspc \prime}$.
    Based on these observations, we conclude that the orbits $\O$ and $\O^{\sspc\prime}$ must satisfy one of the following:
    \begin{itemize}[leftmargin=1.4cm]
        \item[\textit{(1)}] \textit{The orbits ${\O}$ and ${\O^{\sspc \prime}}$ cross distinct leaves in ${\partial_L (C_\O \cap C_{{\O^{\sspc \prime}}})}\spc$:}\\ \hspace*{-1cm}  - There exists \( \phi, \phi' \in \partial_L (C_\O \cap C_{\O^{\sspc \prime}}) \), with $\phi \neq \phi'$, such that
        \(
        \phi \in C_\O \ \text{and} \ \phi' \in C_{\O^{\sspc \prime}}.
        \)
        \mycomment{0.2cm}
        \item[\textit{(2)}] \textit{The orbit ${\O^{\sspc \prime}}$ do not cross any leaf in ${\partial_L (C_\O \cap C_{\O^{\sspc \prime}})}$, but ${\O}$ does:}\\ \hspace*{-1cm}
        - The set $\partial_L (C_\O \cap C_{\O^{\sspc \prime}})$ coincides with $\partial_L C_{\O^{\sspc \prime}}$ and, moreover, $C_{\O} $ intersects $\partial_L C_{\O^{\sspc \prime}}$.
        \mycomment{0.2cm}
        \item[\textit{(3)}] \textit{The orbit ${\O}$ do not cross any leaf in ${\partial_L (C_\O \cap C_{\O^{\sspc \prime}})}$, but ${\O^{\sspc \prime}}$ does:}\\ \hspace*{-1cm}
        - The set $\partial_L (C_\O \cap C_{\O^{\sspc \prime}})$ coincides with $\partial_L C_{\O}$ and, moreover, $C_{\O^{\sspc \prime}} $ intersects $\partial_L C_{\O}$.
        \mycomment{0.2cm}
        \item[\textit{(4)}] \textit{The orbits ${\O}$ and ${\O^{\sspc \prime}}$ do not cross any leaf in ${\partial_L (C_\O \cap C_{{\O^{\sspc \prime}}})}$:}\\ \hspace*{-1cm}
        - The set $\partial_L (C_\O \cap C_{\O^{\sspc \prime}})$ coincides with both $\partial_L C_\O$ and $\partial_L C_{\O^{\sspc \prime}}$.
    \end{itemize}
    \mycomment{0.25cm}
    Similar observations apply to the behavior of the orbits $\O$ and $\O^{\sspc \prime}$ relative to $\partial_R (C_\O \cap C_{\O^{\sspc \prime}})$.

    By observing that, whenever $\O$ and $\O$ satisfy property \textit{(k)} for some $k \in \{1,2,3,4\}$, they also satisfy $\O\lesssim_L \O^{\sspc \prime}$ or $\O^{\sspc \prime} \lesssim_L \O$ according to the respective condition \textup{(Lk)}. This proves that the relation $\lesssim_L$ is total on every subset of the form $\orbphi$, for $\phi \in \F$. The same reasoning applies to the relation $\lesssim_R$ using the analogous versions of conditions \textit{(1)}, \textit{(2)}, \textit{(3)}, and \textit{(4)}.

    Note that the reflexivity of $\lesssim_L$ and $\lesssim_R$ is evident, as every orbit $\O \in \orb$ trivially satisfies conditions \textup{(L4)} and (R4) with itself.
    To conclude the proof of item \textup{(i)}, we show  that $\lesssim_L$ is transitive on the set $\orbphi$, for every $\phi \in \F$. The proof for $\lesssim_R$ is analogous. 

    \mycomment{0.2cm}
  \textup{\textbf{Proof of the transitivity of $\lesssim_L\sspc$:}}
   
     \mycomment{0.2cm}
    Let $\phi \in \F$, and let $\O, \O^{\sspc\prime}, \O^{\sspc\prime\prime}\in \orbphi$ be three orbits satisfying $\O\lesssim_L \O^{\sspc \prime}$ and $\O^{\sspc \prime} \lesssim_L \O^{\sspc \prime\prime}$. To prove the transitivity of the relation $\lesssim_L$, one must show that $\O \lesssim_L \O^{\sspc \prime \prime}$. To do so, we analyze the possible cases that can arise from the definitions of $\lesssim_L$, and prove that the inequality $\O \lesssim_L \O^{\sspc \prime \prime}$ holds in each of these cases.

For that, we consider the following sets
\begin{align*}
    &B := C_\O \cap C_{\O^{\sspc \prime}} \cap C_{\O^{\sspc \prime\prime}} \cap L(\phi),\\
    &B_1:= C_\O \cap C_{\O^{\sspc \prime}} \cap L(\phi),\\
    &B_2:= C_{\O^{\sspc \prime}} \cap C_{\O^{\sspc \prime\prime}} \cap L(\phi).
\end{align*}
These sets are non-empty and satisfy $B\subset B_1$ and $ B \subset B_2$. We claim that the set $B$ is equal to either $B_1$ or $B_2$, with the equality $B=B_1 = B_2$ being possibly true. Indeed, suppose by contradiction that the set $B$ is different from both $B_1$ and $B_2$. In this case, there exists a pair of leaves $\phi_1, \phi_2 \in \partial_L B$ such that $\phi_1 \in B_1$ and $\phi_2 \in B_2$.
    Note that, if these leaves $\phi_1$ and $\phi_2$ coincide, then the orbits $\O$, $\O^{\sspc \prime}$, and $\O^{\sspc \prime\prime}$ have to cross the leaf $\phi_1$. 
    This would imply that $\phi_1\in B$, which is a contradiction. Meanwhile, if $\phi_1$ and $\phi_2$ are distinct, this means that the orbit $\O^{\sspc \prime}$ intersects both sets $L(\phi_1)$ and $L(\phi_2)$. This also gives us a contradiction, proving our initial claim.

    \vspace*{-0.2cm}
    \begin{figure}[h!]
        \center 
        \mycomment{0.3cm}\begin{overpic}[width=4cm, tics=10]{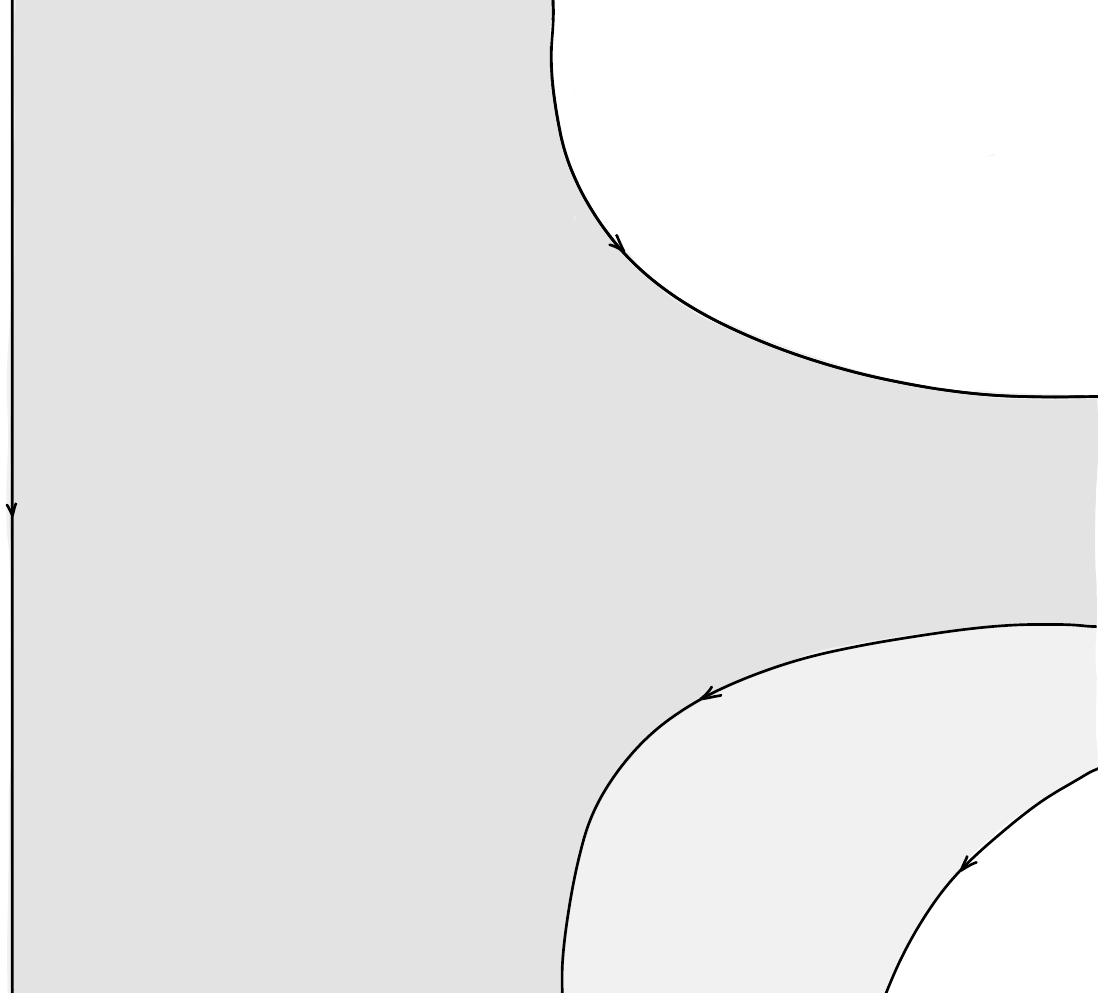}
            \put (26,41.5) {{\color{black}\large$\displaystyle B=B_1$}}
            \put (65,13) {{\color{black}\large$\displaystyle B_2$}}
            \put (-10,42) {{\color{black}\large$\displaystyle \phi$}}
    \end{overpic}
    \hspace*{1.2cm}\begin{overpic}[width=4cm, tics=10]{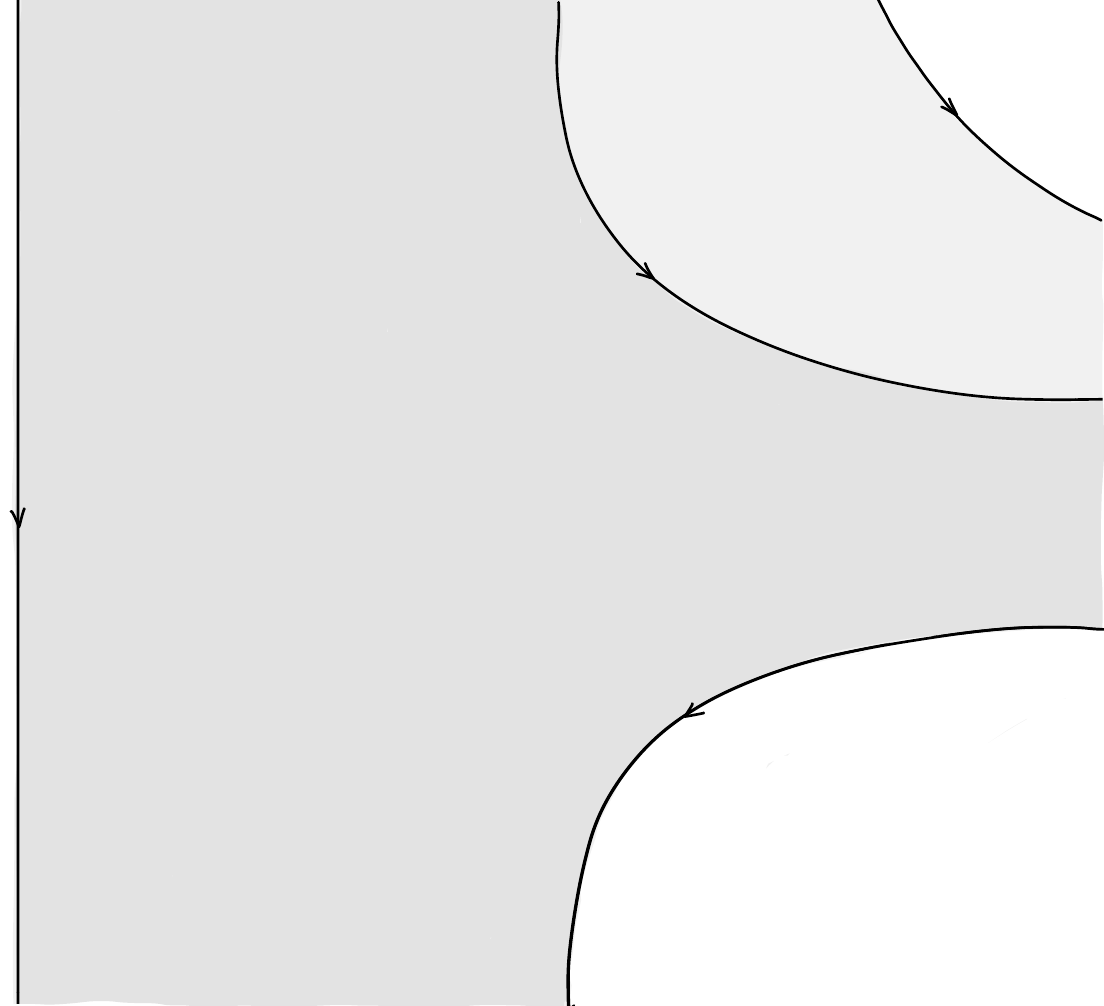}
        \put (26,41.5) {{\color{black}\large$\displaystyle B=B_2$}}
        \put (65,69) {{\color{black}\large$\displaystyle B_1$}}
        \put (-10,42) {{\color{black}\large$\displaystyle \phi$}}
        
    \end{overpic}
    \hspace*{1.2cm}\begin{overpic}[width=4cm, tics=10]{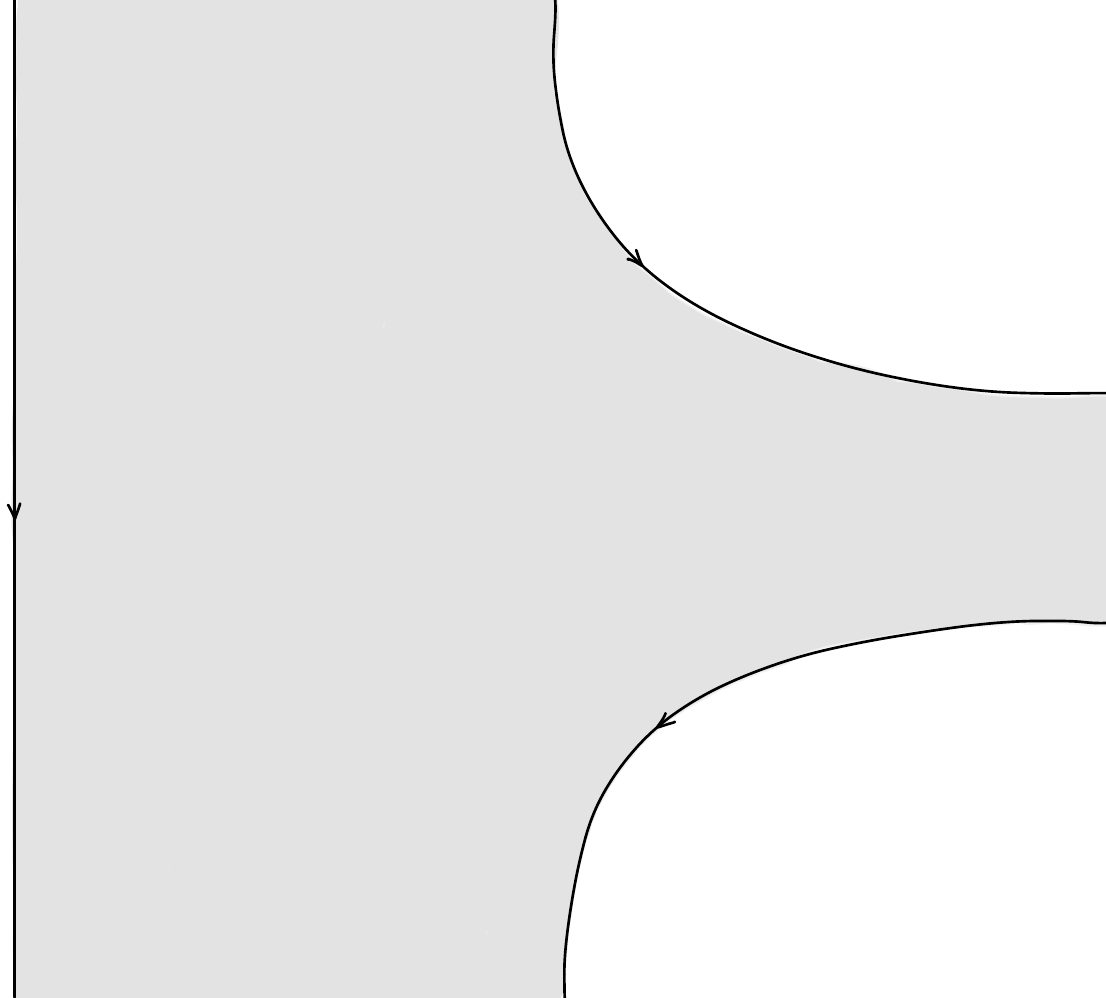}
        \put (16,41.5) {{\color{black}\large$\displaystyle B=B_1=B_2$}}
        \put (-10,42) {{\color{black}\large$\displaystyle \phi$}}
    \end{overpic}
    \end{figure}

    \noindent \textbf{Proof under the assumption that $B = B_1\neq B_2\sspc$:}
    
    \indent In this situation, we have the equality $\partial_L B = \partial_L (C_\O \cap C_{\O^{\sspc\prime}})$ and moreover, there exists some leaf $\phi_2 \in \partial_L (C_\O \cap C_{\O^{\sspc\prime}})$ that is crossed by both orbits $\O^{\sspc\prime}$ and $\O^{\sspc\prime\prime}$.

    \mycomment{0.1cm}
    \indent Since the orbits $\O$ and $\O^{\sspc\prime}$ satisfy $\O \lesssim_L \O^{\sspc \prime}$, we fall into two possible cases:
    \mycomment{0cm}
    \begin{itemize}[leftmargin=1.5cm]
        \item[\textup{(a)}] \textit{The orbit $\O$ do not cross any leaf in $\partial_L (C_\O \cap C_{\O^{\sspc\prime}})\sspc$:}\\ \hspace*{-0.8cm}  In this case, we have $\partial_L C_\O = \partial_L (C_\O \cap C_{\O^{\sspc\prime}})$ and $\phi_2$ must belong to $ \Lbot \O$.
        
        \mycomment{0.1cm}\item[\textup{(b)}] \textit{The orbit $\O$ cross some leaf in $\partial_L (C_\O \cap C_{\O^{\sspc\prime}})\sspc$:}\\ \hspace*{-0.8cm}
        In this case, there exists $\phi_1 \in \partial_L(C_\O \cap C_{\O^{\sspc \prime}})$, with $\phi_1 \prec \phi_2$, such that $\phi_1 \in C_\O$.
    \end{itemize}

    \vspace*{-0.3cm}
    \begin{figure}[h!]
        \center 
        \mycomment{0.5cm}\begin{overpic}[width=4.4cm, height=3.8cm, tics=10]{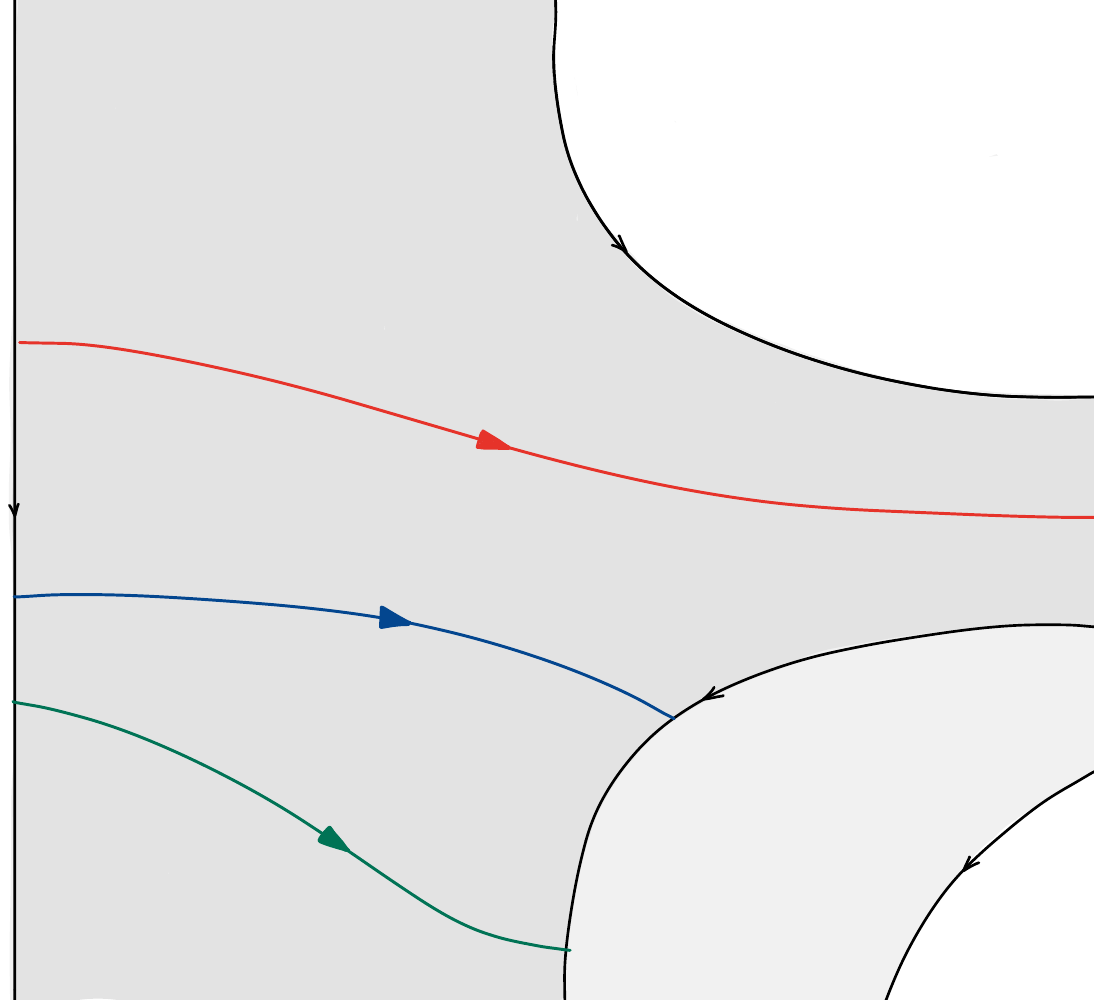}
            \put (25.5,56) {{\large\color{myRED}$\displaystyle \Gamma_{\O} $}}
            \put (9,9) {{\large\color{myGREEN}$\displaystyle \Gamma_{\O^{\sspc\prime\prime}} $}}
            \put (10,38) {{\large\color{myBLUE}$\displaystyle \Gamma_{\O^{\sspc\prime}} $}}
            \put (-15,78) {{\color{black}$\displaystyle \textup{(a)}\ $}}
            \put (66,16.5) {{\color{black}$\displaystyle \phi_2\ $}}
            \put (-10,42) {{\color{black}\large$\displaystyle \phi$}}
    \end{overpic}
    \hspace*{2.5cm}\begin{overpic}[width=4.4cm,height=3.8cm, tics=10]{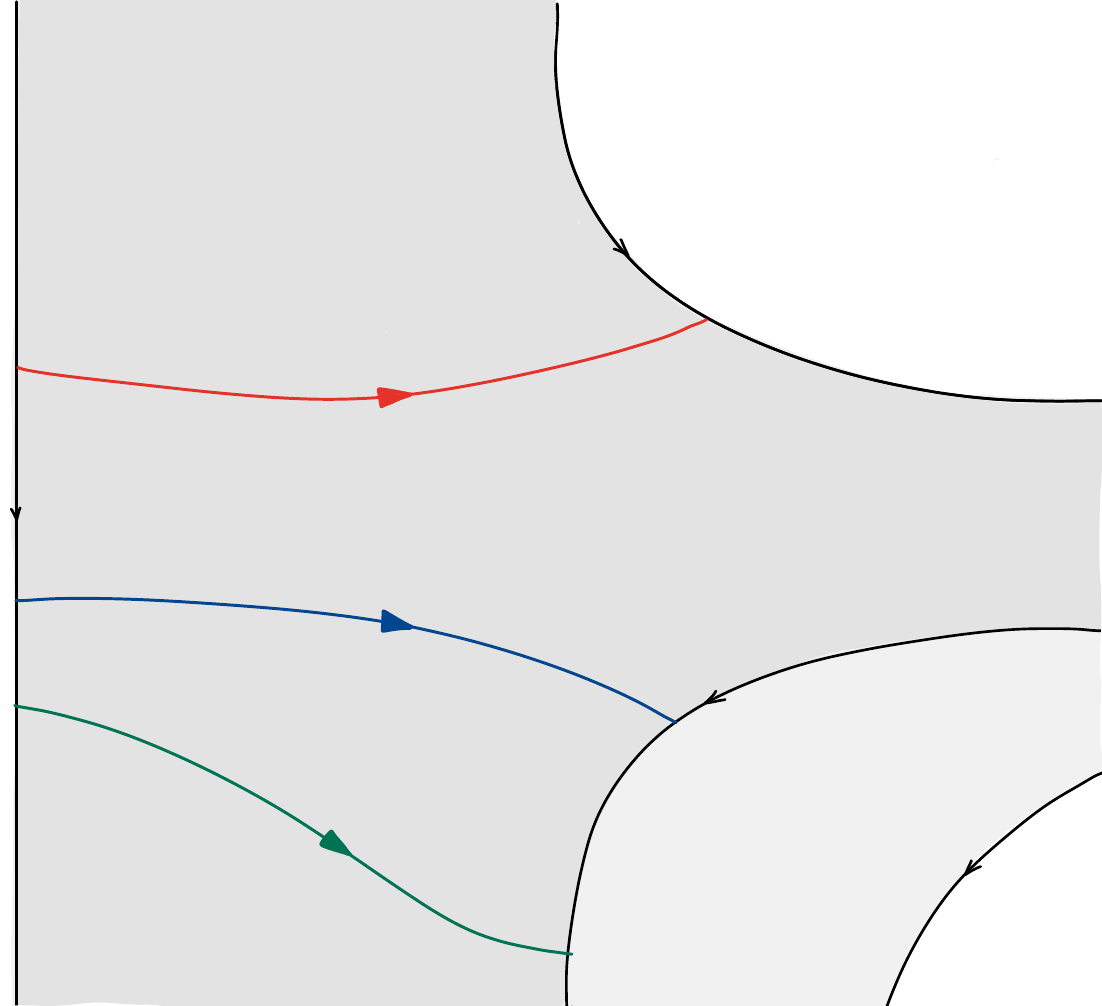}
            \put (23,58) {{\large\color{myRED}$\displaystyle \Gamma_{\O} $}}
            \put (10,38) {{\large\color{myBLUE}$\displaystyle \Gamma_{\O^{\sspc\prime}} $}}
            \put (9,9) {{\large\color{myGREEN}$\displaystyle \Gamma_{\O^{\sspc\prime\prime}} $}}
        \put (-15,78) {{\color{black}$\displaystyle \textup{(b)}\ $}}
        \put (66,16.5) {{\color{black}$\displaystyle \phi_2\ $}}
        \put (65,67) {{\color{black}$\displaystyle \phi_1\ $}}
        \put (-10,42) {{\color{black}\large$\displaystyle \phi$}}

    \end{overpic}
    \end{figure}

    \noindent In case (a), we have $\O\lesssim_L \O^{\sspc\prime\prime}$ through condition \textbf{(L3)}, and in (b), through condition \textbf{(L1)}.
 We remark that the proof under the assumption $B = B_2 \neq B_1$ is completely analogous.

\vspace*{0.3cm}
    \noindent \textbf{Proof under the assumption that $B = B_1= B_2\sspc$:}

In this situation, we have that $\partial_L B = \partial_L (C_\O \cap C_{\O^{\sspc\prime}}) = \partial_L (C_{\O^{\sspc\prime}} \cap C_{\O^{\sspc\prime\prime}}) = \partial_L (C_{\O} \cap C_{\O^{\sspc\prime\prime}}).$ We separate the proof into four different subcases, depending if $\O$, $\O^{\sspc\prime}$ or $\O^{\sspc\prime\prime}$ cross some leaf in $\partial_L B$.

\mycomment{0.2cm}
\underline{\textit{Subcase (1): None of the orbits $\O$, $\O^{\sspc\prime}$ and $\O^{\sspc\prime\prime}$ cross a leaf in $\partial_L B$: }} 
 In this case, we have that $\partial_L B = \partial_L C_\O = \partial_L C_{\O^{\sspc \prime}}= \partial_L C_{\O^{\sspc \prime\prime}}$.
\noindent This means that both  inequalities $\O \lesssim_L \O^{\sspc \prime}$ and $\O^{\sspc \prime} \lesssim_L \O^{\sspc \prime\prime}$ are guaranteed by condition (L4). 
Consequently, we have the inclusion
 $ \partial_L^{\spc\text{top}} \O \subset \partial_L^{\spc\text{top}} {\O^{\sspc \prime}} \subset \partial_L^{\spc\text{top}} {\O^{\sspc \prime\prime}}$, which implies $\O \lesssim_L \O^{\sspc \prime\prime}$ by condition (L4).

\vspace*{-0.2cm}
 \begin{figure}[h!]
    \center 
    \mycomment{0.25cm}\begin{overpic}[width=4cm, height=3.8cm, tics=10]{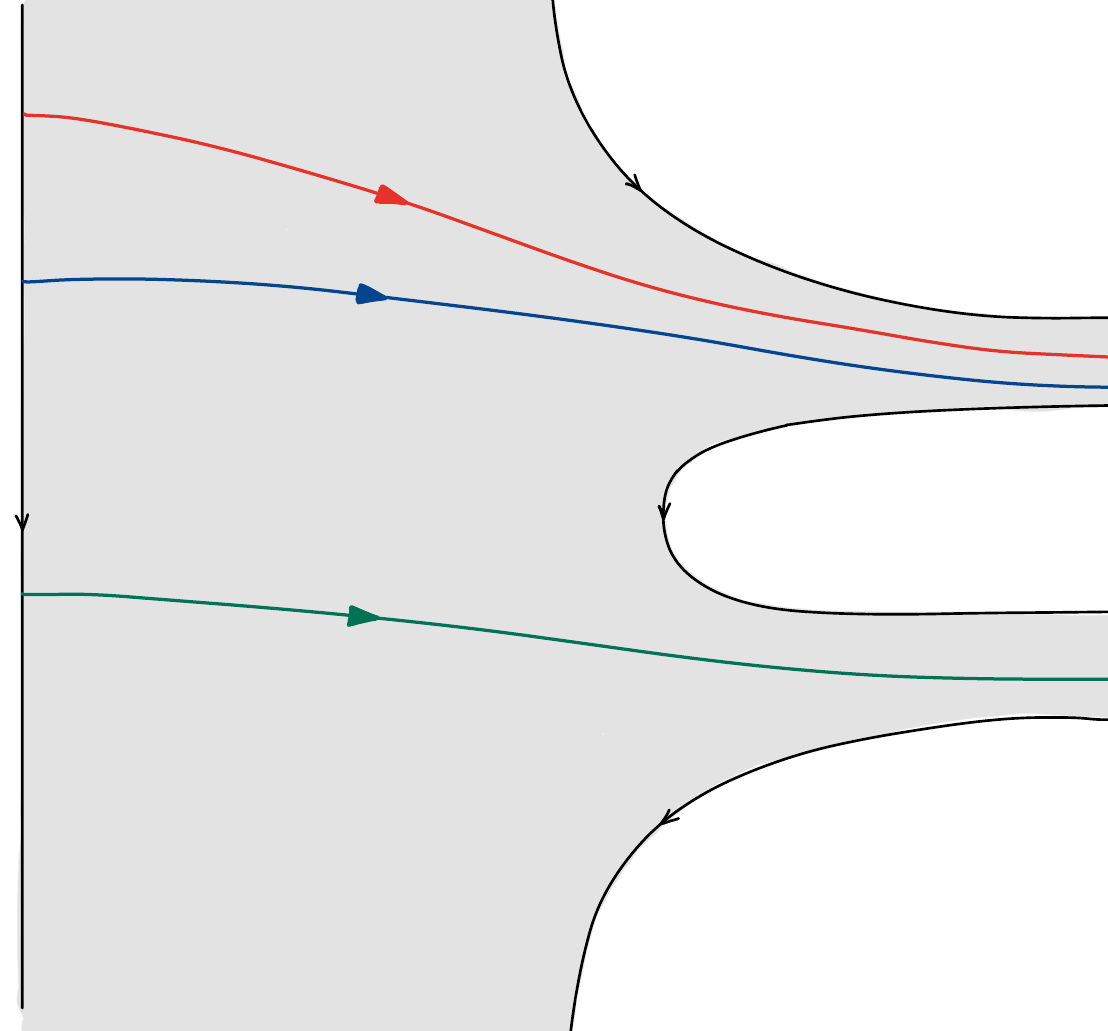}
        \put (33.4,81) {{\large\color{myRED}$\displaystyle \Gamma_{\O} $}}
        \put (25,23) {{\large\color{myGREEN}$\displaystyle \Gamma_{\O^{\sspc\prime\prime}} $}}
        \put (26,50) {{\large\color{myBLUE}$\displaystyle \Gamma_{\O^{\sspc\prime}} $}}
        \put (-10,42) {{\color{black}\large$\displaystyle \phi$}}
\end{overpic}
\end{figure}

\underline{\textit{Subcase (2): One orbit among $\O$, $\O^{\sspc\prime}$ and $\O^{\sspc\prime\prime}$ cross a leaf in $\partial_L B$: }} In this case, assume that the orbit $\O$ crosses a leaf in $\partial_L B$, while the other two orbits do not. The remaining possible combinations are somehow analogous, with the case where the orbit $\O^{\sspc\prime}$ crosses a leaf in $\partial_L B$ needing a slight modification in the arguments.
Here, we have that $\partial_L B = \partial_L C_{\O^{\sspc \prime}}= \partial_L C_{\O^{\sspc \prime\prime}}$, with some leaf $\phi_1 \in \partial_L B$ being crossed by the orbit $\O$, meaning, $\phi_1 \in C_\O$. In this context, $\O \lesssim_L \O^{\sspc \prime}$ is given by condition (L2),  implying that $ \phi_1 \in \Ltop {\O^{\sspc\prime}}.$
Meanwhile, $\O^{\sspc \prime} \lesssim_L \O^{\sspc \prime\prime}$ is given by condition (L4), implying that  $\partial_L^{\spc\text{top}} {\O^{\sspc \prime}} \subset \partial_L^{\spc\text{top}} {\O^{\sspc \prime\prime}}.$
This shows that $\phi_1 $ belongs to $ \partial_L^{\spc\text{top}} {\O^{\sspc \prime\prime}}\sspc$, thus proving that the inequality $\O \lesssim_L \O^{\sspc \prime\prime}$ holds via condition (L2).

\vspace*{-0.2cm}
\begin{figure}[h!]
    \center 
    \mycomment{0.25cm}\begin{overpic}[width=4cm, height=3.7cm, tics=10]{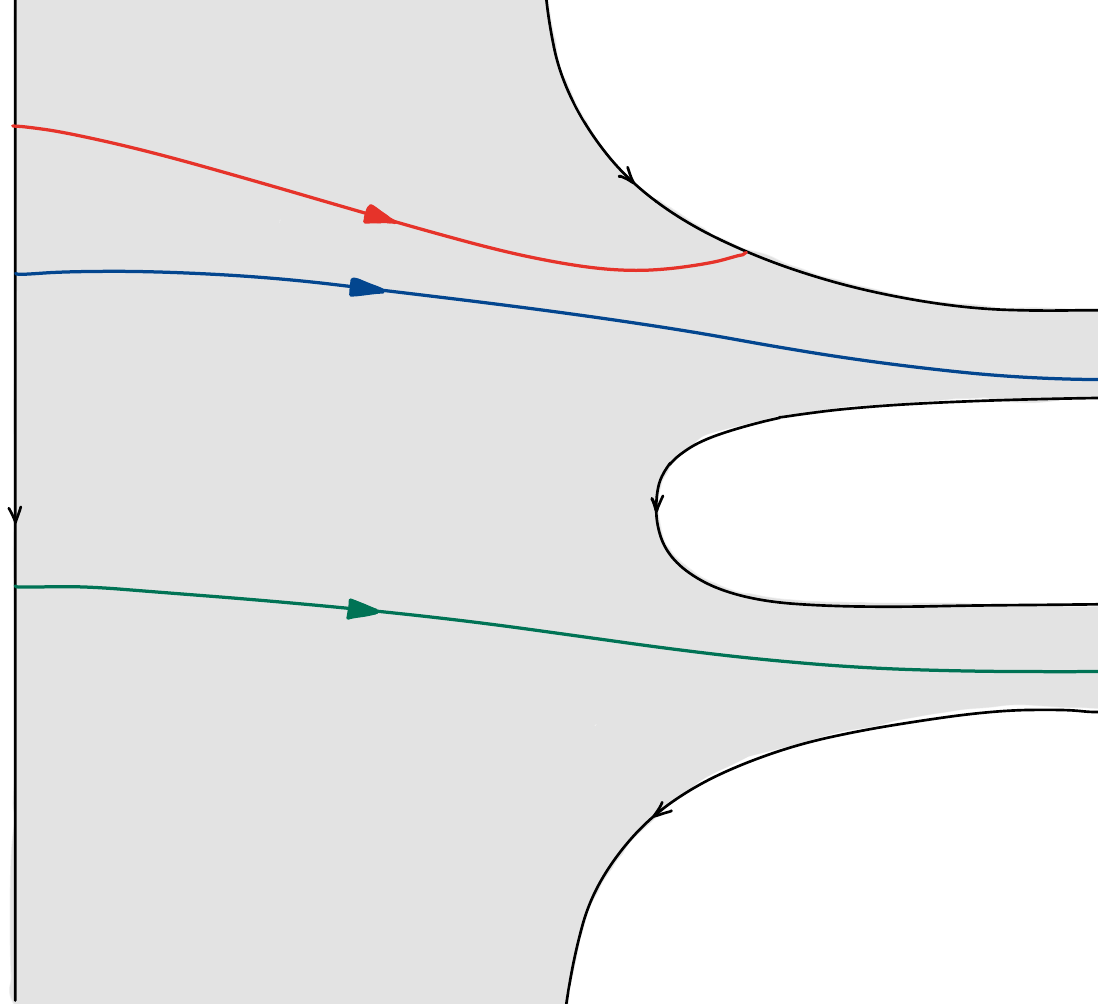}
        \put (30,80) {{\large\color{myRED}$\displaystyle \Gamma_{\O} $}}
        \put (25,23) {{\large\color{myGREEN}$\displaystyle \Gamma_{\O^{\sspc\prime\prime}} $}}
        \put (26,50) {{\large\color{myBLUE}$\displaystyle \Gamma_{\O^{\sspc\prime}} $}}
        \put (-10,42) {{\color{black}\large$\displaystyle \phi$}}
        \put (60.5,78.5) {{\color{black}\large$\displaystyle \phi_1$}}
\end{overpic}
\end{figure}

\vspace*{-0.2cm}
\underline{\textit{Subcase (3): Two orbits among $\O$, $\O^{\sspc\prime}$ and $\O^{\sspc\prime\prime}$ cross a leaf in $\partial_L B$: }} Now, assume that both $\O$ and $\O^{\sspc\prime}$ cross leaves in $\partial_L B$, while $\O^{\sspc\prime\prime}$ does not. The remaining possibilities are analogous, with some necessary modifications when $\O$ and $\O^{\sspc\prime\prime}$ are the ones crossing leaves in $\partial_L B$. 
Here, we have $ \partial_L B =\partial_L C_{\O^{\sspc \prime\prime}}$, with two leaves $\phi_1, \phi_2 \in \partial_L C_{\O^{\sspc \prime\prime}}$ satisfying $\phi_1 \in C_\O$ and $\phi_2 \in C_{\O^{\sspc \prime}}$. By the initial assumption, the leaves $\phi_1$ and $ \phi_2$ belong to $ \partial_L (C_\O \cap C_{\O^{\sspc\prime}})$ and are distinct. Thus, the inequality $\O \lesssim_L \O^{\sspc \prime}$ is given by condition (L1), implying that $ \phi_1 \prec \phi_2.$ Meanwhile, $\O^{\sspc \prime} \lesssim_L \O^{\sspc \prime\prime}$ is given by condition (L3), implying that $\phi_2 \in \Ltop {\O^{\sspc \prime\prime}}$. This shows that $\phi_1 $ belongs to $\partial_L^{\spc\text{top}} {\O^{\sspc \prime\prime}}$, thus implying $\O \lesssim_L \O^{\sspc \prime\prime}$ through (L2).

\vspace*{-0.2cm}
\begin{figure}[h!]
    \center 
    \mycomment{0.25cm}\begin{overpic}[width=4cm, height=3.7cm, tics=10]{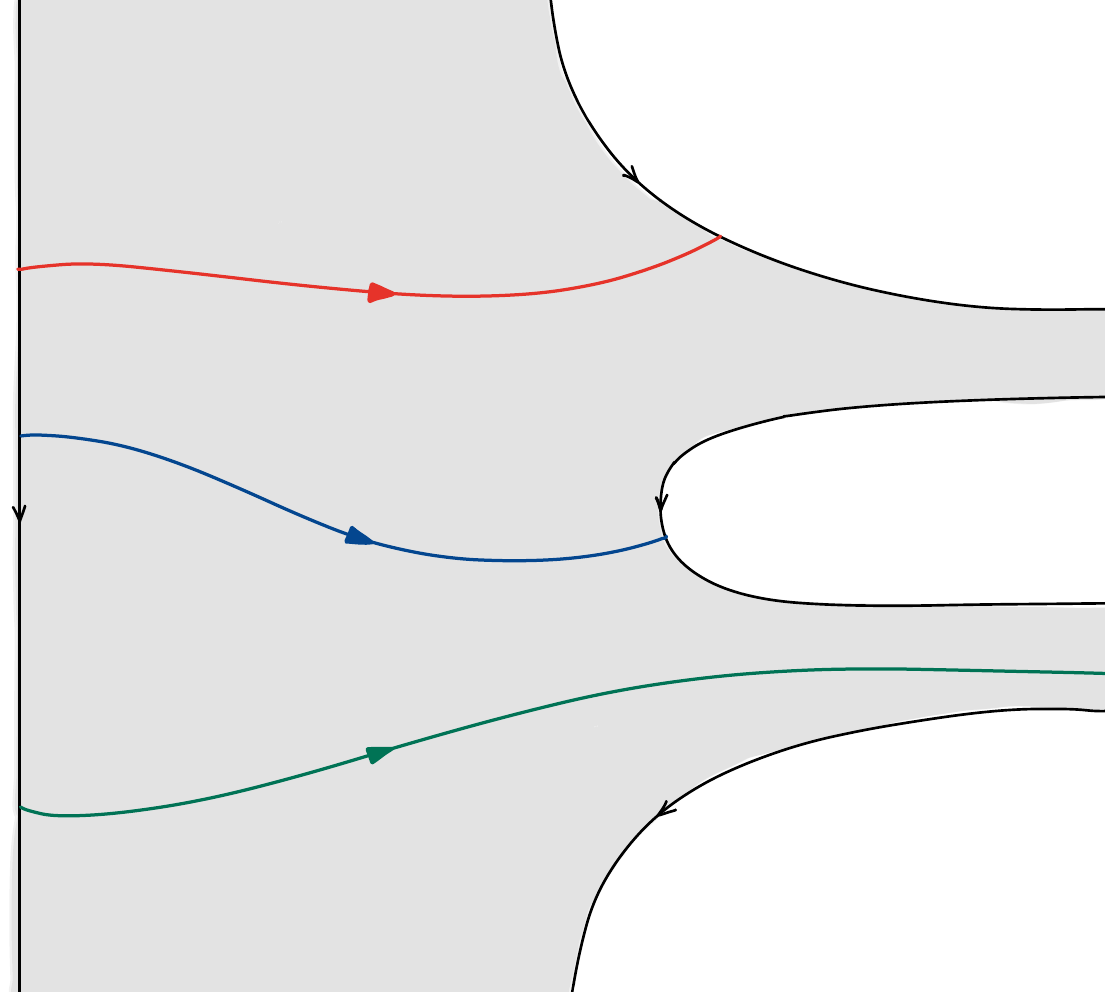}
        \put (30,71.5) {{\large\color{myRED}$\displaystyle \Gamma_{\O} $}}
        \put (22,26) {{\large\color{myGREEN}$\displaystyle \Gamma_{\O^{\sspc\prime\prime}} $}}
        \put (26,50) {{\large\color{myBLUE}$\displaystyle \Gamma_{\O^{\sspc\prime}} $}}
        \put (-10,42) {{\color{black}\large$\displaystyle \phi$}}
        \put (60.5,78.5) {{\color{black}\large$\displaystyle \phi_1$}}
        \put (67,43) {{\color{black}\large$\displaystyle \phi_2$}}
\end{overpic}
\end{figure}


\vspace*{-0.2cm}

\underline{\textit{Subcase (4): Each of the orbits $\O$, $\O^{\sspc\prime}$ and $\O^{\sspc\prime\prime}$ cross a leaf in $\partial_L B$.}}\\[0.4ex] In this case, there exist $\phi_1, \phi_2, \phi_3 \in \partial_L B$ satisfying $\phi_1 \in C_\O$, $\phi_2 \in C_{\O^{\sspc\prime}}$ and $\phi_3 \in C_{\O^{\sspc\prime\prime}}$. 
By the initial assumption, $\partial_L B = \partial_L (C_\O \cap C_{\O^{\sspc\prime}}) = \partial_L (C_{\O^{\sspc\prime}} \cap C_{\O^{\sspc\prime\prime}})=  \partial_L (C_{\O} \cap C_{\O^{\sspc\prime\prime}})$, with the three leaves  $\phi_1, \phi_2, \phi_3 \in \partial_L B$ being distinct. Thus, both $\O \lesssim_L \O^{\sspc \prime}$ and $\O^{\sspc \prime} \lesssim_L \O^{\sspc \prime\prime}$ are given by condition (L1), meaning that, $\phi_1 \prec \phi_2\prec \phi_3$. This implies $\O \lesssim_L \O^{\sspc \prime\prime}$ via condition (L1) and, thus, concludes the proof of transitivity of $\lesssim_L$ on $\orbphi$. This also ends the proof of (i).

\vspace*{-0.2cm}
\begin{figure}[h!]
    \center 
    \mycomment{0.2cm}\begin{overpic}[width=4cm, height=3.7cm, tics=10]{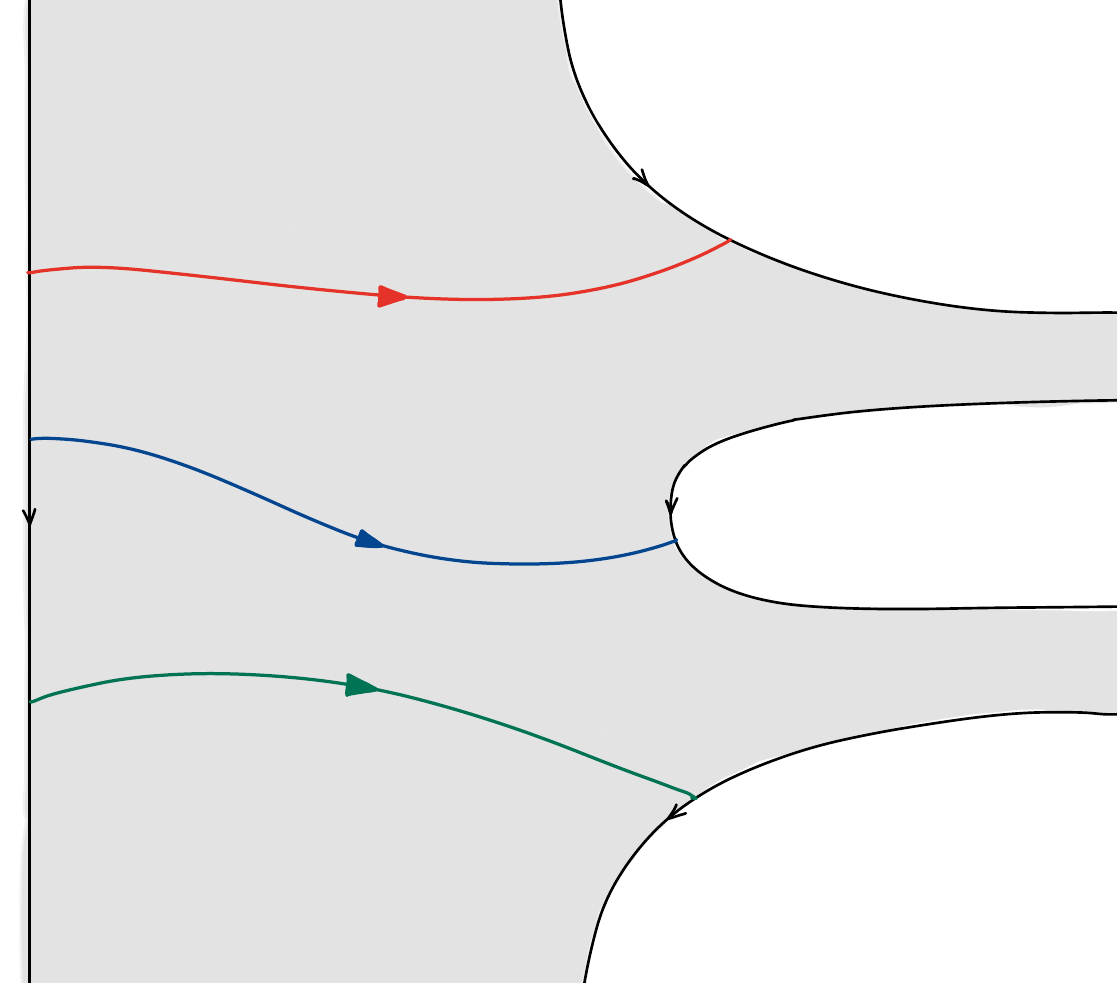}
        \put (30,70) {{\large\color{myRED}$\displaystyle \Gamma_{\O} $}}
        \put (25,16) {{\large\color{myGREEN}$\displaystyle \Gamma_{\O^{\sspc\prime\prime}} $}}
        \put (26,48) {{\large\color{myBLUE}$\displaystyle \Gamma_{\O^{\sspc\prime}} $}}
        \put (-10,42) {{\color{black}\large$\displaystyle \phi$}}
        \put (60,78) {{\color{black}\large$\displaystyle \phi_1$}}
        \put (67,41) {{\color{black}\large$\displaystyle \phi_2$}}
        \put (65,7) {{\color{black}\large$\displaystyle \phi_3$}}
\end{overpic}
\end{figure}


\mycomment{0cm}
\textit{Proof of (ii):} We prove the first equivalence, and the second one follows analogously.

$(\implies)$
Let $\O,\O^\pp \in \orb$ be two forward $\F$-asymptotic orbits. Then $\O$ and $\O^\pp$ satisfy condition \textup{(L4)} in both senses, which implies that $\O \lesssim_L \O^\pp$ and $\O^\pp \lesssim_L \O$. 

\mycomment{-0.1cm}
$(\impliedby)$
Conversely, let $\O,\O^\pp \in \orb$ be two orbits such that $\O \lesssim_L \O^\pp$ and $\O^\pp \lesssim_L \O$. Observe that the only condition in the definition of $\lesssim_L$ admitting symmetry is condition \textup{(L4)}.  Thus, the orbits $\O$ and $\O^\pp$ must satisfy condition (L4) in both senses, meaning that
\mycomment{-0.12cm}
$$ C_\O \cap C_{\O^\pp} \neq \varnothing\sspc, \quad \partial_L (C_\O \cap C_{\O^\pp}) = \partial_L C_\O = \partial_L C_{\O^\pp}\sspc, \quad  \Ltop  \O = \Ltop \O^\pp\sspc.$$

\mycomment{-0.23cm}
\noindent Note that the equalities $\partial_L C_\O = \partial_L C_{\O^\pp}$ and $\Ltop \O = \Ltop {\O^\pp}$ imply that $\Lbot \O =  \Lbot {\O^\pp}$.
 This concludes that the orbits $\O$ and $\O^{\sspc \prime}$ are forward $\F$-asymptotic.
\end{proof}

\mycomment{-0.2cm}
\begin{remark}
    The relations $\lesssim_L$ and $\lesssim_R$ may fail to be transitive on the entire set $\orb$. To see that, it suffices to consider the simple example where $\O, \O^{\sspc\prime}, \O^{\sspc\prime\prime} \in \orb$ satisfy
    $$ C_\O \cap C_{\O^{\sspc\prime}} \neq \varnothing\sspc, \quad C_{\O^{\sspc\prime}} \cap C_{\O^{\sspc\prime\prime}} \neq \varnothing\sspc, \quad C_\O \cap C_{\O^{\sspc\prime\prime}} = \varnothing.$$
    In this case, transitivity fails because $\O$ and $\O^{\sspc\prime\prime}$ are not comparable under $\lesssim_L$ and $\lesssim_R$.

 \end{remark}

\begin{remark}\label{sec:description_order_proper_trajectories}
The configuration of proper transverse trajectories can provide information about the ordering of the associated orbits with respect to $\lesssim_L$ and $\lesssim_R$. For instance:

\vspace*{0.2cm}

    Let $\phi \in \F$ be a leaf, and suppose that two orbits $\O, \O^\pp \in \orbphi$ admit respective proper transverse trajectories $\Gamma_\O$ and $\Gamma_{\O^\pp}$ satisfying the following conditions:
\begin{itemize}[leftmargin=1.3cm]
    \item The trajectories $\Gamma_\O$ and $\Gamma_{\O^\pp}$ are disjoint on the left side of $\phi$, that is,
    \mycomment{-0.15cm}
    $$\Gamma_\O \cap \Gamma_{\O^\pp} \cap \overline{\rule{0cm}{0.35cm}L(\phi)} = \varnothing.$$

    \mycomment{-0.15cm}
    \item The intersection points $p_\O \in \Gamma_\O \cap \phi$ and $p_{\O^\pp} \in \Gamma_{\O^\pp} \cap \phi$ are ordered as
    \mycomment{-0.15cm}
    $$p_\O < p_{\O^\pp} \quad \text{ according to the orientation of } \phi.$$
\end{itemize}
\mycomment{0.1cm}
In this scenario, the orbits $\O$ and $\O^\pp$ are forced to satisfy the inequality
$\O \lesssim_L \O^\pp$. However, this alone does not provide enough information to conclude whether \(\O^\pp \lesssim_L \O\) also holds, or not. 

\mycomment{0.3cm}
We highlight that the existence of the proper transverse trajectories \(\Gamma_\O\) and \(\Gamma_{\O^\pp}\), as described above, is assumed at this stage. Theorem \ref{prop:pairwise_disj_traj} later guarantees the existence of such proper transverse trajectories under a generic condition on the foliation \(\F\).
\end{remark}

\begin{remark}

If \(\Gamma_\O\) and \(\Gamma_{\O^\pp}\) are not proper, they fail entirely to convey information about how the orbits \(\O\) and \(\O^\pp\) are related under $\lesssim_L$ and $\lesssim_R$. 
For example, consider the configuration illustrated below, in which the orbits $\O$ and $\O^{\sspc\prime}$ do not satisfy $\O \lesssim_L \O^{\sspc\prime}$, yet they admit non-proper transverse trajectories \(\Gamma_\O\) and \(\Gamma_{\O^\pp}\) that satisfy
$$ \Gamma_\O \cap \Gamma_{\O^{\sspc\prime}} \cap \overline{\rule{0mm}{0.36cm}L(\phi)} = \varnothing\sspc \quad \text{ and } \quad p_\O < p_{\O^{\sspc\prime}} \text{ according to the orientation of } \phi,$$
where \(p_\O\) and \(p_{\O^{\sspc\prime}}\) are the intersection points of $\Gamma_\O$ and $\Gamma_{\O^\pp}$ with the leaf \(\phi\), respectively.

\mycomment{0.5cm}
\begin{figure}[h!]
    \center
    \hspace*{0.3cm}\begin{overpic}[width=0.65\textwidth,height=4cm, tics=10]{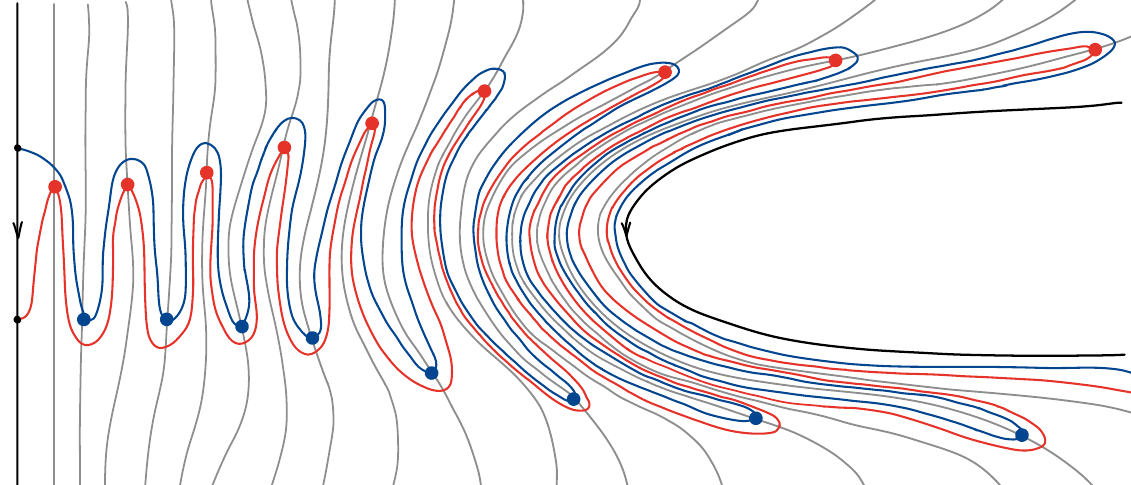}
        \put (67,18.5) {\color{myRED}\large$\displaystyle \O^{\sspc\prime}\spc$\color{black}\normalsize $\displaystyle\lesssim_L \spc$\color{myBLUE}\large$\displaystyle \O$ }
        \put (-6,25) {\color{myBLUE}\large$\displaystyle p_\O $}
        \put (-6,10) {\color{myRED}\large$\displaystyle p_{\O^{\sspc\prime}} $}
        \put (-3,18) {\color{black}$\displaystyle \phi$}
        \put (13.5,30) {\colorbox{white}{\color{myBLUE}\large$\displaystyle \Gamma_\O$}}
        \put (11,4) {\colorbox{white}{\color{myRED}\large$\displaystyle \Gamma_{\O^{\sspc\prime}}$}}
\end{overpic}
\end{figure}

\end{remark}

    \mycomment{-0.5cm}
\newpage

\section{Weak $\F$-transverse intersections}\label{sec:def_weak_transverse_intersections}

In this section, we introduce the notion of weak $\F$-transverse intersection between orbits, which evaluates orbits that are ordered in contradictory ways by the relations $\lesssim_L$ and $\lesssim_R$.

    To recall, from Definition \ref{def:weak_transverse_intersection}, two orbits $\O, \O^{\sspc \prime} \in \orb$ are said to have a \textit{weak $\F$-transverse intersection} if they satisfy the following conditions:
    \begin{itemize}[leftmargin=1.15cm]
\item[\textbf{(i)}] The orbits $\O$ and $\O^\pp$ are neither forward nor backward $\F$-asymptotic.
\item[\textbf{(ii)}] The orbits $\O$ and $\O^\pp$ are oppositely ordered by $\lesssim_L$ and $\lesssim_R\sspc$, that is, they satisfy 
\begin{align*}
    \text{Either } \quad &\O \lesssim_L \O^{\sspc \prime} \spc \text{ and } \spc \O^{\sspc \prime} \lesssim_R \O,& \  \\ 
    \text{or } \quad &\O^{\sspc \prime} \lesssim_L \O \spc \text{ and } \spc\O \lesssim_R \O^{\sspc \prime}.& \  
\end{align*}
\end{itemize}

\vspace*{0.1cm}

Before we proceed, we clarify the definition above through a brief discussion:

\vspace*{0.2cm}
\noindent Let $\O, \O^{\sspc \prime} \in \orb$ be two orbits having a weak $\F$-transverse intersection. By condition (ii),  the orbits $\O$ and $\O^{\sspc \prime}$ are comparable under $\lesssim_L$ and $\lesssim_R$, and this implies that $C_\O \cap C_{\O^{\sspc \prime}} \neq \varnothing$. 
Next, by condition (i) and Lemma \ref{lemma:dual_equivalence}, we know that for any pair of proper transverse trajectories $\Gamma_\O$ and $\Gamma_{\O^{\sspc \prime}}$ associated to $\O$ and $\O^{\sspc \prime}$, there exist leaves $\phi_R, \phi_L \in C_\O \sspc \cap \sspc C_{\O^{\sspc \prime}}$ satisfying
$$ \quad L(\phi_L) \subset L(\phi_R) \quad \text{ and } \quad \Gamma_\O \cap \Gamma_{\O^{\sspc \prime}} \cap \overline{\rule{0cm}{0.35cm}R(\phi_R) \cap L(\phi_L)} = \varnothing\sspc.$$
Finally, as noted in Remark \ref{sec:description_order_proper_trajectories}, the proper transverse trajectories $\Gamma_\O$ and $\Gamma_{\O^{\sspc \prime}}$ intersect the leaves $\phi_R$ and $\phi_L$ at opposite orders, reflecting the fact that \(\O\) and \(\O^\pp\) are oppositely ordered by the relations $\lesssim_L$ and $\lesssim_R$. Observe that this arrangement forces the proper transverse trajectories $\Gamma_\O$ and $\Gamma_{\O^{\sspc \prime}}$ to intersect at least once in between the leaves $\phi_L$ and $\phi_R$.

\vspace*{0.3cm}
\begin{figure}[h!]
    \center \begin{overpic}[width=7cm, height=3.5cm, tics=10]{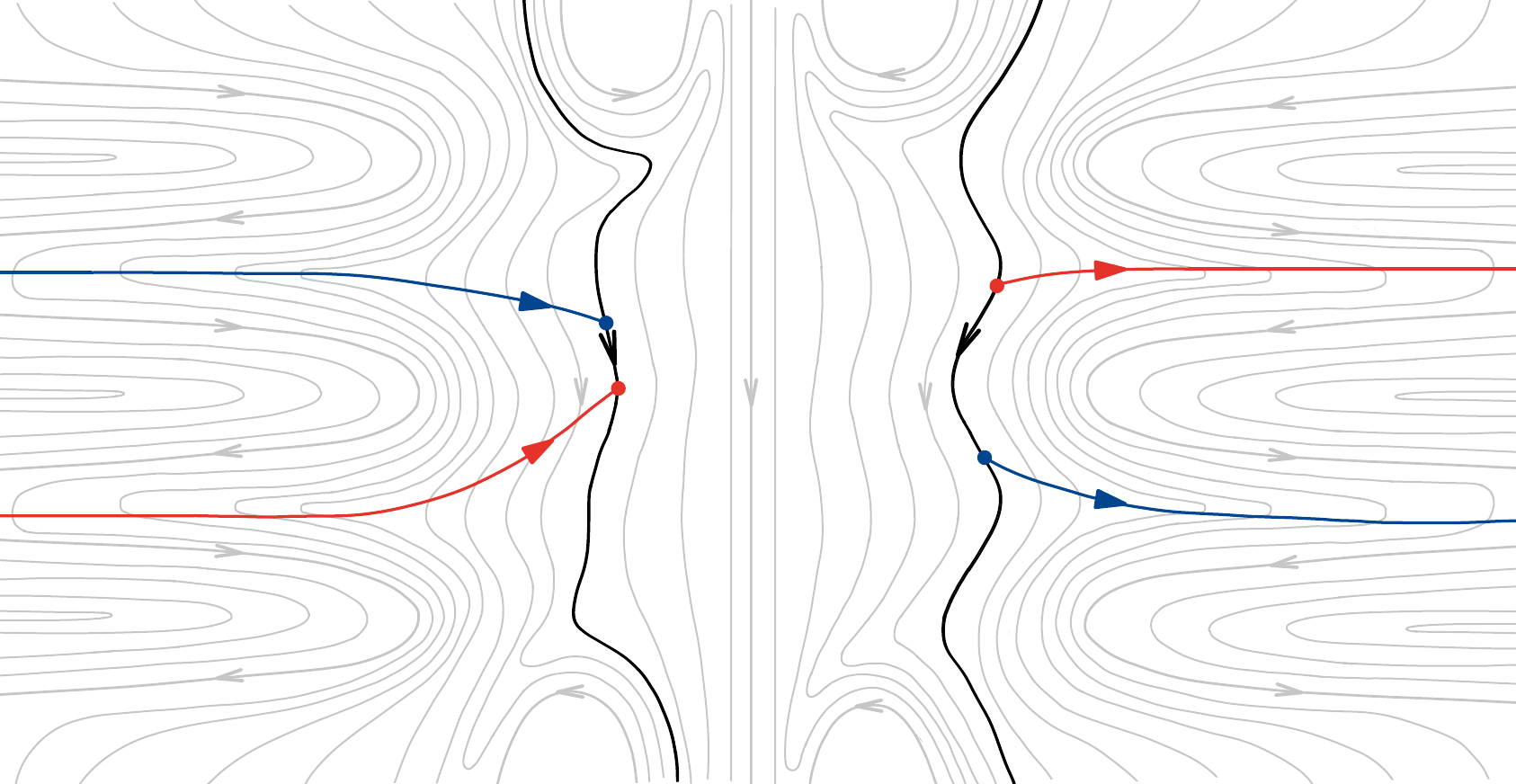}
        \put (33.5,53) {{\color{black}\large$\displaystyle \phi_R$}}
    \put (65,53) {{\color{black}\large$\displaystyle \phi_L$}}
        \put (-10,31) {{\large\color{myBLUE}$\displaystyle \Gamma_{\O^{\sspc\prime}} $}}
        \put (-9,14) {{\large\color{myRED}$\displaystyle \Gamma_{\O} $}}
        \put (102.5,31) {{\large\color{myRED}$\displaystyle \Gamma_{\O} $}}
        \put (102.5,14) {{\large\color{myBLUE}$\displaystyle \Gamma_{\O^{\sspc\prime}} $}}
\end{overpic}
\end{figure}


Below, we illustrate a feel examples of $\O$ and $\O^{\prime \prime}$ having a weak $\F$-transverse intersection:

\begin{figure}[h!]
    \center
    
    \begin{overpic}[width=6.5cm, height=3.3cm, tics=10]{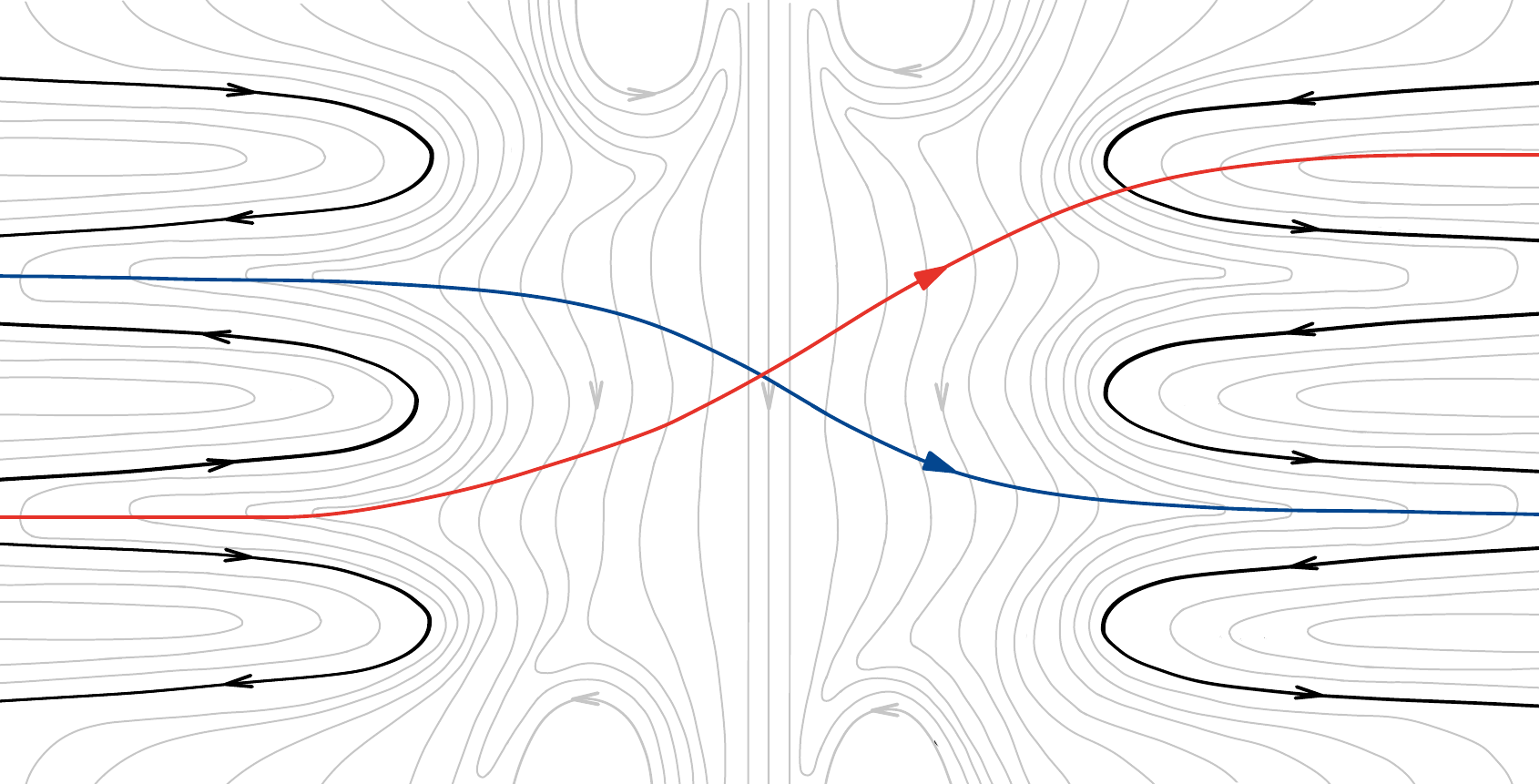}
        \put (54.9,13.3) {\colorbox{white}{$\rule{0cm}{0.3cm}\  \ $}}
        \put (54.9,36) {\colorbox{white}{$\rule{0cm}{0.3cm}\ \  $}}
         \put (55,36.5) {{\color{myRED}$\displaystyle \Gamma_{\O}$}}
         \put (55,14) {{\color{myBLUE}$\displaystyle \Gamma_{\O'}$}}
\end{overpic}
\hspace{1cm}
\begin{overpic}[width=6.5cm, height=3.3cm, tics=10]{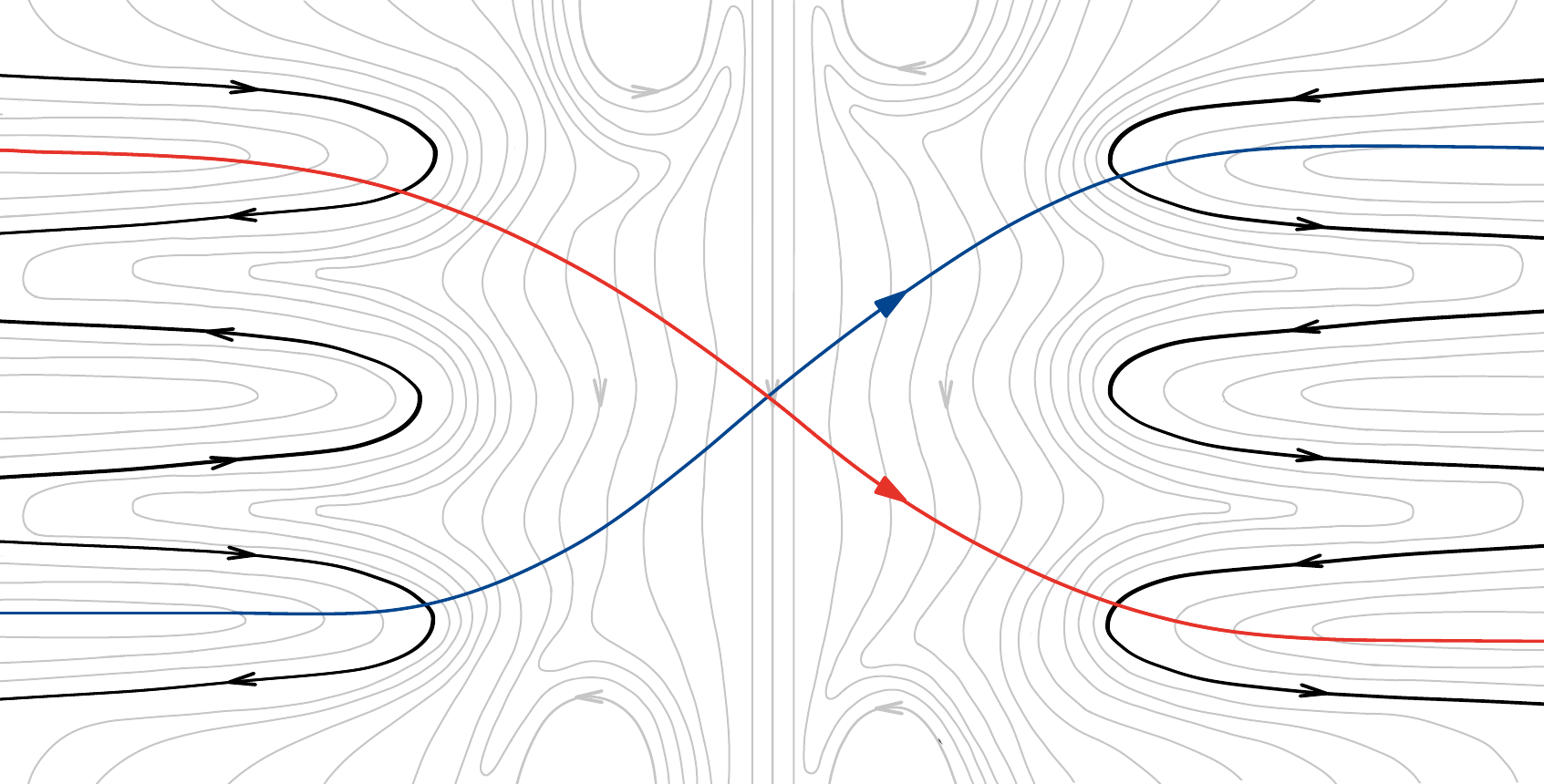}
        \put (54.9,9.8) {\colorbox{white}{$\rule{0cm}{0.3cm}\  \ $}}
        \put (54.9,36) {\colorbox{white}{$\rule{0cm}{0.3cm}\ \  $}}
         \put (55,36.5) {{\color{myBLUE}$\displaystyle \Gamma_{\O}$}}
         \put (55,10) {{\color{myRED}$\displaystyle \Gamma_{\O'}$}}
\end{overpic}
\end{figure}

\begin{figure}[h!]
    \center
    \vspace*{-0.12cm}
\begin{overpic}[width=6.5cm, height=3.3cm, tics=10]{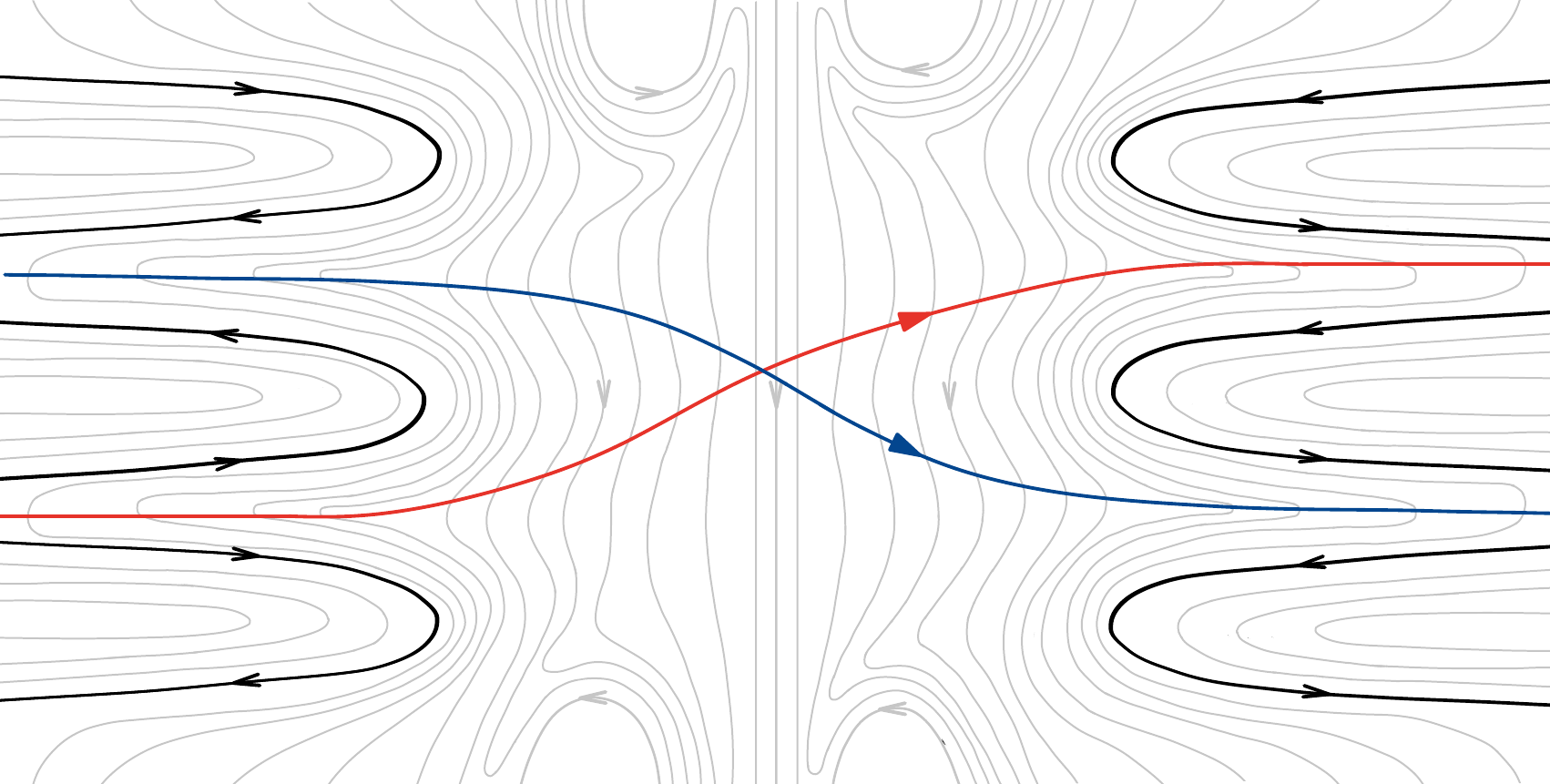}
        \put (54.9,13.3) {\colorbox{white}{$\rule{0cm}{0.3cm}\  \ $}}
        \put (54.9,33.8) {\colorbox{white}{$\rule{0cm}{0.3cm}\ \  $}}
         \put (55,34.3) {{\color{myRED}$\displaystyle \Gamma_{\O}$}}
         \put (55,14) {{\color{myBLUE}$\displaystyle \Gamma_{\O'}$}}
\end{overpic}
\hspace{1cm}
\begin{overpic}[width=6.5cm, height=3.3cm, tics=10]{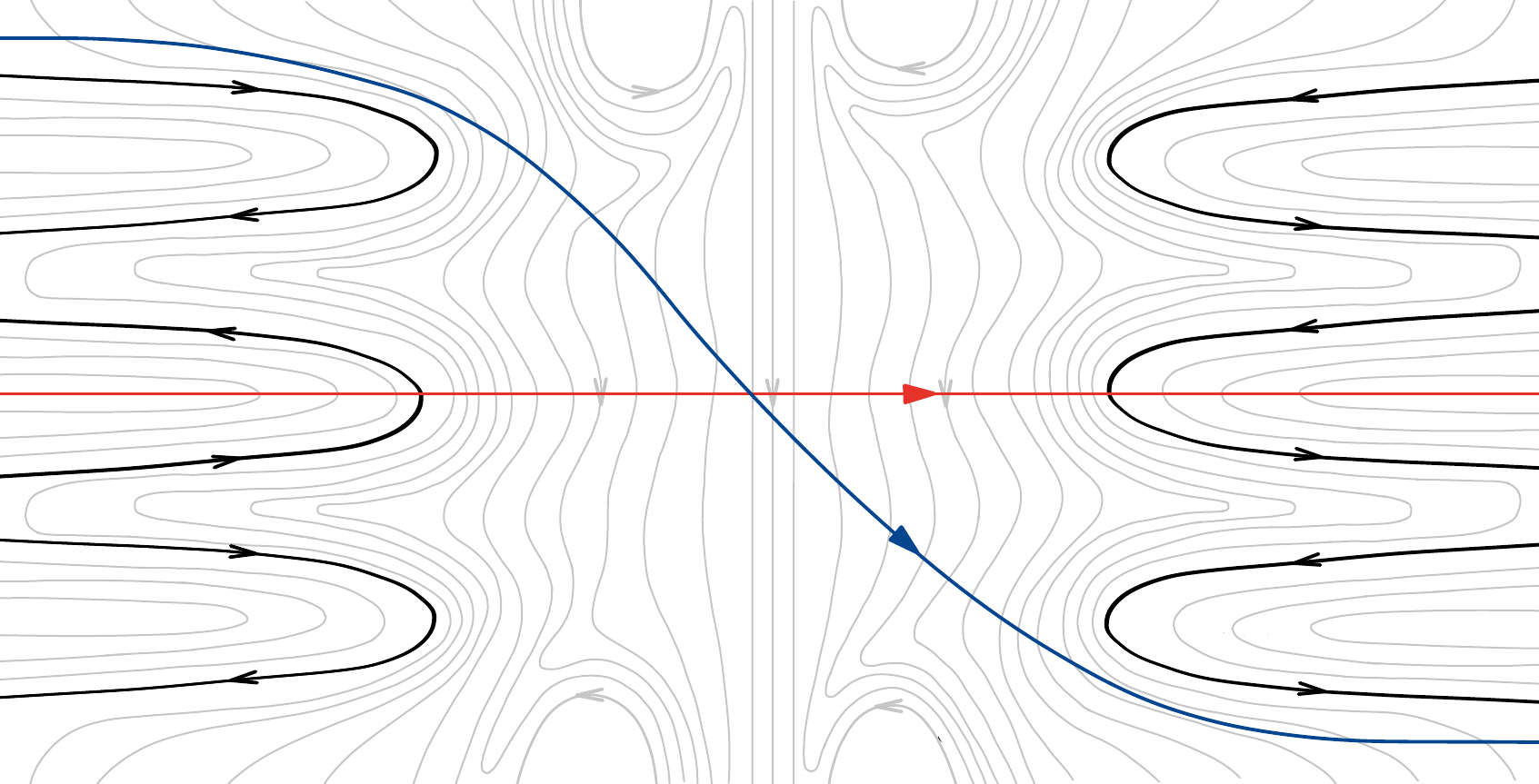}
        \put (52.4,9) {\colorbox{white}{$\rule{0cm}{0.2cm}\  \ \spc $}}
        \put (54.9,29) {\colorbox{white}{$\rule{0cm}{0.3cm}\ \  $}}
         \put (55,29.4) {{\color{myRED}$\displaystyle \Gamma_{\O}$}}
         \put (52.3,8.8) {{\color{myBLUE}$\displaystyle \Gamma_{\O'}$}}
\end{overpic}
\end{figure}

\newpage

\begin{remark}
    To any reader well-versed in the framework of Forcing Theory for surface homeomorphisms, the name \textit{weak $\F$-transverse intersection} may sound familiar. Indeed, this notion strictly generalizes the concept of \(\F\)-transverse intersections introduced in \cite{LCTal2018}, which plays a central role in Forcing Theory. This means that, any two orbits $\O$ and $\O^{\sspc\prime}$ that have exhibit an \(\F\)-transverse intersection, also have a weak $\F$-transverse intersection. 

    To confirm that, it suffices to observe that the notion of $\F$-transverse intersection in \cite{LCTal2018} can be equivalently defined in our context as follows:
     Two orbits $\O, \O^{\sspc \prime} \in \orb$ are said to have an \textit{\(\F\)-transverse intersection} if they satisfy each of the following conditions:
    \begin{itemize}[leftmargin=1.15cm]
        \item[\textbf{(i)}] The set $C_\O \cap C_{\O^{\sspc \prime}}$ is non-empty.
        \item[\textbf{(ii)}] There exist distinct leaves $\phi_L, \phi_L^\pp \in  \partial_L (C_\O \cap C_{\O^{\sspc \prime}})$ satisfying $\phi_L \in C_\O$ and $\phi_L^\pp \in C_{\O^{\sspc \prime}}$,  as well as two distinct leaves $\phi_R, \phi_R^\pp \in  \partial_R (C_\O \cap C_{\O^{\sspc \prime}})$ satisfying $\phi_R \in C_\O$ and $\phi_R^\pp \in C_{\O^{\sspc \prime}}$.
        \item[\textbf{(iii)}] The leaves $\phi_L$ and $\phi_L^\pp$ are ordered by $\prec$ oppositely to the leaves $\phi_R$ and $\phi_R^\pp$, that is, 
        \i
        \begin{align*}
            \quad \text{Either } \quad & \phi_L \prec \phi_L^\pp \spc \text{ and } \spc \phi_R^\pp \prec \phi_R,& \  \\ 
    \quad \text{or } \quad &\phi_L^\pp \prec \phi_L \spc \text{ and } \phi_R \prec \phi_R^\pp,& \  
        \end{align*}
    \end{itemize}

    \begin{figure}[h!]
    \center
    \mycomment{-0.1cm}\begin{overpic}[width=10cm, height=4.5cm, tics=10]{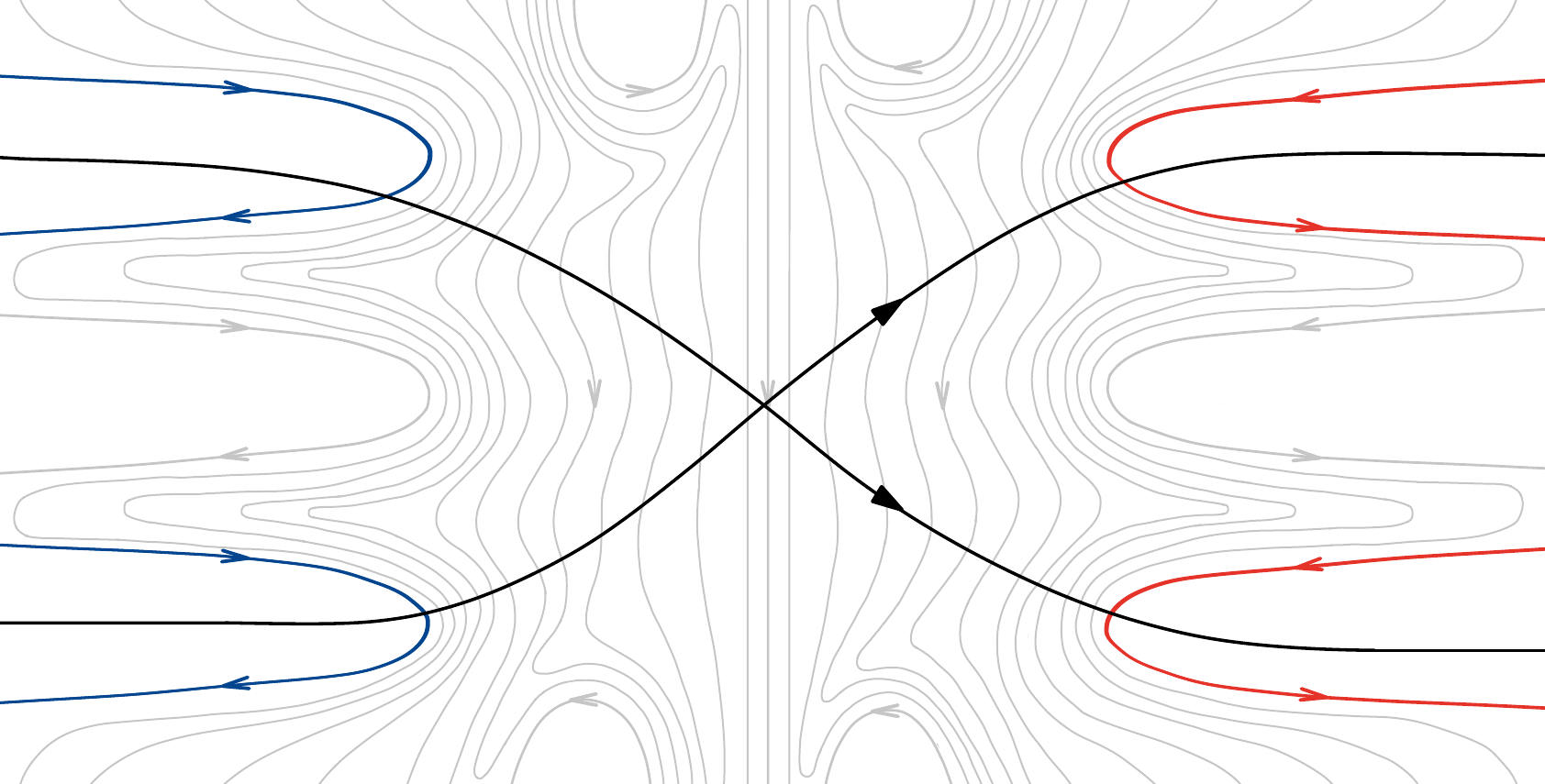}
        \put (53,30.4) {\colorbox{white}{\color{black}\normalsize$\rule{0cm}{0.27cm} \ \  $}}
        \put (53.5,31) {{\color{black}\normalsize$\displaystyle \Gamma_{\O^{\sspc\prime}}$}}
        \put (53,10.5) {\colorbox{white}{\color{black}\normalsize$\rule{0cm}{0.27cm} \ \  $}}
        \put (53.5,10.7) {{\color{black}\normalsize$\displaystyle \Gamma_{\O}$}}
        \put (102,39) {\color{myRED}$\displaystyle \phi_R^\pp$}
        \put (102,13) {\color{myRED}$\displaystyle \phi_R$}
        \put (-6,12.5) {\color{myBLUE}$\displaystyle \phi_L^\pp$}
        \put (-6,39) {\color{myBLUE}$\displaystyle \phi_L$}
\end{overpic}
\end{figure}
\end{remark}


\section{Maximal leaf domains and critical leaves}\label{sec:critical_leaves_and_core regions}

In this section, we show how a finite collection of orbits \(\OO \subset \orb\) induces a finite decomposition of the set \(\bigcup_{\O \in \OO} 
C_\O \subset \F\) into a collection of leaf domains and critical leaves.  

\vspace*{0.3cm}

Let \(\OO \subset \orb\) be a finite collection of orbits.
Recall that the boundary \(\partial C_\O\) is formed by a locally-finite collection of leaves. 
Thus, the union $\bigcup_{\O \in \OO} \partial C_\O$ is also a locally-finite collection of leaves, since \(\OO\) is a finite set.
In this context, we introduce the following definitions.

\begin{definition}[$\sspc $Maximal leaf domains | Critical leaves$\sspc$]\label{def:critical_leaves_and_core_regions}\ 

    \begin{itemize}
        \item A saturated set $D\subset \R^2$ is said to be a \textbf{maximal leaf domain} of \(\OO\) if it is a connected component of $\R^2\setminus \bigcup_{\O \in \OO} \partial C_\O$ that satisfies $D \subset \bigcup_{\O \in \OO} C_\O$.
        \item A leaf $\phi \in \F$ is said to be a \textbf{critical leaf} of \(\OO\) if it is a leaf contained in $\bigcup_{\O \in \OO} \partial C_\O$ that is also contained in the set $\bigcup_{\O \in \OO} C_\O$.
    \end{itemize}
    \mycomment{0.2cm}
The sets of all maximal leaf domains and critical leaves of \(\OO\) are denoted by $\D_\OO$ and $\W_\OO$.
\end{definition}

Roughly speaking, maximal leaf domains of $\OO$ constitute open connected sets of leaves that are crossed by exactly the same (at least one) orbits in $\OO$, while critical leaves are the leaves that are crossed by at least one orbit in $\OO$ and separate two maximal leaf domains. 
Moreover, since the collection of orbits $\OO$ is finite, both sets $\D_\OO$ and $\W_\OO$ should also be finite. 

These properties are listed and proved in Proposition \ref{prop:maximalleafdom} below.

\newpage

\begin{proposition}\label{prop:maximalleafdom}
    The following statements are true:
    \begin{itemize}[leftmargin=1cm]
        \item[\textup{\textbf{(i)}}] Every $D \in \D_
        \OO$ is a leaf domain of $\F$ (see Definition \ref{def:leaf_domain}).
        \item[\textup{\textbf{(ii)}}] For each $D \in \D_{\OO}$, there exists a non-empty subset $\OO(D) \subset \OO$ such that
        \begin{align*}
            D \subset \bigcap_{\O \in \OO(D)} C_\O \quad \text{ and }\quad
            \overline{\rule{0pt}{3.45mm}D}\sspc \cap  \bigcup_{\O \in \OO\setminus\OO(D)} C_\O = \varnothing.
        \end{align*}
        \item[\textup{\textbf{(iii)}}] The sets of maximal leaf domains $\D_{\OO}$ and critical leaves $\W_{\OO}$ are finite.
    \end{itemize}
\end{proposition}

\begin{proof}We prove each item separately.
    
    \vspace*{0.1cm}
    \textit{Proof of item (i):} Let $D \in \D_{\OO}$ be a leaf domain of $\OO$.
    Since $\bigcup_{\O \in \OO} \partial C_\O$ is a closed set, each connected component of $\R^2 \setminus \bigcup_{\O \in \OO} \partial C_\O$ is open. In particular, $D$ is an open and connected saturated subset of $\R^2$ and, consequently, it is formed by infinitely many leaves.
    Consider two distinct leaves $\phi,\phi'\in D$. Since $D\subset \bigcup_{\O \in \OO} C_\O$, there exists an orbit $\O \in \OO$  such that $\phi \in C_\O$. If $\phi'$ does not belong to $C_\O$ as well, then there exists $\phi''\in \partial C_\O$ that separates $\phi$ from $\phi'$ on the plane. This contradicts the definition of maximal leaf domains. Therefore, the leaf $\phi'$ also belongs to $C_\O$. This implies either $L(\phi) \subset L(\phi')$ or $L(\phi') \subset L(\phi)$. We then conclude the proof that $D$ is a leaf domain of $\F$ (see Definition \ref{def:leaf_domain}).

    \vspace*{0.2cm}
    \textit{Proof of item (ii):} According to the argument presented in the proof of item (i), for any pair of leaves $\phi,\phi'\in D$ and any orbit $\O \in \OO$, it holds that $\phi \in C_\O$ if, and only if, $\phi'\in C_\O$. Therefore, there exists a subset $\OO(D) \subset \OO$ such that, for all $\phi \in D$ and $\O\in \OO$, it holds
    \begin{align*}
        \phi \in C_\O \iff \O \in \OO(D).
    \end{align*}
    \noindent This proves that $D \subset \bigcap_{\O \in \OO(D)} C_\O$. Now, since $\bigcup_{\O \in \OO\setminus\OO(D)} C_\O$ is open, we have that
     $$\biggl(\overline{\rule{0pt}{3.5mm}D} \cap \bigcup_{\O \in \OO\setminus\OO(D)} C_\O\biggr)\  \subset\  \biggl(\sspc\overline{\rule{0pt}{4.2mm}D \cap \bigcup_{\O \in \OO\setminus\OO(D)} C_\O}\sspc\biggr) = \varnothing.$$
    \noindent The set $\OO(D)$ is non-empty because $D$ is contained in the set $\bigcup_{\O \in \OO} C_\O$. This proves (ii).

    \vspace*{0.2cm}
    \textit{Proof of item (iii):} According to Lemma \ref{lemma:intersection_of_domains}, for any pair of orbits $\O,\O^{\sspc\prime} \in \OO$, the set $C_\O \cap C_{\O^{\sspc\prime}}$ is a leaf domain of $\F$ (possibly empty). In particular, $C_\O \cap C_{\O^{\sspc\prime}}$ is  connected. Consequently, the set $C_\O \cap \partial C_{\O^{\sspc\prime}}$ contains at most two leaves. Moreover, note that
    $$ \W_{\OO} = \bigcup_{\O,\O^{\sspc\prime}\in \OO} C_\O \cap \partial C_{\O^{\sspc\prime}}.$$
    Since $\OO$ is a finite set, we conclude that the set $\W_{\OO}$ is also finite. This further implies that, for each orbit $\O \in \OO$, the set $C_\O \setminus \W_{\OO}$ has finitely many connected components. Since each maximal leaf domain in $\D_{\OO}$ is a connected component of the set $C_\O \setminus \W_{\OO}$ for some $\O \in \OO$, we conclude that $\D_{\OO}$ is also a finite set. This concludes the proof of item (iii).
\end{proof}

\vspace*{0.2cm}
To illustrate the concepts of $\D_{\OO}$ and $\W_{\OO}$ to the reader, we present an example that follows the color conventions listed below:
\begin{itemize}
    \item Transverse trajectories of orbits in $\OO$ are represented in red.
    \item Leaves in $\W_{\OO} \subset \bigcup_{\O \in \OO} \partial C_\O$, called critical leaves, are represented in black.
    \item Leaves in $\bigcup_{\O \in \OO} \partial C_\O\setminus \W_\OO$ are represented by a dashed line in black for distinction.
    \item Leaf domains in $\D_{\OO}$ are represented in light and dark shades of blue, in order to easily differentiate adjacent leaf domains.
    \item Regions in $\R^2 \setminus \bigcup_{\O \in \OO} C_\O$ are represented in white.
\end{itemize}

\newpage

    \begin{figure}[h!]
        \center
        \mycomment{0.5cm}\begin{overpic}[width=14cm,height=7.8cm, tics=10]{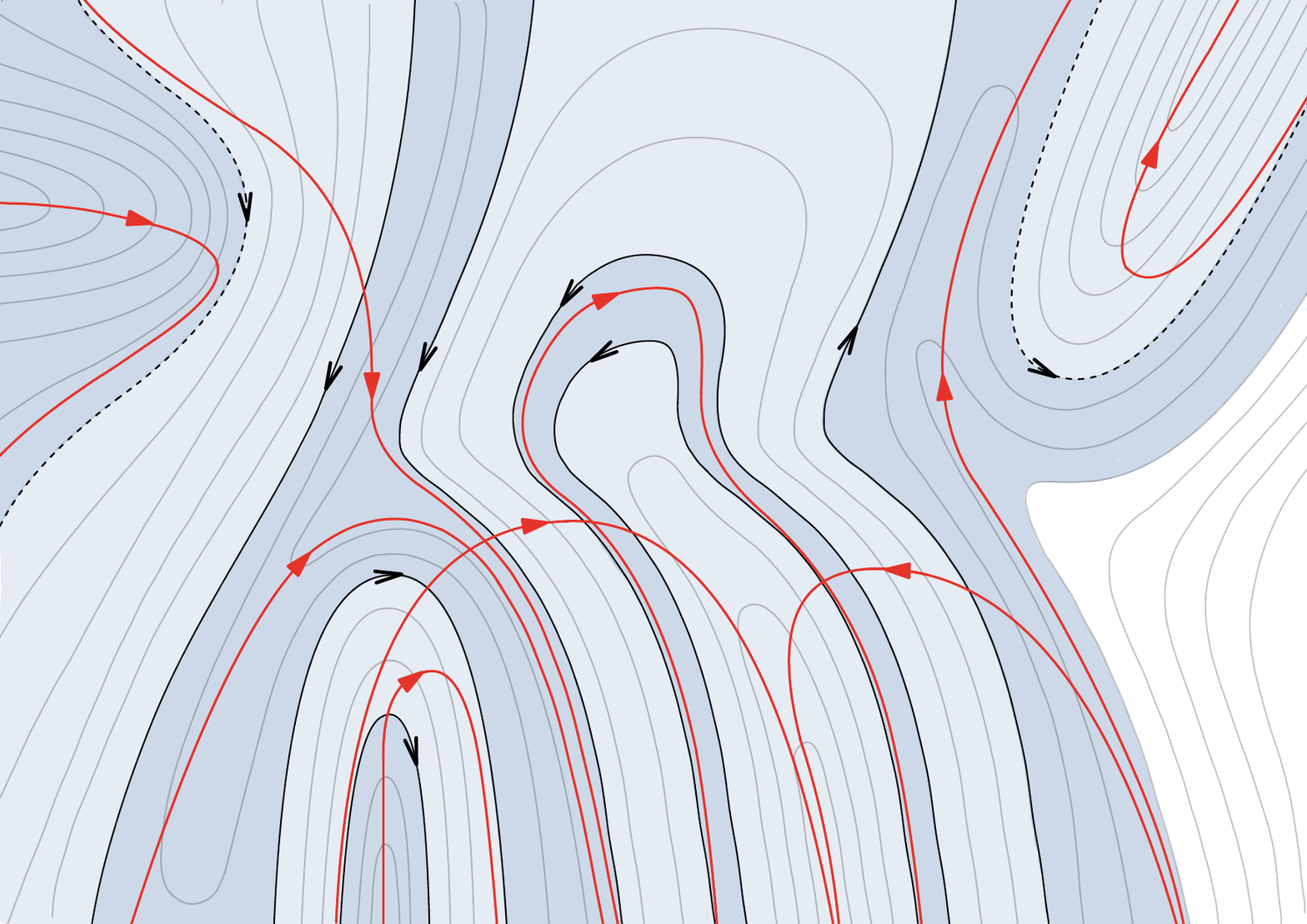}
    \end{overpic}
    \end{figure}

    \mycomment{0.2cm}
\begin{remark} Below, we list some key facts about maximal leaf domains and critical leaves:
    \begin{itemize}[leftmargin=0.75cm]
        \item[\textbf{(i)}] The set $\W_\OO \cup \bigcup_{D\in \D_{\OO}} D$ may not cover the entire plane, as it coincides with $\bigcup_{\O \in \OO} C_\O$.
    
        \item[\textbf{(ii)}] Each $D \in \D_{\OO}$ consists in a saturated set $D \subset \R^2$ homeomorphic to a trivially foliated plane, and its boundary $\partial D$ is a non-separating and locally-finite collection of leaves.
        
        \item[\textbf{(iii)}] The leaves in a maximal leaf domain $D \in \D_{\OO}$ are crossed by exactly the same orbits in $\OO$.  Specifically, they are crossed by every orbit in \( \OO(D) \) and by no orbit in \( \OO \setminus \OO(D) \).
        
        \item[\textbf{(iv)}] Each critical leaf $\phi \in \W_{\OO}$ is crossed by some orbit in $\OO$, but never by all of them. 
        
        \item[\textbf{(v)}] For each critical leaf $\phi \in \W_{\OO}$ there exist  $D,D^\pp\in \D_{\OO}$ that satisfy $\overline{\rule{0cm}{0.35cm}D}\cap \overline{\rule{0cm}{0.35cm}D^\pp} = \{\phi\}$.

         \item[\textbf{(vi)}] If a leaf $\phi \in \F$ satisfies  $\overline{\rule{0cm}{0.35cm}D}\cap \overline{\rule{0cm}{0.35cm}D'} = \{\phi\}$ and is not crossed by any orbit in $\OO$, then $\phi$ is a not critical leaf of $\OO$ (as represented by the dashed black lines in the example above).


         \item[\textbf{(vii)}] Two maximal leaf domains \( D, D' \in \D_{\OO} \) can be crossed by exactly the same orbits in $\OO$, meaning $\OO(D) = \OO(D')$, and still be distinct. This is illustrated in the example above.
    \end{itemize}
\end{remark}

\subsection{Orbits traveling across maximal leaf domains}\label{sec:visited_core regions}

To each orbit $\O \in \OO$, we associate an integer $m_\O>0$ and a sequence $(D_i)_{1\leq i \leq m_\O}$ of maximal leaf domains in $\D_{\OO}$ such that the following properties hold:
\begin{itemize}[leftmargin=1.15cm]
    \item For any $i\in \{1, ..., m_\O\}$, the maximal leaf domain $D_i$ is contained in the set $C_\O$.
    \item For any $i\in \{1, ..., m_\O -1\}$, there exists $\phi_i \in \W_{\OO}$ such that $\overline{\rule{0cm}{0.34cm}D_i}\cap \overline{\rule{0cm}{0.34cm}D_{i+1}} = \{\phi_i\}$.
    \item The set $C_\O$ is decomposed as $C_\O = \bigcup_{i=1}^{m_\O} \spc D_i \spc \cup \spc  \bigcup_{i=1}^{m_\O-1}\{\phi_i\}$.
\end{itemize}
\mycomment{0.3cm}
We denote the first and last terms of this sequence by
$ D_\alpha(\O) := D_1$ and $D_\omega(\O) := D_{m_\O}$.

\begin{figure}[h!]
    \center
    \mycomment{0.2cm}\begin{overpic}[width=0.85\textwidth,height=3.6cm, tics=10]{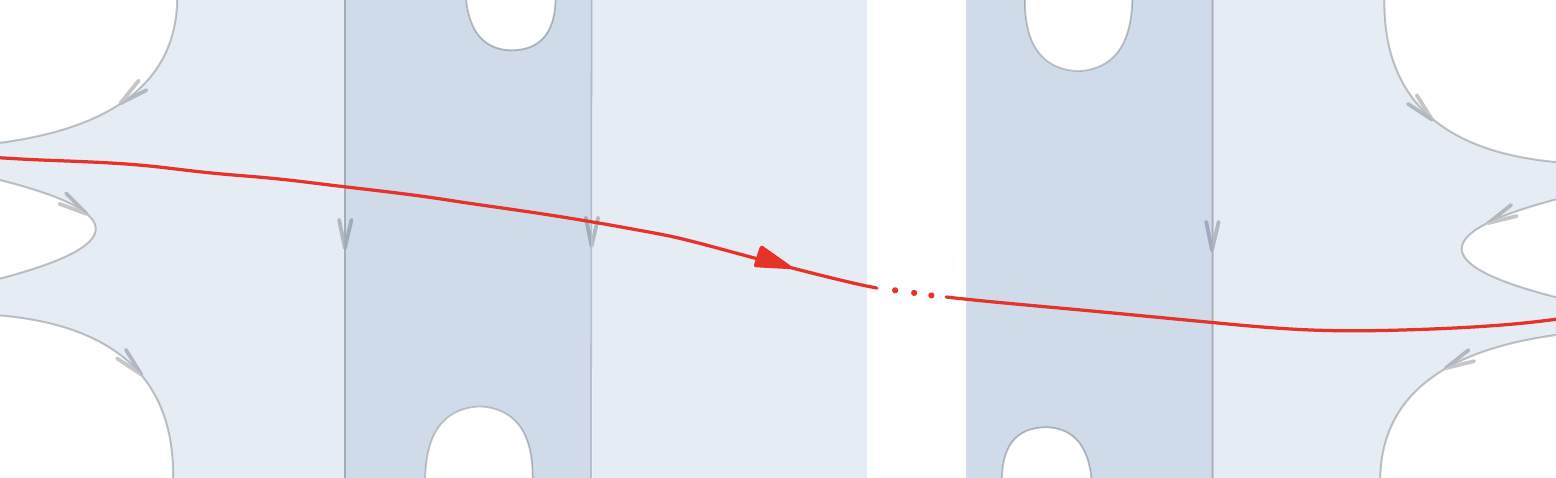}
        \put (47,8.5) {\color{myRED}\large$\displaystyle \Gamma_\O $}
        \put (14,19.3) {\color{myBLUE}\large$\displaystyle D_1 $}
        \put (27.9,19.2) {\color{myBLUE}\large$\displaystyle D_2 $}
        \put (45,19.2) {\color{myBLUE}\large$\displaystyle D_3 $}
        \put (65,19.2) {\color{myBLUE}\large$\displaystyle D_{m_\O-1} $}
        \put (81,19.2) {\color{myBLUE}\large$\displaystyle D_{m_\O} $}
\end{overpic}
\end{figure}

\mycomment{-0.2cm}

Conversely, for each maximal leaf domain $D \in \D_{\OO}$, we define the sets of orbits
\mycomment{-0.1cm}
\begin{align*}
    \OO_\alpha(D):= \{\O \in \OO \mid D_\alpha(\O) = D \} \quad \text{ and } \quad 
    \OO_\omega(D):= \{\O \in \OO \mid D_\omega(\O) = D \}.
\end{align*}

\mycomment{-0.2cm}
\noindent
Their complementary sets are denoted as
\mycomment{-0.1cm}
\begin{align*}
    \OOstart(D) := \OO(D) \setminus \OO_\alpha(D) \quad \text{ and } \quad 
    \OOend(D) := \OO(D) \setminus \OO_\omega(D).
\end{align*}

\mycomment{-0.2cm}
\noindent
Intuitively speaking, the sets \(\OO_\omega(D)\) and \(\OO_\alpha(D)\) comprise orbits in \(\OO\) that eventually stay within the maximal leaf domain $D$ under forward and backward iterations, respectively,  while \(\OOstart(D)\) and \(\OOend(D)\) comprise orbits in $\OO$ that eventually enter and eventually exit $D$.

\mycomment{0.2cm}
\begin{figure}[h!]
        \center
        \hspace*{-4cm}\begin{overpic}[width=9cm,height=4.2cm, tics=10]{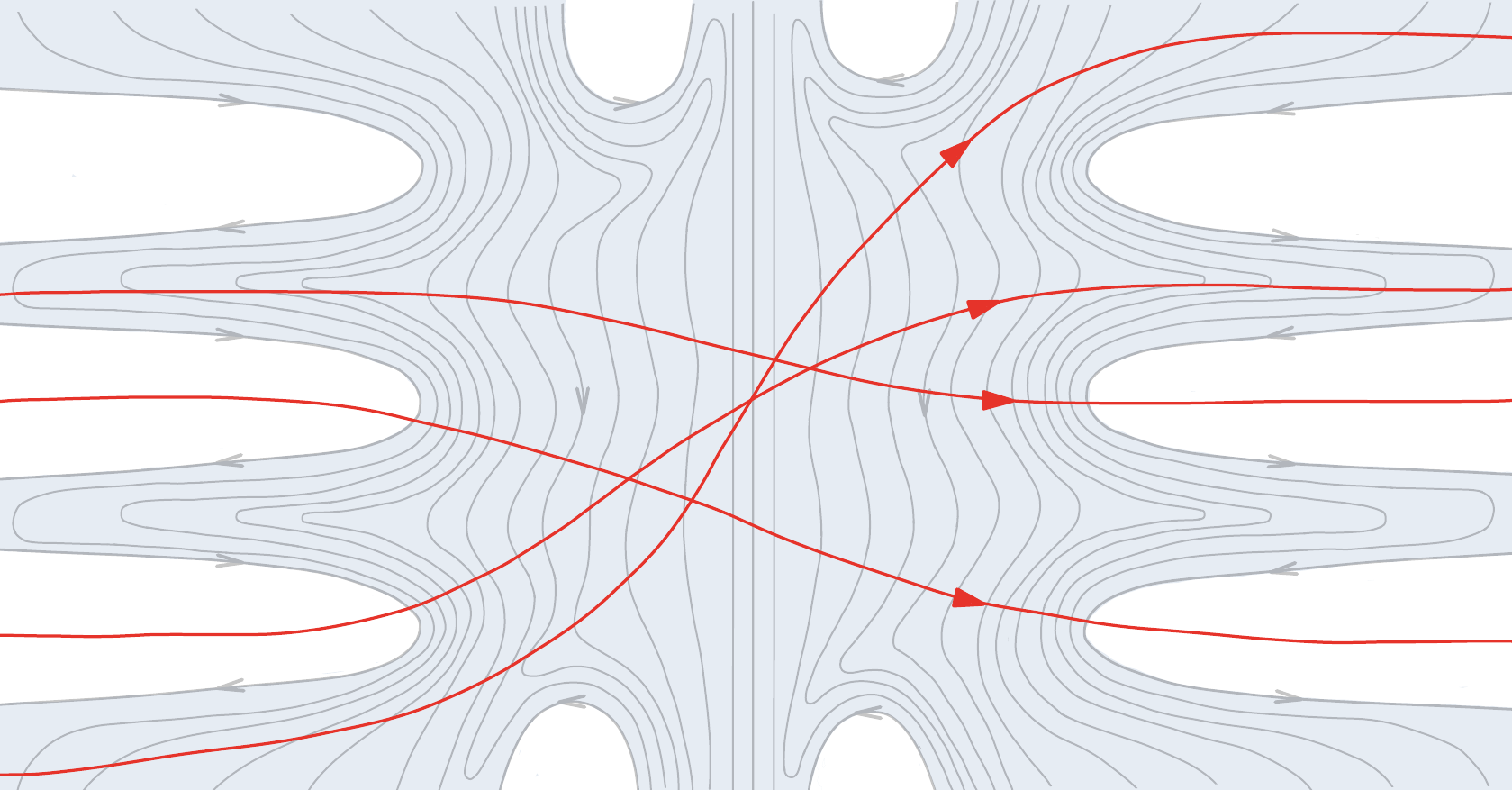}
            \put (-9.5,9) {\color{myRED}\large$\displaystyle \Gamma_{\O_{\sspc 3}} $}
            \put (-9.3,-1) {\color{myRED}\large$\displaystyle \Gamma_{\O_{\sspc 4}} $}
            \put (-9.5,20) {\color{myRED}\large$\displaystyle \Gamma_{\O_{\sspc 2}} $}
            \put (-9.5,29) {\color{myRED}\large$\displaystyle \Gamma_{\O_{\sspc 1}} $}
            \put (35.6,33.3) {\color{myBLUE}\large$\displaystyle D $}
            \put (109,39) {\color{black}\large$\displaystyle \OO_\alpha(D) =\{\O_{\sspc 1}, \sspc \O_{\sspc 4}\} $}
            \put (109,27.5) {\color{black}\large$\displaystyle \OO_\omega(D) =\{\O_{\sspc 3}, \sspc \O_{\sspc 4}\} $}
            \put (108,16) {\color{black}\large$\displaystyle \OOstart(D) =\{\O_{\sspc 2}, \sspc \O_{\sspc 3}\} $}
            \put (106,5) {\color{black}\large$\displaystyle \OOend(D) =\{\O_{\sspc 1}, \sspc \O_{\sspc 2}\} $}

    \end{overpic}
    \end{figure}

Since each \( D \in \D_{\OO} \) is a leaf domain of $\F$, we can define the following boundary subsets
\[
\partial_R D := \partial D \cap \biggl(\ \bigcap_{\phi \in D} R(\phi)\biggr) \quad \text{ and } \quad 
\partial_L D := \partial D \cap \biggl(\ \bigcap_{\phi \in D} L(\phi)\biggr).
\]

\mycomment{0.05cm}
\noindent Recall that both sets $\partial_R D$ and $\partial_L D$ are ordered by the relation $\prec$ defined in Section \ref{sec:order_limit_leaves}. 
These sets will help us to describe the behavior of the orbits $\OO$ inside each $D \in \D_{\OO}$.

Observe that, for any $D \in \D_{\OO}$, the following properties hold:
\mycomment{-0.12cm}
\begin{align*}
    \partial_L D = \partial_L C_\O, \quad \forall \O \in \OO_\omega(D),\\
    \partial_R D = \partial_R C_\O, \quad \forall \O \in \OO_\alpha(D).
\end{align*}

\mycomment{-0.23cm}
\noindent
Hence, each orbit $\O \in \OO_\omega(D)$ induces a cut $(\spc \Ltop \O,\spc \Lbot \O\spc )$  of the ordered set $(\sspc\partial_L D, \prec\sspc)$, and each $\O \in \OO_\alpha(D)$ induces a cut $(\spc \Rtop \O,\spc \Rbot \O\spc )$ of the ordered set $(\sspc\partial_R D, \prec\sspc)$.

\mycomment{0.1cm}
\begin{figure}[h!]
        \center\begin{overpic}[width=9cm,height=4.2cm, tics=10]{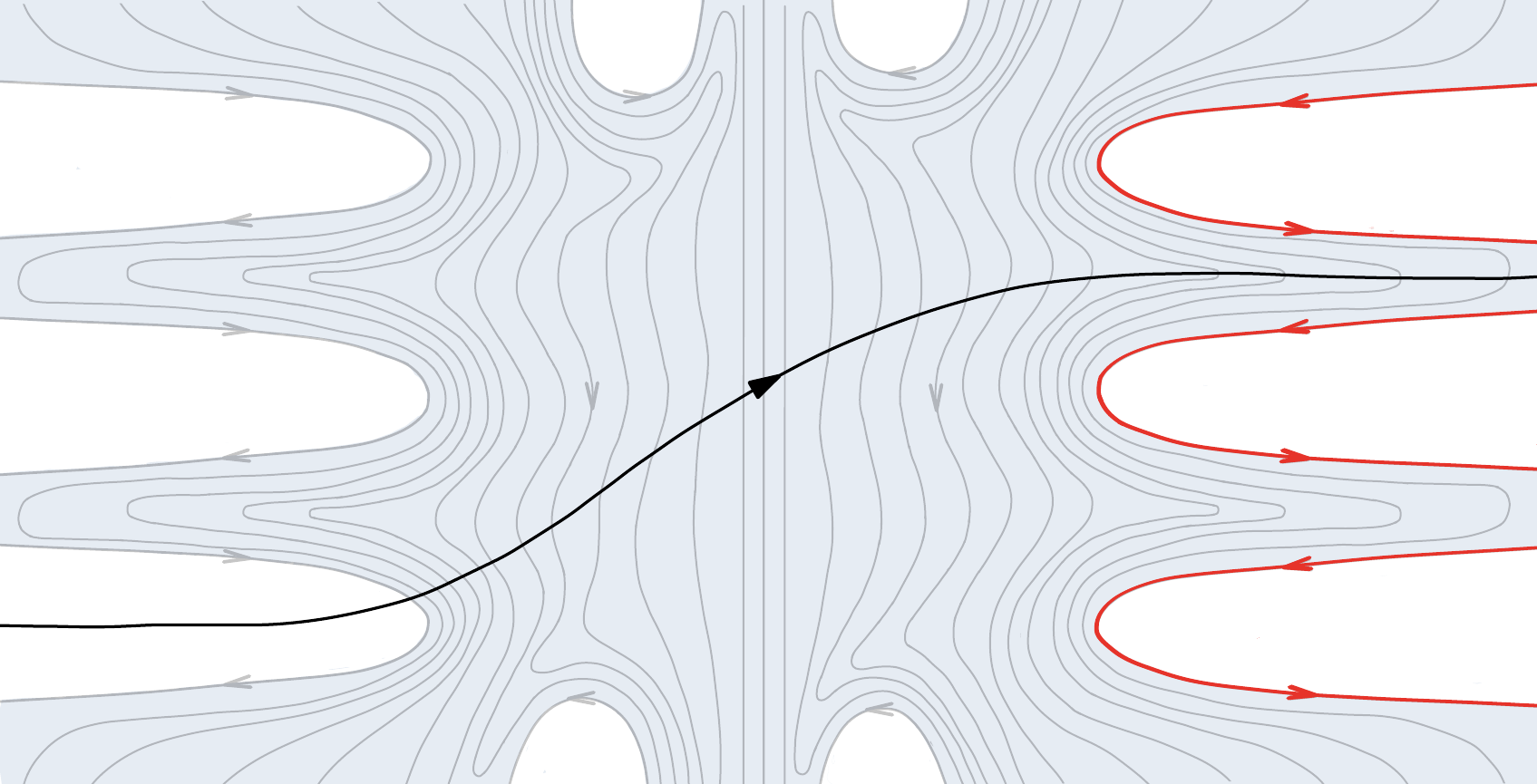}
            \put (-8.5,8.5) {\color{black}\large$\displaystyle \Gamma_{\O} $}
            \put (35.6,33.3) {\color{myBLUE}\large$\displaystyle D $}
            \put (85,21.8) {\color{myRED}\large$\displaystyle \Lbot \O $}
            \put (85,35) {\color{myRED}\large$\displaystyle \Ltop \O $}
            \put (85,8) {\color{myRED}\large$\displaystyle \Lbot \O $}

    \end{overpic}
    \end{figure}

\mycomment{-0.2cm}

For each leaf $\phi \in \F$, let $\OO(\phi)\subset \OO$ be the subset of orbits that cross $\phi$. Observe that,  for any $D \in \D_{\OO}$, the subsets of orbits $\OOstart(D)$ and $\OOend(D)$ can be decomposed as
\mycomment{-0.13cm}
\begin{align*}
        \OOstart(D)= \bigcup_{\phi \in \partial_R D} \OO(\phi) \quad \text{ and } \quad
        \OOend(D)= \bigcup_{\phi \in \partial_L D} \OO(\phi).
\end{align*}
In other words, the orbits in $\OOstart(D)$ must enter $D$ through a leaf in $\partial_R D$, while the orbits in $\OOend(D)$ must exit $D$ through a leaf in $\partial_L D$. If this was not the case, there would exist leaves in $D$ that are not crossed by exactly the same orbits in $\OO$.



\mycomment{-0.3cm}
\begin{remark}
    Note that, for any $D \in \D_{\OO}$, the sets $\OO_\omega(D)$ and $\OOend(D)$ are stable under the equivalence relation $\fasym$, while the sets $\OO_\alpha(D)$ and $\OOstart(D)$ are stable under the relation $\basym$. 
    
    More precisely, for any two orbits $\O,\O^{\sspc\prime} \in \OO(D)$ such that $\O \fasym \O^{\sspc\prime}$, it holds that
    \begin{align*}
        \O \in \OO_\omega(D) \iff \O^{\sspc\prime} \in \OO_\omega(D),\ \ \\
        \O \in \OOend(D) \iff \O^{\sspc\prime} \in \OOend(D).
    \end{align*}

    Similarly, for any two orbits $\O,\O^{\sspc\prime} \in \OO(D)$ such that $\O \basym \O^{\sspc\prime}$, it holds that
    \mycomment{-0.12cm}
    \begin{align*}
        \O \in \OO_\alpha(D) \iff \O^{\sspc\prime} \in \OO_\alpha(D), \ \\
        \O \in \OOstart(D) \iff \O^{\sspc\prime} \in \OOstart(D).
    \end{align*}
\end{remark}

\subsection{Directed planar forest of maximal leaf domains}\label{sec:oriented_trees}

Let $\D_{\OO}$ and $\W_{\OO}$ be the sets of maximal leaf domains and critical leaves of a given finite collection of orbits $\OO \subset \orb$. These sets naturally define an oriented graph $(\D_{\OO},\W_{\OO})$,  where the set of vertices is given by \( \D_{\OO} \), and the set of oriented edges consists of all ordered pairs of maximal leaf domains \( (D, D^{\sspc\prime}) \in \D_{\OO}\times \D_{\OO}\) for which there exists \( \phi \in \W_{\OO} \) satisfying $\overline{\rule{0cm}{0.35cm}D} \cap \overline{\rule{0cm}{0.35cm}D^{\sspc\prime}} = \{\phi\}$ and $D\subset R(\phi)$, or equivalently, $D^{\sspc\prime} \subset L(\phi)$.

Each oriented edge  \( (D, D^{\sspc\prime}) \) in the oriented graph $(\D_{\OO},\W_{\OO})$ is denoted by \( D \leadsto D^{\sspc\prime} \).

\mycomment{-0.1cm}
\begin{lemma}\label{lemma:planar_embedding}
    The graph $(\D_{\OO}, \W_{\OO}) $ is an oriented planar forest, with each tree component corresponding to a connected component of the set $\spc \bigcup_{\O \in \OO}C_\O$.
\end{lemma}

\mycomment{-0.3cm}
\begin{proof} For each maximal leaf domain $D \in \D_{\OO}$, let $p_D\in \R^2$ be a point in the region formed as union of all leaves in $D$. For any $D,D^\pp\in \D_{\OO}$ satisfying $D \leadsto D^\pp$, denote by  $\phi_{D,D^\pp} \in \W_{\OO}$ the critical leaf in between $D$ and $D^\pp$, meaning, $\overline{\rule{0cm}{0.35cm}D} \cap \overline{\rule{0cm}{0.35cm}D^{\sspc\prime}} = \{\phi_{D,D^\pp}\}$. Observe that, since
    \mycomment{-0.13cm}
    $$ \phi_{D,D^\pp} \in \partial_L D \quad \text{ and } \quad \phi_{D,D^\pp} \in \partial_R D^\pp,$$

    \mycomment{-0.23cm}
\noindent
we can connect the point $p_D$ to $p_{D^\pp}$ via a positively transverse path $\gamma_{D,D^\pp}:[0,1]\longrightarrow \R^2$. By associating each core region $D \in \D_{\OO}$ with the point $p_D$, and each oriented edge $D\leadsto D^\pp$ with the path $\gamma_{D,D^\pp}$, we obtain a planar embedding of the oriented graph \( (\D_{\OO}, \W_{\OO}) \).

    \begin{figure}[h!]
        \center
        \mycomment{0.2cm}\begin{overpic}[width=9.3cm,height=5cm, tics=10]{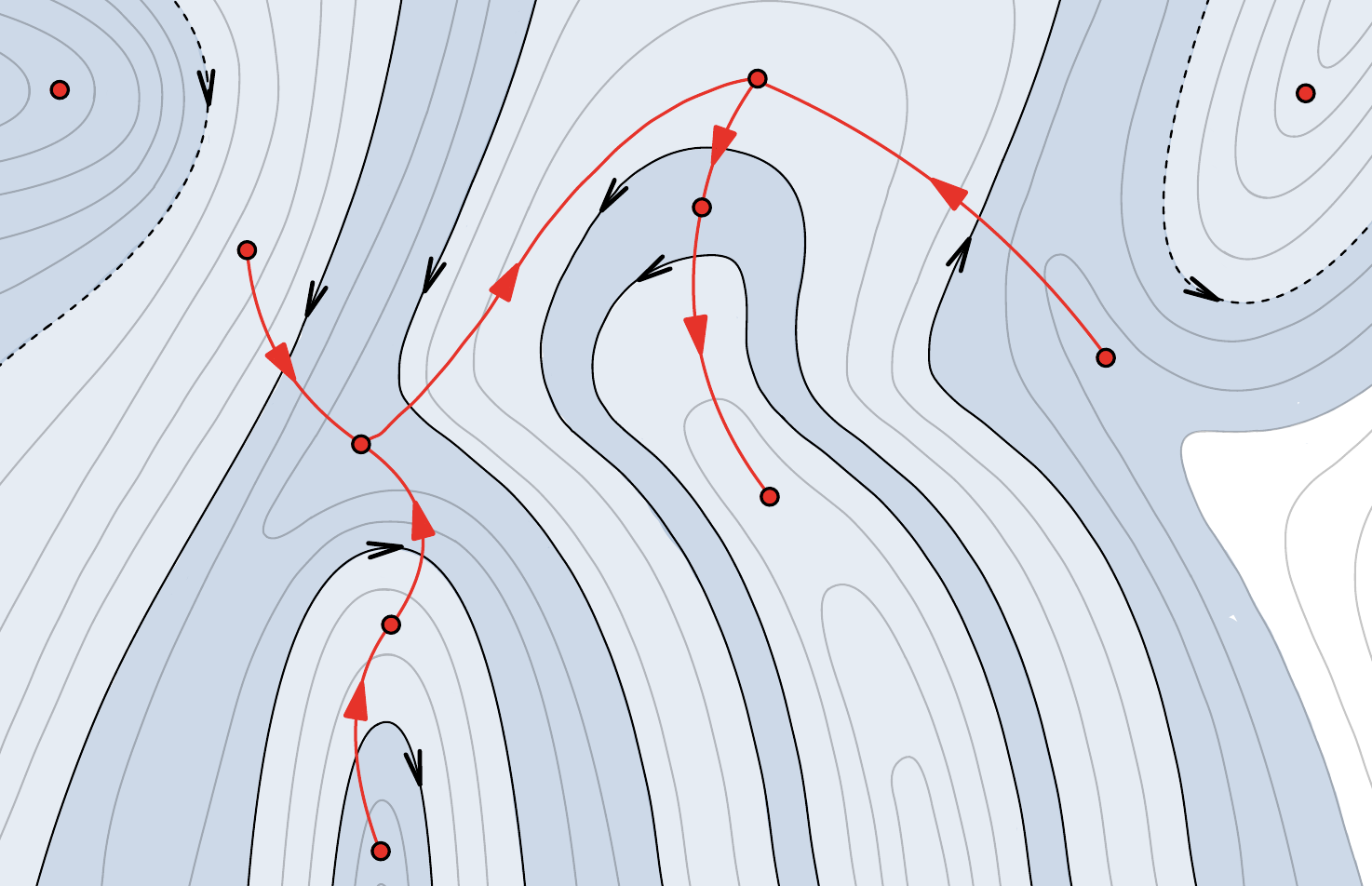}
    \end{overpic}
    \end{figure}

    This planar embedding shows us that the oriented graph $(\D_{\OO},\W_{\OO})$ is an oriented forest, that is, a disjoint union of oriented trees. For that, it suffices to prove that each connected component of the planar graph is no longer connected if one of its edges is deleted. Indeed, the path $\gamma_{D,D^\pp}$ is the unique edge of the planar graph intersecting the leaf $\phi_{D,D^\pp}$. Therefore, by deleting $\gamma_{D,D^\pp}$, we disconnect the connected component of the planar graph that contains the vertices $p_D$ and $p_{D^\pp}$ into two components, each lying on a different side of the leaf $\phi_{D,D^\pp}$. This allows us to conclude that each connected component of $(\D_{\OO},\W_{\OO})$ is an oriented tree corresponding to a connected component of $\spc \bigcup_{\O \in \OO}C_\O$. 
\end{proof}

\section{Proper transverse trajectories with minimal intersections}\label{chap:minimal_intersections}

\mycomment{-0.5cm}
In this section, we prove Theorem \ref{prop:pairwise_disj_traj}, which leverages the relations $\lesssim_L$ and $\lesssim_R$ to study the arrangement of proper transverse trajectories of a finite collection of orbits $\OO \subset \orbphi$ for some fixed leaf $\phi \in \F$. Recall that the transverse foliation \(\F\) is  said to be \emph{generic} for \(\OO\) if no leaf in $\F$ intersects more than one orbit in \(\OO\). This condition is indeed generic, as any transverse foliation can be perturbed to satisfy it (see Section \ref{sec:foliated_brouwer}).

\begin{thmx}\label{prop:pairwise_disj_traj}
    Let $\phi \in \F$, and let $\OO\subset \orbphi$ be a finite collection of orbits crossing $\phi$.  Assume that \(\F\) is generic with respect to \(\OO\). Then, there exists a family $\bigl\{\Gamma_{\O}\bigr\}_{\O \in \OO}$ of proper transverse trajectories associated with the orbits in $\OO$ that satisfy
    \mycomment{-0.12cm}
    $$\Gamma_{\O} \cap \Gamma_{\O^\pp} \cap \overline{\rule{0pt}{3.6mm}L(\phi)} = \varnothing, \quad \forall \sspc \O,\O^\pp \in \OO,\  \O \neq \O^\pp.$$

    \mycomment{-0.23cm}
    \noindent Moreover, the order at wich the intersection points $p_\O \in \Gamma_{\O} \cap \phi$ appear along the leaf $\phi$ is compatible with the relation $\lesssim_L$, in the sense that $p_\O < p_{\O^\pp}$ along the leaf $\phi$ only if $\O \lesssim_L \O^\pp$. 
\end{thmx}


\begin{figure}[h!]
    \center\begin{overpic}[width=8.58cm, height=4cm, tics=10]{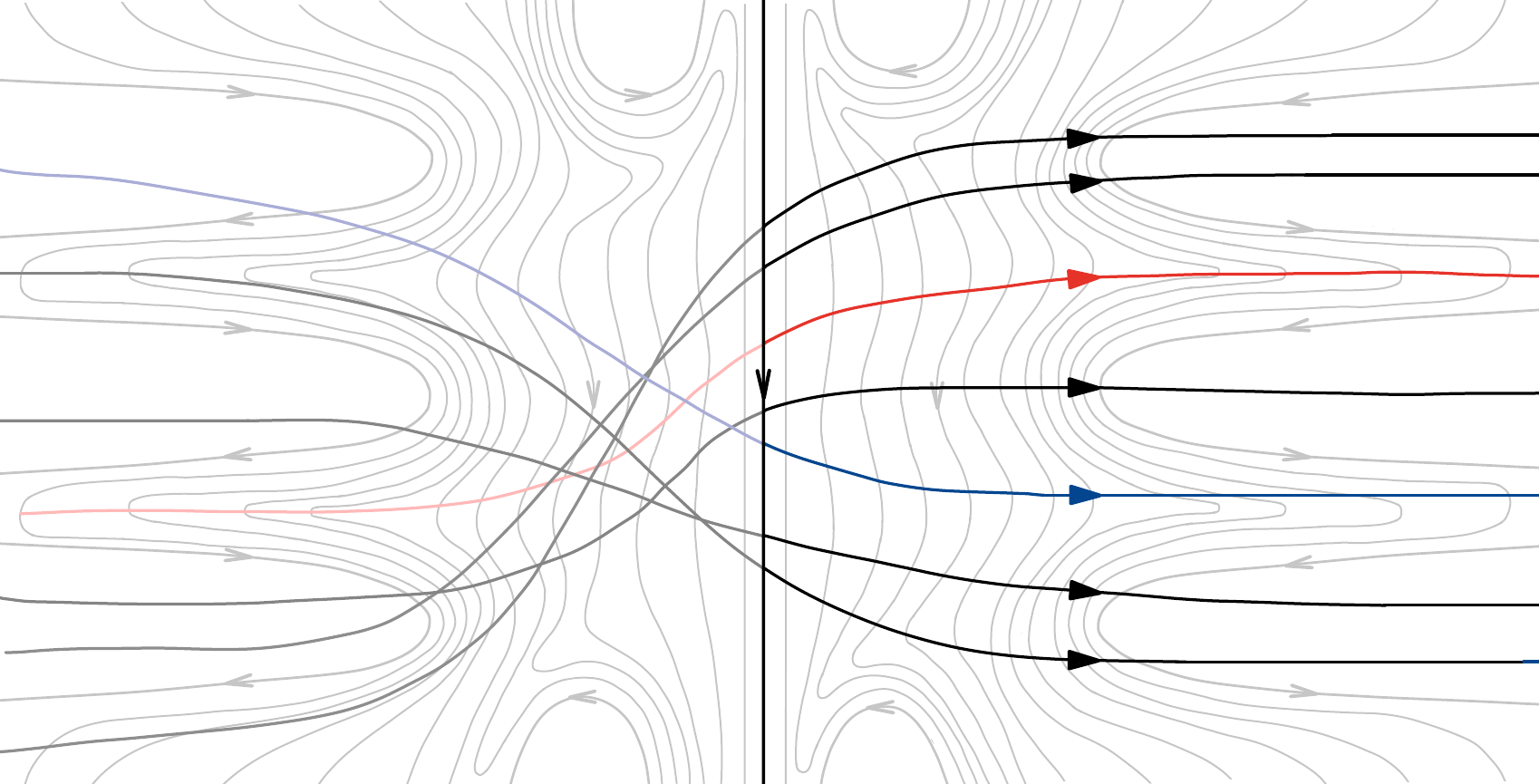}
        \put (44.5,-5) {\colorbox{white}{\color{black}\large$\displaystyle\  \phi \spc$}}
         \put (103.5,15) {{\color{myBLUE}\large$\displaystyle \Gamma_{\O^\pp} $}}
         \put (103.5,28.2) {{\color{myRED}\large$\displaystyle \Gamma_\O$}}
         \put (-25,20) {{\color{myRED}\large$\displaystyle \O \ $}{\color{black}\large$\displaystyle \lesssim_L$}{\color{myBLUE}\large $\spc\displaystyle \O^\pp$}}
\end{overpic}
\end{figure}

\mycomment{0.1cm}
A result similar to Theorem \ref{prop:pairwise_disj_traj} holds for the right side of the leaf $\phi$, with the same assumptions and conditions, but using the relation $\lesssim_R$ instead of $\lesssim_L$. 


\mycomment{0.05cm}
Later on this section, Theorem \ref{prop:pairwise_disj_traj} will be used to prove Theorem \ref{thm:main_Part1}.

\subsection{Technical lemmas for the proof of Theorem \ref{prop:pairwise_disj_traj}}\label{sec:proof_pairwise_disj_traj}

 Before we prove Theorem \ref{prop:pairwise_disj_traj}, we need four technical lemmas. In each of these lemmas, we assume that the transverse foliation $\F$ is generic for the finite collection of orbits $\OO \subset \orb$ is question. 

\subsubsection*{Technical Lemma 1: Flexibility to obtain pairwise disjoint transverse paths}

\begin{lemma}\label{lemma:untangle}
Let $\phi, \phi' \in \F$ be two distinct leaves that satisfy $L(\phi') \subset L(\phi)$. Moreover,  let $\OO \subset \orbphi \cap \textup{Orb}(\phi')$ be a finite collection of orbits crossing both leaves $\phi$ and $\phi'$. Then, for any indexing $\sspc\OO=\{\O_{\sspc 1},\sspc \ldots, \sspc \O_{\sspc r}\}$, there exists a family $\{\gamma_i\}_{1\leq i \leq r}$ of pairwise disjoint paths $\gamma_i : [0,1] \longrightarrow \R^2$ positively transverse to $\F$ such that, for every $1\leq i \leq r$, the following holds:
\begin{itemize}[leftmargin=0.7cm]
\item The endpoints satisfy $\spc\gamma_i(0) \in \phi\spc$ and $\spc\gamma_i(1) \in \phi'\spc$.
\item The fragment of orbit $\overline{\rule{0pt}{3.6mm} L(\phi) \cap R(\phi')} \cap \O_i$ is contained in $\gamma_i\spc$.
\item The endpoints are ordered $\gamma_i(0) < \gamma_{i+1}(0)$ according to the orientation of $\phi$.
\end{itemize}
\end{lemma}
\vspace*{-0.2cm}
\begin{figure}[h!]
    \center
    \mycomment{0.2cm}\begin{overpic}[width=8cm, height=3cm, tics=10]{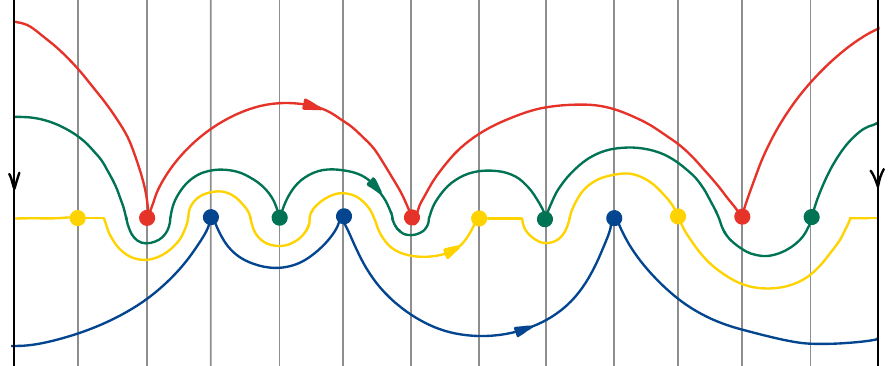}
        \put (-5,18) {\color{black}\large$\displaystyle \phi $}
        \put (20,-8) {\colorbox{white}{\color{black}\large$\displaystyle \OO=\{\sspc$\color{myRED}$\O_{\sspc 1}$\color{black}$\spc,\spc$\color{myGREEN}$\O_{\sspc 2}$\color{black}$\spc,\spc$\color{myYELLOW}$\O_{\sspc 3}$\color{black}$\spc,\spc$\color{myBLUE}$\O_{\sspc 4}$\color{black}$\sspc\}$}} 
        \put (101.5,18) {\color{black}\large$\displaystyle \phi' $}
\end{overpic}
\end{figure}

\mycomment{0.3cm}
\begin{proof}
    Consider an arbitrary indexing of the set of orbits $\OO=\{\O_{\sspc 1},\sspc \ldots, \sspc \O_{\sspc r}\}$. 
    Observe that, by means of the Homma-Schoenflies theorem, up to conjugacy, we can assume that: 
    \begin{itemize}
        \item $\phi=\{0\} \times \R$ and $\phi'= \{1\}\times \R$, both oriented downwards.
        \item The square $[0,1]^2$ is vertically foliated by $\F$.
        \item For each $i \in \{1,...,r\}$, we have that $\O_i \cap \overline{\rule{0cm}{0.35cm}L(\phi) \cap R(\phi)}\subset [0,1] \times \{1/i\}$.
    \end{itemize}
    \mycomment{0.2cm}
   The third point only holds because of the generic condition imposed on $\F$ with respect to $\OO$,  as it ensures that each leaf of $\F$ intersects at most one orbit in $\OO$.
    For each $i \in \{1,...,r\}$,  let $\gamma_{\O_i}:[0,1]\longrightarrow \R^2$ be the horizontal segment given by $\gamma_{\O_i}(t) = (t, 1/i)$. This concludes the proof, as the family $\{\gamma_{\O_i}\}_{1\leq i \leq r}$ satisfies the conditions of the lemma.
\end{proof}

\subsubsection*{Technical Lemma 2: Transverse half-lines for classes of $\F$-asymptotic orbits}

\begin{lemma}\label{lemma:common_half_traj}
    Let $\phi \in \F$, and let $\OO\subset \orbphi$ be a finite set of orbits crossing the leaf $\phi$. Then, for every class of forward $\F$-asymptotic orbits $[\O\spc]^+ \in \OO/{\fasym}\spc$, there exists a half-line $\spc\Gamma^{^+}:[\spc 0,\infty)\longrightarrow\R^2$ that is
    positively transverse to $\F$ and satisfies the following conditions:
    \begin{itemize}
        \item The endpoint $\sspc\Gamma^{^+}(0)$ lies on the leaf $ \phi$.
        \item The forward orbit $\sspc\overline{\rule{0pt}{3.6mm}L(\phi)} \cap \O\spc $ is contained in $\Gamma^{^+}$, for all $\O \in [\O\sspc]^+.$
    \end{itemize}
\end{lemma}

\begin{figure}[h!]
    \center
    \hspace*{-1cm}\begin{overpic}[width=7cm, height=3.7cm, tics=10]{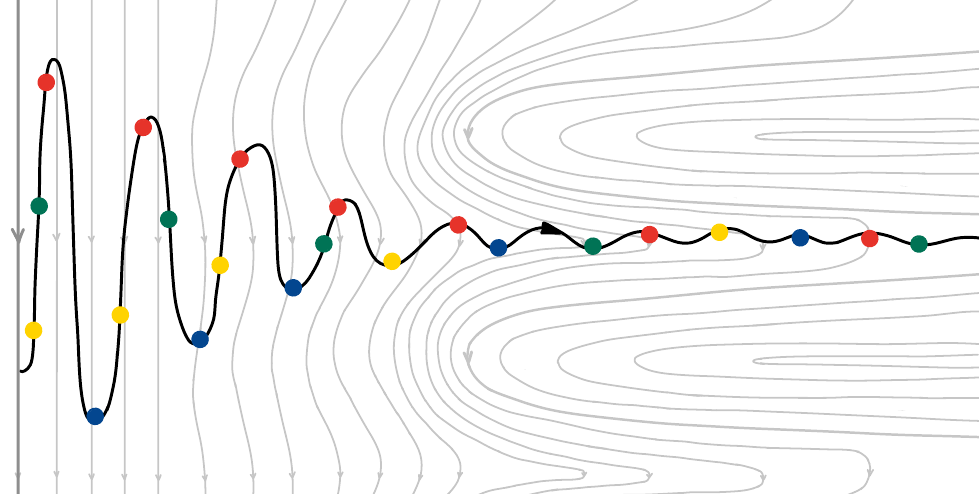}
        \put (-7,27) {\colorbox{white}{\color{myDARKGRAY}\large$\displaystyle \phi $}}
        \put (103,27) {{\color{black}\large$\displaystyle {\Gamma^{^+}}$}}
        \put (-5,-6) {\colorbox{white}{\rule{0cm}{0.5cm}\color{black}\large$\quad\quad\quad\quad\ \   \ $\color{black}$[\sspc\O\spc]^+ = \{\spc$\color{myRED}$\O_{\sspc 1}$\color{black}$\spc,\spc$\color{myGREEN}$\O_{\sspc 2}$\color{black}$\spc,\spc$\color{myYELLOW}$\O_{\sspc 3}$\color{black}$\spc,\spc$\color{myBLUE}$\O_{\sspc 4}$\color{black}$\sspc\}$}}
\end{overpic}
\end{figure}

\begin{proof}
    The proof follows from the proof of Proposition \ref{lemma:equivalence} (see Remark \ref{rmk:adapt}).
\end{proof}

\subsubsection*{Technical Lemma 3: Disjoint transverse half-lines for $\F$-asymptotic classes}

\begin{lemma}\label{lemma:shared_disjoint}
    Let $\phi \in \F$, and let $\OO\subset \orbphi$ be a finite set of orbits crossing the leaf $\phi$.  We enumerate $\OO/{\fasym} = \{[\O_{\sspc 1}]^+,\sspc \ldots, \sspc [\O_{\sspc r}]^+\}$ so that \([\O_{\sspc i}]^+\sspc \lesssim_L \sspc[\O_{\sspc i+1}]^+\) for all \(i \in \{1, \ldots, r-1\}\).
    Then, there exists a family $\spc\{\Gamma^{^+}_i\}_{1\leq i \leq r} \spc $ of pairwise disjoint half-lines $\spc\Gamma^{^+}_{i}:[\spc 0,\infty)\longrightarrow\R^2$ positively transverse to $\F$ such that, for every $1\leq i \leq r$, the following holds:
    \begin{itemize}
\item The endpoints $\spc\Gamma_i^{^+}(0)$ lies on the leaf $ \phi\sspc$.
\item The forward orbit $\sspc\overline{\rule{0pt}{3.6mm}L(\phi)} \cap \O\spc $ is contained in $\Gamma^{^+}_i$, for all orbits $\O \in [\O_{\sspc i}]^+.$
\item The endpoints are ordered $\Gamma_i^{^+}(0) < \Gamma_{i+1}^{^+}(0)$ according to the orientation of $\phi$.
\end{itemize}

\end{lemma}

\mycomment{0cm}
\begin{figure}[h!]
    \center
    \hspace*{-1.5cm}\begin{overpic}[width=6.5cm, height=3.8cm, tics=10]{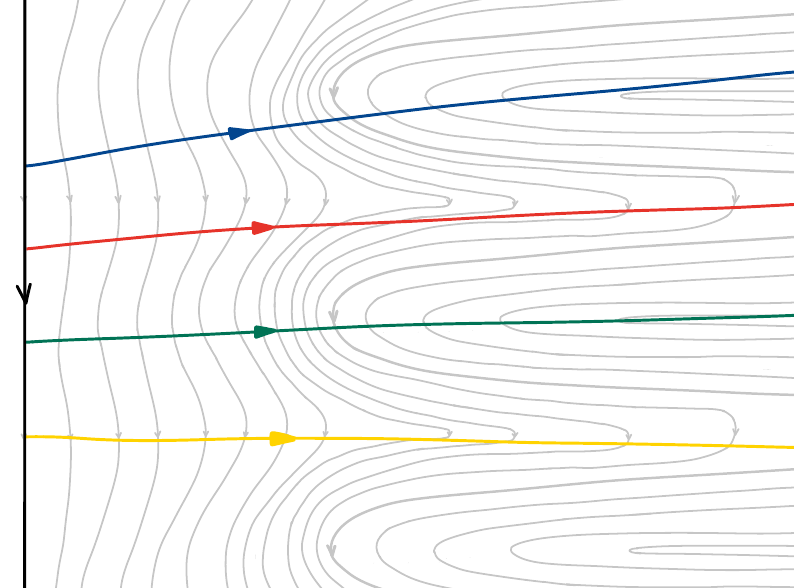}
        \put (-6,33) {\colorbox{white}{\color{black}\large$\displaystyle \phi $}}
        \put (102,42) {{\color{myRED}\large$\displaystyle {\Gamma^{^+}_{[\O\sspc_2]^+}}$}}
        \put (102,28) {{\color{myGREEN}\large$\displaystyle {\Gamma^{^+}_{[\O\sspc_3]^+}}$}}
        \put (102,14) {{\color{myYELLOW}\large$\displaystyle {\Gamma^{^+}_{[\O\sspc_4]^+}}$}}
        \put (102,56) {{\color{myBLUE}\large$\displaystyle {\Gamma^{^+}_{[\O\sspc_1]^+}}$}}
        \put (-19,-6) {\colorbox{white}{\rule{0cm}{0.5cm}\color{black}\large$\ \  \ \OO/{\fasym}\ =\{\spc$\color{myBLUE}$[\O\sspc_1]^+${\normalsize\color{black}$\spc\lesssim_L\spc$}{\color{myRED}$[\O\sspc_2]^+$}{\normalsize\color{black}$\spc\lesssim_L\spc$}\color{myGREEN}$[\O\sspc_3]^+${\normalsize\color{black}$\spc\lesssim_L\spc$}\color{myYELLOW}$[\O\sspc_4]^+$\color{black}$\spc\}\quad \quad \quad $ }} 
\end{overpic}
\end{figure}
\mycomment{0.1cm}
\newpage

\mycomment{-1.7cm}
\begin{proof}
    Note that, we can enumerate $\OO/{\fasym} = \{\sspc[\O_{\sspc 1}]^+,\sspc \ldots, \sspc [\O_{\sspc r}]^+\sspc\}$ in an increasing order according to $\lesssim_L$ because this relation induces a total order on $\OO/{\fasym}$ (see Proposition \ref{prop:preorders}).
    Let $\D_{\OO}$ be the decomposition maximal leaf domains induced by $\OO$, and let $\W_{\OO}$ be the set of critical leaves of $\OO$ (see Definition \ref{def:critical_leaves_and_core_regions}).
    Since the leaf $\phi$ is crossed by all orbits in $\OO$,  we know that $\phi$ is not a critical leaf of $\OO$. Hence, there exists a maximal leaf domain $D \in \D_{\OO}$ such that $\phi \in D$ and, consequently, $\OO(D) = \OO.$ We separate the proof into two cases:

    \mycomment{0.2cm}

     \noindent \underline{\textit{Case (1): Assuming that $\OOend(D) = \varnothing\spc$:}}
    
    \mycomment{0.1cm}
    According to Lemma \ref{lemma:common_half_traj}, each class $[\O_{\sspc i}]^+ \in \OO/{\fasym}$ admits a half-line $\Gamma^{^+}_i:[\spc 0,\infty)\longrightarrow\R^2$
    that is positively transverse to $\F$, has its endpoint $\sspc\Gamma^{^+}_i(0) \in \phi$, and satisfies

    \mycomment{-0.11cm}
    \begin{equation*}
    \sspc\overline{\rule{0pt}{3.6mm}L(\phi) \cap \O}\spc \subset \Gamma^{^+}_i
    , \quad \forall \O \in [\O_{\sspc i}]^+.
    \end{equation*}

    \mycomment{-0.22cm}
\noindent According to Proposition \ref{lemma:equivalence}, for any $i,j \in \{1, \ldots, r\sspc\}$ with $i \neq j$, there exists $t>0$ so that 
    \mycomment{-0.11cm}
    \[
        \Gamma^{^+}_i\bigl(\sspc[\spc t, \infty)\sspc\bigr) \cap \Gamma^{^+}_j\bigl(\sspc[\spc t, \infty)\sspc\bigr) = \varnothing.
    \]

    \mycomment{-0.22cm}
    \noindent
    Since $\OO$ is finite, this implies that for each $i \in \{1,...,r\}$, there exists  $\phi_i\in \F$ such that
    \mycomment{-0.11cm}
     $$\Gamma^{^+}_i \cap \Gamma^{^+}_j \cap \overline{\rule{0pt}{3.6mm}L(\phi_i)} = \varnothing, \quad \forall j \neq i.$$

     \mycomment{-0.22cm}
     \noindent
     Observe that, in the present setting, it holds that $\OO = \OO_\omega(D).$ Thus, we have $\phi_i \in D$ for every $i \in \{1,...,r\}$. Consider a leaf $\phi^* \in D$ that satisfies 
    $ L(\phi^*) \subset  L(\phi_i)$ for every $i \in \{1,...,r\}$. 
\begin{figure}[h!]
    \center
    \mycomment{-0.3cm}\begin{overpic}[width=9cm, height=4.2cm, tics=10]{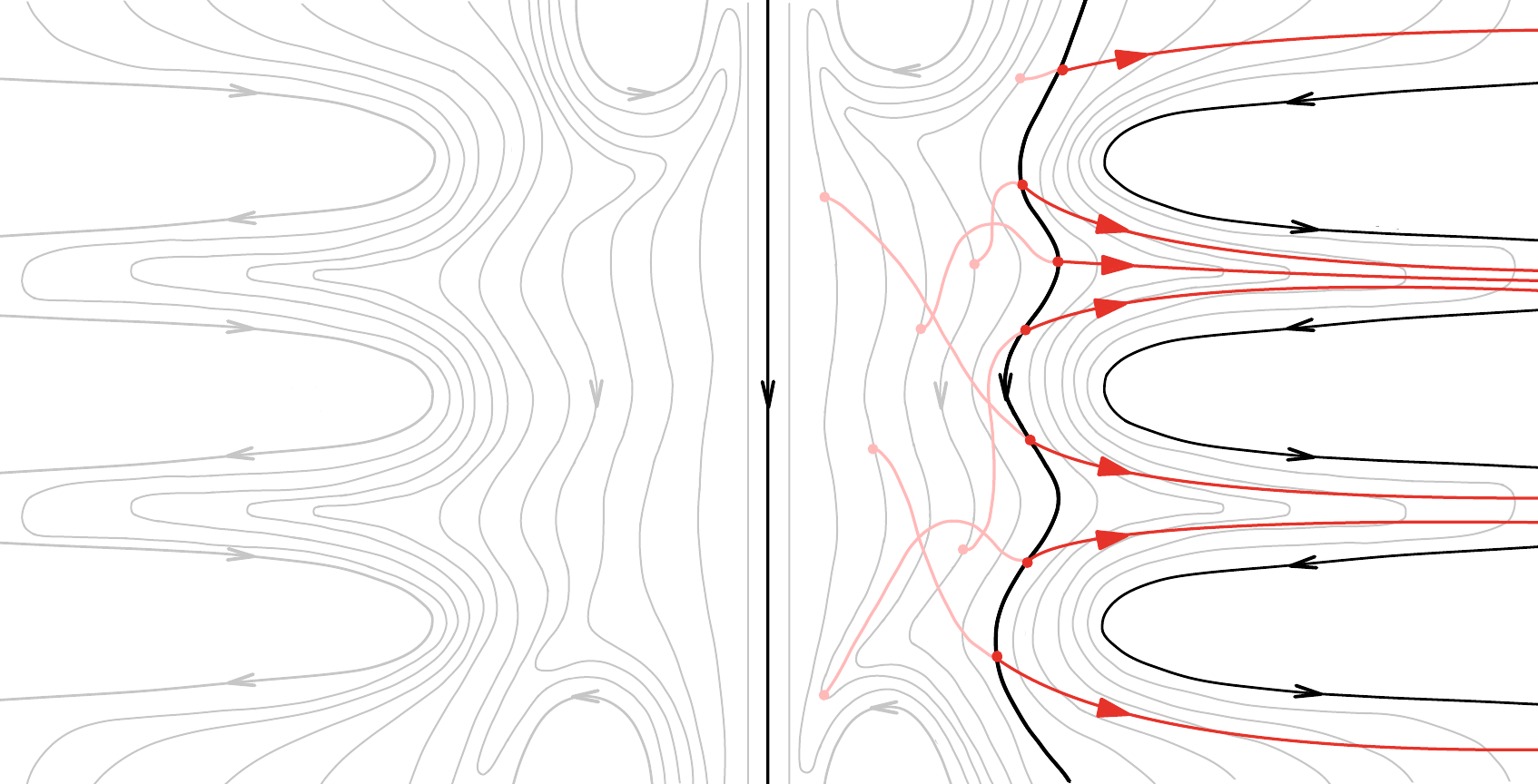}
        \put (69.5,49) {\color{black}\normalsize$\displaystyle \phi^*$}
        \put (48.5,49) {\color{black}\normalsize$\displaystyle \phi$}
         \put (103,22) {\color{myRED}\large$\displaystyle \left\{\Gamma^{^+}_{\sspc i}\right\}_{1\leq i \leq r}$}
        \put (40,22) {\colorbox{white}{\color{black}\large$\displaystyle \rule{0cm}{0.27cm}\ \ $}} 
        \put (40,22) {\color{black}\large$\displaystyle D $} 
\end{overpic}
\end{figure}

\mycomment{-0.3cm}

    By reparametrizing the half-lines in $\{\Gamma^{^+}_i\}_{1\leq i \leq r}\sspc$, we can obtain some $t^*>0$ such that the point $\spc\Gamma^{^+}_i(t^*)\sspc$ lies on the leaf $\phi^*$ for every $i \in \{1,...,r\}$. We remark that, since the half-lines in the family $\{\Gamma^{^+}_i\}_{1\leq i \leq r}$ are pairwise disjoint on the left side $\overline{\rule{0cm}{0.36cm}L(\phi^*)}$, it follows from property described in Section \ref{sec:description_order_proper_trajectories}  that the points in the family $\{\Gamma^{^+}_i(t^*)\}_{1\leq i \leq r}$ must be ordered as
    \mycomment{-0.11cm}
    $$\Gamma_i^{^+}(t^*) < \Gamma_{i+1}^{^+}(t^*) \quad \text{ along the leaf } \phi^*, \quad \forall i \in \{1, \ldots, r-1\}.$$

    \mycomment{-0.22cm}
    \noindent
    Using Lemma \ref{lemma:untangle}, we can replace the arcs $\{\Gamma^{^+}_i(\spc[\spc 0, t^* \spc]\spc)\}_{1 \leq i \leq r}$ so that they become pairwise disjoint, while still preserving the endpoint $\Gamma^{^+}_i(t^*)$ and satisfying 
    \mycomment{-0.11cm}
    $$ \overline{\rule{0pt}{3.6mm}L(\phi) \cap \O \cap R(\phi_i)} \subset \Gamma^{^+}_i(\sspc[\sspc 0, t^* \sspc]\sspc), \quad \forall \O \in [\O_i\sspc]^+.$$

    \mycomment{-0.22cm}
    \noindent
    This concludes the proof under the assumption that $\OOend(D) = \varnothing.$

    \noindent \underline{\textit{Case (2): Assuming that $\OOend(D) \neq \varnothing\spc$:}}
    
    First, consider a leaf $\phi_{\textup{out}} \in D$ such that every leaf contained in $D \cap L(\phi_{\textup{out}} )$ is disjoint from the set $\bigcup_{\O \in \OOend(D)} \O$. Moreover, we can assume that $L(\phi_{\textup{out}}) \subset L(\phi)$.

    Consider the partition of the set $\{1,...,r\}$ into the following two sets
    \mycomment{-0.11cm}
    \begin{align*}
        I_\omega =\{1\leq i \leq r \mid [\O_{\sspc i}]^+ \in \OO_\omega(D)/{\fasym}\spc\}, \ \ \\
        I_{\text{out}} = \{1\leq i \leq r \mid [\O_{\sspc i}]^+ \in \OOend(D)/{\fasym}\spc\}.
    \end{align*}
    
    \mycomment{-0.22cm}
    Note that we can apply the result proved in the previous case ($\OOend(D) = \varnothing$) to the classes of orbits $[\O_{\sspc i}]^+ \in \OO_\omega(D)/{\fasym}$ and the leaf $\phi_{\textup{out}}$. This yields a family $\{\Gamma^{^+}_i\}_{i \in I_\omega}$ of pairwise disjoint half-lines positively transverse to $\F$ such that, for all $i \in I_\omega$, the following conditions hold:
    \begin{itemize}
        \item The endpoint $\Gamma^{^+}_i(0) $ lies on the leaf $\phi_{\textup{out}}$.
        \item The half-orbit $\sspc\overline{\rule{0pt}{3.6mm}L(\phi_{\textup{out}}) \cap \O}$ is contained in $\sspc\Gamma^{^+}_i\sspc$, for all $\O \in [\O_{\sspc i}]^+$.
        \item The points $\left\{\sspc q_i:=\sspc\Gamma^{^+}_i(0)\right\}_{i \in I_\omega}$ are ordered increasingly along $\phi_{\textup{out}}$, that is,
        \mycomment{-0.11cm}
        $$q_i < q_{j} \quad \text{ along the leaf } \phi_{\textup{out}} \sspc, \quad \forall i < j \in I_\omega.$$
    \end{itemize}

    \mycomment{-0.15cm}
    Meanwhile, the set $I_{\text{out}}$ admits itself a partition into sets of the form
    \mycomment{-0.13cm}
    $$ I_{\text{out}} (\phi_L) = \{ i \in I_{\text{out}} \mid [\O_{\sspc i}]^+ \in \OO(\phi_L)/{\fasym}\spc\}, \quad \text{ where } \phi_L \in \partial_L D.$$

    \mycomment{-0.25cm}
    \noindent
    In other words, for each leaf $\phi_L \in \partial_L D$, we have a set $I_{\text{out}}(\phi_L)$ formed by the indices of the classes of orbits $[\O_{\sspc i}]^+ \in \OOend(D)/{\fasym}$ that are leaving the maximal leaf domain $D$ through the leaf $\phi_L \in \partial_L D$. Now, for each leaf $\phi_L \in \partial_L D$, we can consider a family of points $\{p_i\}_{i \in I_{\text{out}}(\phi_L)}$ lying on the leaf $\phi_L$ and ordered increasingly according to the orientation of $\phi_L$, that is,
    \mycomment{-0.15cm}
    $$p_i < p_j \quad \text{ along the leaf } \phi_L , \quad \forall i < j \in I_{\text{out}}(\phi_L).$$

    \mycomment{-0.27cm}
    \noindent
    Next, we consider a family of points $\{q_i\}_{i \in I_{\text{out}}}$ lying on the leaf $\phi_{\textup{out}}$ in such a way that the whole family $\{q_i\}_{1 \leq i \leq r}$ is ordered increasingly according to the orientation of $\phi_{\textup{out}}$, that is,
    \mycomment{-0.15cm}
    $$q_i < q_j \quad \text{ along the leaf } \phi_{\textup{out}}, \quad \forall 1\leq i< j \leq r.$$

    \mycomment{-0.27cm}
    \noindent
    This setting is illustrated in the figure below.

    \vspace*{-0.1cm}
    \begin{figure}[h!]
    \center
    \mycomment{0.2cm}\begin{overpic}[width=9.4cm, height=5cm, tics=10]{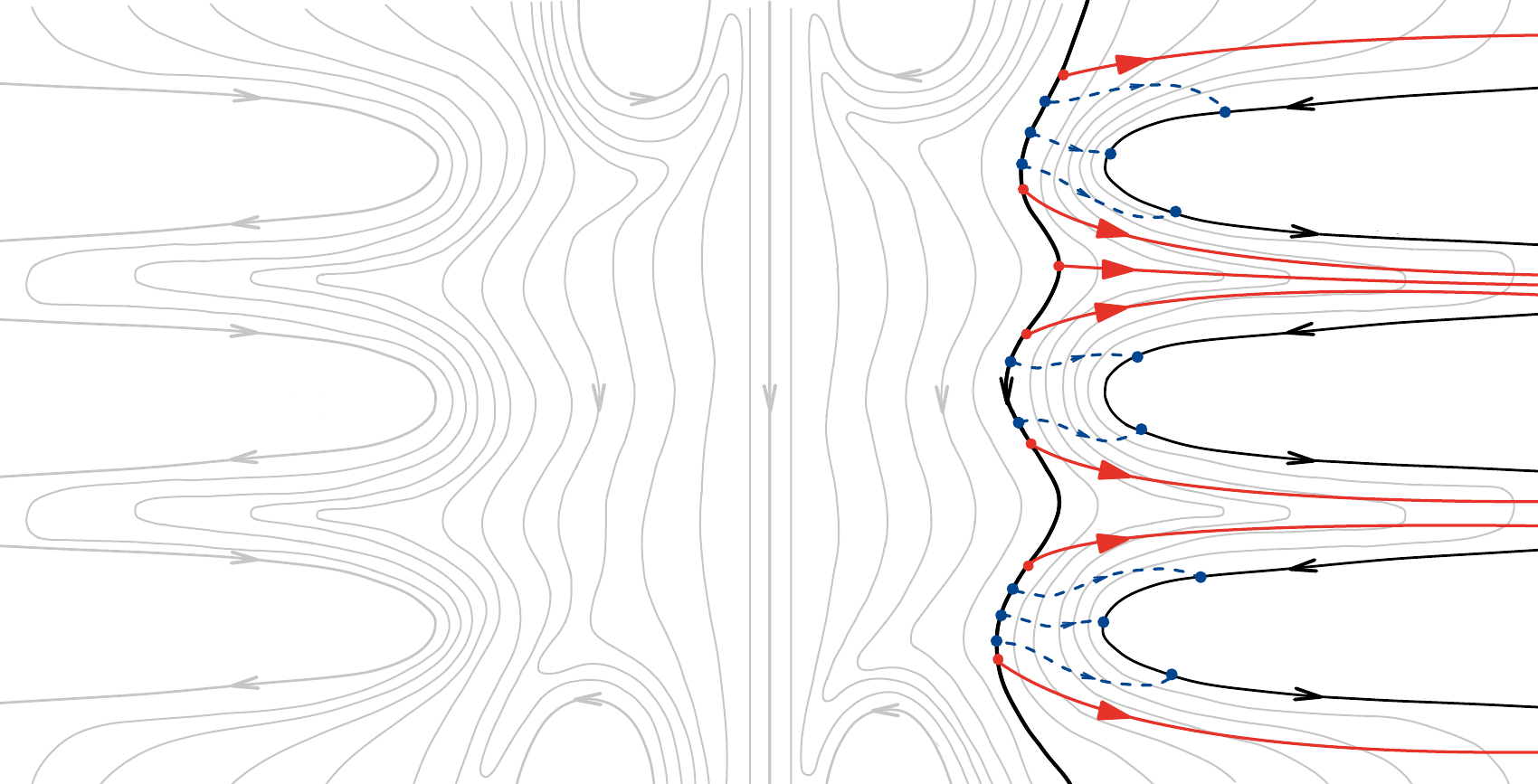}
        \put (69.5,55.8) {\color{black}$\displaystyle \phi_{\text{out}}$}
         \put (103,25) {\color{myRED}\large$\displaystyle \left\{\Gamma^{^+}_{\sspc i}\right\}_{i \in I_\omega}$}
        \put (42,25) {\colorbox{white}{\color{black}\large$\displaystyle \rule{0cm}{0.27cm}\ \ $}} 
        \put (43,25) {\color{black}\large$\displaystyle D $} 
\end{overpic}
\end{figure}

    Note that, for each $i \in I_{\text{out}}$,
    we can construct a path $\gamma_i:[0,1]\longrightarrow\R^2$ that is positively transverse to $\F$ and connects the point $q_i$ to the point $p_i$. By constructing these paths one at a time, we can ensure that the resulting family $\{\gamma_i\}_{i \in I_{\text{out}}}$ satisfies the following properties:
    \begin{itemize}[leftmargin=1.3cm]
        \item $\gamma_i \sspc \cap \sspc \gamma_j=\varnothing$, for any $i,j \in I_{\text{out}}$ with $i \neq j$.
        \item $\gamma_i \sspc \cap \sspc \Gamma^{^+}_j=\varnothing$, for any $i \in I_{\text{out}}$ and $j \in I_\omega$.
    \end{itemize}
    \mycomment{0.2cm}
    Using Lemma \ref{lemma:untangle}, we can extend all half-lines in the family $\{\Gamma^{^+}_i\}_{i \in I_\omega}$ and all paths in the family  $\{\gamma_i\}_{i \in I_\text{out}}$ so that their endpoints $\Gamma^{^+}_i(0) $ and $\gamma_i(0) $ lie on the leaf $\phi$, while preserving pairwise disjointness and ensuring that no orbit point lying between $\phi_{\textup{out}}$ and $\phi$ is missed.

    Finally, by repeating this same construction for each leaf $\phi_L \in \partial_L D$, we extend the paths in the family $\{\gamma_i\}_{i \in I_\text{out}}\sspc$ until we get a family $\{\Gamma^{^+}_i\}_{i \in I_\text{out}}$ of pairwise disjoint half-lines satisfying
    \mycomment{-0.3cm}
    $$\overline{\rule{0pt}{3.6mm}L(\phi) \cap \O}\spc \subset \Gamma^{^+}_i, \quad \forall \O \in [\O_{\sspc i}]^+.$$

    \mycomment{-0.23cm}
    \noindent
    This family, together with the family $\{\Gamma^{^+}_i\}_{i \in I_\omega}$ that we already had before, gives us a family of half-lines $\{\Gamma^{^+}_i\}_{1\leq i \leq r}$ that satisfies the conditions of the lemma. This concludes the proof.
\end{proof}

\subsubsection*{Technical Lemma 4: Blowing-up proper transverse trajectories}

\mycomment{-0.1cm}
\begin{lemma}\label{lemma:blow-up}
    Let $\phi \in \F$, and let $\OO\subset \orbphi$ be a finite set of orbits crossing the leaf $\phi$. 
    Consider a class $[\O\spc]^+ \in \OO/{\fasym}$ and a half-line $\sspc\Gamma^{^+}_{[\O\sspc]^+}: [0, \infty) \longrightarrow\R^2$ as in Lemma \ref{lemma:common_half_traj}. Moreover, consider any open neighborhood $\spc U_{[\O\sspc]^+}\subset \R^2$ of the half-line $\sspc\Gamma^{^+}_{[\O\sspc]^+}$. Therefore,  for any indexing $\sspc[\O\spc]^+ = \{\O_{\sspc 1},\sspc \ldots,\sspc \O_{\sspc r}\}$ there exists a family $\bigl\{\Gamma^{^+}_{\O}\bigr\}_{\O \in [\O\sspc]^+}$ of pairwise disjoint half-lines $\Gamma^{^+}_{\O} : [0, \infty) \longrightarrow\R^2$, each positively transverse to $\F$, such that:
\begin{itemize}[leftmargin=0.8cm]
    \item The half-line $\sspc\Gamma^{^+}_{\O}\sspc$ is contained in $\sspc U_{[\O\sspc]^+}$, for all $\O \in [\O\sspc]^+$.
    \item The endpoint $\sspc\Gamma^{^+}_{\O}(0)$ lies on the leaf $ \phi$, for all $\O \in [\O\sspc]^+$.
        \item The forward orbit $\sspc\overline{\rule{0pt}{3.6mm}L(\phi)} \cap \O\spc $ is contained in $\Gamma^{^+}_\O$, for all $\O \in [\O\sspc]^+.$
    \item The endpoints $\bigl\{\Gamma^{^+}_{\O}(0)\bigr\}_{\O \in [\O\sspc]^+}$ are ordered along the leaf $\phi$ according to $\leq$.
\end{itemize}  
\end{lemma}

\vspace*{-0.3cm}
\begin{figure}[h!]
    \center
    \hspace*{-0cm}\begin{overpic}[width= 10cm, height=4.4cm, tics=10]{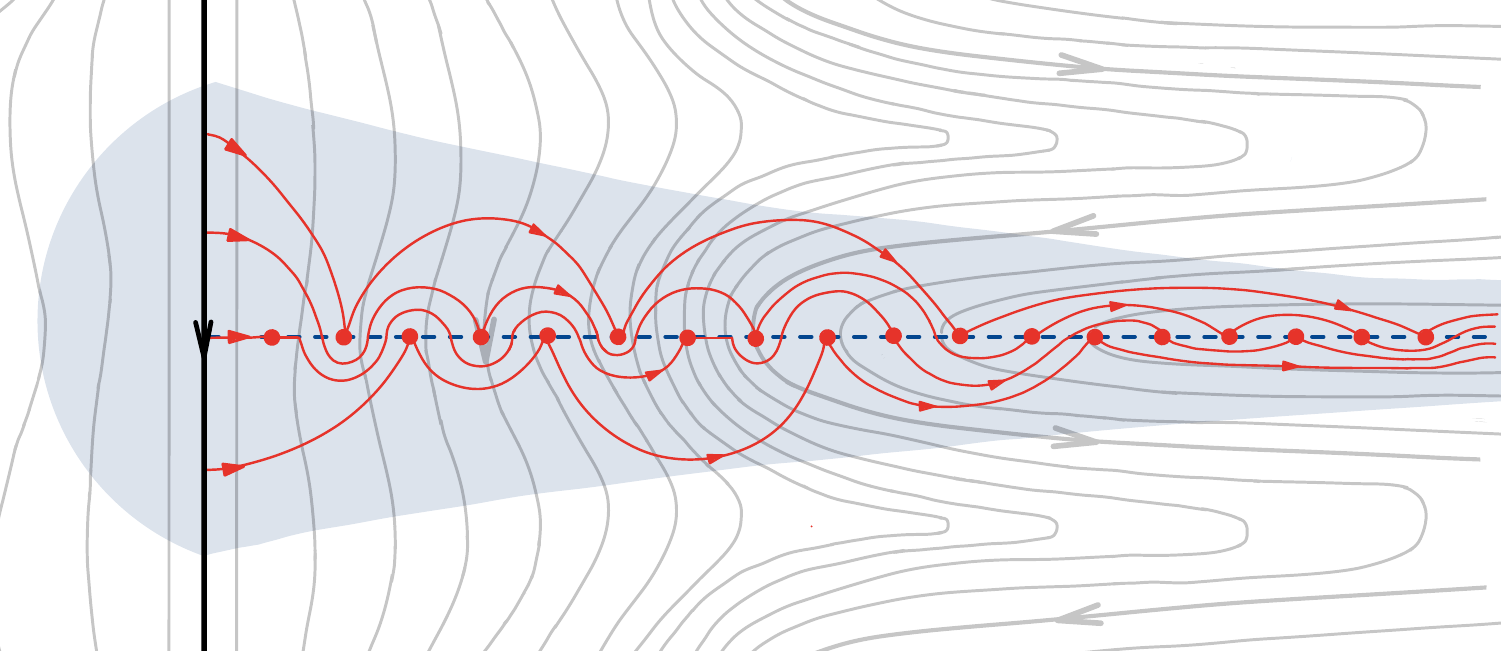}
        \put (-9,21.5) {{\color{myBLUE}\large$\displaystyle {U_{[\O\sspc]^+}} $}}
        \put (7,40) {\colorbox{white}{\color{black}\large$\displaystyle \phi$}}
        \put (103,19) {\color{myRED}\large$\displaystyle \left\{\Gamma^{^+}_{\sspc \O}\right\}_{\O \in [\O\sspc]^+}$}
\end{overpic}
\end{figure}

\vspace*{-0.2cm}

\begin{proof} Using the Homma-Schoenflies theorem, we can assume that:
    \begin{itemize}
        \item $\Gamma^{^+}_{[\O\spc]^+}(\sspc[\sspc0,\infty)\sspc) = [\sspc0,\infty) \times \{0\}$, oriented positively in the first coordinate.
        \item $\phi = \{0\} \times \R$, oriented downwards.
        \item $[\sspc0,\infty)\times [-1,1] \subset U_{[\O\spc]^+}$ is vertically foliated by $\F$.
    \end{itemize}

    \mycomment{0.3cm}
    Consider an arbitrary enumeration of the class $[\O\spc]^+ = \{\O_{\sspc 1},\ldots,\O_{\sspc r}\}$. Note that every two orbit $\O_{\sspc i}, \O_{\sspc j} \in [\O\spc]^+$, we have that $\partial_L C_{\O_{\sspc i}} = \partial_L C_{\O_{\sspc j}}$. Thus, we denoted this set by $\partial_L C_{[\O\spc]^+}.$
    Now, since the closed set $[\sspc0,\infty)\times [-1,1]$ is trivially foliated by $\F$ and contains the image of the transverse half-line $\Gamma^{^+}_{[\O\spc]^+}$, we conclude that $[\sspc0,\infty)\times [-1,1]$ is disjoint from $\partial_L C_{[\O\spc]^+}$. According to Lemma \ref{lemma:trivial_neighborhoods}, there exists a family $\{V_\phi\}_{\phi \in \partial_L C_{[\O\sspc]^+}}$ of sufficiently small trivial-neighborhoods for the leaves in $\partial_L C_{[\O\sspc]^+}$ such that each $V_\phi$ is disjoint from $ [\sspc0,\infty)\times [-1,1] $.

    Recall that each leaf of $\F$ is assumed to intersect at most one point in the set $\bigcup_{i=1}^r \O_i$. Thus, by moving the points of each orbit $\O_{\sspc i} \in[\O\spc]^+ $ vertically, we obtain a homeomorphism 
    \mycomment{-0.12cm}
    $$ \varphi: [\sspc0,\infty)\times [-1,1] \longrightarrow [\sspc0,\infty)\times [-1,1]$$

    \mycomment{-0.22cm}
    \noindent
    that fixes each vertical and satisfies $\varphi(\sspc\O_{\sspc i} \sspc\cap\sspc \overline{\rule{0cm}{0.34cm}L(\phi)}\sspc) \subset [\sspc0,\infty)\times \{1/i\}$ for all $ i\in \{1,\ldots,r\}.$
    For each $i \in \{1,...,n\}$, consider the path $\sspc\Gamma^+_{i}:[0,\infty\sspc) \longrightarrow \R^2$ defined by $\sspc\Gamma^+_{i}(t) = (\sspc t,\sspc 1/i\sspc ).$


    \mycomment{-0.12cm}
    Observe that each path $\varphi^{-1}\circ\Gamma^+_{i}$ is positively transverse to $\F$ and disjoint from all trivial-neighborhoods in $\{V_\phi\}_{\phi \in \partial_L C_{[\O\sspc]^+}}$. Therefore, as shown in the proof of Theorem \ref{thmx:proper_trajectories_restate}, this is property is enough to imply that each $\varphi^{-1}\circ\Gamma^+_{i}$ is a half-line (see Section 
    \ref{sec:proper_transverse_trajectories}).
    
    Thus, we ended up with a family $\{\varphi^{-1} \circ \Gamma^+_i\}_{i=1}^r$ of pairwise disjoint half-lines, each of them positively transverse to $\F$ and contained in the open neighborhood $U_{[\O\sspc]^+}$, with their endpoints $\{\varphi^{-1} \circ \Gamma^+_i(0)\}_{i=1}^r$ lying on the leaf $\phi$ and ordered as
    $$ \varphi^{-1} \circ \Gamma^+_i(0) < \varphi^{-1} \circ \Gamma^+_{i+1}(0) \quad \text{ according to the orientation of } \phi, \quad \forall i \in \{1, \ldots, r-1\}.$$
    This concludes the proof of the lemma.
\end{proof}

\newpage

\subsection{Proof of Theorem \ref{prop:pairwise_disj_traj}}\label{sec:proof_pairwise_disj_traj_indeed}
\mycomment{-0.3cm}

\begin{proof}
    \vspace*{-0.1cm}
    Before we start, we pay attention to the fact that the enumeration $\OO = \{\O_{\sspc 1},\ldots,\O_{\sspc r}\}$ fixed at the statement of the proposition satisfies $\O_{\sspc i} \lesssim_L \O_{\sspc i+1}$ for all $i \in \{1, \ldots, r-1\}$. 
    This implies that, each class $[\O\sspc]^+ \in \OO/{\fasym}$ is associated to an interval of $\{1,...,r\}$ given by
    $$I_{[\O\sspc]^+} = \{\sspc i \in \{1, \ldots, r\} \mid [\O_{\sspc i}]^+ = [\O\sspc]^+\sspc \}.$$ 
    Moreover, the intervals in $\{I_{[\O\sspc]^+}\}_{[\O\sspc]^+ \in \OO/{\fasym}}$ are pairwise disjoint and their arrangement inside the set $\{1, \ldots, r\}$ is determined by the total order $\lesssim_L$ on $\OO/{\fasym}$.

    By applying Lemma \ref{lemma:shared_disjoint} to the leaf $\phi$ and the set of orbits $\OO$, we obtain a family of pairwise disjoint half-lines $\{\sspc\Gamma^{^+}_{[\O\sspc]^+}\sspc\}_{[\O\sspc]^+ \in \OO/{\fasym}}\spc$, each positively transverse to $\F$ and satisfying 
    \mycomment{0.2cm}
    \begin{itemize}
        \item The endpoints $\sspc\Gamma^{^+}_{[\O\sspc]^+}(0)$ lie on the leaf $\phi$.
        \item The forward orbit $\sspc\overline{\rule{0pt}{3.6mm}L(\phi)} \cap \O_{\sspc i} $ is contained in $\Gamma^{^+}_{[\O\sspc]^+}\spc$, for all $i \in I_{[\O\sspc]^+}$.
        \item For any two distinct classes $[\O\sspc]^+, [\O\sspc'\sspc]^+ \in \OO/{\fasym}\spc$ satisfying $[\O\sspc]^+ \lesssim_L [\O\sspc']^+$, we have
        \mycomment{-0.12cm}
        $$\Gamma^{^+}_{[\O\sspc]^+}(0) < \Gamma^{^+}_{[\O\sspc']^+}(0) \quad \text{ along the leaf } \phi .$$
    \end{itemize}
    Next, consider a family $\{\sspc\spc U_{[\O\sspc]^+}\sspc\}_{[\O\sspc]^+ \in \OO/{\fasym}}$ of pairwise disjoint open sets satisfying 
    $$ \Gamma^{^+}_{[\O\sspc]^+} \subset U_{[\O\sspc]^+}, \quad \forall [\O\sspc]^+ \in \OO/{\fasym}.$$
    Finally, for each class $[\O\sspc]^+ \in \OO/{\fasym}$, we apply Lemma \ref{lemma:blow-up} to the half-line $\Gamma^{^+}_{[\O\sspc]^+}$ and the open neighborhood $U_{[\O\sspc]^+}$, with respect to the enumeration $[\O\sspc]^+ = \{\O_{\sspc i}\}_{i \in I_{[\O\sspc]^+}}\spc$. This yields a family of pairwise disjoint half-lines $\{\Gamma^{^+}_{\sspc i}\}_{i \in I_{[\O\sspc]^+}}$, each positively transverse to $\F$ and satisfying:
    \begin{itemize}[leftmargin=1cm]
        \item The half-line $\Gamma^{^+}_{\sspc i}$ is contained in the open neighborhood $U_{[\O\sspc]^+}$.
        \item The forward orbit $\sspc\overline{\rule{0pt}{3.6mm}L(\phi)} \cap \O_{\sspc i} $ is contained in $\Gamma^{^+}_{\sspc i}$.
        \item The endpoint $\Gamma^{^+}_{\sspc i}(0)$ lies on the leaf $\phi$.
        \item The points $\{\Gamma^{^+}_{\sspc i}(0)\}_{i \in I_{[\O\sspc]^+}}$ are ordered $\Gamma^{^+}_{\sspc i}(0) < \Gamma^{^+}_{\sspc i+1}(0)$ along the leaf $\phi$.
    \end{itemize}
    

\mycomment{-0.1cm}
    By repeating this process for each class $[\O\sspc]^+ \in \OO/{\fasym}$, we obtain a family of pairwise disjoint half-lines $\{\Gamma^{^+}_{\sspc i}\}_{1\leq i \leq r}$ that satisfies the conditions stated in the proposition.
\end{proof}

\vspace*{-0.1cm}
By combining Remark \ref{sec:description_order_proper_trajectories} with Theorem \ref{prop:pairwise_disj_traj}, we obtain directly the following corollary, which provides a (generic) definition of $\lesssim_L$ and $\lesssim_R$ in terms of proper transverse trajectories.

\vspace*{-0.1cm}
\begin{corollary}
\label{cor:alternative_def_preorders}
Let $\OO \subset \orb$ be a finite set of orbits, and assume that $\F$ is generic for $\OO$.
     Let $\phi \in \F$ be a leaf, and let $\O,\O^\pp \in \OO(\phi)$ be two distinct orbits. Then, we have that $\O\lesssim_L \O^\pp$ if, and only if, there exists proper transverse trajectories $\Gamma_\O$ and $\Gamma_{\O^\pp}$ of the orbits $\O$ and $\O^\pp$, respectively, such that the following conditions hold:
\begin{itemize}[leftmargin=1.3cm]
    \item[\textup{\textbf{(i)}}] The trajectories $\Gamma_\O$ and $\Gamma_{\O^\pp}$ are disjoint on the left side of $\phi$, that is,
    \mycomment{-0.1cm}
    $$\Gamma_\O \cap \Gamma_{\O^\pp} \cap \overline{\rule{0cm}{0.35cm}L(\phi)} = \varnothing.$$

    \mycomment{-0.15cm}
    \item[\textup{\textbf{(ii)}}] The intersection points $p_\O \in \Gamma_\O \cap \phi$ and $p_{\O^\pp} \in \Gamma_{\O^\pp} \cap \phi$ satisfy
    \mycomment{-0.05cm}
    $$p_\O < p_{\O^\pp} \quad \text{ along the leaf } \phi \text{ (according to its orientation)}.$$
\end{itemize}
\mycomment{0cm}
We obtain an analogous result for $\lesssim_R$ by replacing each $L$ with $R$ in the above statement.
\end{corollary}
In Section \ref{sec:def_weak_transverse_intersections}, we have seen that orbits having a weak $\F$-transverse intersection do not admit a pair of disjoint proper transverse trajectories. Now, with Corollary \ref{cor:alternative_def_preorders}, we can show that the converse is also true in the same generic setting. Yielding the following corollary.

\newpage

\begin{corollary}
\label{cor:alternative_def_weak_F}
Distinct orbits $\O,\O^\pp \in \OO$ have a weak $\F$-transverse intersection if, and only if, every pair of proper transverse trajectories $\Gamma_\O$ and $\Gamma_{\O^\pp}$ associated to $\O$ and $\O^\pp$ intersect.
\end{corollary}

\begin{proof}
    Having in mind the discussion of Section \ref{sec:def_weak_transverse_intersections}, to prove the corollary it suffices to show that any two orbits $\O,\O^\pp \in \OO$ that do not have a weak $\F$-transverse intersection admit a pair of disjoint proper transverse trajectories. Note that, in the case where $C_\O \cap C_{\O^\pp} = \varnothing$, this result is trivial. Thus, assume that $C_\O \cap C_{\O^\pp} \neq \varnothing$ and consider a leaf $\phi \in C_\O \cap C_{\O^\pp}$ that is disjoint from $\O$ and $\O^\pp$. Now, we can consider leaves $\phi_L, \phi_R \in C_\O \cap C_{\O^\pp}$ with $L(\phi_L) \subset L(\phi_R)$ that are sufficiently close to $\phi$ so that the set $ \overline{\rule{0cm}{0.35cm}L(\phi_R) \cap R(\phi_L)}$ is also disjoint from $\O$ and $\O^\pp$.
    By applying Corollary \ref{cor:alternative_def_preorders}, we can obtain proper transverse trajectories $\Gamma_\O$ and $\Gamma_{\O^\pp}$ for the orbits $\O$ and $\O^\pp$ that are disjoint on the left side of $\phi_L$ and on the right side of $\phi_R$. Since the orbits $\O$ and $\O^\pp$ do not have a weak $\F$-transverse intersection, these trajectories $\Gamma_\O$ and $\Gamma_{\O^\pp}$ can be taken to intersect the leaves $\phi_L$ and $\phi_R$ in the same order. Thus, we can easily modify $\Gamma_\O$ and $\Gamma_{\O^\pp}$ in the set $\overline{\rule{0cm}{0.35cm}L(\phi_R) \cap R(\phi_L)}$ so that they become disjoint, while preserving the properties of being proper transverse trajectories. This concludes the proof.
\end{proof}

\section{Adaptive orders across maximal leaf domains}\label{sec:refinements_definitions}

We begin this section with an illustrative scenario that motivates the main idea behind the proof of Theorem \ref{thm:main_Part1}. This scenario goes as follows:

Let $\OO\subset \orb$ be a finite collection of orbits, and let $\D_\OO$ be the decomposition by maximal leaf domains induced by $\OO$. Consider a maximal leaf domain $D\in \D_\O$ and a pair of sufficiently close leaves $\phi_L, \phi_R \in D$ satisfying \(L(\phi_L) \subset L(\phi_R)\). Now, assume that there exists a family of proper transverse trajectories $\{\Gamma_\O\}_{\O \in \OO(D)}$ for the orbits in $\OO(D)$ such that 
$$ \bigcap_{\O \in \OO(D)}\Gamma_\O \cap \overline{\rule{0cm}{0.39cm}L(\phi_L) \cap R(\phi_R) \cap D} = \varnothing.$$

\mycomment{0.1cm}
\begin{figure}[h!]
    \center
    \mycomment{0cm}\begin{overpic}[width=10cm, height=5cm, tics=10]{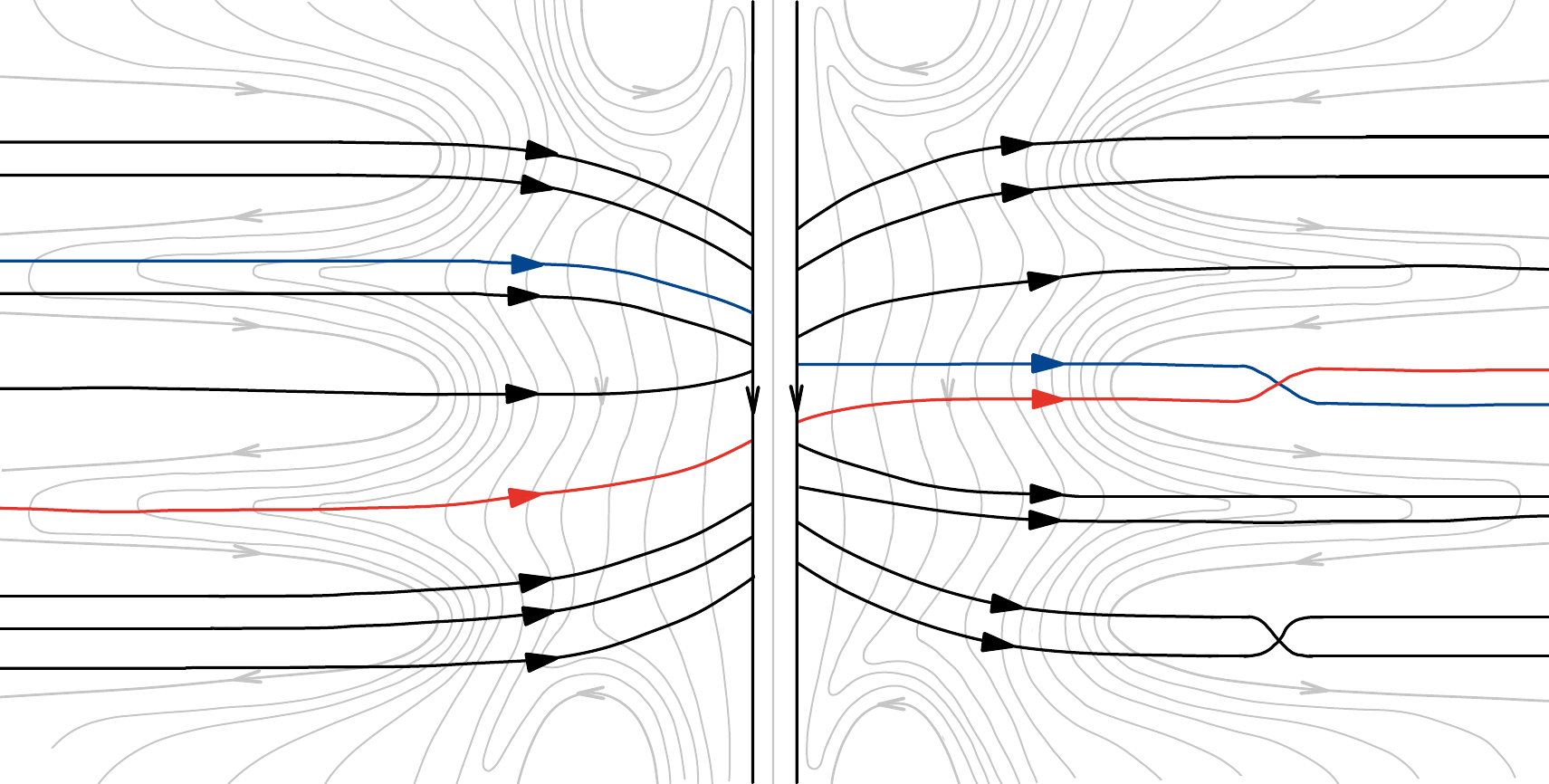}
        \put (44,52.5) {\color{black}\normalsize$\displaystyle \phi_R \ \  \phi_L$}
         \put (102,25.2) {\color{myRED}\large$\displaystyle \Gamma_{\O^\pp}$}
         \put (102,18.4) {\color{myBLUE}\large$\displaystyle \Gamma_\O $}
        \put (-8.5,30) {\color{myBLUE}\large$\displaystyle \Gamma_\O $}
        \put (-8.5,15) {\color{myRED}\large$\displaystyle \Gamma_{\O^\pp} $}
\end{overpic}
\end{figure}

Moreover, we want that the family $\{\Gamma_\O\}_{\O \in \OO(D)}$ to follow the following principles:
\begin{itemize}[leftmargin=1.3cm]
    \item If two orbits $\O,\O^\pp \in \OO(D)$ have a weak $\F$-transverse intersection, we ask ourselves
        $$\textup{\ \ \ \ \ ``Does there exist some $D^\pp \in \D_\OO$ satisfying $D \leadsto D^\pp$ for which $\O,\O^\pp \in \OO(D^\pp)$\sspc?''}$$
        \underline{If the answer is ``YES''\sspc:} then we want that the trajectories $\Gamma_\O$ and $\Gamma_{\O^\pp}$ to be disjoint inside $D$, leaving their obligatory intersection to happen in $D^
        \pp$, or later if possible.

        \noindent\underline{If the answer is ``NO''\sspc:} then we want the obligatory intersection between $\Gamma_\O$ and $\Gamma_{\O^\pp}$ to happen in a single point inside $D$, specifically, in between the leaves $\phi_L$ and $\phi_R$.
    
    \item If two orbits $\O,\O^\pp \in \OO(D)$ do not have a weak $\F$-transverse intersection, the we simply want the trajectories $\Gamma_\O$ and $\Gamma_{\O^\pp}$ to be disjoint.
\end{itemize}

\mycomment{0.2cm}
Observe that, any family $\{\Gamma_\O\}_{\O \in \OO(D)}$ as described above would prove Theorem \ref{thm:main_Part1} in the particular case where $\OO = \OO(D)$. Therefore, in order to prove Theorem \ref{thm:main_Part1} in the general case, we not only need to construct such a family $\{\Gamma_\O\}_{\O \in \OO(D)}$ for each leaf domain $D \in \D_\OO$, but also to do so in a coherent way, so that the resulting family $\{\Gamma_\O\}_{\O \in \OO}$ solves the theorem.

This problem will be solved by defining a ``standard'' total order $\leq_\phi$ on the set $\OO(\phi)$ of orbits crossing a leaf $\phi \in \F$, and an ``adaptive'' total order $\leq_{D}$ on the set $\OO(D)$ of orbits crossing a maximal leaf domain $D \in \D_\OO$. 

Now, we come back to the general setting. Let $\OO\subset \orb$ be a finite collection of orbits, and let $\D_\OO$ be the decomposition by maximal leaf domains induced by $\OO$. 
Before we proceed with the definitions of the standard and adaptive orders, we chose to denote the equivalence relation of backward and forward $\F$-asymptotic orbits in $\OO$ by $\sim$, that is,
\mycomment{-0.1cm}
$$ \O \sim \O^\pp \iff \O \fasym \O^\pp \text{ and } \O \basym \O^\pp.$$

\mycomment{-0.25cm}
\noindent Then, we fix an arbitrary total order $\leq_{[\O\sspc]}$ on each equivalence class $[\O\sspc] \in \OO/\sim$. This choice of total orders $\{\leq_{[\O\sspc]}\}_{[\O\sspc] \in \OO/\sim}$ determine the standard and adaptive orders defined below.

\mycomment{-0.3cm}
\subsubsection*{Standard ordering over leaves:}

\mycomment{-0.15cm}
For each leaf $\phi \in \F$, we consider the relation $\leq_\phi$ on the set $\OO(\phi)$, defined as follows:
\mycomment{0.15cm}
$$
\O\leq_\phi \O^{\sspc \prime} \iff \begin{cases} \ \O\lesssim_R \O^{\sspc \prime} \quad\ \  \text{ if } \O\not\basym\O^{\sspc \prime},\\
    \ \O\lesssim_L \O^{\sspc \prime} \quad \ \ \text{ if } \O\basym\O^{\sspc \prime} \text{ and } \O\not\fasym\O^{\sspc \prime},\\ 
    \ \O\leq_{[\O\sspc]} \O^{\sspc \prime} \quad\sspc \text{ if } \O\sim \O^{\sspc \prime}.
\end{cases}
$$

\mycomment{0.15cm}
\noindent Roughly speaking, the relation $\leq_\phi$ primarily compares orbits in $\OO(\phi)$ using the relation $\lesssim_R$.  If the compared orbits are backward $\F$-asymptotic, then it tries to compare them using the relation $\lesssim_L$. At last, if the orbits are both forward and backward $\F$-asymptotic, then it uses the total order $\leq_{[\O\sspc]}$ defined on the equivalence class $[\O\sspc] \in \OO/{\sim}$ of these orbits.

\begin{lemma}
\label{lemma:order_standard}
    For any leaf $\phi \in \F$, the relation $\leq_\phi$ is a total order on the set $\OO(\phi)$.
\end{lemma}

\begin{proof}
    Let $\phi \in \F$ be a leaf, and recall that $\lesssim_R$ induces a total order on the set $\OO(\phi)/{\basym}$.  Thus, we can enumerate the set $\OO(\phi)/{\basym}$ increasingly with respect to $\lesssim_R$, that is,
    $$\OO(\phi)/{\basym} = \{\OO_{\sspc 1},\OO_{\sspc 2},\ldots,\OO_{\sspc r}\}\sspc , \quad  \text{ with }\  \OO_{\sspc i} \lesssim_R \OO_{\sspc i+1} \  \text{ for all } \ 1\leq i < r.$$
    Then, for each $i \in \{1, \ldots, r\}$, since $\lesssim_L$ induces a total order on the set $\OO_{\sspc i}/{\fasym}$, we can enumerate it increasingly with respect to $\lesssim_L$, that is,
    $$\OO_{\sspc i}/{\fasym} = \{\OO_{\sspc i,1},\OO_{\sspc i,2},\ldots,\OO_{\sspc i,m(i)}\}\sspc , \quad  \text{ with }\  \OO_{\sspc i,j} \lesssim_L \OO_{\sspc i,j+1} \  \text{ for all } \ 1\leq j< m(i).$$
    At last, note that, for any $i \in \{1, \ldots, n\}$ and $j \in \{1, \ldots, m(i)\}$, it holds that $\OO_{\sspc i,j} \in \OO(\phi)/{\sim}$. Thus, there exists an equivalence class $[\O\sspc]_{\sspc i,j} \in \OO/{\sim}$ such that $\OO_{\sspc i,j} \subset [\O\sspc]_{\sspc i,j}$. Therefore, we can enumerate $\OO_{\sspc i,j}$ increasingly with respect to the initially fixed total order $\leq_{[\O\sspc]_{\sspc i,j}}$, that is,
    $$\OO_{\sspc i,j} = \{\O_{\sspc i,j,1},\ldots,\O_{\sspc i,j,n(i,j)}\}\sspc , \quad  \text{ with }\ \O_{\sspc i,j,k} \leq_{[\O\sspc]_{\sspc i,j}} \O_{\sspc i,j,k+1} \  \text{ for all } \ 1\leq k< n(i,j).$$
  At the end, we obtained a triple index enumeration of the set $\OO(\phi)=\{\O_{\sspc i,j,k}\}_{i,j,k}$
    where the indices run through $i \in \{1, \ldots, r\}$, $j \in \{1, \ldots, m(i)\}$ and $k \in \{1, \ldots, n(i,j)\}$.
    
    \noindent By construction, the relation $\leq_\phi$ on $\OO(\phi)=\{\O_{\sspc i,j,k}\}_{i,j,k}$ coincides with the lexicographic order
    \mycomment{0cm}
    $$\O_{\sspc i,j,k} \leq_\phi \O_{\sspc i',j',k'} \iff \begin{cases}
        \ i < i' & \text{ or,}\\
        \ i = i' \text{ and } j < j' & \text{ or,}\\
        \ i = i' \text{ and } j = j' \text{ and } k < k'.
    \end{cases}$$
    \mycomment{0.3cm}
    This proves that $\leq_\phi$ is a total order on the set $\OO(\phi)$, and concludes the proof.
\end{proof}

\newpage
\subsection*{Adaptive ordering over maximal leaf domains:}

\mycomment{-0.15cm}
For each  $D \in \D_{\OO}$, we consider an equivalence relation $\Dasym$ on the set $\OO(D)$ that is slightly weaker than the relation $\fasym$. For any two orbits $\O,\O^\pp \in \OO(D)$, we say that $\O \Dasym \O^\pp$ if:
\begin{itemize}[leftmargin=1.3cm]
    \item Either $\O,\O^\pp\in\OO_\omega(D)$ and they satisfy $\O \fasym \O^\pp$.
    \item Or $\O,\O^\pp \in \OOend(D)$ and there exists a leaf $\phi_L \in \partial_L D$ such that $\O, \O^\pp \in \OO(\phi_L)$. 
\end{itemize}
\mycomment{0.2cm}
We observe that, if the orbits $\O$ and $\O^\pp$ satisfy $\O\fasym \O^\pp$ then they must also satisfy $\O \Dasym \O^\pp$, but the converse does not hold in general, only if the orbits $\O$ and $\O^\pp$ are both in $\OO_\omega(D)$.
Moreover, note that the relation $\lesssim_L$ induces a total order on the set $\OO(D)/{\sspc\Dasym}$. This follows from the fact that if the orbits $\O,\O^\pp,\O^{\sspc \prime\prime} \in \OOend(D)$ satisfy $\O \lesssim_L \O^\pp \lesssim_L \O^{\sspc \prime\prime}$ and $\O \Dasym \O^{\sspc \prime\prime}$, then they must also satisfy $\O\Dasym \O^{\sspc \prime\prime}$.

With this equivalence relation, we can define the relation $\leq_D$ on the set $\OO(D)$ as follows:
\mycomment{0.2cm}
$$
\O\leq_D \O^{\sspc \prime} \iff \begin{cases} \ \O\lesssim_L \O^{\sspc \prime} \quad\ \  \text{ if } \O\not\Dasym\O^{\sspc \prime},\\
    \ \O\lesssim_R \O^{\sspc \prime} \quad \ \ \text{ if } \O\Dasym\O^{\sspc \prime} \text{ and } \O\not\basym\O^{\sspc \prime},\\ 
    \ \O\leq_{[\O\sspc]} \O^{\sspc \prime} \quad\sspc \text{ if } \O\sim \O^{\sspc \prime}.
\end{cases}
$$

\mycomment{0.2cm}
\noindent Roughly speaking, the relation $\leq_D$ primarily compares orbits in $\OO(\phi)$ using the relation $\lesssim_L$.  If the compared orbits exit the domain $D$ through the same leaf $ \phi_L \in \partial_L D$, then it tries to compare them using $\lesssim_R$. At last, if these orbits are both forward and backward $\F$-asymptotic, then it uses the total order $\leq_{[\O\sspc]}$ on their equivalence class $[\O\sspc] \in \OO/{\sim}$ to compare them.

Note that $\leq_D$ may be viewed as an analogue of $\leq_\phi$, with the roles of $\lesssim_L$ and $\lesssim_R$ essentially reversed, except in the case where the compared orbits exit $D$ through the same leaf in $\partial_L D$, in which case the relation $\lesssim_R$ continues to be applied.

\begin{lemma}
\label{lemma:order_adaptive}
    For any $D \in \D_\OO$, the relation $\leq_D$ is a total order on the set $\OO(D)$.
\end{lemma}

\begin{proof}
    The proof is completely analogous to the proof of Lemma \ref{lemma:order_standard}.
\end{proof}

\section{Proof of Theorem \ref{thm:main_Part1}}\label{sec:proof_main_Part1}
This section is dedicated to the proof of Theorem \ref{thm:main_Part1}. The version of the theorem stated here is slightly different from the one presented in the introduction (see Remark \ref{remark:proof_main_Part1}).

\begin{thmx}\label{thm:main_Part1}
     Let $\OO\subset \orb$ be a finite collection of orbits, and assume \(\F\) is generic for $\OO$. Then,
    there exists a family $\{\Gamma_\O\}_{\O \in \OO}$ of proper transverse trajectories of $\OO$ that satisfies
    \begin{itemize}[leftmargin=1.3cm]
        \item[\textbf{(i)}]  $\#\sspc(\sspc \Gamma_\O \sspc \cap\sspc \Gamma_{\O^\pp}\sspc ) \leq 1$, for any pair of distinct orbits $\O,\O^\pp \in \OO$.
        \item[\textbf{(ii)}]  $\#\sspc(\sspc \Gamma_\O \sspc \cap\sspc \Gamma_{\O^\pp}\sspc ) = 1$ if, and only if, $\O$ and $\O^\pp$ have a weak $\F$-transverse intersection.
    \end{itemize}
\end{thmx}

\begin{remark}
\label{remark:proof_main_Part1}
    The version of Theorem \ref{thm:main_Part1} enunciated in the introduction is a direct consequence of the statement above combined with the following argument. Enumerate $\OO = \{\O_{\sspc 1}, \ldots, \O_{\sspc r}\}$, and consider the family of proper transverse trajectories $\{\Gamma_{i}\}_{1\leq i \leq r}$ given by the theorem above. Then, consider a collection of planar foliations $\{\F_j\}_{1\leq j \leq n}$ which are pairwise transverse to each other, and such that $\F_0 = \F$ and for each $1\leq i \leq r$ there exists $1\leq j \leq n$ such that $\Gamma_i \in \F_j$.\break The homeomorphism $\varphi:\R^2 \longrightarrow \mathbb D^2$ in the statement of Theorem \ref{thm:1C-alternative} in the introduction is given by the circle compactification of $\R^2$ via ideal boundaries of $\{\F_j\}_{1\leq j \leq n}$, as described in \cite{Bon24}.
\end{remark}

\begin{proof}
Let $\D_{\OO}$ be the decomposition into maximal leaf domain induced by $\OO$, and let $\W_{\OO}$ be the set of critical leaves for $\OO$ (see Section \ref{sec:critical_leaves_and_core regions}). Next, fix an arbitrary total order $\leq_{[\O\sspc]}$ on each equivalence class $[\O\sspc] \in \OO/{\sim}$, where $\sim$ consists in the equivalence relation of backward  and forward $\F$-asymptotic orbits in $\OO$. As explained in Section \ref{sec:refinements_definitions}, this choice of total  orders $\{\leq_{[\O\sspc]}\}_{[\O\sspc] \in \OO/\sim}$ determines the standard order $\leq_\phi$ and the adaptive order $\leq_D$, which are total orders on the sets $\OO(\phi)$ and $\OO(D)$, respectively, for each $\phi \in \F$ and each $D \in \D_{\OO}$.

For each leaf $\phi \in \W_{\OO}$, we consider a family of pairwise distinct points $\{p_{\phi}(\O)\}_{\O \in \OO(\phi)}$ lying on the leaf $\phi$, such that the points $p_{\phi}(\O)$ are ordered along the leaf $\phi$ according to $\leq_\phi$, that is
$$\O \leq_\phi \O^\pp \iff p_{\phi}(\O) \leq p_{\phi}(\O^\pp) \ \ \text{ according to the orientation of } \phi.$$
In the case where the leaf $\phi \in \W_{\OO}$ intersects some orbit $\O \in \OO(\phi)$, then we can choose the point $p_{\phi}(\O)$ to be exactly the intersection point $\phi \cap \O = \{p_{\phi}(\O)\}$.

Fix a leaf domain $D \in \D_{\OO}$, and let $\phi_L, \phi_R \in D$ be two leaves with $ L(\phi_L) \subset L(\phi_R)$ that are sufficiently close and sufficiently to the left so that they satisfy 
$$ \overline{\rule{0cm}{0.35cm}L(\phi_R) \cap R(\phi_L)} \cap \bigcup_{\O \in \OO(D)} \O  = \varnothing \quad \text{ and } \quad \overline{\rule{0cm}{0.35cm}L(\phi_L)} \cap D \cap \bigcup_{\O \in \OOend(D)} \O = \varnothing.$$
\noindent Now, consider a family of pairwise distinct points $\{q_{_{R,D}}(\O)\}_{\O \in \OO(D)}$ lying on the leaf $\phi_R$ and ordered according to the standard total order $\leq_{\phi_R}$ on the set $\OO(D) = \OO(\phi_R)$, meaning that
$$\O \leq_{\phi_R} \O^\pp \iff q_{_{R,D}}(\O) \leq q_{_{R,D}}(\O^\pp) \ \ \text{ according to the orientation of } \phi_R.$$
\noindent Observe that, the order $\leq_{\phi_R}$ allows us to apply the right-analogous version of Theorem \ref{prop:pairwise_disj_traj}  to the leaf $\phi_R$ and to the set orbits $\OO(D)$ indexed increasingly by $\leq_{\phi_R}$, thus obtaining a family of pairwise disjoint half-lines $\left\{\Gamma^{^-}_{D, \O}\right\}_{\O \in \OO(D)}$, each positively transverse to $\F$, that satisfies
\begin{itemize}
    \item The endpoint $\sspc\Gamma^{^-}_{D, \O}(0)$ lies on the leaf $ \phi_R$, for all $\O \in \OO(D)$.
    \item The backward orbit $\sspc\overline{\rule{0pt}{3.6mm}R(\phi_R)} \cap \O\spc $ is contained in $\Gamma^{^-}_{D, \O}$, for all $\O \in \OO(D).$
    \item The endpoints $\left\{\sspc\Gamma^{^-}_{D, \O}(0)\sspc\right\}_{\O \in \OO(D)}$ are ordered along the leaf $\phi_R$ according to $\leq_{\phi_R}$.
\end{itemize}

\vspace*{0.3cm}

We remark that the proof of Theorem \ref{prop:pairwise_disj_traj} gives us enough flexibility to choose the half-lines in the family $\left\{\Gamma^{^-}_{D, \O}\right\}_{\O \in \OO(D)}$ so that they satisfy the following additional properties:
\begin{itemize}
    \item $\Gamma^{^-}_{D, \O}(0) = q_{_{R,D}}(\O)$, for any $\O \in \OO(D)$.
    \item $p_{\phi^\pp}(\O) \in \Gamma^{^-}_{D, \O}$, for any $\O \in \OOstart(D)$, where $\phi^\pp$ denotes the unique leaf in $\partial_R D \cap C_\O$.
\end{itemize}

\vspace*{0.25cm}
\noindent At last, for each orbit $\O \in \OOstart(D)$, we cut  $\Gamma^{^-}_{D, \O}$ at its intersection point with $\partial_R D$
so that it stays inside $D$ and it joins $\partial_R D$ to $\phi_R$. In other words, we redefine $\Gamma^{^-}_{D, \O}:[0,1]\longrightarrow\R^2$  as the 
sub-arc of the originally defined half-line $\Gamma^{^-}_{D, \O}$  that ends at the point $q_{_{R,D}}(\O) \in \phi_R$, and starts at the point $p_{\phi^\pp}(\O) \in \phi^\pp$, where $\phi^\pp$ is the unique leaf in $\partial_R D \cap C_\O$.

 Now, consider a family of pairwise distinct points $\{q_{_{L,D}}(\O)\}_{\O \in \OO(D)}$ lying on the leaf $\phi_L$ and ordered according to the standard total order $\leq_{D}$ on the set $\OO(D)$, meaning that
$$\O \leq_{D} \O^\pp \iff q_{_{L,D}}(\O) \leq q_{_{L,D}}(\O^\pp) \ \ \text{ according to the orientation of } \phi_L.$$
Observe that, the order $\leq_{D}$ only allows us to apply Theorem \ref{prop:pairwise_disj_traj} to the leaf $\phi_L$ and the subset $\OO_\omega(D) \subset \OO(D)$ indexed increasingly by $\leq_{D}$, as the order $\leq_{D}$ is not necessarily compatible with the relation $\lesssim_L$ on orbits leaving $D$ through the same leaf in $\partial_L D$. However, this is enough to obtain a family of pairwise disjoint half-lines $\left\{\Gamma^{^+}_{D, \O}\right\}_{\O \in \OO_\omega(D)}$, each positively transverse to $\F$, that satisfies the following conditions:
\begin{itemize}
    \item The endpoint $\sspc\Gamma^{^+}_{D, \O}(0)$ coincides with the point $q_{_{L,D}}(\O)\in\phi_L$, for all $\O \in \OO_\omega(D)$.
    \item The forward orbit $\sspc\overline{\rule{0pt}{3.6mm}L(\phi_L)} \cap \O\spc $ is contained in $\Gamma^{^+}_{D, \O}$, for all $\O \in \OO_\omega(D).$
\end{itemize}
\mycomment{0.3cm}
The current setting is illustrated in the following figure for sake of clarity.

\vspace*{0.1cm}
\begin{figure}[h!]
    \center
    \mycomment{0.6cm}\begin{overpic}[width=8cm, height=4.3cm, tics=10]{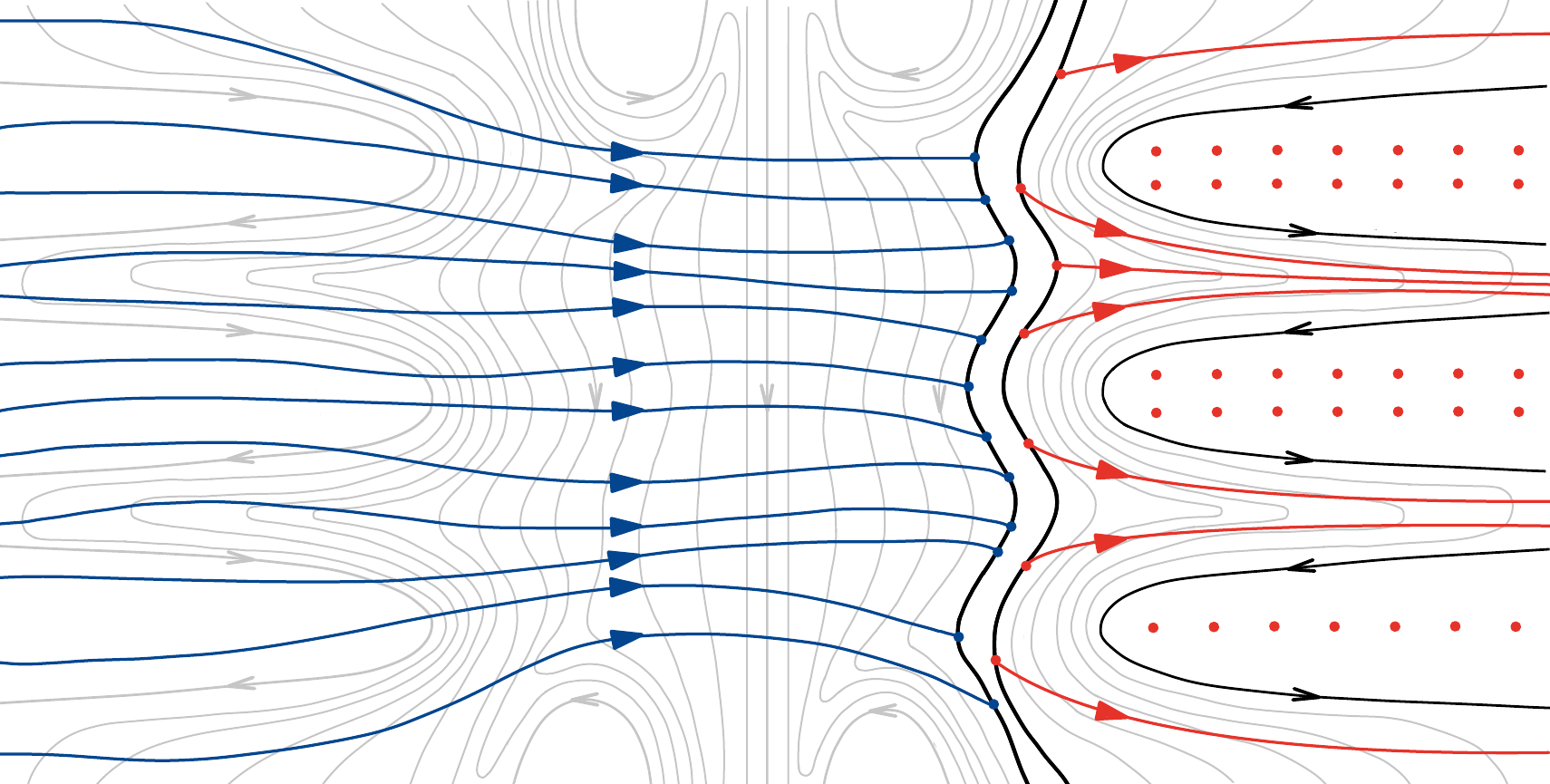}
        \put (62,56) {\color{black}\normalsize$\displaystyle \phi_R \  \phi_L$}
         \put (103,31) {\color{myRED}\large$\displaystyle \left\{\Gamma^{^+}_{D, \O}\right\}_{\O \in \OO_\omega(D)}$}
        \put (-36,25) {\color{myBLUE}\large$\displaystyle \left\{\Gamma^{^-}_{D, \O}\right\}_{\O \in \OO(D)} $}
        
\end{overpic}
\end{figure}

\mycomment{-0.1cm}
Observe that, for each orbit $\O \in \OOend(D)$, we have a leaf $\phi^\pp \in \partial_L D$ such that $\O \in \OO(\phi^\pp)$, and we can construct a path $\Gamma_{D,\O}^{^+}:[-1,0\sspc]\longrightarrow\R^2$ positively transverse to $\F$ that connects the point $q_{_{L,D}}(\O)\in \phi_L$ to the point $p_{\phi^\pp}(\O) \in \phi^\pp$. By constructing these paths one at a time,  we can construct them so that the family $\{\Gamma_{D,\O}^{^+}\}_{\O \in \OOend(D)}$ is pairwise disjoint and, moreover, 
 so that they are each disjoint from the already constructed half-lines $\{\Gamma^{^-}_{D, \O}\}_{\O \in \OO_\omega(D)}$.

To summarize, we have constructed two families $\{\Gamma^{^-}_{D, \O}\}_{\O \in \OO(D)}$ and $\{\Gamma^{^+}_{D, \O}\}_{\O \in \OO(D)}$ formed by pairwise disjoint half-lines and paths that are positively transverse to $\F$ and such that, for each orbit $\O \in \OO(D)$, the following properties hold:
\begin{align*}
    \Gamma^{^-}_{D, \O}(0) = q_{_{R,D}}(\O) \in \phi_R \quad &\text{ and } \quad \overline{\rule{0pt}{3.5mm}R(\phi_R)\cap D} \cap \O \subset \Gamma^{^-}_{D, \O}\spc, \\
    \Gamma^{^+}_{D, \O}(0) = q_{_{L,D}}(\O) \in \phi_L \quad &\text{ and } \quad \overline{\rule{0pt}{3.5mm}L(\phi_L)\cap D} \cap \O \subset \Gamma^{^+}_{D, \O}\spc.
\end{align*}

\vspace*{0.2cm}
Now, using the Homma-Schoenflies theorem, we can assume up to conjugacy that:  
\begin{itemize}[leftmargin=1.3cm]
    \item[$\sbullet$] The leaves \( \phi_R \) and \( \phi_L \) are vertical, with \( \phi_R = \{0\} \times \mathbb{R} \) and \( \phi_L = \{1\} \times \mathbb{R} \).
    \item[$\sbullet$] The square \( [0,1] \times [0,1] \) is foliated vertically by \( \mathcal{F} \).
    \item[$\sbullet$] The endpoints \( \{q_{_{R,D}}(\O)\}_{\O \in \OO(D)} \) lie on side of the square \( \{0\} \times [0,1] \).
    \item[$\sbullet$] The endpoints \( \{q_{_{L,D}}(\O)\}_{\O \in \OO(D)} \) lie on the side of the square \( \{1\} \times [0,1] \).
\end{itemize}  

\vspace*{-0.1cm}

\begin{figure}[h!]
    \center
    \mycomment{0.3cm}\begin{overpic}[width=4.5cm, height=3.2cm, tics=10]{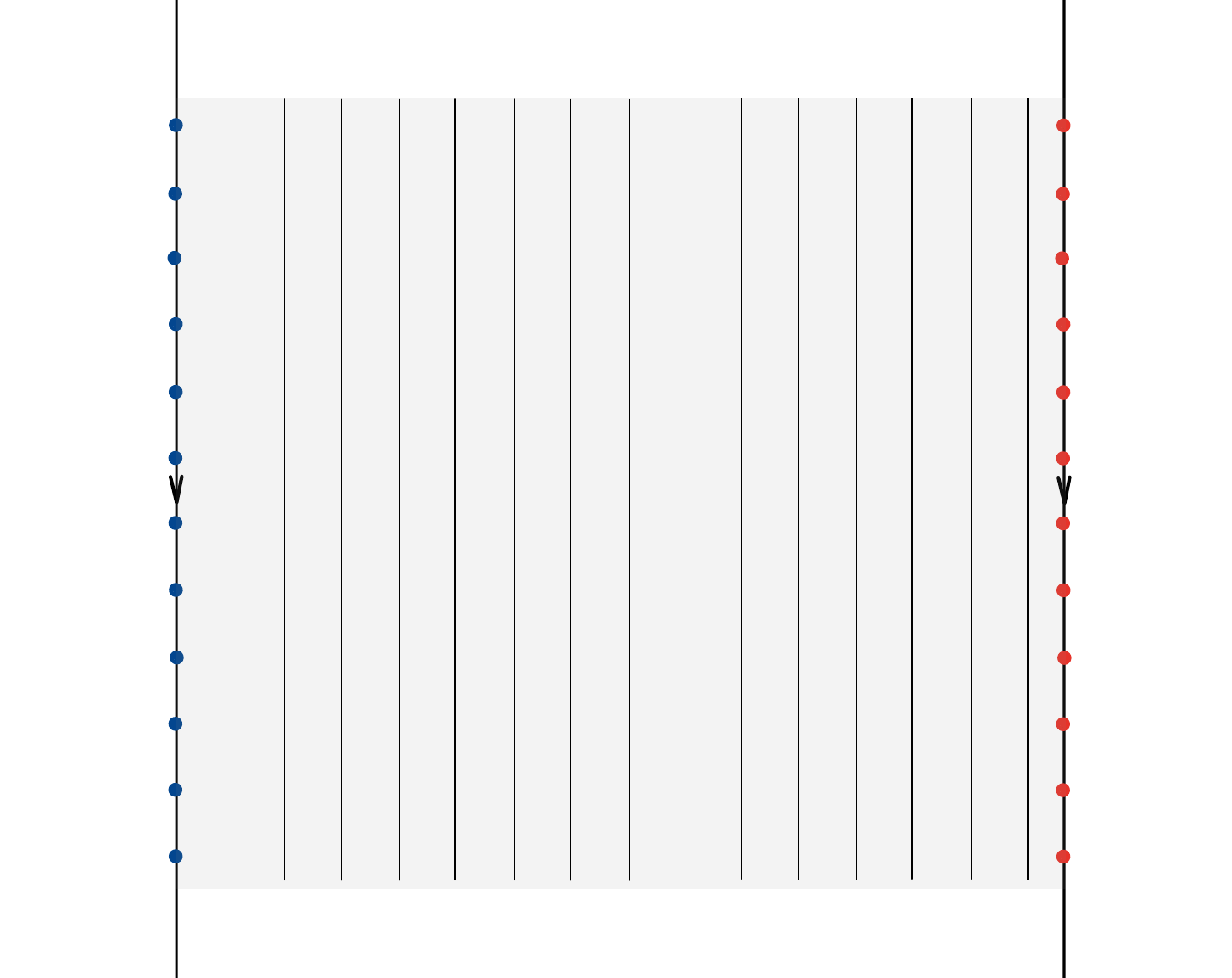}
        \put (17.8,71) {\color{black}\normalsize$\displaystyle \phi_R $}
        \put (76,71) {\color{black}\normalsize$\displaystyle \phi_L$}
         \put (91.5,34) {\color{myRED}\large$\displaystyle \{q_{_{L,D}}(\O)\}_{\O \in \OO(D)}$}
        \put (-54,34) {\color{myBLUE}\large$\displaystyle \{q_{_{R,D}}(\O)\}_{\O \in \OO(D)}$}
        
\end{overpic}
\end{figure}

\vspace*{-0.3cm}
\noindent For each \( \O \in \OO(D) \), let \( \gamma_{_{D,\O}}:[0,1] \longrightarrow \mathbb{R}^2 \) be the line segment joining \( q_{_{R,D}}(\O) \) to \( q_{_{L,D}}(\O) \). 
Observe that, from the definition of the standard order $\leq_{\phi_R}$ and the adaptive order $\leq_D$,  the segments in the family \( \{\gamma_{_{D,\O}}\}_{\O \in \OO(D)} \) satisfy the following properties:
\begin{itemize}[leftmargin=0.5cm]
    \item \( \#(\sspc\gamma_{_{D,\O}} \cap \sspc \gamma_{_{D,\O^\pp}}) = 1 \) if and only if the orbits \( \O \) and \( \O^{\sspc \prime} \) have a weak $\F$-transverse intersection and there exists no maximal leaf domain \( D' \in \D_{\OO} \) such that \( D \leadsto D' \) and \( \O, \O^{\sspc \prime} \in \OO(D') \).
    \item \( \#(\sspc\gamma_{_{D,\O}} \cap \gamma_{_{D,\O^\pp}}) = 0 \) otherwise.
\end{itemize}  

\vspace*{0.25cm}
\noindent At last, for each orbit \( \O \in \OO(D) \), we define the concatenated path 
\mycomment{-0.1cm}
$$\Gamma_{D,\O} := \Gamma^{^-}_{D, \O} \cup \gamma_{_{D,\O}} \cup \Gamma^{^+}_{D, \O}.$$

\mycomment{-0.1cm}
To conclude the proof, we repeat the construction for each maximal leaf domain \( D \in \D_{\OO} \). Then, we recall that for any orbit \( \O \in \OO(D) \), we have a sequence of maximal leaf domains
\mycomment{-0.1cm}
$$ D_0\leadsto D_1 \leadsto \ldots \leadsto D_{m_\O} $$ 

\mycomment{-0.2cm}
\noindent such that $\O$ belongs to the sets $\OO_\alpha(D_0)$ and $\OO_\omega(D_{m_\O})$. And this allows us to construct a proper transverse trajectory of the orbit \( \O \) as the concatenation of the paths 
$$ \Gamma_\O := \prod_{i=0}^{m_\O} \Gamma_{D_i,\O}. $$ 
 The resulting family of proper transverse trajectories \( \{\Gamma_\O\}_{\O \in \OO} \) proves the theorem.
\end{proof}

\setstretch{1.5}	

    \bibliography{biblio1.bib}

\end{document}